\documentclass[reqno,12pt]{article}
\usepackage{a4wide,eucal,enumerate,mathrsfs,xspace}
\usepackage{amsmath,amssymb,epsfig,bbm}
\usepackage[toc,page]{appendix}
\usepackage[usenames]{color}
\usepackage{ifthen}
\numberwithin{equation}{section}

\newtheorem{theorem}{Theorem}[section]

\newtheorem{corollary}[theorem]{Corollary}
\newtheorem{lemma}[theorem]{Lemma}
\newtheorem{proposition}[theorem]{Proposition}
\newtheorem{definition}[theorem]{Definition}

\newtheorem{example}[theorem]{Example}

\newtheorem{remark}[theorem]{Remark}

\usepackage[normalem]{ulem}
\usepackage{pdfsync}
\usepackage{hyperref}

\newcommand{\C}{\mathbb{C}}
\newcommand{\D}{\mathbb{D}}

\renewcommand{\H}{\mathbb{H}}
\renewcommand{\L}{\mathbb{L}}
\newcommand{\M}{\mathbb{M}}
\newcommand{\N}{\mathbb{N}}

\newcommand{\Q}{\mathbb{Q}}
\newcommand{\R}{\mathbb{R}}

\newcommand{\V}{\mathbb{V}}

\newcommand{\LL}{\mathscr{L}}

\newcommand{\PP}{\mathscr{P}}

\newcommand{\cA}{{\ensuremath{\mathcal A}}}
\newcommand{\cB}{{\ensuremath{\mathcal B}}}
\newcommand{\calC}{{\ensuremath{\mathcal C}}}

\newcommand{\cE}{{\ensuremath{\mathcal E}}}

\newcommand{\cH}{{\ensuremath{\mathcal H}}}

\newcommand{\cL}{{\ensuremath{\mathcal L}}}

\newcommand{\cQ}{{\ensuremath{\mathcal Q}}}

\newcommand{\cS}{{\ensuremath{\mathcal S}}}

\newcommand{\cU}{{\ensuremath{\mathcal U}}}

\newcommand{\ii}{{\mbox{\boldmath$i$}}}

\newcommand{\mm}{{\mbox{\boldmath$m$}}}
\newcommand{\nn}{{\mbox{\boldmath$n$}}}

\newcommand{\rr}{{\mbox{\boldmath$r$}}}

\newcommand{\mmu}{{\mbox{\boldmath$\mu$}}}
\newcommand{\nnu}{{\mbox{\boldmath$\nu$}}}
\newcommand{\ppi}{{\mbox{\boldmath$\pi$}}}

\newcommand{\ssigma}{{\mbox{\boldmath$\sigma$}}}

\newcommand{\sfa}{{\sf a}}

\newcommand{\sfd}{{\sf d}}

\newcommand{\sfg}{{\sf g}}

\newcommand{\sfv}{{\sf v}}

\newcommand{\sfF}{{\sf F}}
\newcommand{\sfG}{{\sf G}}
\newcommand{\sfH}{{\sf H}}

\newcommand{\sfJ}{{\sf J}}

\newcommand{\sfP}{{\sf P}}
\newcommand{\sfQ}{{\sf Q}}

\newcommand{\sfS}{{\sf S}}

\newcommand{\sfV}{{\sf V}}

\newcommand{\frg}{{\frak g}}

\newcommand{\rmc}{{\mathrm c}}

\newcommand{\rme}{{\mathrm e}}

\newcommand{\rmv}{{\mathrm v}}

\newcommand{\rmC}{{\mathrm C}}
\newcommand{\rmD}{{\mathrm D}}

\newcommand{\rmI}{{\mathrm I}}

\newcommand{\Kliminf}{K\kern-3pt-\kern-2pt\mathop{\rm lim\,inf}\limits}  
\newcommand{\supp}{\mathop{\rm supp}\nolimits}   

\newcommand{\diam}{\mathop{\rm diam}\nolimits}   
\newcommand{\Lip}{\mathop{\rm Lip}\nolimits}          
\renewcommand{\d}{{\mathrm d}}
\newcommand{\dt}{{\d t}}

\newcommand{\restr}[1]{\lower3pt\hbox{$|_{#1}$}}
\newcommand{\topref}[2]{\stackrel{\eqref{#1}}#2}
\newcommand{\Leb}[1]{{\mathscr L}^{#1}}      
\newcommand{\la}{{\langle}}                  
\newcommand{\ra}{{\rangle}}
\newcommand{\down}{\downarrow}              
\newcommand{\up}{\uparrow}
\newcommand{\weakto}{\rightharpoonup}

\newcommand{\eps}{\varepsilon}  
\newcommand{\nchi}{{\raise.3ex\hbox{$\chi$}}}

\newcommand{\media}{\mkern12mu\hbox{\vrule height4pt           %
          depth-3.2pt                                 
          width5pt}\mkern-16.5mu\int\nolimits}        

\newcommand{\forevery}{\text{for every }}

\def\qed{\ifmmode 
  \else \leavevmode\unskip\penalty9999 \hbox{}\nobreak\hfill
  \fi               
    \qquad           \hbox{\hskip.5em $\square$
                \hskip.1em}}
\def\endproofsym{\qed}
\def\endnobox{\def\endproofsym{}}

\newenvironment{proof}[1][Proof]{\def\endproofsym{\qed}\trivlist\item[\hskip\labelsep{%
\noindent{\normalfont\emph{#1}.}\hskip .321429\parindent}]\ignorespaces}
{\endproofsym\endtrivlist}

\newcommand{\nc}{\normalcolor}

\newcommand{\AAA}{\color{black}}
\newcommand{\GGG}{\color{black}}
\newcommand{\GGGG}{\color{black}}
\newcommand{\EEE}{\normalcolor\normalsize} 

\renewcommand{\mm}{\mathfrak m} 

\newcommand{\Gbil}[2]{\Gamma\big(#1,#2\big)}   
\newcommand{\Gq}[1]{\Gamma\big(#1\big)}          

\newcommand{\Probabilities}[1]{\mathscr P(#1)}          
\newcommand{\Probabilitiesac}[2]{\mathscr P^{ac}(#1,#2)}          
\newcommand{\ProbabilitiesTwo}[1]{\mathscr P_2(#1)}     

\newcommand{\Probabilitiesp}[1]{\mathscr P_p(#1)}     

\newcommand{\entv}{\mathrm{Ent}_{\mm}}                    

\newcommand{\RCD}[2]{\ensuremath{\mathrm{RCD}(#1,#2)}}

\newcommand{\CD}[2]{\ensuremath{\mathrm{CD}(#1,#2)}}

\newcommand{\CDS}[2]{\ensuremath{\mathrm{CD}^*(#1,#2)}}
\newcommand{\RCDS}[2]{\ensuremath{\mathrm{RCD}^*(#1,#2)}}
\newcommand{\BE}[2]{\mathrm{BE}(#1,#2)}   
\newcommand{\MC}[1]{\mathrm{DC}(#1)} 
\newcommand{\MCreg}[1]{\mathrm{DC}_{reg}(#1)} 

\newcommand{\frQ}{\mathfrak Q}
\newcommand{\weight}[2]{\frQ(#1,#2)}
\newcommand{\wname}{\frQ}
\newcommand{\fnd}[3]{\sfg_{\ifthenelse{\equal{#2}{}}{}{(#2)}}(#3,#1)}
\newcommand{\Reny}[2]{\mathcal U_{#1}(#2)} 
 \newcommand{\Mod}[4]{\cA_{#2}^{(#1)}(#3;#4)} 
\newcommand{\EMod}[3]{\cA_{\frak #1}(#2;#3)} 
\newcommand{\Action}[3]{\mathcal A_{#1}(#2;#3)} 

\newcommand{\Vhom}[1]{\V_{\kern-2pt
\ifthenelse{\equal{#1}{1}}{\cE}{#1}}}
\newcommand{\Vdual}[1]{{\V}_{\kern-2pt
\ifthenelse{\equal{#1}{1}}{\cE}{\rm #1}}'}
\newcommand{\Ddual}{\D_{\kern-1pt \cE}'}
\newcommand{\tGq}[2]{\Gamma_{\kern-2pt#1}(#2)}  
\newcommand{\tGqs}[2]{\Gamma_{\kern-2pt#1}^*(#2)}  
\newcommand{\tGbil}[3]{\Gamma_{\kern-2pt#1}(#2,#3)}
\newcommand{\class}[2]{#1_{#2}}
\newcommand{\classd}[2]{(\tilde #1)_{#2}}
\newcommand{\Sobolev}[4]{W^{1,2}(0,#1;#2,#3)
\ifthenelse{\equal{#4}{}}{}{\cap
  L^{#4}_{\ifthenelse{\equal{#4}{\infty}}w{}}(0,T;L^{#4}(X,\mm))}}
\newcommand{\GGamma}{\mathbf\Gamma}
\newcommand{\Prim}{V}
\newcommand{\Phim}{\mathcal V}
\newcommand{\cEs}[1]{\cE^*_{#1}}
\renewcommand{\nn}{\mathfrak n}
\newcommand{\ND}[2]{\mathcal{ND}(#1,#2)}
\newcommand{\AC}[3]{\mathrm{AC}^{#1}(#2;#3)}
\renewcommand{\C}{\mathsf{Ch}}
\newcommand{\DeltaE}{\mathrm L}
\newcommand{\duality}[4]{{\vphantom\la}_{#1}\la#3,#4\ra_{#2}}

\newcommand{\fn}{\normalcolor}

\newcommand{\rh}{\rho}

\newcommand{\vare}{\varepsilon}

\newcommand{\vph}{\varphi}

\newcommand{\factor}[3]{\sigma_{#1}^{(#2)}(#3)}

\def\XXint#1#2#3{{\setbox0=\hbox{$#1{#2#3}{\int}$}
\vcenter{\hbox{$#2#3$}}\kern-.5\wd0}}

\title{Nonlinear diffusion equations and curvature\\ conditions in 
  metric measure spaces
}
\begin{document}

\author{Luigi Ambrosio\
   \thanks{Scuola Normale Superiore, Pisa. email: \textsf{luigi.ambrosio@sns.it}.}
   \and
   Andrea Mondino
   \thanks{Universit\"at Z\"urich. email: \textsf{andrea.mondino@math.uzh.ch}.}
 \and
   Giuseppe Savar\'e\
   \thanks{Universit\`a di Pavia. email:
   \textsf{giuseppe.savare@unipv.it}. Partially supported by
   PRIN10/11 grant from MIUR for the project \emph{Calculus of Variations}.}
   }

\maketitle

\begin{abstract} 
Aim of this paper is to provide new characterizations of the curvature dimension condition in the context of metric measure
spaces $(X,\sfd,\mm)$. On the geometric side, our new approach takes into
account suitable weighted action functionals which provide the natural modulus of $K$-convexity when
one investigates the convexity properties of $N$-dimensional entropies. On the side of diffusion semigroups and evolution
variational inequalities, our new approach uses the nonlinear diffusion semigroup induced by the $N$-dimensional entropy, in place of the heat flow. Under suitable assumptions (most notably the quadraticity of Cheeger's energy relative to the metric measure structure) 
both approaches are shown to be equivalent to the strong $\CDS K N$ condition of Bacher-Sturm. 
\end{abstract}

\tableofcontents


\section{Introduction}  

Spaces with Ricci curvature bounded from
below play an important role in many probabilistic and analytic
investigations, that reveal various deep connections between different
fields.

Starting from the celebrated paper by \textsc{Bakry-\'Emery} \cite{Bakry-Emery84},
the curvature-dimension condition based on $\Gamma$-calculus and
the $\Gamma_2$-criterium in Dirichlet spaces 
provides crucial tools for proving 
refined estimates on Markov semigroups and many
functional inequalities, of Poincar\'e, Log-Sobolev, Talagrand,
and concentration type
(see, e.g.~\cite{Ledoux01,Ledoux04,Ledoux11,Bakry06,Ane-et-al00,Bakry-Gentil-Ledoux14}).

In the framework of optimal transport, the importance of curvature bounds  
has been deeply analyzed in \cite{Otto-Villani00,Cordero-McCann-Schmuckenschlager01,Sturm-VonRenesse05}.
These and other important results led \textsc{Sturm}
\cite{Sturm06I,Sturm06II} and \textsc{Lott-Villani} \cite{Lott-Villani09} 
to introduce a new synthetic notion of the curvature-dimension
condition, in the general framework of a metric-measure space 
$(X,\sfd,\mm)$. 

In recent years more than one paper has been devoted to the investigation of the relation between the differential and metric structures, 
particularly in connection with Dirichlet forms, see for instance \cite{Koskela_Zhou}, \cite{Koskela-Shanmugalingam-Zhou12}, 
\cite{Sturm_intrinsic}, \cite{AGS11b} and \cite{AGS12}. In particular, under a suitable
infinitesimally Hilbertian assumption on the metric measure structure (and very mild regularity assumptions), 
thanks to the results of the last 
two papers we know that the optimal transportation point of view provided by the \textsc{Lott}-\textsc{Sturm}-\textsc{Villani} theory coincides 
with the point of view provided by \textsc{Bakry}-\textsc{\'Emery} 
when the inequalities do not involve any upper bound on the
dimension: both the approaches can 
thus be equivalently used to characterize the class of $\RCD K\infty$ spaces
with Riemannian Ricci curvature bounded from below by $K\in \R$.
More precisely, the logarithmic entropy functional
\begin{equation}
  \label{eq:6}
  \mathcal U_\infty(\mu):=\int_X \varrho\log\varrho\,\d\mm
  \quad\text{if }\mu=\varrho\mm\ll \mm,
\end{equation}
satisfies 
the
$K$-convexity inequality along geodesics  $(\mu_s)_{s\in [0,1]}$ 
induced by the transport distance $W_2$ (i.e. with
cost equal to the square of the distance) 
\begin{equation}
  \label{eq:24}
\mathcal U_\infty(\mu_s)\le (1-s)\mathcal U_\infty(\mu_0)+s\mathcal
  U_\infty(\mu_1)
  -\frac K2 s(1-s)W_2^2(\mu_0,\mu_1)
\end{equation}
if and only if $\Gamma_2(f)\geq K\,\Gamma(f)$. 

A natural and relevant  question is then to establish a similar equivalence when
upper bounds on the dimension are imposed; more precisely one is interested in the equivalence between the condition
\begin{equation}\label{eq:intro0}
\Gamma_2(f)\geq K\,\Gamma(f)+\frac 1 N(\DeltaE f)^2
\end{equation}
(where $\DeltaE$ is the infinitesimal generator of the semigroup associated to the Dirichlet form)
and the curvature-dimension conditions based on optimal transport.  In the dimensional case, the logarithmic entropy functional \eqref{eq:6} is replaced by  the ``$N$-dimensional'' R\'eny entropy  
\begin{equation}
  \label{eq:59}
  \mathcal U_N(\mu):=\int_X U_N(\varrho)\,\d\mm=N-N \int_X \varrho^{1-\frac{1}{N}} \, \d\mm  \qquad\text{if }\mu=\varrho\mm+\mu^\perp, \quad
  \mu^\perp\perp\mm.
\end{equation}
Except for the case $K=0$, which can be formulated by means of a geodesic convexity condition
analogous to \eqref{eq:24}, the case $K\neq0$ involves a much more
complicated property \cite{Sturm06II,Bacher-Sturm10}, that gives raise
to difficult technical questions. 

Aim of this paper is precisely to provide new characterizations of the curvature dimension condition in the context of metric measure
spaces $(X,\sfd,\mm)$. On the geometric side, our new approach takes into
account suitable weighted action functionals of the form
\begin{equation}
\Mod tN\mu\mm=\int_0^1
\int_X\sfg(s,t)\varrho^{1-1/N}(x,s)\bar v^2(x,s)\,\d\mm\,\d
s,\label{eq:60}
\end{equation}
where $\mu_s=\varrho_s\mm$, $s\in [0,1]$, is a Wasserstein geodesic,
$\sfg$ is a weight function and $\bar v$ is the minimal velocity
density of $\mu$, a new concept that extends to general metric spaces
the notion of Wasserstein velocity vector field 
developed for Euclidean spaces 
\cite[Chap.~8]{AGS08}. Functionals like \eqref{eq:60}  
provide the natural modulus of $K$-convexity when
one investigates the convexity properties of 
the $N$-dimensional R\'eny entropy \eqref{eq:59}. 
 On the side of diffusion semigroups and evolution
variational inequalities, our new approach uses the nonlinear diffusion semigroup induced by the $N$-dimensional entropy, in place of the heat  flow.
Under suitable assumptions (most notably the quadraticity of Cheeger's energy relative to the metric measure structure) 
both approaches are shown to be equivalent to the strong $\CDS K N$
condition of \textsc{Bacher-Sturm}
\cite{Bacher-Sturm10}.

Apart from the stated equivalence between the \textsc{Lott-Sturm-Villani} and the
\textsc{Bakry-\'Emery} approaches, our results and techniques 
can hardly be compared with the recent work \cite{Erbar-Kuwada-Sturm}
of \textsc{Erbar}-\textsc{Kuwada}-\textsc{Sturm}, motivated by the
same questions. Instead of the R\'eny entropies \eqref{eq:59}, in their approach 
an $N$-dependent modification of the logarithmic entropy \eqref{eq:6} is considered, namely
the logarithmic entropy power 
\begin{equation}\label{eq:defSnSturm}
\cS_N(\mu):=\exp\left(-\frac 1 N \mathcal U_\infty 
(\mu)\right),
\end{equation}
and convexity inequalities as well as evolution variational inequalities are stated in terms of $\cS_N$, proving equivalence
with the strong $\CDS K N$ condition.
A conceptual and technical advantage of their approach is the use of essentially the same objects (logarithmic entropy, heat flow)
of the adimensional theory. On the other hand, since power-like nonlinearities appear in a natural way ``inside the integral''
in the optimal transport approach to the curvature dimension theory, we believe it is interesting to pursue a different line of thought,
using the Wasserstein gradient flow 
induced by the R\'eny entropies (in the same spirit of the seminal \textsc{Otto}'s paper \cite{Otto01}
on convergence to equilibrium for porous medium equations). 
The only point in common of the two papers is that both provide the equivalence between the differential
curvature-dimension condition \eqref{eq:intro0} and 
the so-called strong $\CDS K N$ condition; however, this equivalence is estabilished passing through convexity and
differential properties which are quite different in the two approaches (for instance some of them do not
involve at all the distorsion coefficients) and have, we believe, an independent interest.

Our paper starts with Section~\ref{subsec:heu}, where we illustrate in the simple framework of a $d$-dimensional
Euclidean space the basic heuristic arguments providing the links between contractivity and convexity. It builds upon
the fundamental papers \cite{Otto-Westdickenberg05} and \cite{Daneri-Savare08}. The main new ingredient here is that the links are provided in terms of monotonicity
of the Hamiltonian, instead of monotonicity of the Lagrangian (see \cite{Liero-Mielke13} for a related discussion of the role of dual
Hamiltonian estimates in terms of the so-called Onsager operator). 
More precisely, if $\sfS_t:\R^d\to\R^d$ is the flow generated by a smooth vector field
$\frak f:\R^d\to \R^d$, and if $\calC(x,y)$ is the cost functional relative to a Lagrangian $\cL$, then we know that the contractivity property
$$\calC(\sfS_tx,\sfS_t y)\leq\calC(x,y) \qquad \text{for all } t\geq 0,$$
is equivalent to the action monotonicity 
\begin{equation}\label{eq:IntroActionMon}
\frac{\d}{\d t} \cL(x(t),w(t)) \leq 0
\end{equation}
 whenever $x(t)$ solves the ODE
$$ \frac{\d}{\d t} x(t) = \frak f (x(t))$$
and $w$ solves the linearized ODE 
$$ \frac{\d}{\d t} w(t) =\rmD\frak f(x(t)) w(t)$$
 (in the applications $w$ arises as the derivative w.r.t. $s$ of a smooth
curve of initial data for the ODE). In Section~\ref{subsec:heu} we use duality arguments to prove the same equivalence when the action monotonicity \eqref{eq:IntroActionMon}
is replaced by the Hamiltonian monotonicity 
\begin{equation}\label{eq:IntroHamMon}
 \frac{\d}{\d t} \cH(x(t),\varphi(t)) \geq 0,
\end{equation} where now $\varphi$ solves the \textit{backward} transposed 
equation 
\begin{equation}\label{eq:IntroBackODE}
 \frac{\d}{\d t} \varphi(t)=-\mathrm D\frak f(x(t))^\intercal\varphi(t),
\end{equation} see Proposition~\ref{prop:easy-back}. Lemma~\ref{le:2ndeasyarg} provides, in the case 
when $\frak f=-\nabla U$ and $\cL$, $\cH$ are quadratic forms, the link between the Hamiltonian monotonicity and another contractivity
property involving both $\calC$ and $U$, see \eqref{eq:199}; this is known to be equivalent to the convexity of $U$ along the geodesics
induced by $\calC$.

In the context of optimal transportation (say on a smooth, compact
Riemannian manifold $(M,\frg)$), the role of the Hamiltonian is played by
$\cH(\varrho,\varphi):=\frac12\int_X |\rmD\varphi|_\frg^2\varrho\,\d\mm$, thanks to \textsc{Benamou-Brenier} formula and 
the \textsc{Otto} formalism:
\begin{equation}
\label{eq:intro1}
    \cL(\varrho,w):=\frac 12 \int_X
    |\rmD \varphi|^2_\frg\varrho\,\d\mm,\qquad
    -\mathrm{div}_\frg (\varrho \nabla\kern-2pt_\frg\varphi)=w.
\end{equation}
  In other words, the cotangent bundle is associated to the velocity gradient $\nabla\kern-2pt_\frg\varphi$ and the duality between
  tangent and cotangent bundle is provided by the possibly degenerate elliptic PDE $-\mathrm{div}_\frg (\varrho \nabla\kern-2pt_\frg \varphi)=w$.
 With a very short computation we show in Example~\ref{ex:1} how the \textsc{Bakry-\'Emery} $\BE 0\infty$ condition 
  corresponds precisely to the Hamiltonian monotonicity, when the vector field is (up to the sign) the gradient vector of the logarithmic 
  entropy functional. If the entropy $\cU(\varrho\,\mm)=\int U(\varrho)\,\d x$ satisfies the (stronger) \textsc{McCann}'s $\MC N$ condition, then the 
  same correspondence holds with $\BE 0 N$, see Example~\ref{ex:2}. In both cases the flow corresponds to the diffusion equation
  \begin{equation}\label{eq:intro2}
    \frac \d\dt \varrho=\Delta_\frg P(\varrho)
  \end{equation}
  with $P(\varrho):=\varrho U'(\varrho)-U(\varrho)$, which is linear
  only in the case of the logarithmic entropy \eqref{eq:6}.
  
  The computations made in Examples~\ref{ex:1} and Example~\ref{ex:2} involve regularity in time and space of the potentials $\varphi$
  in \eqref{eq:intro1}, whose proof is not straightforward already in the smooth Riemannian context. Another difficulty arises from the 
  degeneracy of the PDE $-\mathrm{div}_\frg (\varrho \nabla\kern-2pt_\frg \varphi)=w$, which forces us to consider weak solutions $\varphi$
  in ``weighted Sobolev spaces''.  
  Keeping in mind these technical difficulties, our goal is then to provide tools to extend the calculations of these examples to a nonsmooth
  context, following on the one hand the $\Gamma$-calculus formalism, on the other hand the calculus in metric measure spaces $(X,\sfd,\mm)$ developed in
  \cite{AGS11a}, \cite{AGS11b}, \cite{Gigli12} and in the subsequent papers. 

Now we pass to a more detailed description of the three main parts of the paper.

\bigskip
\centerline{\bf Part I}
\smallskip

This first part, which consists of Section~3 and Section~4, 
is written in the context of a Dirichlet form $\cE$ on $L^2(X,\mm)$, for some measurable space $(X,\cB)$ endowed 
with a $\sigma$-finite measure $\mm$. We adopt the notation $\H$ for $\L^2(X,\mm)$, $\V$ for the domain of the Dirichlet form, $-\DeltaE$ for the linear
monotone map from $\V$ to $\V'$ induced by $\cE$, $\sfP_t$ for the semigroup whose infinitesimal generator is $\DeltaE$. 

We already mentioned
the difficulties related to the degeneracy of our PDE; in addition, since we don't want to assume a spectral gap, we need also to take
into account the possibility that the kernel $\{f:\ \cE(f,f)=0\}$ of the Dirichlet form is not trivial. We then consider the abstract completion $\V_\cE$ of the
quotient space of $\V$ and the realization $\V_\cE'$ of the dual of $\V_\cE$ as the finiteness domain of the quadratic form $\cE^*:\V'\to [0,\infty]$
defined by
$$
\frac 12 \cE^*(\ell,\ell):=\sup_{f\in\V}\langle\ell,f\rangle-\frac 12\cE(f,f).
$$
Section~\ref{subsec:completion} is indeed devoted to basic functional analytic properties relative to the completion of quotient spaces w.r.t. a seminorm 
(duality, realization of the dual, extensions of the action of $\DeltaE$). The spaces $\V$, $\V_\cE$ and their duals are the basic
ingredients for the analysis, in Section~\ref{subsec:A-1}, of the
nonlinear diffusion equation 
\begin{equation}\label{eq:64}
\frac{\d}{\d t}\varrho-\DeltaE P(\varrho)=0
\end{equation}
(which
corresponds to \eqref{eq:intro2}) in the abstract context, for regular monotone nonlinearities $P$; the basic existence and uniqueness
result is given in  Theorem~\ref{thm:nonlin-diff}, which provides also the natural apriori estimates and contractivity properties.

Chapter 4 is devoted to the linearizations of the diffusion
equation \eqref{eq:64}. We first consider in
Theorem~\ref{prop:backward-linearization} the (backward) PDE 
$$
\frac {\d} {\d t}\varphi+P'(\varrho)\DeltaE\varphi =\psi
$$
which is the adjoint to the linearized equation and corresponds, when $\psi=0$, to the backward transposed ODE \eqref{eq:IntroBackODE} of the heuristic Section~\ref{subsec:heu}.
Existence, uniqueness and stability for this equation is provided in the class $W^{1,2}(0,T;\D,\H)$ of $L^2(0,T;\D)$ maps with derivative in $L^2(0,T;\H)$, where
$\D$ is the space of all $f\in\V$ such that $\DeltaE f\in\H$, endowed with the natural norm.\\
In Theorem~\ref{thm:reg2} we consider the linearized PDE 
\begin{equation}
\frac{\d}{\d t} w=\DeltaE (P'(\varrho)w);\label{eq:72}
\end{equation}
since \eqref{eq:72} is in ``divergence form'' we can use the regularity of $P'(\varrho)$ to provide existence and uniqueness (as well as stability) in the 
large class $W^{1,2}(0,T;\H,\D_\cE')$ of $L^2(0,T;\H)$ maps with derivative in $L^2(0,T;\D_\cE')$. Here
$\D_\cE'$ is the space of all $\ell\in\D'$ such that, for some constant $C$, 
$|\langle\ell,f\rangle|\leq C\|\DeltaE f\|_\H$ for all $f\in\D$ (endowed with the natural norm
provided by the minimal constant $C$). In Theorem~\ref{thm:precise-limit} we prove that the PDE is indeed the linearization of 
\eqref{eq:64}
by considering suitable families of initial conditions and their derivative.

\bigskip
\centerline{\bf Part II}
\smallskip

This part is devoted to the metric side of the theory and builds upon the papers \cite{AGS11a}, \cite{Lisini07}, \cite{AGS11b}, \cite{Daneri-Savare08}
with some new developments that we now illustrate.

Chapter~5 is mostly devoted to the introduction of preliminary and by now well estabilished concepts in metric spaces $(X,\sfd)$, as absolutely continuous 
curves $\gamma_t$, metric derivative $|\dot\gamma_t|$, $p$-action $\cA_p(\gamma)=\int|\dot\gamma|^p\,\d t$, slope $|\rmD f|$ and its one-sided 
counterparts $|\rmD^\pm f|$. In Section~\ref{sec:5.2} we
recall the metric/differential properties of the map
$$
Q_tf(x):=\inf_{y\in X}f(y)+\frac 1{2t}\sfd^2(x,y)\qquad x\in X,
$$
given by the Hopf-Lax formula (which provides a semigroup if $(X,\sfd)$ is a length space). Section~\ref{sec:5.3} and Section~\ref{sec:5.4} cover basic material
on couplings, $p$-th Wasserstein distance $W_p$, absolutely continuous curves w.r.t. $W_p$ and dynamic plans. Particularly important for us  is the 
1-1 correspondence between absolutely continuous curves $\mu_t$ in $(\Probabilities{X},W_p)$ and time marginals probability measures $\ppi$ in
$\rmC([0,1];X)$ with finite $p$-action $\mathscr A_p(\ppi):=\int \cA_p(\gamma)\,\d\ppi(\gamma)$, provided in \cite{Lisini07}. In general only
the inequality $|\dot\mu_t|^p\leq  \int |\dot\gamma_t|^p\,\d\ppi(\gamma)$ holds, and \cite{Lisini07} provides existence of a distinguished plan
$\ppi$ for which equality holds, that we call $p$-tightened to $\mu_t$. \\
Section~\ref{subsec:Cheeger} introduces a key ingredient of the metric theory, the Cheeger energy that we shall denote by $\C$ and the relaxed slope
$|\rmD f|_w$, so that $\C(f)=\frac 12\int_X|\rmD f|_w^2\,\d\mm$.
The energy $\C$ is by construction lower semicontinuous in $L^2(X,\mm)$; furthermore, under an additional quadraticity assumption it has been shown in 
\cite{AGS11b,Gigli12} 
that $\C$ provides a strongly  local Dirichlet form, whose Carr\'e du Champ is given by
$$
\Gamma(f,g)=\lim_{\eps\downarrow 0}\frac{|\rmD (f+\epsilon g)|_w^2-|\rmD f|_w^2}{2\epsilon}.
$$  
Motivated by the necessity to solve the PDE $-\mathrm{div}_\frg (\varrho \nabla\kern-2pt_\frg \varphi)=\ell$, whose abstract counterpart is
\begin{equation}\label{eq:intro4}
\int_X\varrho \Gamma(\varphi,f)\,\d\mm=\langle \ell,f\rangle\qquad\forall f\in\V,
\end{equation}
in Section~\ref{subsec:weighted} we consider natural weighted spaces $\V_\rho$ arising from the completion of the seminorm
$\sqrt{\int_X\varrho\Gamma(f)\,\d\mm}$, and the extensions of $\Gamma$ to these spaces, denoted by $\Gamma_\varrho$. In connection
with these spaces we investigate several stability properties which play a technical role in our proofs.\\

Section~\ref{sec:abscurWas} provides a characterization of $p$-absolutely continuous curves $\mu_s:[0,1]\to\Probabilities{X}$
in terms of the following control on the increments (where $|\rmD^*\varphi|$ is the usc relaxation of the slope $|\rmD\varphi|$):
$$
\biggl|\int_X\varphi\,\d\mu_s-\int_X\varphi\,\d\mu_t\biggr|\leq
\int_s^t\int_X|\rmD^*\varphi| v \, \d\mu_r\,\d r\qquad \varphi\in\Lip_b(X),\,\,0\leq s\leq t\leq 1.
$$
Any function $v$ in $L^p(X\times (0,1),\mm\otimes\Leb{1})$ will be called $p$-velocity density. In Theorem~\ref{thm:MVD} we show
that for all $p\in (1,\infty)$ a $p$-velocity density exists if and only if $\mu_t\in\AC{p}{[0,1]}{\Probabilities{X}}$ (see also \cite{Gigli-BangXian} for closely
related results). In addition we identify a crucial relation between
the unique $p$-velocity density $\bar v$ with minimal $L^p$ norm 
and any plan $\ppi$ $p$-tightened to $\mu$, namely
\begin{equation}\label{eq:intro3}
\bar v(\gamma_t,t)=|\dot\gamma_t|\quad\text{for $\ppi$-a.e. $\gamma$, for $\Leb{1}$-a.e. $t\in (0,1)$.}
\end{equation}
Heuristically, this means that even though branching cannot be ruled out, the metric velocity of the curve $\gamma$ in the support of $\ppi$
depends only on time and position of the curve, and it is independent of $\ppi$. 

In Section~\ref{sec:weights} we use the minimal velocity density $\bar v$ to define, under the additional assumption $\mu_s=\varrho_s\mm$,
the weighted energy functionals
\begin{equation}\label{eq:intro6}
  \begin{aligned}
    \Action{\wname}{\mu}\mm:= &
    \int_0^1\int_X
    \weight s{\varrho_s}
    \,\bar v^p\varrho_s\,\d\mm\,ds,
  \end{aligned}
\end{equation}
where $\weight{s}{r}: [0,1]\times [0,\infty)\to [0,\infty]$ is a suitable weight function (the typical choice will be
$\weight{s}{r}=\omega(s)Q(r)$ with $Q(r)=rP'(r)-P(r)$). 
Notice that when $\frQ\equiv 1$ we have the usual action $\int_0^1\int_X\bar v^p\,\d\mu_s\,\d s=\cA_p(\mu)$, 
which makes sense even for curves not made of absolutely continuous measures.
If $\ppi$ is a dynamic plan $p$-tightened to $\mu$ (recall that this means $\mathscr A_p(\ppi)=\cA_p(\mu)$),
we can use \eqref{eq:intro3} to obtain an equivalent expression in terms of $\ppi$:
$$
    \begin{aligned}
      \Action
      {\wname}\mu\mm &= 
      \int_0^1\int_X \weight {s}{\varrho_s(\gamma_s)}
      |\dot\gamma_s|^p\,\d\ppi(\gamma)\,\d s.
    \end{aligned}
$$
In Theorem~\ref{le:stabilitySigma} we provide, by Young measures techniques, continuity and lower semicontinuity properties of 
$\mu\mapsto\Action {\wname}\mu\mm$ under the assumption that the $p$-actions are convergent. 

In Section~\ref{sec:dynamic} we restrict ourselves to the case when $p=2$ and $\C$ is quadratic. For curves $\mu_s=\varrho_s \mm$ having
uniformly bounded densities w.r.t. $\mm$ we show in
Theorem~\ref{thm:dynamicKant} that $(\mu_s)_{s \in [0,1]}$ belongs to
$\AC2{[0,1]}{(\Probabilities X,W_2)}$ if and only if there exists
$\ell\in L^2(0,1;\V')$ satisfying, for all $f\in\V$,
$$
\frac{\d}{\d s}\int_X f\varrho_s\,\d\mm=\ell_s(f)\qquad\text{in
  $\mathscr D'(0,1)$. }  
$$
In addition $\ell_s\in\V_{\varrho_s}'$ for $\Leb{1}$-a.e. $s\in (0,1)$ and they are linked to the minimal velocity $\bar v$ by
$\cE^*(\ell_s,\ell_s)=\int_X|\bar v_s|^2\varrho_s\,\d s$.
Thanks to this result, we can obtain by duality the potentials $\phi_s$ associated to the curve, linked to $\ell_s$ by \eqref{eq:intro4}.

In Section~\ref{sec:9} we enter into the core of the matter, by providing on the one hand a characterization of strong $\CDS K N$ spaces whose
Cheeger energy is quadratic in terms
of convexity inequalities involving weighted action functionals and on the other hand a characterization involving
evolution variational inequalities. These characterizations extend \eqref{eq:24}, known \cite{AGS11b, AGMR12} in
the case $N=\infty$: the logarithmic entropy and the 
Wasserstein distance are now replaced by a nonlinear entropy and weighted action functionals. Section~\ref{subsec:9.1} provides basic 
results on weighted convexity inequalities and the distorsion multiplicative coefficients $ \factor \kappa t\delta$ (see \eqref{eq:170}), 
which appear in the formulation of the $\CDS K N$ condition. Section~\ref{subsec:regularent} introduces the basic entropies and 
their regularizations. In Section~\ref{subsec:9.3} we recall the basic definitions of $\CD K\infty$ space, of strong $\CD K\infty$ space
(involving the $K$-convexity \eqref{eq:24}
of the logarithmic entropy along \textit{all} geodesics) and  Proposition~\ref{prop:scdk} states their main properties, following
\cite{Rajala-Sturm12}. We then pass to the part of the theory involving dimensional bounds, by recalling the Baker-Sturm $\CDS K N$ condition
which involves a convexity inequality along $W_2$-geodesics for the
R\'eny entropies $\mathcal U_M$ defined in \eqref{eq:59}
and the distortion coefficients $ \factor {K/M} t\delta$, see \eqref{eq:177}, for all $M\geq N$.\\
Theorem~\ref{thm:Andrea_complete} is our first main result, providing a characterization of strong $\CDS K N$ spaces in terms of the convexity inequality
\begin{equation}\label{eq:ConvUNK}
      \cU_N(\mu_t)\le 
      (1-t) \cU_N (\mu_0)+ t \,
      \cU_N(\mu_1)-K\Mod t{N}{\mu}\mm\quad
    \forevery t\in [0,1].
\end{equation}
Here $\Mod t{N}{\mu}\mm$ is the $(t,N)$-dependent weighted action functional as in \eqref{eq:intro6} given by the choice
$\frQ^{(t)}(s,r):=\sfg(s,t)r^{-1/N}$, where $\sfg$ is the Green function defined in \eqref{eq:3}, so that
$$
  \Mod tN\mu\mm=\int_0^1
  \int_X\sfg(s,t)\varrho^{-1/N}(x,s)\bar v^2(x,s)\varrho_s\,\d\mm\,\d s.
$$
Comparing with the $\CDS K N$ definition \eqref{eq:177}, we can say
that the  distortion due to the lower bound $K$ on the Ricci tensor appears
just as a multiplicative factor, and that the distortion coefficients $\factor {K/M} t\delta$ 
are now replaced by the $(t,N)$-dependent weighted action functional. Hence,
$K$ and $N$ have more distinct roles, compared to the original definition.  \AAA Let us mention that  a convexity inequality in the same spirit of \eqref{eq:ConvUNK} was also obtained in \cite[Remark 4.18]{Erbar-Kuwada-Sturm}, the main difference being that in the present paper  $\cU_N$ is the Reny entropy functional \eqref{eq:59} while in \cite{Erbar-Kuwada-Sturm} a similar notation is used to denote the logarithmic entropy power  \eqref{eq:defSnSturm}. \fn \\
Our second main result is given in Theorem~\ref{thm:CD-EVI} and Theorem~\ref{thm:EVI-CD}. More precisely, in Theorem~\ref{thm:CD-EVI} we prove 
that in strong $\CDS K N$ spaces whose Cheeger energy is a quadratic form, for any regular entropy $U$ in McCann's class $\MC N$ the
induced functional $\cU$ as in \eqref{eq:59} satisfies the evolution variational inequality
\begin{equation}\label{eq:IntroEVIKN}
\frac 12\frac{\d^+}{\d t}W_2^2({\sf S}_t\varrho\,\mm,\nu)+\cU({\sf S}_t\varrho\,\mm)\leq
\cU(\nu)-K\Action{\omega Q}{\mu_{\cdot,t}}{\mm},
\end{equation}
where ${\sf S}$ is the nonlinear diffusion semigroup studied in Part I,
$\omega(s)=(1-s)$, $Q(r)=P(r)/r=U'(r)-U(r)/r$ and $\{\mu_{s,t}\}_{s\in [0,1]}$ is the unique geodesic connecting
$\mu_t={\sf S}_t\varrho\,\mm$ to $\nu$. The proof of this result follows the lines of \cite{AGS11b} ($N=\infty$, $\mm(X)<\infty$)
and \cite{AGMR12} (where the assumption on the finiteness of $\mm$ was removed) and uses the calculus tools developed
in \cite{AGS11a}, in particular in the proof of \eqref{eq:aug3_4}. In Theorem~\ref{thm:EVI-CD}, independently of the quadraticity
assumption, we adapt the ideas of \cite{Daneri-Savare08} to prove that the evolution variational inequality above (for all
regular entropies $U\in\MC N$) implies
the strong $\CDS K N$ condition. Moreover we can use Lemma~\ref{le:CDSKN-CDKI} to get the $\CD K \infty$ condition and then 
apply the characterization of $\RCD K\infty$ spaces provided in \cite{AGS11b} to obtain that $\C$ is quadratic.
Hence, under the quadraticity assumption on $\C$, the strong $\CDS K N$
condition and the evolution variational inequality are equivalent; without this assumption, as in the case $N=\infty$, the evolution
variational inequality is stronger.

\bigskip
\centerline{\bf Part III}
\smallskip
This last part is really the core  of the work, where all the tools developed in Parts I and II are combined to prove the main results. The natural setting is provided  by a Polish topological space $(X,\tau)$ endowed with
a $\sigma$-finite reference Borel measure $\mm$ and  a strongly local symmetric Dirichlet form $\cE$ in $L^2(X,\mm)$
enjoying a \emph{Carr\'e du Champ} $\Gamma:D(\cE)\times D(\cE)\to L^1(X,\mm)$ and a $\Gamma$-calculus. All the estimates about the  \textsc{Bakry-\'Emery} condition  discussed in Section~\ref{sec:BE}  and  the action estimates for nonlinear diffusion equations provided in Section \ref{sec:NDaction} do not really need an underlying compatible metric structure. In any case, in Section \ref{sec:EquivBERCDS}, they will be applied to the case of the Cheeger energy (thus assumed to be quadratic) of the metric measure space $(X,\sfd,\mm)$ in order to prove the main results of the paper.
Let us now discuss in more detail the content of Part III. 

In Section~\ref{sec:BE}  we recall the basic assumptions related to the \textsc{Bakry-\'Emery} condition and we prove some important properties related to them; in particular, in the case of a locally compact space, we establish  useful local and nonlinear criteria to check this condition.  
More precisely, we introduce the multilinear form 
 $\GGamma_2$ given by
$$  \GGamma_2(f,g;\varphi):= 
  \frac 12\int_X \Big(\Gbil fg\, {\DeltaE \varphi}-
  \Gbil f{\DeltaE g}\varphi-\Gbil g{\DeltaE f}\varphi\Big)\,\d\mm,
$$
with   $(f,g,\varphi)\in D_\V(\DeltaE)\times D_\V(\DeltaE) \times
  D_{L^\infty}(\DeltaE)$, where we set    $\V_\infty:=\V\cap L^\infty(X,\mm)$,
$\D_\infty:=\D\cap L^\infty(X,\mm)$,
\begin{equation}
  \nonumber
  \begin{cases}
  D_{L^p}(\DeltaE):=\big\{f\in \D{}\cap L^p(X,\mm):\ \DeltaE f\in
  L^p(X,\mm)\big\}\qquad p\in [1,\infty],\\
  D_\V(\DeltaE)=\big\{f\in \D:\ \DeltaE f\in \V\big\}.
  \end{cases}
\end{equation}
When $f=g$ we also set
$$ \GGamma_2(f;\varphi):=\GGamma_2(f,f;\varphi)=
  \int_X \Big(\frac 12\Gq f\, {\DeltaE \varphi}-
  \Gbil f{\DeltaE f}\varphi\Big)\,\d\mm. $$
The $\mathbf\Gamma_2$ form provides a weak version 
(see Definition \ref{def:BE} inspired by \cite{Bakry06,Bakry-Ledoux06}) of 
the Bakry-\'Emery $\BE K N$ condition \cite{Bakry-Emery84,Bakry92} 
  \begin{equation}
  \label{eq:Intro9}
  \GGamma_2(f;\varphi)\ge K\int_X \Gq f\,\varphi\,\d\mm+\frac 1N\int_X
  (\DeltaE f)^2\varphi\,\d\mm, \quad \forall (f,\varphi)\in D_\V(\DeltaE)\times D_{L^\infty}(\DeltaE), \varphi\ge0.
\end{equation} 
We say that a metric measure space 
$(X,\sfd,\mm)$ (see \S~\ref{subsec:Cheeger}) satisfies the 
\emph{metric} $\BE KN$ condition if 
the Cheeger energy is quadratic, the associated Dirichlet form
$\cE$ satisfies $\BE KN$, and any $f\in \V_\infty$ with $\Gq f\in L^\infty(X,\mm)$ has a $1$-Lipschitz representative.

In Section~\ref{subsec:BE-Dirichlet}, by an approximation lemma, on the one hand we show that  in order to get the full $\BE  K N$ it is enough to check the validity of \eqref{eq:Intro9} just for  every $f \in D_{\V}(\DeltaE)\cap D_{L^\infty}(\DeltaE)$ and every nonnegative $\varphi\in D_{L^\infty} (\DeltaE)$. On the other hand, thanks to the improved integrability of $\Gamma$ given by Theorem~\ref{thm:interpolation}, in Corollary~\ref{cor:weaker-assumption-G2} we extend the domain of $\GGamma_2$ to the whole $(\D_\infty)^3$ and we give an equivalent reformulation of  the $\BE K N$ condition for functions in this larger space.  Local and nonlinear characterizations of the $\BE K N$ condition for locally compact spaces are investigated  in  Section~\ref{subsec:localBE}: in Theorem~\ref{thm:localization}  we show that in  order to get the full $\BE  K N$  it is enough to check the validity of  \eqref{eq:Intro9} just for every  $f \in D_{\V}(\DeltaE)\cap D_{L^\infty}(\DeltaE)$ and  $\varphi\in D_{L^\infty} (\DeltaE)$ \emph{with compact support}, and in Theorem~\ref{thm:nonlinearBE} we give a new nonlinear characterization of the $\BE K N $ condition in terms of regular entropies,
 namely
\begin{equation}
    \label{eq:164_intro}
    \GGamma_2(f;P(\varphi))
    +\int_X R(\varphi) \,(\DeltaE f)^2\,\d\mm
    \ge K\int_X \Gq f\,P(\varphi)\,\d\mm.
  \end{equation}
This last formulation will be very convenient later in the work in order to make a bridge between the curvature of the space and  the  contraction properties of non linear diffusion semigroups.    

Chapter~\ref{sec:NDaction} is devoted to action estimates along a
nonlinear diffusion semigroup. The aim is to give a rigorous proof of
the crucial estimate briefly discussed in the formal calculations of
Example~\ref{ex:2}.  To this purpose, in Theorem~\ref{thm:crucial-estimate} we prove that if $\varrho_t$
(resp. $\varphi_t$)  is a  sufficiently regular solution to  the
nonlinear diffusion equation $\partial_t \varrho_t- \DeltaE P
(\rho_t)=0$ (resp.~to the backward 
linearized equation $\partial_t \varphi_t+P'(\varrho_t) \DeltaE \varphi_t=0$) then the map  
  $t\mapsto \cE_{\varrho_t}(\vph_t)=\int_X \rh_t\Gq{\vph_t}\,\d\mm$ 
  is absolutely continuous  and we have
  \begin{equation}
    \label{eq:Intro58}
    \frac\d{\d t}\frac 12\int_X \rh_t\Gq{\vph_t}\,\d\mm
    =\mathbf\Gamma_2(\vph_t;P(\rh_t))+
    \int_X R(\rh_t)(\DeltaE\vph_t)^2\,\d\mm\quad
    \text{$\Leb{1}$-a.e.  in } (0,T).
  \end{equation}
  Notice that this formula is exactly  the derivative of the hamiltonian  $\frac 12\int_X \rh_t\Gq{\vph_t}\,\d\mm$ along the nonlinear diffusion semigroup. It is  clear from \eqref{eq:164_intro} and \eqref{eq:Intro58} that the metric $\BE K N $ condition implies a lower bound on the derivative of the hamiltonian, more precisely in 
  Theorem~\ref{thm:BEKNaction} we show that  the metric $\BE K N $ condition implies
   \begin{equation}
      \label{eq:Intro91}
      \frac\d{\d t}\frac 12\int_X \rh_t\Gq{\vph_t}\,\d\mm\ge
      K \int_X P(\varrho_t)\Gq{\vph_t}\,\d\mm
      \qquad
      \text{$\Leb 1$-a.e.\ in }(0,T),
    \end{equation}
and its natural counterparts in terms of the potentials $\phi_t$ (introduced in Part II) associated to the curve $\varrho_t \mm$. 
The inequality \eqref{eq:Intro91} should be considered as the appropriate nonlinear
version of the \AAA Bakry-\'Emery inequality \cite{Bakry-Emery84} (see also  \cite{AGS12} for the non-smooth formulation, and 
\cite{Bakry-Ledoux06} for dimensional improvements) \fn for solutions $\varrho_t,\varphi_{T-t}$
to the \emph{linear}
Heat flow
   \begin{equation}
      \label{eq:kk}
      \frac\d{\d t}\frac 12\int_X \rh_t\Gq{\vph_{t}}\,\d\mm\ge
      K \int_X \varrho_t\Gq{\vph_{t}}\,\d\mm
      \qquad
      \text{$\Leb 1$-a.e.\ in }(0,T),
    \end{equation}
which characterizes the $\BE K\infty$ condition. In this case, due to the linearity of the
Heat flow and to the self-adjointness of the Laplace operator, 
the backward evolution $\varphi_t$ can be easily constructed
by using the time reversed Heat flow and it is independent of
$\varrho$. 

In the last Chapter~\ref{sec:EquivBERCDS} we combine all the estimates and tools in order  to prove the equivalence between metric $\BE K N$ and $\RCDS K N$. 
To this aim, in  Section~\ref{subsec:RegularCurves}  we show some technical lemmas about approximation of $W_2$-geodesics via regular curves and about regularization of entropies.  
Section~\ref{subsec:BEEVI} is devoted to the proof of Theorem \ref{cor:MAIN} stating that $\BE K N$ implies $\CDS K N$.  This is achieved by showing that the nonlinear diffusion semigroup associated to a regular entropy provides the unique solution of the Evolution Variational Inequality \eqref{eq:IntroEVIKN} which characterizes $\RCDS K N$.
In the same section we prove the facts of independent interest that   $\BE K N$ implies contractivity in $W_2$ of the nonlinear diffusion semigroup induced by a regular entropy (see Theorem~\ref{thm:BEKNContr}), and that $\BE K N$ implies monotonicity of the action $\cA_2$ computed on a curve which is moved by a nonlinear diffusion semigroup (see Theorem~\ref{thm:BEEVI1}). 
 \\The last Section~\ref{sec:RCDtoBE} is devoted to the proof of the
 converse implication, namely that  if $(X,\sfd, \mm)$ is an $\RCDS K
 N$ space then the Cheeger energy  satisfies $\BE K N$. The rough idea
 here is to differentiate the 2-action of an arbitrary $W_2$-curve
 along the nonlinear diffusion semigroup and use the arbitrariness of
 the curve to show that  this yields the nonlinear characterization
 \eqref{eq:164_intro} of $\BE K N$ obtained in
 Theorem~\ref{thm:nonlinearBE}. 
 {\AAA The perturbation technique used to generate a sufficiently large
 class of curves is similar to the one independently proposed by \cite{Bolley-Gentil-Guillin-Kuwada14}.

\smallskip
{\bf Acknowledgement.} The quality of this manuscript greatly improved after the revision made by the reviewer, whose
careful reading has been extremely helpful for the authors. 
\newpage
}

\centerline{\large {\bf Main notation}}
\medskip

\halign{#\hfil\qquad&#\hfil\cr
\text{$\cE(f,f)$, $\Gbil f g$}&\text{Symmetric Dirichlet form $\cE$
  and its Carr\'e du Champ, Sect.~\ref{subsec:Dir-Form}}\cr
\text{$\H$, $\V$}&\text{$L^2(X,\mm)$ and the domain of $\cE$ }\cr
\text{$\V_\infty$}&\text{$\V\cap L^\infty(X,\mm)$}\cr
\text{$\sfP$}&\text{Markov semigroup induced by $\cE$}\cr
\text{$\DeltaE$}&\text{Infinitesimal generator of $\sfP$}\cr
\text{$\D$}&\text{Domain of $\DeltaE$, \eqref{eq:defDD}}\cr
\text{$\D_\infty$}& $\D\cap L^\infty(X,\mm)$\cr
\text{$\Vhom 1$}&\text{Homogeneous space associated to $\cE$,
  Sect.~\ref{subsec:completion}}\cr
\text{$\cQ^*(\ell,\ell)$}&\text{Dual of a quadratic form $\cQ$, \eqref{eq:107}}\cr
\text{$\cE_\varrho(f,f)$, $\Gamma_\varrho(f,g)$}&\text{Weighted quadratic form and Carr\'e du Champ}\cr
\text{$\Vhom\varrho$}&\text{Abstract completion of the domain of $\cE_\varrho$}\cr
\text{$-A_\varrho^*$}&\text{Riesz isomorphism between $\V_\varrho'$ and $\V_\varrho$, \eqref{eq:112}}\cr
\text{$|\dot\gamma|$}&\text{metric velocity, or speed, Sect.~\ref{subsec:AC}}\cr
\text{$\AC p{[a,b]}{(X,\sfd)}$}&\text{$p$-absolutely continuous
  paths}\cr
\text{$\cA_p(\gamma)$}&\text{$p$-action of a path $\gamma$, \eqref{eq:21}}\cr
\text{$\Lip(X),\,\,\Lip_b(X)$}&\text{Lipschitz and bounded Lipschitz functions $f:X\to\R$}\cr
\text{$\Lip(f)$}&\text{Lipschitz constant of $f\in\Lip(X)$}\cr
\text{$|\rmD f|$, $|\rmD^\pm f|$, $|\rmD^* f|$}&\text{Slopes of $f$, \eqref{eq:5o}, \eqref{eq:18}}\cr
\text{$\sfQ_t$}&\text{Hopf-Lax semigroup, \eqref{eq:11o}}\cr
\text{$\mathscr B(X),\,\,\Probabilities{X}$}&\text{Borel sets and Borel probability measures in $X$}\cr
\text{$\Probabilitiesp X$}&\text{Probability measures with finite $p$-moment}\cr
\text{$\Probabilitiesac X\mm$}&\text{Absolutely continuous probability measures}\cr
\text{$W_p(\mu,\nu)$}&\text{$p$-Wasserstein extended distance in $\Probabilities{X}$}\cr
\text{$\rme_s$}&\text{Evaluation maps $\gamma\mapsto\gamma_s$ at time $s$}\cr
\text{$\cA_p(\ppi)$}&\text{$p$-action of $\ppi\in\Probabilities{\rmC([0,1];X)}$, \eqref{eq:18bis}}\cr
\text{$\mathrm{GeoOpt}(X)$}&\text{Optimal geodesic plans, \eqref{eq:178}}\cr
\text{$\C(f)$}&\text{Cheeger relaxed energy, \eqref{eq:14}}\cr
\text{$|\rmD f|_w$}&\text{Minimal weak gradient, \eqref{eq:53}}\cr
\text{$\Action{\wname}{\mu}\mm$}&\text{Weighted energy functional induced by $\wname$, \eqref{eq:17}}\cr
\text{$\Action{\wname}{\mu_0;\mu_1}\mm$}&\text{Weighted energy functional along a geodesic from $\mu_0$ to $\mu_1$}\cr
\text{$\sfg$}&\text{Green function on $[0,1]$, \eqref{eq:3}}\cr
\text{$\sigma_\kappa^{(t)}(\delta)$}&\text{Distorted convexity
  coefficients, \eqref{eq:170}}\cr
\text{$U$, $\mathcal U$}&\text{Entropy function and the induced
  entropy functional, \eqref{eq:40}, \eqref{eq:41}}\cr
\text{$P$}&\text{Pressure function induced by $U$, \eqref{eq:40}}\cr
\text{$\sfS$}&\text{Nonlinear diffusion semigroup associated to an
  entropy $U$}\cr
\text{$\MC N$}&\text{Entropies satisfying the
  $N$-dimensional McCann condition, Def.~\ref{def:McCann}}\cr
\text{$\CD K\infty, \CDS KN$}&\text{Curvature dimension conditions, Sect.~\ref{subsec:9.3}}\cr
\text{$\RCD K\infty$}&\text{Riemannian curvature dimension
  condition, Def.~\ref{def:RCDI}}\cr
\text{$\Gamma_2,\BE KN$}&\text{$\Gamma_2$ tensor and Bakry-\'Emery
  curvature dimension
  condition, Sect.~\ref{subsec:BE-Dirichlet}
}\cr
}

\newpage

\section{Contraction and convexity via Hamiltonian estimates: an heuristic argument}
\label{subsec:heu}

Let us consider a smooth Lagrangian $\cL:\R^d\times\R^d\to [0,\infty)$,
convex and $2$-homogeneous w.r.t. the second variable,
 which is the Legendre transform of
a smooth and convex Hamiltonian $\cH:\R^d\times(\R^d)^*\to [0,\infty)$, i.e.
\begin{equation}
  \label{eq:191}
  \cL(x,w)=\sup_{\varphi\in (\R^d)^*}\langle
  w,\varphi\rangle-\cH(x,\varphi),\qquad
  \cH(x,\varphi)=\sup_{w\in \R^d}\langle
  \varphi,w\rangle-\cL(x,w);
\end{equation}
We consider the cost functional
\begin{equation}
  \label{eq:2}
  \calC(x_0,x_1):=\inf\Big\{\int_0^1 \cL(x(s),\dot x(s))\,\d s:x\in
  \rmC^1([0,1];\R^d),\ x(i)=x_i,\ i=0,1\Big\}
\end{equation}
and the flow $\sfS_t:\R^d\to\R^d$ given by a smooth vector field
$\frak f:\R^d\to \R^d$, i.e.~$x(t)=\sfS_t (\bar x)$ is the solution of
\begin{equation}
  \label{eq:19}
  \frac\d\dt x(t)=\frak f(x(t)),\qquad x(0)=\bar x.
\end{equation}
We are interested in necessary and sufficient conditions for the contractivity of 
the cost $\calC$ under the action of the flow $\sfS_t$.

As a direct approach, for every solution $x$ of the ODE \eqref{eq:19} 
one can consider the
linearized equation
\begin{equation}
  \label{eq:20}
  \frac\d{\dt}w(t)= \mathrm D\frak f(x(t))w(t).
\end{equation}
It is well known that 
if $s\mapsto \bar x(s)$ is a smooth curve of initial data for \eqref{eq:19} and
$x(t,s):=\sfS_t \bar x(s)$ are the corresponding solutions, then
$\partial_s x(t,s)$ solves \eqref{eq:20} for all $s$, i.e.
\begin{equation}
  \label{eq:192}
  w(t,s):=\frac\partial{\partial s}x(t,s)\quad\text{satisfies}\quad
  \frac\partial{\partial t}w(t,s)=\mathrm D\frak f(x(t,s)) w(t,s),\quad
  w(0,s)=\dot{\bar x}(s).
\end{equation}

It is one of the basic tools of \cite{Otto-Westdickenberg05} 
to notice that $\sfS$ satisfies the contraction property
\begin{equation}
    \label{eq:115}
    \calC(\sfS_T \bar x_0,\sfS_T\bar x_1)\le 
    \calC( \bar x_0,\bar x_1)\quad\forevery \bar x_0,\bar x_1\in \R^d,\,\,T\geq 0
  \end{equation} 
if and only if for every solution $x$ of \eqref{eq:19} and every
solution $w$ of \eqref{eq:20} one has
\begin{equation}
  \label{eq:193}
  \frac \d{\dt} \cL(x(t),w(t))\le 0.
\end{equation}

As we will see in the next sections, in some situations it is easier to
deal with the Hamiltonian $\cH$ instead of the Lagrangian $\cL$.
In order to get a useful condition, we thus introduce
the backward transposed equation
\begin{equation}
  \label{eq:73}
  \frac\d{\dt}\varphi(t)=-\mathrm D\frak f(x(t))^\intercal\varphi(t).
\end{equation}
It is easy to check that $w'(t)=A(t)w(t)$ and $\varphi'(t)=-A(t)^\intercal \varphi(t)$ imply that the duality pairing $\langle w(t),\varphi(t)\rangle$
is constant. Hence, choosing $A(t)=\mathrm D\frak f(x(t))$ gives
\begin{equation}\label{eq:basic}
\text{$t\mapsto \langle w(t),\varphi(t)\rangle$ is constant, whenever $w$ solves \eqref{eq:20}, $\varphi$ solves \eqref{eq:73}.}
\end{equation}
In the next proposition we assume a mild coercitivity property on $\cL$, namely
$$
\cL(x,w)\geq \gamma(|x|)|w|^2\qquad\text{with}\qquad\lim_{R\to\infty}\int_0^R \sqrt{\gamma}(r)\,dr=\infty
$$
for some continuous function $\gamma:[0,\infty)\to (0,\infty)$.
Under this assumption, by differentiating the function $t\mapsto\int_{|x(0)|}^{|x(t)|}\sqrt{\gamma}(r)\,\d r$, it is easily seen that
\begin{equation}\label{eq:coerci_star}
\sup_n|x_n(0)|+\int_0^1 \cL(x_n(s),\dot{x}_n(s))\,\d s<\infty\quad\Longrightarrow\quad
\sup_n \max_{[0,1]}|x_n|<\infty.
\end{equation}

\begin{proposition}[Contractivity is equivalent to Hamiltonian monotonicity]
  \label{prop:easy-back}
  The flow $(\sfS_t)_{t\ge0}$ satisfies the contraction property
  \eqref{eq:115}   
  if and only if
  \begin{equation}
    \label{eq:187}
    \frac \d{\d t} \cH(x(t),\varphi(t))\ge 0\quad
    \text{whenever $x$ solves \eqref{eq:19} and $\varphi$ solves \eqref{eq:73}.}
  \end{equation}
\end{proposition}

Notice that the monotonicity condition \eqref{eq:187} can be equivalently stated in differential form as
\begin{equation}
  \label{eq:194}
  \langle \cH_x(x,\varphi),\frak f(x)\rangle-\langle 
  \cH_\varphi(x,\varphi),
  \mathrm D\frak f(x(t))^\intercal\varphi
  \rangle\ge0\quad\forevery x\in \R^d,\
  \varphi\in (\R^d)^*.
\end{equation}
\begin{proof}
  Let us first prove that \eqref{eq:187} yields \eqref{eq:115}.
  Let $\bar x\in \rmC^1([0,1];\R^d)$ be a curve connecting $\bar x_0$
  to $\bar x_1$, let $x(t,s):=\sfS_t \bar x(s)$ and 
  $w(t,s):=\partial_s x(t,s)$. The thesis follows if we show that
  \begin{equation}
  \cL(x(T,s),w(T,s))\le \cL(\bar x(s),\dot{\bar x}(s))\quad
  \forevery s\in [0,1],\ T\ge0,\label{eq:195}
\end{equation}
since then
\begin{displaymath}
  \calC(x(T,0),x(T,1))\le 
  \int_0^1 \cL(x(T,s),w(T,s))\,\d s\le 
  \int_0^1 \cL(\bar x(s),\dot{\bar x}(s))\,\d s
\end{displaymath}
and it is sufficient to take the infimum of the right hand side
w.r.t.~all the curves connecting $\bar x_0$ to $\bar x_1$.

  For a fixed $s\in [0,1]$ and $T>0$ we consider a sequence 
  $\bar \varphi_n(s)$ such that 
  \begin{equation}
    \label{eq:196}
    \cL(x(T,s),w(T,s))=\lim_{n\to\infty} \langle
    w(T,s),\bar\varphi_n(s)\rangle -\cH(x(T,s),\bar\varphi_n(s)),
  \end{equation}
  and we consider the solution $\varphi_n(t,s)$ of the backward differential equation
  with terminal condition
  \begin{displaymath}
    \frac\partial{\partial
      t}\varphi_n(t,s)=-\mathrm D\frak f(x(t,s))^\intercal\varphi_n(t,s),\quad 0\le t\le T,\quad
    \varphi_n(T,s)=\bar\varphi_n(s).
  \end{displaymath}
  By \eqref{eq:192} and \eqref{eq:basic} we get $
    \langle w(T,s),\bar\varphi_n(s)\rangle=
    \langle w(0,s),\varphi(0,s)\rangle$.
   In addition we can use the monotonicity assumption \eqref{eq:187} to get
  \begin{displaymath}
    \cH(x(T,s),\bar\varphi_n(s))\ge
    \cH(\bar x(s),\varphi_n(0,s)).
  \end{displaymath}
  It follows that 
  \begin{eqnarray*}
    \langle
    w(T,s),\bar\varphi_n(s)\rangle -\cH(x(T,s),\bar\varphi_n(s))&\le& 
    \langle
    w(0,s),\varphi_n(0,s)\rangle -\cH(\bar x(s),\varphi_n(0,s))\\&\le& 
    \cL(\bar x(s),w(0,s))
  \end{eqnarray*}
  and passing to the limit as $n\to\infty$, by \eqref{eq:196} we get
  \eqref{eq:195} since $w(0,s)=\dot{\bar x}(s)$.

  In order to prove the converse implication,
  let us first prove the asymptotic formula
  \begin{equation}\label{eq:asymcost}
  \lim_{\delta\downarrow 0}\frac{\calC (x(0),x(\delta))}{\delta^2}=\cL(x(0),\dot x(0))
  \end{equation}
  for any curve $s\mapsto x(s)$ right differentiable at $0$. Indeed, 
  notice first that the inequality
  $$
  \limsup_{\delta\downarrow 0}\frac{\calC (x(0),x(\delta))}{\delta^2}\leq\cL(x(0),\dot x(0))
  $$
  immediately follows considering an affine function connecting $x(0)$ and $x(\delta)$.
  In order to get the $\liminf$ inequality, notice
  that for any curve $s\mapsto y(s)$ and any vector $\varphi$ one has
  \begin{eqnarray*}
    \int_0^1 \cL(y(s),\dot y(s))\,\d s&\ge&
    \int_0^1
    \Big(\langle \dot y(s),\varphi\rangle-
    \cH(y(s),\varphi)\Big)\,\d s
    \\&=&
       \langle y(1)-y(0),\varphi\rangle-\int_0^1 \cH(y(s),\varphi)\,\d s,
  \end{eqnarray*}
  so that choosing an almost (up to the additive constant $\delta^3$) 
  minimizing curve $y=x_\delta:[0,1]\to\R^d$ connecting $x(0)$ to $x(\delta)$ and replacing
  $\varphi$ by $\delta \varphi$ with $\delta\in (0,1)$,
  the $2$-homogeneity of $\cH$ yields
  \begin{displaymath}
   \delta+ \frac{\calC (x(0),x(\delta))}{\delta^2}\ge 
    \left\langle\frac {x(\delta)-x(0)}\delta,\varphi\right\rangle-
    \int_0^1 \cH(x_\delta(s),\varphi)\,\d s.
  \end{displaymath}
  Since \eqref{eq:coerci_star} provides the relative compactness of $x_\delta$ in $\rmC ([0,1];\R^d)$ and since
  $\gamma>0$, it is easily seen that $x_\delta$ uniformly converge to the constant $x(0)$ as $\delta\downarrow 0$.
  Therefore, passing to the limit as $\delta\downarrow 0$ we get
  \begin{equation*}
    \liminf_{\delta\down0}\frac{\calC (x(0),x(\delta))}{\delta^2}\ge 
    \left\langle \dot x(0),\varphi\right\rangle-
    \cH(x(0),\varphi)
  \end{equation*}
  and eventually we can take the supremum w.r.t. $\varphi$ to obtain \eqref{eq:asymcost}.
  
  If \eqref{eq:115} holds and $x(t)$ and $\varphi(t)$ are solutions to \eqref{eq:19} and \eqref{eq:73} respectively,
we fix $t_0\ge0$ and $w\in \R^d$ such that 
  \begin{equation}
    \label{eq:237}
    \cH(x(t_0),\varphi(t_0))=\langle w,\varphi(t_0)\rangle-
    \cL(x(t_0),w).
  \end{equation}
  We then consider 
  the curve $s\mapsto x(t_0)+sw$ and we set $x(t,s)=\sfS_{t-t_0}(x(t_0)+sw)$,
  so that $w(t)=\partial_s x(t,s)\restr{s=0}$ is a solution of 
  \eqref{eq:20} with Cauchy condition $w(t_0)=w$.
  For $t>t_0$ we can use twice \eqref{eq:asymcost} 
  and \eqref{eq:basic} once more to obtain
  \begin{eqnarray*}
    \cH(x(t),\varphi(t))&\ge& \langle\varphi(t),w(t)\rangle-
    \lim_{\delta\down0} \frac{\calC(x(t,0),x(t,\delta))}{\delta^2}
    \\&\topref{eq:115}\ge&
    \langle\varphi(t_0),w\rangle-
    \lim_{\delta\down0} \frac{\calC(x(t_0,0),x(t_0,\delta))}{\delta^2} 
    \\&=& \langle\varphi(t_0),w\rangle- \cL(x(t_0),w)
    \topref{eq:237}=\cH(x(t_0),\varphi(t_0)).
  \end{eqnarray*}
\end{proof}

We can refine the previous argument to gain further insights when
$\cH, \,\cL$ are quadratic forms and $\frak f$ is the gradient of a
potential $U$. More precisely, we will suppose that 
\begin{equation}
  \label{eq:197}
  \cL(x,w)=\frac 12 \langle\sfG(x)w,w\rangle,\quad
  \cH(x,\varphi)=\frac 12 \langle \varphi,\sfH(x)\varphi\rangle,\quad
  \sfH(x)=\sfG(x)^{-1},
\end{equation}
and $\sfG(x)$ are symmetric and positive definite linear maps from
$\R^d$ to $(\R^d)^*$, smoothly depending on $x\in \R^d$.
The vector field $\mathfrak f$ is the (opposite) gradient of $U:\R^d\to \R$ with respect to 
the metric induced by
$G$ if
\begin{equation}
  \label{eq:198}
  \langle \sfG(x)\frak f(x) ,w\rangle=\langle -\rmD U(x),w\rangle\quad
  \forevery w\in \R^d,\quad\text{i.e.\ }\frak f(x)=-\sfH(x)\rmD U(x).
\end{equation}
In \cite{Daneri-Savare08} it is shown that 
$U$ is geodesically convex along the distance induced by the cost
$\calC$ if and only if 
\begin{equation}
  \label{eq:199}
\calC(\bar x_0,\sfS_t \bar x_1)+t \Big(U(\sfS_t(\bar x_1))- U(\bar x_0)\Big)\le 
\calC(\bar x_0,\bar x_1)
\quad\forevery \bar x_0,\,\bar x_1\in \R^d,\ t\ge0.
\end{equation}
Here is a simple argument to deduce \eqref{eq:199} from
\eqref{eq:187}.

\begin{lemma}
  \label{le:2ndeasyarg}
  Let $\cL$, $\cH$, $\frak f$ be given by \eqref{eq:197} and \eqref{eq:198}. 
  Then \eqref{eq:187} yields \eqref{eq:199}.
\end{lemma}
\begin{proof}
  Let us consider a curve $\bar x(s)$ connecting $\bar x_0$ to $\bar
  x_1$ and let us set 
   \begin{displaymath}
    y(t,s):=\sfS_{t} (\bar x(s)),\quad
    x(t,s):=y(st,s),\quad
    z(t,s):=\frac\partial{\partial s}y(t,s),\quad
    w(t,s):=\frac\partial{\partial s}x(t,s),
  \end{displaymath}
  so that \eqref{eq:198} gives
   \begin{equation}\label{eq:formulaxw}
    w(t,s)=z(st,s)+t\mathfrak f(x(t,s))=
    z(st,s)-t\sfH(x(t,s))\rmD U(x(t,s)).
  \end{equation}
  Clearly for every $t\ge0$ the curve $s\mapsto x(t,s)$ connects
  $\bar x_0$ to $\sfS_t\bar x_1$ and therefore
  \begin{displaymath}
    \calC(\bar x_0,\sfS_t \bar x_1)\le 
    \int_0^1 \cL(x(t,s),w(t,s))\,\d s,\quad
    U(\sfS_t\bar x_1)-U(\bar x_0)=
    \int_0^1 \langle \rmD U(x(t,s)),w(t,s)\rangle\,\d s.
  \end{displaymath}
  It follows that 
  \begin{displaymath}
    \calC(\bar x_0,\sfS_t \bar x_1)+t \Big(U(\sfS_t(\bar x_1))- U(\bar
    x_0)\Big)
    \le \int_0^1 \Big(\cL(x(t,s),w(t,s))+t \langle \rmD
    U(x(t,s)),w(t,s)\rangle\Big)\,\d s.
  \end{displaymath}
  For a fixed $s\in [0,1]$ the integrand satisfies
  \begin{eqnarray}\label{eq:2014_1}
    &&\cL(x(t,s),w(t,s))+t \langle \rmD
    U(x(t,s)),w(t,s)\rangle\\&=&\nonumber
    \sup_{\psi\in (\R^d)^*} \langle \psi+t\rmD
    U(x(t,s)),w(t,s)\rangle-\cH(x(t,s),\psi)
    \\&=&\nonumber
    \sup_{\varphi\in (\R^d)^*}
    \langle \varphi,w(t,s)\rangle-\cH(x(t,s),\varphi-t\rmD U (x(t,s))).
  \end{eqnarray}
  Substituting the expression \eqref{eq:formulaxw} and recalling \eqref{eq:197} we get
  \begin{align*}
    \langle &\varphi,w(t,s)\rangle-\cH(x(t,s),\varphi-t\rmD U
    (x(t,s)))
    \\&=
    \langle \varphi,z(st,s)\rangle-
    t\langle \varphi,\sfH(x(t,s))\rmD U
    (x(t,s))\rangle-\cH(x(t,s),\varphi-t\rmD U (x(t,s)))
    \\&= 
    \langle \varphi,z(st,s)\rangle-
    \cH(x(t,s),\varphi)
    -
    t^2\cH(x(t,s),\rmD U (x(t,s)))
    \le \langle \varphi,z(st,s)\rangle-
    \cH(x(t,s),\varphi).
  \end{align*}
  Choosing now an arbitrary curve $\varphi(s)$ and solutions
  $\varphi(\tau,s)$ of 
   \begin{displaymath}
     \frac\partial{\partial
      \tau}\varphi(\tau,s)=-\rmD \frak f(y(\tau,s))^\intercal\varphi(\tau,s),\quad 0\le \tau\le st,\quad
    \varphi(st,s)=\varphi(s)
  \end{displaymath}
  we can use the monotonicity assumption and \eqref{eq:basic} to obtain
  \begin{align*}
    \langle \varphi(s),z(st,s)\rangle-
    \cH(x(t,s),\varphi(s))&\le 
    \langle \varphi(0,s),z(0,s)\rangle-
    \cH(x(0,s),\varphi(0,s))\\&=
    \langle \varphi(0,s),w(0,s)\rangle-
        \cH(x(0,s),\varphi(0,s))\\&\le 
    \cL(\bar x(s),w(0,s)).
  \end{align*}
  Since $\varphi(s)$ is arbitrary and $w(0,s)=\dot{\bar x}(s)$, considering a maximizing sequence
  $(\varphi_n(s))$ in \eqref{eq:2014_1} we eventually get
  \begin{displaymath}
    \calC(\bar x_0,\sfS_t \bar x_1)+t \Big(U(\sfS_t(\bar x_1))- U(\bar
    x_0)\Big)
    \le \int_0^1  \cL(\bar x(s),\bar x'(s))\,\d s
  \end{displaymath}
  and taking the infimum w.r.t.~the initial curve $\bar x$ we conclude.
\end{proof}

In order to understand how to apply the previous
arguments for studying contraction and convexity in Wasserstein space,
let us consider two basic examples. 
For simplicity we will consider the case
of a compact Riemannian manifold $(\M^d,\sfd,\mm)$ endowed
with the
distance and measure associated to the Riemannian metric tensor $\frg$.

\begin{example}[The Bakry-\'Emery condition for the linear heat
  equation]
  \label{ex:1}
  \upshape
In the subspace of smooth probability densities
  (identified with the corresponding measures) 
  the Wasserstein distance cost $\calC(\varrho_0,\varrho_1)=
  \frac 12W_2^2(\varrho_0\mm,\varrho_1\mm)$ is
  naturally associated to the Hamiltonian
  \begin{equation}
    \label{eq:200}
    \cH(\varrho,\varphi):=\frac12\int_X |\rmD\varphi|_\frg^2\varrho\,\d\mm:
  \end{equation}
  in fact the Otto-Benamou-Brenier interpretation yields
  \begin{equation}
    \label{eq:114}
    \calC(\varrho_0,\varrho_1)=\inf\Big\{\int_0^1
    \cL(\varrho_s,\dot\varrho_s)\,\d s:\ 
    s\mapsto \varrho_s\text{ connects $\varrho_0$ to $\varrho_1$}\Big\},
  \end{equation}
  where
  \begin{equation}
    \label{eq:201}
    \cL(\varrho,w):=\frac 12 \int_X
    |\rmD \varphi|^2_\frg\varrho\,\d\mm,\qquad
    -\mathrm{div}_\frg (\varrho \nabla\kern-2pt_\frg\varphi)=w,
  \end{equation}
  i.e.
  \begin{equation}
    \label{eq:202}
    \cL(\varrho,w)=\sup_{\varphi} \, \int_X \varphi\,w\,\d\mm-\frac 12 \int_X
    |\rmD\varphi|^2_\frg\varrho\,\d\mm.
  \end{equation}
  In other words, the cotangent bundle is associated to the velocity gradient $\nabla\kern-2pt_\frg\varphi$ and the duality between
  tangent and cotangent bundle is provided by the possibly degenerate elliptic PDE 
  \begin{equation}\label{eq:deg_PDE}
  -\mathrm{div}_\frg (\varrho \nabla\kern-2pt_\frg \varphi)=w.
  \end{equation}
  If we consider the logarithmic entropy 
  functional $\cU_\infty(\varrho):=\int_X \varrho\log\varrho\,\d\mm$,
  then its Wasserstein gradient flow corresponds to the linear differential
  equation
  \begin{equation}
    \label{eq:203}
    \frac \d{\dt}\varrho=\Delta_\frg\varrho
  \end{equation}
  and thus  the backward equation \eqref{eq:73} for $\varphi$ corresponds to
  \begin{equation}
    \label{eq:204}
    \frac \d{\dt}\varphi=-\Delta_\frg\varphi.
  \end{equation}
  Evaluating the derivative of the Hamiltonian one gets
  \begin{align*}
    \frac \d{\dt}\cH(\varrho_t,\varphi_t)&=
    \frac 12\frac \d{\dt}\int_X \varrho_t|\rmD\varphi_t|^2_\frg\,\d\mm=
    \frac 12\int_X \Delta_\frg\varrho_t|\rmD\varphi_t|^2_\frg\,\d\mm+
    \int_X \varrho_t \langle \rmD \varphi_t,\rmD
    (-\Delta_\frg\varphi_t)\rangle_\frg\,\d\mm
    \\&=\int_X \varrho_t\Big(\frac 12 \Delta_\frg
    |\rmD\varphi_t|^2_\frg-
    \langle \rmD \varphi_t,\rmD
    \Delta_\frg\varphi_t\rangle_\frg\Big)\,\d\mm.
  \end{align*}
  Since $\varrho\ge0$ and $\varphi$ are arbitrary, \eqref{eq:187} 
  corresponds to 
  the Bakry-\'Emery  $\BE 0\infty$ condition 
  $$
  \Gamma_2(\varphi):=\frac 12 \Delta_\frg 
    |\rmD\varphi_t|^2_\frg-
    \langle \rmD \varphi_t,\rmD
    \Delta_\frg\varphi_t\rangle_\frg
  \ge0\qquad\text{ for every
  $\varphi$}.
  $$
  
  It is remarkable that the above calculations correspond to the 
  Bakry-Ledoux \cite{Bakry-Ledoux06} derivation of the $\Gamma_2$ tensor: if $\sfP_t$
  denotes the heat flow associated to \eqref{eq:203}, it is well known
  that 
  \begin{displaymath}
    \Gamma_2\ge0\quad\Longleftrightarrow\quad
    \frac\d{\d s }\Big(\sfP_s |\rmD \sfP_{t-s}\varphi|^2_\frg\Big)\ge0
  \end{displaymath}
  or, in the integrated form,
  \begin{displaymath}
    \frac12\frac\d{\d s }\int_X \sfP_s \varrho\, |\rmD
    \sfP_{t-s}\varphi|_\frg^2\,\d\mm=
    \frac\d{\d s }\cH(\sfP_s \varrho,
    \sfP_{t-s}\varphi)
    \ge0
    \quad\forevery \varphi,\varrho\ge0.
  \end{displaymath}
  Thus the combination of the forward flow $\sfP_s\varrho_s$ and the backward
  flow $\sfP_{t-s}\varphi$ in the derivation of $\Gamma_2$ tensor
  corresponds to the Hamiltonian monotonicity \eqref{eq:187}.  
\end{example}

\begin{example}[The Bakry-\'Emery condition for nonlinear diffusion]
  \label{ex:2}
  \upshape
  If we apply the previous argument to the entropy functional
  $\cU(\varrho)=\int_X U(\varrho)\,\d\mm$, we are led to study the 
  nonlinear diffusion equation
  \begin{equation}
    \label{eq:205}
    \frac \d\dt \varrho=\Delta_\frg P(\varrho)\qquad\text{with}\qquad
    P(\varrho):=\varrho U'(\varrho)-U(\varrho).
  \end{equation}
  The corresponding linearized backward transposed flow is
  \begin{equation}
    \label{eq:206}
    \frac\d\dt \varphi=-P'(\varrho)\Delta_\frg \varphi
  \end{equation}
  and, setting $R(\varrho):=\varrho P'(\varrho)-P(\varrho)$, we get
   \begin{align*}
    & \frac \d{\dt}\cH(\varrho_t,\varphi_t)\\&=
    \frac 12\frac \d{\dt}\int_X \varrho_t|\rmD\varphi_t|^2_\frg\,\d\mm\\&=
    \frac 12\int_X \Delta_\frg P(\varrho_t)|\rmD\varphi_t|^2_\frg\,\d\mm-
    \int_X \varrho_t \langle \rmD \varphi_t,\rmD
    (P'(\varrho_t)\Delta_\frg\varphi_t)\rangle_\frg\,\d\mm
    \\&=\int_X P(\varrho_t)\Gamma_2(\varphi_t)\,\d\mm+
    \int_X P(\varrho_t)\langle\rm D\varphi_t,\rm D\Delta_\frg\varphi_t\rangle_\frg
    -\int_X \varrho_t \langle \rmD \varphi_t,\rmD
    (P'(\varrho_t)\Delta_\frg\varphi_t)\rangle_\frg\,\d\mm
    \\&=\int_X P(\varrho_t)\Gamma_2(\varphi_t)\,\d\mm+\int_X \bigl(-P(\varrho_t)+\varrho_tP'(\varrho_t)\bigr)\big(\Delta_\frg\varphi_t\big)^2\,\d\mm
    -\int_X\Delta_\frg\varphi_t\langle\rm D \varphi_t,\rm D P(\varrho_t)\rangle_\frg\,\d\mm\\
    &+\int_XP'(\varrho_t)\Delta_\frg\varphi_t\langle\rm D \varphi_t,\rm D\varrho_t\rangle_\frg\,\d\mm
        \\&=\int_X P(\varrho_t)\Gamma_2(\varphi_t)\,\d\mm
    +\int_X R(\varrho_t)\big(\Delta_\frg\varphi_t\big)^2\,\d\mm.
  \end{align*}
  If $U$ satisfies McCann's condition $\MC N$, so that
  $R(\varrho)\ge -\frac 1N P(\varrho)$, and 
  the Bakry-\'Emery condition $\BE 0N$ holds, so that 
  $\Gamma_2(\varphi)\ge \frac 1N (\Delta_\frg\varphi)^2$, we still get 
  $\frac\d\dt \cH(\varrho,\varphi)\ge0$.
\end{example}

\begin{example}[Nonlinear mobilities]
  \upshape
  As a last example, consider \cite{Dolbeault-Nazaret-Savare08b}  the case of an Hamiltonian 
  associated to a nonlinear positive mobility $h$
  \begin{equation}
    \label{eq:208}
    \cH(\varrho,\varphi):=\frac 12\int_X h(\varrho)|\rmD\varphi|_\frg^2\,\d\mm
  \end{equation}
  under the action of the linear heat flows 
  \eqref{eq:203} and \eqref{eq:204}: with computations similar to those of the previous examples we get
   \begin{align*}
     \frac \d{\dt}&\cH(\varrho_t,\varphi_t)=
    \frac 12\frac \d{\dt}\int_X h(\varrho_t)|\rmD\varphi_t|^2_\frg\,\d\mm\\&=
    \frac 12\int_X h'(\varrho_t)\Delta_\frg \varrho_t|\rmD\varphi_t|^2_\frg\,\d\mm-
    \int_X h(\varrho_t) \langle \rmD \varphi_t,\rmD
    \Delta_\frg\varphi_t)\rangle_\frg\,\d\mm
    \\&=
    \frac 12\int_X \Delta_\frg( h(\varrho_t))|\rmD\varphi_t|^2_\frg\,\d\mm-
    \int_X h(\varrho_t) \langle \rmD \varphi_t,\rmD
    \Delta_\frg\varphi_t)\rangle_\frg\,\d\mm
    -
    \frac 12 \int_X h''(\varrho_t)|\rmD \varrho_t|_\frg^2\,|\rmD\varphi_t|_\frg^2\,\d\mm
    \\&=
    \int_X h(\varrho_t)\Gamma_2(\varphi_t)\,\d\mm
    -
    \frac 12 \int_X h''(\varrho_t)|\rmD \varrho_t|_\frg^2\,|\rmD\varphi_t|_\frg^2\,\d\mm.
  \end{align*}
  If $h$ is concave and the Bakry-\'Emery condition $\BE 0\infty$
  holds, 
  \GGG we still have
  $\frac\d\dt \cH(\varrho,\varphi)\ge0$.
  According to Proposition \ref{prop:easy-back}, 
  this property formally corresponds to the contractivity
  of the $h$-weighted Wasserstein distance $\mathcal W_h$ 
  associated to the Hamiltonian \eqref{eq:208}
  along the Heat flow, a property that has been
  proved in \cite[Theorem 4.11]{Carrillo-Lisini-Savare-Slepcev09}
  by a different method.
  \EEE
\end{example}

\part{Nonlinear diffusion equations 
and their linearization in Dirichlet spaces}

\section{Dirichlet forms, homogeneous spaces
  and nonlinear diffusion}

\subsection{Dirichlet forms}
\label{subsec:Dir-Form}

In all this first part we will deal with a measurable space $(X,\cB)$, which is complete with respect to 
a $\sigma$-finite measure $\mm:\cB\to[0,\infty]$.  {\color{black} We denote by $\H$ the Hilbert space
$L^2(X,\mm)$ and we are given}  a symmetric Dirichlet form 
$\cE:\H=L^2(X,\mm)\to [0,\infty]$ 
(see e.g.~\cite{Bouleau-Hirsch91} as a general reference)
with proper domain 
\begin{equation}
  \label{eq:253}
  \V=D(\cE):=\big\{f\in L^2(X,\mm):\cE(f)<\infty\big\},\quad
  \text{with}\quad
  \V_\infty:=\V\cap L^\infty(X,\mm).
\end{equation}
$\V$ is a Hilbert space endowed with the norm
\begin{equation}
  \label{eq:261}
  \|f\|^2_\V:=\|f\|^2_{L^2(X,\mm)}+\cE(f);
\end{equation}
the inclusion of $\V$ in $\H$ is always continuous and
we will assume that it is also dense (we will write
$\V\stackrel{ds}\hookrightarrow\H$);
we will still denote by $\cE(\cdot,\cdot):\V\to\R$ the symmetric
bilinear form associated to $\cE$. 
Identifying $\H$ with its dual $\H'$, $\H$ is also continuously and densely
imbedded in the dual space $\V'$, so that
\begin{equation}
  \label{eq:266}
  \V\stackrel{ds}\hookrightarrow\H\equiv \H'\stackrel{ds}\hookrightarrow\V'\quad\text{is a standard Hilbert triple}
\end{equation}
and we have
\begin{equation}
  \label{eq:262}
  \duality{}{}fg=\int_X fg\,\d\mm\quad\text{whenever }f\in \H,\,
  g\in \V
\end{equation}
where $\duality{}{}\cdot\cdot=\duality{\V'}\V\cdot\cdot$
denotes the duality pairing between $\V'$ and $\V$,
when there will be no risk of confusion.

The locality of $\cE$ and the $\Gamma$-calculus are not needed
at this level: they will play a crucial role
in the next parts. On the other hand, 
we will repeatedly use the following properties of
Dirichlet form: 
\begin{subequations}
  \label{subeq:dir}
  \begin{equation}
    \label{eq:Dir-Alg}
    \V_\infty\text{ is an algebra,}\quad
    \cE^{1/2}(fg)\le \|f\|_\infty\cE^{1/2}(g)+
    \|g\|_\infty\cE^{1/2}(f)
    \quad f,\,g\in \V_\infty,
    \tag{DF1}
  \end{equation}
  \begin{equation}
    \label{eq:Dir-Lip}
    \begin{gathered}
      \text{if $P:\R\to\R$ is $L$-Lipschitz with $P(0)=0$ then the map
        $f\mapsto P\circ f$}\\
      \text{is well defined and continuous from $\V$ to $\V$
        with}\quad \cE(P\circ f)\le L^2\,\cE(f),
    \end{gathered}
    \tag{DF2}
  \end{equation}
  \begin{equation}
    \label{eq:Dir-Mon}
    \begin{gathered}
      \text{if $P_i:\R\to\R$ are Lipschitz and nondecreasing with
        $P_i(0)=0$, $i=1,2$, then}\\
      \cE(P_1\circ f,P_2\circ f)\ge 0\quad \forevery f\in \V.
    \end{gathered}
    \tag{DF3}
  \end{equation}
\end{subequations}
We will denote by $-\DeltaE$ the
linear monotone operator induced by $\cE$,
\begin{equation}
  \label{eq:265}
  \DeltaE:\V\to\V',\quad
  \duality{}{}{-\DeltaE f}g:=\cE(f,g)\quad\forevery f,g\in \V,
\end{equation}
satisfying
\begin{equation}\label{eq:599}
  \big|\duality{}{}{\DeltaE f}g\big|^2\le \cE(f,f)\,\cE(g,g)
  \quad\forevery f,\,g\in \V,
\end{equation}
and by $\D$ the Hilbert space
\begin{equation}
  \label{eq:263}
  \D:=\big\{f\in \V:\DeltaE f\in \H\big\}\quad\text{endowed with the Hilbert norm}\quad
  \|f\|_\D^2:=\|f\|_\V^2+\|\DeltaE f\|_\H^2.
\end{equation}
Thanks to the interpolation estimate
\begin{equation}\label{eq:interpola_2015}
  \|\varrho\|_\V\le C\|\varrho\|_\H^{1/2}\,\|\varrho\|_\D^{1/2}\quad
  \forevery \varrho\in\D,
\end{equation}
which easily follows by the identity
$\cE(\varrho,\varrho)=-\int_X\varrho\DeltaE\varrho\,\d\mm$, the norm of $\D$ is equivalent
to the norm $\|f\|_\H^2+\|\DeltaE f\|_\H^2$.

We also introduce the dual quadratic form $\cE^*$ on $\V'$, defined by
\begin{equation}
  \label{eq:107}
  \frac 12\cE^*(\ell,\ell):=\sup_{f\in \V}\,
  \langle \ell,f\rangle-\frac 12\cE(f,f).
\end{equation}
It is elementary to check that the right hand side in \eqref{eq:107} satisfies the parallelogram
rule, so our notation $\cE^*(\ell,\ell)$ is justified 
(and actually we will prove that $\cE^*$, when restricted
to its finiteness domain, is canonically associated to the dual Hilbert norm \AAA{of a suitable quotient space; see the 
following Section~\ref{subsec:completion} for the details\fn).

The operator $\DeltaE$ generates a 
Markov semigroup $(\sfP_t)_{t\ge0}$ 
in each $L^p(X,\mm)$, $1\leq p\leq\infty$:
for every $f\in \H$ the curve $f_t:=\sfP_t f$ 
belongs to $\rmC^1((0,\infty);\H)\cap\rmC^0((0,\infty);\D)$ 
and it is the unique solution in this class of the Cauchy problem
\begin{equation}
  \label{eq:267}
  \frac\d{\d t}f_t=\DeltaE f_t\quad t>0,\qquad
  \lim_{t\downarrow0}f_t=f\quad\text{strongly in }\H.
\end{equation} 
The curve $(f_t)_{t\ge0}$ belongs to $\rmC^1([0,\infty);\H)$ 
if and only if $f\in \D$ and in this case
\begin{equation}
  \label{eq:269}
  \lim_{t\down0}\frac{f_t-f}t=\DeltaE f\quad\text{strongly in }\H.
\end{equation}
$(\sfP_t)_{t\ge0}$ is in fact an analytic semigroup of linear
contractions in $\H$ and in 
each $L^p(X,\mm)$ space, $p\in (1,\infty)$, satisfying
the regularization estimate
\begin{equation}
  \label{eq:268}
  \frac 12 \|\sfP_t f\|_\H^2+
  t\cE(\sfP_t f)+t^2 \|\DeltaE\sfP_t f\|_\H^2\le 
  \frac 12 \|f\|_\H^2\quad \forevery t>0.
\end{equation}
The semigroup $(\sfP_t)_{t\ge0}$ is said to be \emph{mass preserving} if 
\begin{equation}
  \label{eq:271}
  \int_X \sfP_t f\,\d\mm=
  \int_X f\,\d\mm\quad\forall t\geq 0\qquad
  \forevery f\in L^1\cap L^2(X,\mm).
\end{equation}
Since $(\sfP_t)_{t\ge0}$ is a strongly continuous semigroup
of contractions in $L^1(X,\mm)$, the mass preserving
property is equivalent to
  \begin{equation}
    \label{eq:228}
    \int_X \DeltaE f\,\d\mm=0\quad 
    \forevery f\in \D\cap L^1(X,\mm)\text{ with }\DeltaE f\in L^1(X,\mm).
  \end{equation}
When $\mm(X)<\infty$ then \eqref{eq:271} is equivalent to 
the property $1\in D(\cE)$ 
with $\cE(1)=0$.

\subsection{Completion of quotient spaces w.r.t.\ a seminorm}
\label{subsec:completion}

Here we recall a simple construction that we will often use in the
following. 

Let $N$ be the kernel of $\cE$ and $\DeltaE $, namely
\begin{equation}
  \label{eq:106}
  N:=\Big\{f\in\V:\ \cE(f,f)=0\Big\}=\Big\{f\in\V:\ \DeltaE f=0\Big\}.
\end{equation}
It is obvious that
$N$ is a closed subspace of $\V$ and that it induces the
equivalence relation
\begin{equation}
  \label{eq:105}
  f\sim g\quad\Longleftrightarrow\quad f-g\in N.
\end{equation}
We will denote by $\widetilde \V:=\V/\!\!\sim$ the quotient space
and by $\tilde f$ the equivalence class of $f$ (still denoted by $f$ when there is no risk of confusion); 
it is well known that we can identify 
the dual of $\widetilde\V$ with the closed subspace $N^\perp$ of $\V'$, i.e.
\begin{equation}
  \label{eq:235}
  \widetilde\V'=N^\perp=\big\{\ell\in\V':\ \langle
  \ell,f\rangle=0\quad\forevery f\in N\big\}.
\end{equation}
Since $\cE$ is nonnegative, we have 
$\cE(f_1,g_1)=\cE(f_0,g_0)$ whenever $f_0\sim f_1$ and $g_0\sim g_1$, 
so that $\cE$ can also be considered a symmetric bilinear form 
on $\widetilde\V$, for which we retain the same notation.
The bilinear form $\cE$ is in fact a scalar product on $\widetilde\V$, so that 
it can be extended to a scalar product
on the abstract completion $\Vhom{\cE}$ of $\widetilde\V$, with respect to
the norm induced by $\cE$. The dual of $\Vhom{\cE}$ will be denoted by $(\Vhom\cE)'$. 

In the next proposition we relate $(\Vhom{\cE})'$ to $\V'$ and to the dual quadratic form
$\cE^*$ in \eqref{eq:107}. 

\begin{proposition} [Basic duality properties]\label{prop:allduals}
  Let 
  $\V,\,\cE,\,\widetilde\V$ be as above and let $\Vhom\cE$ 
  be the abstract completion of
  $\widetilde\V$ w.r.t.~the scalar product $\cE$.
  Then the following properties hold:
\begin{itemize}
\item[(a)] $(\Vhom{\cE})'$ can be canonically and isometrically
  realized as the finiteness domain of $\cE^*$ in $\V'$, 
  endowed with the norm induced by $\cE^*$, that we will denote as $\Vdual\cE$. 
\item[(b)] If $\ell\in \Vdual\cE$ and $(f_n)$ is a maximizing sequence in \eqref{eq:107}, then the corresponding elements
  in $\Vhom\cE$ strongly converge in $\Vhom\cE$ 
  to $f\in \Vhom\cE$ satisfying
  \begin{equation}
    \frac 12\cE^*(\ell,\ell)=\langle\ell,f\rangle-\frac 12
    \cE(f,f),\qquad
    \cE^*(\ell,\ell)=\cE(f,f).
    \label{eq:264}
\end{equation}
\item[(c)] The operator $\DeltaE$ 
  in \eqref{eq:265} maps $\V$ into $\Vdual\cE$; it 
  can be extended to a continuous and linear operator $\DeltaE_\cE$ from
  $\Vhom\cE$ to $\Vdual{\cE}$ 
  and $-\DeltaE_\cE:\Vhom\cE\to \Vdual\cE$ 
  is the Riesz isomorphism associated to the scalar product 
  $\cE$ on $\Vhom\cE$.
\item[(d)] $\cE^*(\ell,-\DeltaE  f)=\langle \ell,f\rangle$ for all 
  $\ell\in \Vdual\cE$, $f\in \V$.
\end{itemize}
\end{proposition}
\begin{proof} (a) The inequality $2|\langle\ell,f\rangle|\leq \cE(f,f)+\cE^*(\ell,\ell)$,
by homogeneity, gives $|\langle\ell,f\rangle|\leq (\cE(f,f))^{1/2}(\cE^*(\ell,\ell))^{1/2}$. Hence, any
element $\ell$ in the finiteness domain of $\cE^*$ 
induces a continuous linear functional
on $\widetilde{\V}$ and therefore an element in  $(\Vhom{\cE})'$, 
with (dual) norm less than $(\cE^*(\ell,\ell))^{1/2}$.
Conversely, any $\ell\in (\Vhom{\cE})'$ induces a continuous linear functional in $\widetilde{\V}$, and then
a continuous linear functional $\ell$ in $\V$, satisfying
$|\ell(f)|\leq  \| \ell \|_{(\Vhom{\cE})'}(\cE(f,f))^{1/2}$. By the
continuity of 
$\cE$, $\ell\in \V'$; in addition, the Young inequality
gives $\cE^*(\ell,\ell)\leq \| \ell \|_{(\Vhom{\cE})'}^2$.

(b) The uniform concavity of $g\mapsto \langle\ell,g\rangle-\frac 12\cE(g,g)$ shows that $\cE(f_n-f_m,f_n-f_m)\to 0$
as $n,\,m\to\infty$. By definition of $\Vhom\cE$ this means that $(f_n)$ is convergent in $\Vhom\cE$. Eventually we use
the continuity of $\langle\ell,\cdot\rangle$ in $\Vhom\cE$ to conclude that the first identity in \eqref{eq:264} holds. {\color{black}
The second identity follows immediately from $2(\cE(f,f))^{1/2}(\cE^*(\ell,\ell))^{1/2}\geq \cE(f,f)+\cE^*(\ell,\ell)$.}

(c) By \eqref{eq:599}, $\DeltaE $ can also be seen as an operator from $\widetilde\V$ to $\Vdual\cE$, with
$\|\DeltaE f\|_{\Vdual\cE }\leq (\cE(f,f))^{1/2}$ for all $f\in \widetilde\V$. It extends therefore to the completion $\Vhom{\cE}$ of $\widetilde\V$.
Denoting by $\DeltaE _\cE$ the extension, let us prove that $-\DeltaE _{\cE}$ is the Riesz isomorphism.

We first prove that $-\DeltaE _{\cE}$ is onto; this follows easily proving that, for given $\ell\in \Vdual\cE $, the maximizer
$f\in\Vhom\cE$ given by (b) satisfies $\ell=-\DeltaE_\cE f$.
Since $\widetilde\V$ is dense in $\Vhom\cE$, we obtain that
\begin{equation}
\label{eq:63}
\Big\{\ell\in \Vdual{\cE}:\ \ell=-\DeltaE  f\ \text{for some }f\in \widetilde\V\Big\}\quad
\text{is dense in }\Vdual\cE .
\end{equation}
Computing $\cE^*(-\DeltaE f,-\DeltaE f)$ for $f\in\widetilde\V$ and using the definition of $\cE$ immediately gives 
$\cE^*(-\DeltaE f,-\DeltaE f)=\cE(f,f)$. By density, this proves that $-\DeltaE_\cE$ is the Riesz isomorphism.

(d) When $\ell=-\DeltaE g$ for some $g\in\widetilde\V$ it follows by polarization of the identity $\cE^*(-\DeltaE h,-\DeltaE h)=\cE(h,h)$, already
mentioned in the proof of (c). The general case follows by \eqref{eq:63}.
\end{proof}
We can summarize the realization in (a) by writing
\begin{equation}
  \label{eq:238}
  \Vdual\cE =D(\cE^*)=\big\{\ell\in \V':\ |\langle \ell,f\rangle|\le 
  C\sqrt {\cE(f,f)}\quad \forevery f\in \V\big\}.
\end{equation}
According to this representation and
the identification $\H=\H'$, 
$f\in\H$ 
belongs to $\Vdual1$ if and only if there exists a constant $C$ such that 
$$
\biggl|\int_X fg\,\d\mm\biggr|\leq C\Bigl(\cE(g,g)\Big)^{1/2}\qquad\forall g\in\V.
$$
If this is the case, we shall write $f\in\H\cap\Vdual1$.
\begin{remark}[Identification of Hilbert spaces]
  \label{rem:careful}
  \upshape
  In the usual framework
  of the variational formulation of parabolic
  problems, one usually considers
  a Hilbert triple as in \eqref{eq:266}
  $V\subset H\equiv H'\subset V'$ so that 
  the duality pairing $\langle \ell,f\rangle$ between $V'$ and $V$
  coincides
  with the scalar product in $H$ whenever $\ell\in H$. 
  In this way
  the definition of the domain $\D$ of $\DeltaE $ as 
  in \eqref{eq:264} 
  makes sense.
  In the case of $\Vhom\cE$, $\Vdual\cE$ 
  one has to be careful 
  that $\Vhom\cE$ is not generally imbedded in
  $\H$ and therefore $\H$ is not imbedded in $\Vdual\cE$,
  unless $\cE$ is coercive with respect to the $\H$-norm;
  it is then possibile to consider the intersection $\H\cap\Vdual\cE $
  (which can be better understood as $\H'\cap \Vdual\cE )$.
  Similarly, 
  $\V$ is imbedded in $\Vhom\cE$ 
  if and only if $\Vdual\cE $ is dense in $\V'$, 
  and this happens if and only if 
  $N=\{0\}$, i.e.\ $\cE$ is a norm on $\V$. 
\end{remark}
The following lemma will be useful.
\begin{lemma}\label{lem:Alica}
The following properties of the spaces $\H$, $\V$ and $\Vdual1$ hold.
\begin{itemize}
\item[(a)] A function $f\in\H$ belongs to $\V$ if and only if 
\begin{equation}\label{eq:Alica1}
\biggl|\int_X f\DeltaE g\,\d\mm\biggr|\leq C\Big(\cE(g,g)\Big)^{1/2}\qquad\forall g\in\D.
\end{equation}
\item[(b)] $\{\DeltaE f:\ f\in \D\}$ is dense in $\Vdual 1$ and, in particular,
$\H\cap\Vdual1$ is dense in $\Vdual1$.
\end{itemize}
\end{lemma}
\begin{proof} 
(a) \AAA If $f\in\V$ we can integrate by parts and conclude via Cauchy-Schwartz inequality by choosing $C= \cE(f,f)^{1/2}$. To show the converse implication first of all note that the property \eqref{eq:Alica1} is stable under the action of the semigroup $(\sfP_t)_{t\ge0}$. Thus we can argue by approximation by observing that if $f \in\D$ one can choose  $g=f$; then integrate by parts 
on the left hand side 
to obtain $\cE(f,f)\leq C^2$. \fn
 
(b)  Let us consider an element $\ell\in \Vdual\cE$ such that 
\begin{displaymath}
  \cE^*(\ell,\DeltaE f)=0\quad\forevery\,f\in \D.
\end{displaymath}
Applying Proposition~\ref{prop:allduals}(d) we get
\begin{displaymath}
  \duality{}{}\ell f=0\quad\forevery\,f\in \D.
\end{displaymath}
Since $\ell\in \V'$ and $\D$ is dense in $\V$ we conclude that 
$\ell=0$.
\end{proof}

\subsection{Nonlinear diffusion}
\label{subsec:A-1}

The aim of this section is to study evolution equations of the form
\begin{equation}
  \label{eq:270}
  \frac\d{\d t}\varrho-\DeltaE P(\varrho)=0,
\end{equation}
where $P:\R\to\R$ is a \emph{regular monotone} nonlinearity 
satisfying 
  \begin{equation}
  \label{eq:A1}
  P\in \rmC^1(\R),\quad
  P(0)=0,\quad
  0<\sfa\le \sfP'(r)\le \sfa^{-1}  \quad
  \forevery r\ge0.
\end{equation}
The results are more or less standard application of
the abstract theory of monotone operators and variational evolution
equations in Hilbert spaces \cite{Brezis70,Brezis71,Brezis71b,Brezis73},
with the only caution described in Remark~\ref{rem:careful}
and the use of a general Markov operator
instead of a particular realization given by a
second order elliptic differential operator.

If $H_0,\,H_1$ are Hilbert spaces continuously imbedded in a common Banach space $B$ and $T>0$ 
is a given final time, we introduce the spaces of time-dependent functions
\begin{equation}
  \label{eq:A16}
  W^{1,2}(0,T;H_1,H_0):=\Big\{u\in W^{1,2}(0,T;B):\
  u\in L^2(0,T;H_1),\,\,\dot u\in L^2(0,T;H_0)\Big\},
\end{equation}
endowed with the norm
\begin{equation}
  \label{eq:A17}
  \|u\|^2_{W^{1,2}(0,T;H_1,H_0)}:=
  \|u\|^2_{L^2(0,T;H_1)}+\|\dot u\|^2_{L^2(0,T;H_0)}.
\end{equation}
Denoting by $(H_0,H_1)_{\vartheta,2}$,
$\vartheta\in (0,1)$,
the family of (complex or real, \cite[2.1 and Thm.~15.1]{Lions-Magenes72}, \cite[1.3.2]{Triebel78})
Hilbert interpolation spaces, the equivalence with the so-called \emph{trace Interpolation method}
\cite[Thm. 3.1]{Lions-Magenes72}, \cite[1.8.2]{Triebel78},
shows that 
\begin{equation}
  \label{eq:11}
  \text{if}\quad H_1\hookrightarrow H_0\quad
  \text{then}\quad
  W^{1,2}(0,T;H_1,H_0)\hookrightarrow \rmC([0,T]; (H_0,H_1)_{1/2,2} ),
\end{equation}
with continuous inclusion.

As a possible example, we will consider
$W^{1,2}(0,T;\V,\Vdual\cE)$ (in this case
$\V$ and $\Vdual\cE$ are continuously imbedded in $\V'$)
and $W^{1,2}(0,T;\D,\H)$. Since
\cite[Prop.~2.1]{Lions-Magenes72} 
$$
\Vdual\cE\subset \Vdual{},\quad 
(\V,\Vdual{})_{1/2,2}=\H,\qquad
\text{and}\qquad
(\D,\H)_{1/2,2}=\V,$$
we easily get
\begin{equation}
  \label{eq:12}
  W^{1,2}(0,T;\V,\Vdual\cE)\hookrightarrow
  \rmC([0,T];\H),\qquad
  W^{1,2}(0,T;\D,\H) \hookrightarrow
  \rmC([0,T];\V),
\end{equation}
Let us fix a regular function $P$ according to \eqref{eq:A1}:
we introduce the set
\begin{equation}
  \label{eq:137}
  \ND 0T:=\Big\{\varrho\in W^{1,2}(0,T;\H)
  \cap \rmC^1([0,T];\Vdual\cE):\ P(\varrho)\in L^2(0,T;\D)\Big\}.
\end{equation}
Notice that
\begin{equation}
  \label{eq:138}
  \ND 0T\subset \rmC([0,T];\V).
\end{equation}
Indeed, if $\varrho\in \ND 0T$ then by the chain rule $P(\varrho)\in W^{1,2}(0,T;\D,\H)$, 
so that $P(\varrho)\in \rmC([0,T];\V)$
thanks to \eqref{eq:12}. Composing
with the Lipschitz map $P^{-1}$ 
provides the continuity of $\varrho$ in $\V$
thanks to \eqref{eq:Dir-Lip}. 
\begin{theorem}[Nonlinear diffusion]
  \label{thm:nonlin-diff}
  Let $P$ be a regular function according to \eqref{eq:A1}.
  For every $T>0$ and every 
  $\bar\varrho\in \H$ there exists a unique curve 
  $\varrho=\sfS\bar\varrho\in W^{1,2}(0,T;\V,\Vdual\cE)$ 
  satisfying
  \begin{equation}
    \label{eq:75}
    \frac\d{\d t}\varrho-\DeltaE P(\varrho)=0\quad\text{$\Leb 1$-a.e.\
      in $(0,T)$, with $\varrho_0=\bar \varrho$.}
  \end{equation}
  Moreover:
  \begin{enumerate}[\rm (ND1)]
  \item For every $t>0$ the map 
    $\bar\varrho\mapsto \sfS_t\bar\varrho$ is a
    contraction with respect to the norm 
    $\Vdual1$, with
    \begin{equation}
      \label{eq:51}
      \|\sfS_t\bar \varrho^1-\sfS_t\bar\varrho^2\|^2_{\Vdual1}+
      2\sfa\, \int_0^t \int_X |\sfS_r\bar\varrho^1-\sfS_r\bar
      \varrho^2|^2\,\d\mm\,\d r\le 
      \|\bar\varrho^1-\bar\varrho^2\|_{\Vdual1}^2.
    \end{equation}
  \item If $W\in \rmC^{1,1}(\R)$ is a nonnegative convex 
    function with $W(0)=0$, then
    \begin{equation}
      \label{eq:123}
      \int_X W(\varrho_t)\,\d\mm
      +\int_0^t \cE(P(\varrho_r),W'(\varrho_r))
      \,\d r=
      \int_X W(\bar\varrho)\,\d\mm
      \quad\forall t\geq 0.
    \end{equation}
    Moreover, for every convex and lower semicontinuous function
    $W:\R\to [0,\infty]$ 
    \begin{equation}
      \label{eq:122}
      \int_X W(\varrho_t)\,\d\mm\le \int_X W(\bar\varrho)\,\d\mm.
    \end{equation}
    In particular, $\sfS_t$ is positivity preserving and
    if $0\le \bar\varrho\le R$ $\mm$-a.e.\ in $X$, then $0\le\varrho_t\le R$
    $\mm$-a.e.\ in $X$ for every $t\ge0$.
    \item If $\bar\varrho\in \V$ then 
      $\varrho \in \ND 0T\subset \rmC([0,T];\V)\cap
      \rmC^1([0,T];\Vdual\cE)$
  and 
    \begin{equation}
      \label{eq:76}
      \lim_{h\rightarrow 0}\frac 1h\big(\varrho_{t+h}-\varrho_t\big)=
      \DeltaE P(\varrho_t)\quad\text{strongly in $\Vdual1$,%
        \quad for all $t\geq 0$.}
    \end{equation}
  \item The maps $\sfS_t$, $t\ge0$, are contractions in $L^1\cap
      L^2(X,\mm)$ w.r.t. the $L^1(X,\mm)$ norm 
      \GGG and they can be uniquely extended to 
      a $\rmC^0$-semigroup of contractions 
      in $L^1(X,\mm)$ (still denoted by $(\sfS_t)_{t\ge0}$).
      \EEE
      For every $\bar\varrho_i\in L^1 
      (X,\mm)$, $i=1,2$,
      \begin{equation}
        \label{eq:274}
        \int_X (\sfS_t \bar\varrho_2-\sfS_t\bar\varrho_1)_+\,\d\mm
        \le 
        \int_X (\varrho_2-\varrho_1)_+\,\d\mm\quad
        \forevery t\ge0.
      \end{equation}
      In particular $\sfS$ is order preserving, 
      i.e.
      \begin{equation}
        \label{eq:276}
        \bar\varrho_1\le \bar\varrho_2\quad \Rightarrow\quad
        \sfS_t\bar\varrho_1\le \sfS_t\bar\varrho_2\quad\forevery t\ge0.
      \end{equation}
      \AAA Moreover, if $\bar{\varrho}\in L^{\infty}(X,\mm)$ with bounded support, then $\sfS_t \bar{\varrho} \to $
      Finally, if $\sfP_t$ is mass preserving then
    \begin{equation}
      \label{eq:222}
      \int_X \sfS_t \bar\varrho\,\d\mm=
      \int_X \bar\varrho\,\d\mm\quad
      \forevery \bar\varrho\in L^1
      (X,\mm),\quad
      t\ge0.
    \end{equation}
    \end{enumerate}
\end{theorem}
We split the proof of the above theorem in various steps.
First of all, we introduce the primitive function of $P$,
\begin{equation}
  \label{eq:74}
  \Prim(r):=\int_0^r P(z)\,\d z,
\end{equation}
which, thanks to \eqref{eq:A1}, satisfies
\begin{equation}\label{eq:74_bis}
  \frac {\sfa}2r^2\le \Prim(r)\le \frac 1{2\sfa} r^2\qquad\forall r\geq 0.
\end{equation}
We adapt to our setting the approach 
of \cite{Brezis71}, showing that the nonlinear equation
\eqref{eq:75} can be viewed as a gradient flow
in the dual space $\Vdual1$ driven by 
the integral functional $\Phim:\V'\to[0,\infty]$ defined by
\begin{equation}\label{eq:54}
\Phim(\sigma):=
  \begin{cases}
    \displaystyle
    \int_X\Prim(\sigma)\,\d\mm&\text{if }\sigma\in 
    \H,\\
    +\infty&\text{if }\sigma\in \V'\setminus\H,
  \end{cases}
\end{equation}
associated to $\Prim$. 

Since $\H$ is not included in $\Vdual1$ in general,
if $\varrho$ is a solution of \eqref{eq:75} with an arbitrary 
$\bar\varrho\in \H$ only the difference $\sigma_t:=\varrho_t-\bar\varrho$,
will belong to $\Vdual1$; therefore it is useful to introduce the family of shifted functionals
$\Phim_\eta:\Vdual{}\to[0,\infty]$, $\eta\in\H$, defined by
\begin{equation}
  \label{eq:54bis}
  \Phim_\eta(\sigma):=\Phim(\eta+\sigma),\quad
  \qquad
  \forevery \sigma\in\Vdual{}.
\end{equation}
Notice that, thanks to \eqref{eq:74_bis},
\GGGG $\Phim_\eta$ is finite
on $\H$. Dealing with subdifferentials and evolutions in
$\Vdual1$, we consider the restriction of $\Phim_\eta$ to $\Vdual1$, with
$D(\Phim_\eta):= \Vdual1\cap \H$ 
\EEE and 
we shall denote by $\partial\Phim_\eta(\cdot)$ the 
$\cE^*$-subdifferential of $\Phim_\eta$, defined at any $\sigma\in D(\Phim_\eta)$
as the collection of all $\ell\in\Vdual1$ satisfying
$$
\cE^*(\ell,\zeta-\sigma)\leq
\Phim_\eta(\zeta)-\Phim_\eta(\sigma)\qquad\forall\zeta\in D(\Phim_\eta).
$$
In the next lemma we characterize the subdifferentiability and the subdifferential of $\Phim_\eta$.

\begin{lemma}[Subdifferential of $\Phim_\eta$]
  \label{le:A1}
  For every $\eta\in \H$ 
  the functional 
  $\Phim_\eta:\Vdual1\to[0,\infty]$ \EEE 
  defined by
  \eqref{eq:54bis} 
  is convex and lower semicontinuous. Moreover, for every $\sigma\in D(\Phim_\eta)$ we have
  \begin{equation}
    \label{eq:118}
    \ell\in \partial\Phim_\eta(\sigma)
    \quad\Longleftrightarrow\quad
    P(\sigma+\eta)\in \V,\quad
    \ell=-\DeltaE P(\sigma+\eta).
  \end{equation}
  In particular $\partial\Phim_\eta$ is single-valued in its domain and $D(\partial\Phim_\eta)=\{\sigma\in
  \H:\ 
  P(\sigma+\eta)\in \V\}.$
\end{lemma}
\begin{proof}
  The convexity of $\Phim_\eta$ is clear. The lower semicontinuity is also easy to prove, since $\Phim_\eta(\sigma_n)\leq C<\infty$
  and $\sigma_n\to \sigma$ weakly in $\Vdual1$ imply that $\sigma\in\H$ and $\sigma_n$ weakly converge
  to $\sigma$ in $\H$, by the weak compactness of $(\sigma_n)$ in the weak topology of $\H$. 
  
  The left implication $\Leftarrow$ in \eqref{eq:118} is immediate, since
  by Proposition~\ref{prop:allduals}(d) and the fact that $\zeta-\sigma\in \H\cap \Vdual1$
  \begin{align*}
    \cE^*(-\DeltaE P(\sigma+\eta),\zeta-\sigma)&=
    \int_X P(\sigma+\eta)(\zeta-\sigma)\,\d\mm=
    \int_X
    P(\sigma+\eta)\,((\zeta+\eta)-(\sigma+\eta))\,\d\mm
    \\&\le 
    \int_X \Big(\Prim(\zeta+\eta)-\Prim(\sigma+\eta)\Big)\,\d\mm
    = \Phim_\eta(\zeta)-\Phim_\eta(\sigma),
  \end{align*}
  where we used the pointwise property $P(x)(y-x)\le
  \Prim(y)-\Prim(x)$
  for every $x,\,y\in \R$.
  
  In order to prove the converse implication $\Rightarrow$, let us suppose
  that $\ell\in\partial\Phim_\eta(\sigma)$; choosing $\zeta=\sigma+\eps \varphi$, with $\varphi\in \H\cap \Vdual1$, 
  we get
  \begin{displaymath}
    \cE^*(\ell,\varphi)\le
    \eps^{-1}\Big(\Phim_\eta(\sigma+\eps\varphi)-\Phim_\eta(\sigma)\Big)\le 
    \int_X P(\sigma+\eta+\eps\varphi)\varphi\,\d\mm. 
  \end{displaymath}
  Passing to the limit as $\eps\down0$ and changing $\varphi$ into $-\varphi$ we get
  \begin{displaymath}
    \cE^*(\ell,\varphi)=\int_X P(\sigma+\eta)\,\varphi\,\d\mm  
    \qquad\forevery \varphi\in \H\cap \Vdual1.
  \end{displaymath}
  Choosing now $\varphi=-\DeltaE f$ with $f\in\D$ we get
  \begin{displaymath}
    -\int_X P(\sigma+\eta)\DeltaE f\,\d\mm\le
    \|\ell\|_{\Vdual1}\Bigl(\cE(f,f)\Bigr)^{1/2},
  \end{displaymath}
  so that Lemma~\ref{lem:Alica}(a) yields $P(\sigma+\eta)\in \V$. Therefore (using 
   Proposition~\ref{prop:allduals}(d) once more in the last equality), we get 
  \begin{eqnarray*}
    \cE^*(\ell,-\DeltaE f)&=&
    -\int_X P(\sigma+\eta)\DeltaE f\,\d\mm=
    \cE(P(\sigma+\eta),f)\\&=&
    -\langle\DeltaE P(\sigma+\eta),f\rangle=
    \cE^*(-\DeltaE P(\sigma+\eta),-\DeltaE f)
  \end{eqnarray*}
  for all $f\in\D$, and this proves that $\ell$ coincides with $-\DeltaE P(\sigma+\eta)$.
\end{proof}

\begin{proof}[Proof of Theorem \ref{thm:nonlin-diff}]
  Let $\bar\varrho\in\H$ and let $\eta\in\H$ be any element such that 
  $\bar\sigma:=\bar\varrho-\eta\in \Vdual1$ (in particular we can choose
  $\eta=\bar\varrho$, so that $\bar\sigma=0$; as a matter of fact, $\eta$ plays only an
  auxiliary role in the proof and the solution $\varrho$ will be independent of $\eta$).
  Setting $\sigma_t:=\varrho_t-\eta$, the equation \eqref{eq:75} is
  equivalent to 
  \begin{equation}
    \label{eq:119}
    \frac\d{\d t}\sigma-\DeltaE P(\sigma+\eta)=0,\quad\text{i.e.}\quad \frac
    \d{\d t}\sigma+\partial\Phim_\eta(\sigma)\ni0,\quad
    \text{with}\quad
    \sigma_0=\bar\sigma,
  \end{equation}
  where $\partial\Phim_\eta$ is the subdifferential of $\Phim_\eta$, characterized in
  \eqref{eq:118}.
  
  \noindent
  {\it Proof of existence of solutions and (ND1).} Since Lemma~\ref{lem:Alica}(b) provides the density of the domain of $\Phim_\eta$ in $\Vdual1$,
  existence of a solution $\sigma\in \rmC([0,T];\Vdual1)$ satisfying 
  $\DeltaE P(\sigma+\eta),\frac \d{\dt}\sigma\in L^2(0,T;\Vdual1)$ (and thus $P(\sigma+\eta)\in L^2(0,T;\V)$) 
  follows by the general theory of equations in Hilbert spaces
  governed by the subdifferential of convex and lower semicontinuous functions 
  \cite{Brezis71}, so that $\varrho_t:=\sigma_t+\eta$ satisfies \eqref{eq:75}. 
  Since $P(\varrho)\in L^2(0,T;\V)$ and $P$ satisfies the regularity property \eqref{eq:A1},
  we also get $\varrho\in L^2(0,T;\V)$; since $\frac \d\dt \varrho\in {L^2(0,T;\Vdual1)}\subset L^2(0,T;\V')$
  we deduce $\varrho\in \rmC([0,T];\H)$ by \eqref{eq:12}.

  The abstract theory also provides the regularization estimates
  \begin{equation}
    \label{eq:121}
    t \,\| \DeltaE P(\sigma+\eta)  \|_{\Vdual 1}^2=
    t \, \cE (P(\varrho_t),P(\varrho_t))\le \int_X \Prim(\bar
    \varrho)\,\d\mm\qquad\forevery t>0,
    \end{equation}
    \begin{equation}
    \label{eq:121bis}
    \lim_{h\down0}\frac{\varrho_{t+h}-\varrho_t}h=\DeltaE P(\varrho_t)
    \text{ in }\Vdual1\quad\forevery t>0,
  \end{equation}
  and the fact that the semigroup $\sfS_t:\bar \rho\mapsto \varrho_t$ is nonexpansive in $\Vdual1$.
  If $\bar\varrho\in \V$ (so that $\bar\sigma\in D(\partial\Phim_\eta)$)
  the limit in \eqref{eq:121bis} holds also at $t=0$.
  Since $\partial\Phim_\eta$ is single-valued, this proves \eqref{eq:76}. 
  
  In order to prove \eqref{eq:51} we simply consider two solutions
  $\varrho^j_t=\sigma^j_t+\eta$, $j=1,\,2$ (we can choose the same $\eta$ since $\bar \rho^1-\bar
  \rho^2\in\Vdual1$), and we evaluate the time
  derivative of $\frac 12\cE^*(\varrho^1_t-\varrho^2_t)$, obtaining
  \begin{align}
    \notag\frac \d{\dt} \frac 12\cE^*(\varrho^1_t-\varrho^2_t)&=
    \frac \d{\dt} \frac 12\cE^*(\sigma_t^1-\sigma_t^2)=
    \cE^*(\sigma_t^1-\sigma_t^2, \DeltaE P(\varrho_t^1)-\DeltaE P(\varrho_t^2))
    \\&=
    \label{eq:132}
    -\int_X (\varrho_t^1-\varrho_t^2)(P(\varrho_t^1)-P(\varrho_t^2))\,\d\mm
    \\&\notag
    \le -\sfa \|\varrho_t^1-\varrho_t^2\|_{L^2(X,\mm)}^2,
  \end{align}
  where $\sfa$ is the constant in \eqref{eq:A1}.
  
  \noindent {\it Proof of (ND2).}
  \GGG We consider the perturbed function $W^\eps(r):=W(r)+\eps V(r)
  $, $r\in \R$, $\eps>0$, \EEE and 
  we can apply Lemma~\ref{le:A1} to the integral functional
  $\mathcal W_\eta^\eps$ defined similarly to $\Phim$, with $W^\eps$ instead
  of $V$; by denoting by $G$ the derivative of $W$ 
  \GGG and by $G^\eps(r):=G(r)+\eps P(r)$ the derivative of $W^\eps$, 
  \EEE
  the $\Vdual1$-subdifferential $\partial\mathcal W_\eta^\eps$ 
  can then be represented as $-\DeltaE
  G^\eps(\sigma+\eta)$ as in \eqref{eq:118} and its domain is contained in
  $D(\partial\Phim_\eta)$. If $\sigma $ is a solution of \eqref{eq:119}, the chain rule
  for convex and lower semicontinuous functionals 
  in Hilbert spaces yields
  \begin{align*}
    \frac\d\dt \int_X W^\eps(\varrho_t)\,\d\mm
    &=
    \frac \d{\dt }\mathcal W_\eta(\sigma_t)=
    -\cE^*(\frac \d{\dt}\sigma_t,\DeltaE G^\eps(\sigma_t+\eta))=
    -\cE^*(\DeltaE P(\sigma_t+\eta),\DeltaE G^\eps(\sigma_t+\eta))
    \\&=
        -\cE(P(\varrho_t),G^\eps(\varrho_t)).
  \end{align*}
  \GGG We can eventually integrate with respect to time and pass to
  the limit as $\eps\down0$ to obtain \eqref{eq:123}. \EEE
  
  The inequality \eqref{eq:122} follows now by \eqref{eq:Dir-Mon} and 
  by a standard approximation procedure,
  e.g. by considering the Moreau-Yosida regularization of $W$.
 Choosing now 
  $W(r):=(r-R)_+^2$ with $R\ge0$ 
  or $W(r):=(R-r)_+^2$ with $R\le 0$,
  we prove the comparison estimates w.r.t.~constants.
  
  \noindent
  {\it Proof of (ND3).} We already proved \eqref{eq:76};
  let us now show that
  $\frac\d\dt \varrho\in L^2(0,T;\H)$ if $\bar\varrho\in \V$.
  This property follows easily by \eqref{eq:51} applied to the couple of
  solutions
  $\varrho^1_t:=\varrho_t$ and $\varrho^2_t:=\varrho_{t+h}$, since it yields
  \begin{align*}
    \frac {2\sfa}{h^2}&\int_0^{T-h} \|\varrho_t-\varrho_{t+h}\|_{L^2(X,\mm)}^2\,\d t\le 
    \frac1{h^2}\|\varrho_h-\varrho_0\|_{\Vdual1}^2\le
                        \GGG \Big(\frac 1h\int_0^h\big\|\tfrac \d\dt
                        \varrho\big\|_{\Vdual1}\,\d t\Big)^2\\
                      &\GGG =
                        \Big(\frac 1h\int_0^h\big\|\DeltaE
                        P(\varrho_t)
                        \big\|_{\Vdual1}\,\d t\Big)^2\le
        \big\|\DeltaE P(\bar\varrho) \EEE
        \big\|^2_{\Vdual1}=
\cE(P(\bar\varrho),P(\bar\varrho))\quad 
    \forevery h\in (0,T),
  \end{align*}
  \GGG where we used the fact that the map $t\mapsto 
  \big\|\DeltaE
  P(\varrho_t)
  \big\|_{\Vdual1}$ is nonincreasing. \EEE
  The regularity 
  $\frac\d{\d t}\varrho\in L^2(0,T;\H)$ 
  yields $P(\varrho)\in L^2(0,T;\D)$
  \GGG and therefore 
  $P(\varrho)\in W^{1,2}(0,T;\D,\H)$, so that 
  the map $t\mapsto P(\varrho_t)$ belongs 
  to $\rmC([0,T];\V)$ by \eqref{eq:12}.
  The differential equation \eqref{eq:35}
  then yields $\varrho\in \rmC^1([0,T];\Vdual1).$
  \EEE

  \noindent
  {\it Proof of (ND4).} For every $\tau>0$ and $\varrho\in \H$ let us consider the resolvent
  equation 
  \begin{equation}
    \label{eq:292}
    \text{find }\varrho'\in H\text{ with }P(\varrho')\in \D\text{
      such that }
    \varrho'-\tau \DeltaE P(\varrho')=\varrho.
  \end{equation}
  By introducing the resolvent operators 
  $\sfJ_{\tau,\eta}:\Vdual1\to D(\partial\Phim_\eta)$, $\tau>0$ and
  $\eta\in \H$, defined by 
  $\sfJ_{\tau,\eta}:=(I+\tau\partial\Phim_\eta)^{-1}$, Lemma
  \ref{le:A1} shows that whenever $\varrho-\eta\in \Vdual1$ 
  a solution $\varrho'\in \H$ with $\varrho'-\eta\in \Vdual1$ can be
  obtained by setting
  \begin{equation}
    \label{eq:226}
    \varrho':=\sfJ_{\tau,\eta}(\varrho-\eta)+\eta.
  \end{equation}
  In particular, the choice $\eta:=\varrho$ ensures the existence of a
  solution to \eqref{eq:292}. We will show that the solution
  $\varrho'$ of
  \eqref{eq:292} is in
  fact unique and independent of the choice of $\eta$ in \eqref{eq:226}.
  More precisely, we will show that if a couple $\varrho_i'\in \H$,
  $i=1,2$, solves \eqref{eq:226} with data $\varrho_i\in \H$ one has
  \begin{equation}
    \label{eq:275}
    \int_X \big(\varrho_1'-\varrho_2'\big)_+\,\d\mm
    \le 
    \int_X \big(\varrho_1-\varrho_2\big)_+\,\d\mm
    \quad \forevery \varrho_1,\,\varrho_2\in \H.
  \end{equation}
  \EEE
  The monotonicity inequality
  \eqref{eq:275}
  can be proved by 
  introducing an increasing sequence  
  of smooth maps approximating
  the Heaviside 
  function:
  \begin{displaymath} 
    f_n\in \rmC^1(\R;[0,1]),\quad
    f_n\equiv 0\quad\text{in $(-\infty,0)$},\quad
    0<f_n'(r)\le n,\quad
    f_n(r)\uparrow 1\quad
    \forevery r>0.
  \end{displaymath}
  \GGGG Since 
  $P(\varrho_i')\in \D$ and 
  $f_n$ is Lipschitz with $f_n(0)=0$,
  $f_n(P(\varrho_1')-P(\varrho_2'))\in L^2
  \cap L^\infty(X,\mm)\cap\V$.
  We thus get by \eqref{eq:226} and the positivity of $f_n$
  \begin{align*}
    \int_X &
    (\varrho_1'-\varrho_2')f_n\big(P(\varrho_1')-P(\varrho_2')\big)\,\d\mm
    +\tau \cE\Big(
    f_n\big(P(\varrho_1')-P(\varrho_2')\big),
    P(\varrho_1')-P(\varrho_2')\Big)
       \\&= 
    \int_X
    (\varrho_1-\varrho_2)f_n\big(P(\varrho_1')-P(\varrho_2')\big)\,\d\mm
    \le \int_X \big(\varrho_1-\varrho_2\big)_+\,\d\mm.
  \end{align*}
  By neglecting the positive contribution
  of the Dirichlet form $\cE$ thanks to \eqref{eq:Dir-Mon},
  we can pass to the limit as $n\to\infty$ 
  by the monotone 
  convergence theorem 
  observing that 
  $(\varrho_1'-\varrho_2')
  f_n\big(P(\varrho_1')-P(\varrho_2')\big)
  \up (  \varrho_1'-\varrho_2')_+$ as $n\to\infty$; when
  $(\varrho_1-\varrho_2)_+\in L^1(X,\mm)$ we thus obtain
  \eqref{eq:275}.
  
  Recalling \eqref{eq:226}
  %
  %
  %
  and the exponential formula $\sfS_t(\bar\varrho)=
  \eta+\lim_{n\to\infty}(\sfJ_{t/n,\eta})^n(\bar\varrho-\eta)$
  strongly in $\Vdual1$ and weakly in $\H$ for some $\eta\in \H$ with $\bar\varrho-\eta\in \Vdual1$,
  we obtain 
  \eqref{eq:274}, the $L^1$-contraction of 
  and the order preserving property
  \eqref{eq:276}.

  \GGGG
  Let us now consider the operator
  \begin{equation}
    A:\varrho\mapsto -\DeltaE P(\varrho) \quad\text{defined in 
      $D(A):=\big\{\varrho\in L^1\cap L^2(X,\mm): \DeltaE P(\varrho)\in
  L^1\cap L^2(X,\mm)\big\}$}
\label{eq:109}
\end{equation}
\EEE
and its multivalued 
  extension obtained by taking the closure of its graph in
  $L^1(X,\mm)$:
  \begin{equation}
    \label{eq:236}
    \bar A\varrho:=\Big\{\xi\in L^1(X,\mm): 
    \exists\,\varrho_n\in D(A):\ \varrho_n\to \varrho,\ 
    A\varrho_n\to\xi\quad\text{in }L^1(X,\mm)\Big\}.
  \end{equation}
  If $\varrho\in D(A)$ it is easy to check  by \eqref{eq:275} that $\bar
  A\varrho=\{A\varrho\}$ 
  and the resolvent $\bar \sfJ_\tau:=(I+\tau\bar A)^{-1}$ 
  of $\bar A$ coincides 
  with the map $\sfJ_\tau:\varrho\to\varrho'$ induced by
  \eqref{eq:292} on $L^1\cap L^2(X,\mm)$. 
  Since $L^1\cap L^2(X,\mm)$
  is dense in $L^1(X,\mm)$, it follows by
  \eqref{eq:275} that $\bar A$ is an $m$-accretive operator in
  $L^1(X,\mm)$.
  By Crandall-Liggett Theorem the limit 
  $\bar \sfS_t(\varrho):=\lim_{n\to\infty} (\bar \sfJ_{t/n})^{n}\varrho$ 
  exists in the strong topology of $L^1(X,\mm)$ uniformly on $[0,T]$
  and provides the unique extension of $(\sfS_t)_{t\ge0}$ to
  continuous 
  semigroup of contractions in $L^1(X,\mm)$. In particular, for every $\varrho\in
  L^1\cap L^2(X,\mm)$ the sequence
  $(\sfJ_{t/n})^{n}\varrho$ converges strongly to $\sfS_t(\varrho)$ 
  in $L^1(X,\mm)$. 
  
  In order to check the mass preserving property \eqref{eq:222}
  in the case when $\sfP$ is mass preserving, 
  it is therefore sufficient to prove that $\sfJ_\tau$ is mass
  preserving on $L^1\cap L^2(X,\mm)$, i.e. \EEE
  \begin{equation}\label{eq:227}
    \int_X \varrho'\,\d\mm=
    \int_X \varrho\,\d\mm\quad\text{whenever 
      \eqref{eq:226} holds.}    
  \end{equation}
  Eventually, \eqref{eq:227} follows by integrating 
 \eqref{eq:226} 
  and recalling \eqref{eq:228}. 
 \end{proof}
\GGGG For later use, we fix some of the results obtained in the last
part of the above
proof in the next Theorem.
\begin{theorem}
  \label{thm:m-accretive}
  Let $P$ be a regular nonlinearity according to \eqref{eq:A1}; the
  operator $\bar A$ 
  defined by \eqref{eq:236} and \eqref{eq:109} is $m$-accretive in
  $L^1(X,\mm)$ with dense domain, its resolvent $\sfJ_\tau:=(I+\tau
  \bar A)^{-1}$ is a contraction
  satisfying \eqref{eq:275} for every $\varrho_i'=\sfJ_\tau\varrho_i$.
  For every $\varrho \in L^1\cap L^2(X,\mm)$ $\sfJ_\tau\varrho$ provides the
  unique solution $\varrho'$ of \eqref{eq:292} and
  the solution $\varrho_t=\sfS_t\varrho$ of \eqref{eq:75} can be obtained by the
  exponential formula $\varrho_t=\lim_{n\to\infty}\sfJ_{t/n}^n
  \varrho$ as strong limit in $L^1(X,\mm)$.
\end{theorem}
\EEE
We only considered nonlinear diffusion
problems associated to regular monotone functions $P$ 
as in \eqref{eq:A1}, since they provide a sufficiently general
class of equations for our aims. Nevertheless, 
starting from Theorem~\ref{thm:nonlin-diff} 
and adapting its arguments, it would not be difficult 
to prove existence and uniqueness results 
under more general assumptions.
The next result is a possible example in this direction:
a proof can be obtained by the same strategy (we omit the details, since we need only
Theorem~\ref{thm:nonlin-diff} in the sequel); notice that the fact that $\sfS_t$ preserves
$L^\infty$ bounds allows to modify 
the behaviour of $P$ for large densities, so that 
its primitive function $V$ has a quadratic growth
and its domain coincides with $L^2(X,\mm)$
when $\mm(X)<\infty$.

\begin{theorem}[Nonlinear diffusion for
  general nonlinearities]
  \label{thm:more-general-P}
  Let $P\in \rmC^0(\R)$ be 
  a nondecreasing function and let us suppose that 
  $\mm(X)<\infty.$ 
  For every $\bar\varrho\in L^\infty(X,\mm)$ 
  there exists a unique curve $\varrho=\sfS\bar\varrho
  \in W^{1,2}(0,T;\Vdual\cE)
  \cap L^\infty(X\times (0,T))$ with
  $P(\varrho)\in L^2(0,T;\V)$ satisfying 
  \eqref{eq:75}.
  $(\sfS_t)_{t\ge0}$ is a semigroup of contractions
  in $\Vdual\cE$ and in $L^1(X,\mm)$ and 
  properties {\upshape (ND2), (ND4)} still hold.
\end{theorem}

\section{Backward and forward linearizations of nonlinear diffusion} 
\label{subsec:ND}

In this section we collect a few results
concerning linearization of the nonlinear diffusion equations 
of the form studied by Theorem~\ref{thm:nonlin-diff}.

The linearized PDE discussed in the next proposition corresponds to \eqref{eq:73} of the heuristic Section~\ref{subsec:heu},
while the evolution semigroup is provided by the nonlinear diffusion
equation of Theorem~\ref{thm:nonlin-diff}. 
Recall the notation 
$W^{1,2}(0,T;\D,\H)=L^2(0,T;\D)\cap W^{1,2}(0,T;\H)$, 
\eqref{eq:137} for $\ND 0T$,
and that, according to \eqref{eq:12} and \eqref{eq:138},
\begin{equation}
  \label{eq:279}
  \ND 0T\subset \rmC([0,T];\V) ,\qquad
  W^{1,2}(0,T;\D,\H)\hookrightarrow \rmC([0,T];\V).
\end{equation}
\begin{theorem}[Backward adjoint linearized equation]
  \label{prop:backward-linearization}
  Let $P$ be a regular monotone nonlinearity as 
  in \eqref{eq:A1} and let $\varrho\in L^2(0,T;\H)$.

  For every $\bar\varphi\in \V$,  $T>0$ and $\psi\in L^2(0,T;\H)$  there exists
  a unique strong solution
  $\varphi\in \Sobolev T\D\H{}$ of 
  \begin{equation}
    \label{eq:82}
    \frac\d\dt \varphi+P'(\varrho)\DeltaE\varphi=\psi,\qquad
    \varphi_T=\bar\varphi.
  \end{equation}
  \begin{enumerate}[\rm ({BA}1)]
  \item For all $r\in [0,T]$, the solution $\varphi$ satisfies
    \begin{equation}
      \label{eq:44}
      \int_r^T \int_X \frac 1{P'(\varrho)}|\dot\varphi|^2\,\d\mm\,\dt
      +\frac12 \cE(\varphi_r,\varphi_r)=
      \int_r^T \int_X \frac 1{P'(\varrho)}\psi\dot\varphi\,\d\mm\,\dt
      +\frac 12 \cE(\bar\varphi,\bar\varphi).
    \end{equation}
    \item If $\bar\varphi\in L^\infty(X,\mm)$ and $\psi\equiv0$, then 
    $\varphi_t\in L^\infty(X,\mm)$ with
    $|\varphi_t|\le \|\bar\varphi\|_{L^\infty(X,\mm)}$ $\mm$-a.e. in $X$ for every $t\in [0,T]$.
  \item If $\varrho^n\to \varrho^\infty,\ \psi^n\to\psi^\infty$ in $L^2(0,T;\H)$,
      $\bar\varphi^n\to \bar\varphi^\infty$ in $\V$ and $\varphi^n$,
      $n\in \N\cup\{\infty\}$, are the
    corresponding solutions of \eqref{eq:82}, 
    then $\varphi^n\to
    \varphi^\infty$ strongly in $W^{1,2}(0,T;\D,\H)$.
  \end{enumerate}
\end{theorem}

\begin{remark}[Forward adjoint linearized equation]
  \label{rem:forward-linear}
  \upshape
  By time reversal, the previous
  Theorem is equivalent to the 
  analogous result 
  for the forward linearized equation
  \begin{align}
    \label{eq:82for}
    \frac\d\dt \zeta-P'(\varrho)\DeltaE\zeta=\psi\in L^2(0,T;\H),
    \qquad
    \zeta_0=\bar\zeta\in \V,
  \end{align}
  that admits a unique solution 
  $\zeta\in \Sobolev T\D\H{}$.  
\end{remark}

The following lower semicontinuity result will often be useful.

\begin{lemma}
  \label{le:sitrovaaltrove?}
  Let $Y$ be a Polish space endowed with a 
  nonnegative $\sigma$-finite Borel measure $\nn$, 
  let $w_n\in L^2(Y,\nn)$ and $Z_n\in L^\infty(Y,\nn)$, $Z_n\ge0$.
  If $w_n\weakto w$ in $L^2(X,\nn)$ and
  $Z_n\to Z$ pointwise $\nn$-a.e.\ in $Y$, then
  \begin{equation}
    \label{eq:134}
    \liminf_{n\to\infty}\int_Y Z_n|w_n|^2\,\d\nn\ge
    \int_Y Z|w|^2\,\d\nn.
  \end{equation}
\end{lemma}
\begin{proof}
  Let us first assume that $\nn(Y)<\infty;$ 
  by Egorov's Theorem, for every $\delta>0$
  we can find a $\nn$-measurable  set $B_\delta\subset Y$ such that 
  $\nn(Y\setminus B_\delta)\le \delta$ and 
  $Z_n\to Z$ uniformly on $B_\delta$. Since $\|w_n\|_{L^2(Y,\nn)}\le C$
  independent of $n$ we obtain
  \begin{align*}
    \liminf_{n\to\infty}&
    \int_Y Z_n|w_n|^2\,\d\nn\ge
    \liminf_{n\to\infty}
    \int_{B_\delta} Z_n|w_n|^2\,\d\nn
    \\&\ge
    -C^2 \limsup_{n\to\infty}\|Z_n-Z\|_{L^\infty(B_\delta,\nn)}+
    \liminf_{n\to\infty}
    \int_{B_\delta} Z|w_n|^2\,\d\nn\ge
    \int_{B_\delta} Z|w|^2\,\d\nn.
  \end{align*}
  By letting $\delta\down0$ we obtain \eqref{eq:134}.
  When $\nn(Y)=\infty$, since $\nn$ is $\sigma$-finite, 
  we can find an increasing sequence
  $Y_k\uparrow Y$ of Borel sets with $\nn(Y_k)<\infty$.
  By the previous claim, we get
  \begin{align*}
    \liminf_{n\to\infty}&
    \int_Y Z_n|w_n|^2\,\d\nn\ge
    \liminf_{n\to\infty}
    \int_{Y_k} Z_n|w_n|^2\,\d\nn
    \ge     \int_{Y_k} Z|w|^2\,\d\nn
  \end{align*}
  for every $k\in \N$. As $k\to\infty$ we recover \eqref{eq:134}.
\end{proof}

\begin{proof}[Proof of Theorem~\ref{prop:backward-linearization}] 
  Let us fix the final time $T$ and set
  $\alpha_t:=P'(\varrho_{T-t})$, $g_t:=\psi_{T-t}$.
  We can thus consider the forward equation
  \begin{equation}
    \label{eq:124}
    \frac \d{\dt} f_t-\alpha_t\DeltaE f_t=g_t\quad\text{in }(0,T),\qquad
    f_0=\bar f=\bar\varphi,
  \end{equation}
  where $\alpha$ is a Borel map satisfying (with $\sfa$ the positive constant in \eqref{eq:A1})
  \begin{equation}
    \label{eq:125}
    0<\sfa\le \alpha\le \frac 1\sfa\qquad \mm\otimes \Leb 1\text{-a.e.\
      in }X\times (0,T).
  \end{equation}
  In order to solve \eqref{eq:124} we use a 
  piecewise constant (in time) discretization 
  of the coefficients $\alpha_t$: we introduce a uniform partition
  of the time interval $(0,T]$ of step 
  $\tau:=T/N$ given by the intervals 
  $I^N_k:=((k-1)\tau,k\tau]$, $k=1,\ldots,N$ and
  we set
  $$\alpha^N_k:= \frac 1\tau \int_{I_k^N} \alpha_r\,\d r, \qquad
  \bar
  \alpha^N_t:=
  \alpha^N_k\quad\text{if $t\in I^N_k$},
  $$
  so that $\sfa\le \bar\alpha^N\le \sfa^{-1}$.
  Applying standard result for evolution
  equation in Hilbert spaces (in particular
  we write the PDE as the gradient flow of $\cE$ w.r.t. the $L^2(X,1/\alpha^N_k\mm)$ norm when $g\equiv 0$
  and in the inhomogeneous case we use Duhamel's principle) we can find
  recursively strong
  solutions $f^N_k\in W^{1,2}(I^N_k;\D,\H)$ of 
  \begin{equation}
    \label{eq:127}
    \frac 1{\alpha^N_k}\frac \d{\d t}f^N_k-\DeltaE f^N_k=\frac
    1{\alpha^N_k}g\quad
    \text{in }I^N_k,
    \qquad f^N_k((k-1)\tau))=f^N_{k-1}((k-1)\tau),
  \end{equation}
  with the convention $f^N_0(0)=\bar f$. Defining the function $f^N(t):=f^N_k(t)$ if $t\in I^N_k$, we easily
  check that 
  $f^N\in W^{1,2}(0,T;\D,\H)$, that $f^N$ is a strong solution of the differential equation
  \begin{equation}
    \label{eq:126}
    \frac \d{\dt} f^N-\bar\alpha^N\DeltaE f^N=g\quad\text{in }(0,T),
  \end{equation}
  and that it satisfies the apriori energy dissipation identity
  \begin{equation}
    \label{eq:45}
    \int_0^s \int_X \frac
    1{\bar\alpha^N}\Big|\frac\d\dt f^N\Big|^2\,\d\mm\,\d t+
    \frac 12 \cE(f^N_s,f^N_s)=
    \int_0^s \int_X \frac
    1{\bar\alpha^N} g\,\frac\d\dt f^N\,\d\mm\,\d t+\frac 12\cE(\bar f,\bar f).
  \end{equation}
  Since $1/\bar\alpha^N\geq \sfa$ and $\bar f\in\V$, this shows in particular that $f^N$ is
  uniformly bounded in $W^{1,2}(0,T;\D,\H)$. 
  Since $\bar\alpha^N\to\alpha$ in $L^2(0,T;\H)$ 
  we can then easily pass to the limit as $N\to\infty$
  (see also the more detailed argument below), 
  obtaining
  \eqref{eq:124}.
  Since \eqref{eq:124} holds in the strong form, we can also write it as
  $$
  \frac{1}{\alpha_t}\frac \d{\dt} f_t-\DeltaE f_t=\frac{g_t}{\alpha_t}\quad\text{in }(0,T),\qquad
  $$
  and then the energy identity corresponding to \eqref{eq:44} follows by multiplying both sides by $\d f_t/\dt$.
  This proves (BA1).

  When $g\equiv0$ and $\bar f\in L^\infty(X,\mm)$ satisfies $|\bar f|\le F$
  $\mm$-a.e.\ in $X$, 
  a standard truncation argument based on \eqref{eq:127} 
  yields the recursive estimate
  \begin{displaymath}
    \|f^N_{k}(k\tau)\|_{L^\infty(X,\mm)}\le 
    \|f^N_k(t)\|_{L^\infty(X,\mm)}\le
    \|f^N_{k-1}((k-1) \tau)\|_{L^\infty(X,\mm)}
    \quad
    \text{for $t$ in }I^N_k,
  \end{displaymath}
   and therefore $|f^N(t)|\le F$ $\mm$-a.e.\ in $X$ for every $t\in
  [0,T]$; this estimate passes to the limit as $N\to\infty$
  providing the statement (BA2).

  Let us now prove the last statement (BA3); we thus consider a 
  sequence $\alpha^n$ satisfying the uniform bounds
  $\sfa\le \alpha^n\le \sfa^{-1}$ 
  and the limit $\alpha^n\to \alpha^\infty$ $\mm\otimes \Leb 1$-a.e.\
  in $X\times (0,T)$, and corresponding solutions $f^n$
  of
  \begin{equation}
    \label{eq:128}
    \frac \d\dt f^n-\alpha^n\DeltaE f^n=g^n,\qquad f^n(0)=\bar f^n,
  \end{equation}
  with $\bar f^n\to\bar f^\infty$ strongly in $\V$ and 
  $g^n\to g^\infty$ strongly in $L^2(0,T;\H)$.
  Using the energy identity \eqref{eq:44} it is easily seen that
  $(f^{n})$ is bounded in $W^{1,2}(0,T;\H)$ and in $\rmC([0,T];\V)$;
  we can also use the PDE \eqref{eq:128}
  to show that $(f_n)$ is bounded  in $L^2(0,T;\D)$. 
  Hence, possibly extracting a suitable subsequence (still denoted by $f^n$),
  we can assume that  
  $f^n\weakto f^\infty$ in $W^{1,2}(0,T;\D,\H)$,
  so that 
  $\frac \d\dt f^n\weakto \frac \d\dt f^\infty$ and
  $\DeltaE f^n\weakto \DeltaE f^\infty$
  in $L^2(0,T;\H)$. 
  Since for every $s\in [0,T]$ 
  the linear operator $f\mapsto f(s)$ 
  is continuous from $W^{1,2}(0,T;\D,\H)$
  to $\V$ thanks to \eqref{eq:279}, 
  we also obtain 
  the weak continuity property 
  $f^n(s)\weakto f^\infty(s)$ in $\V$ for every 
  $s\in [0,T]$.
  In particular $f^\infty$ satisfies \eqref{eq:128} 
  with $n=\infty$.
   
  {\color{black} Taking also Lemma~\ref{le:sitrovaaltrove?} into account,} it follows that 
  \begin{displaymath}
    \liminf_{n\to\infty} \cE(f^n(s),f^n(s))\ge \cE(f^\infty(s),f^\infty(s)),
  \end{displaymath}
  \begin{displaymath}
    \liminf_{n\to\infty} 
    \int_0^s \int_X \frac 1{\alpha^n}
    \Big|\frac\d\dt f^n\Big|^2\,\d\mm\,\d r\ge
    \int_0^s \int_X \frac 1{\alpha^\infty}
    \Big|\frac\d\dt f^\infty\Big|^2\,\d\mm\,\d r
  \end{displaymath}
  and
  \begin{displaymath}
    \lim_{n\to\infty} \cE(\bar f^n,\bar f^n)= \cE(\bar f^\infty, \bar f^\infty),\quad
    \lim_{n\to\infty} 
    \int_0^s \int_X \frac {g^n}{\alpha^n} \frac\d\dt f^n\,\d\mm\,\d r=
    \int_0^s \int_X \frac {g^\infty}{\alpha^\infty} \frac\d\dt f^\infty\,\d\mm\,\d r,
  \end{displaymath}
  so that by \eqref{eq:44} we obtain
  \begin{align*}
    \lim_{n\to\infty} 
    \Big( \int_0^s \int_X \frac 1{\alpha^n}\Big|\frac\d\dt f^n\Big|^2\,\d\mm
    &+
    \frac 12\cE(f^n(s),f^n(s))\Big)
    =
    \int_0^s \int_X \frac {g^\infty}{\alpha^\infty}  \frac\d\dt
    f^\infty\,\d\mm\,\d r
    +\frac 12 \cE(\bar f^\infty,\bar f^\infty)
    \\&=
    \int_0^s \int_X \frac 1{\alpha^\infty}
    \Big|\frac\d\dt f^\infty\Big|^2\,\d\mm\,\d r+
    \frac 12\cE(f^\infty(s),f^\infty(s)).
  \end{align*}
  We conclude (see Remark~\ref{rem:trick} below)
  that $\sqrt{\frac 1{\alpha^n}}\frac \d\dt f^n\to \sqrt{\frac 1{\alpha^\infty}}\frac \d\dt f^\infty$ 
  strongly in $L^2(0,T;\H)$, {\color{black} so that we can use the strong convergence and the uniform 
  boundedness from below of $\alpha^n$ to conclude that}
  $f^n\to f^\infty$ strongly in $W^{1,2}(0,T;\D,\H)$.
\end{proof}
\begin{remark} \label{rem:trick} 
  \upshape 
  We will repeatedly use the following 
  simple property, valid for 
  sequences $(a_n)$, $(b_n)$ of nonnegative real numbers: 
  if 
  $$ 
  \liminf_{n\to\infty}a_n\ge a,\quad\liminf_{n\to\infty} b_n\ge b,\quad 
  \limsup_{n\to\infty} \,(a_n+b_n)\le (a+b), 
  $$ 
  then 
  $$ 
\lim_{n\to\infty}a_n=a\qquad\text{and}\qquad\lim_{n\to\infty}b_n=b. 
  $$ 
\end{remark} 
 The next proposition provides existence and regularity for the linearization of the nonlinear diffusion
equation of Theorem~\ref{thm:nonlin-diff}.

In the statement we will make use of the space $\D'$, the dual of $\D$, and
\begin{equation}
  \label{eq:70}
  \Ddual:=\Big\{\ell\in \D':\ |\langle \ell,f\rangle|\le C\|\DeltaE
  f\|_{\H}\quad\forevery f\in \D\Big\}.
\end{equation}
Since 
$\D\hookrightarrow^{ds}\V$ we have
$\H\hookrightarrow^{ds}\V'\hookrightarrow^{ds}\D'$
with continuous and dense inclusions; 
the duality pairing between $\D'$ and $\D$ is an extension
of the one between $\V'$ and $\V$ and of the scalar product in $\H$,
and we will still denote it as $\duality{}{}\cdot\cdot$
whenever no misunderstanding are possible. 
Denoting by $\|\ell\|_{\Ddual}$ the least constant $C$ in
\eqref{eq:70}, $\Ddual$ is also a Hilbert space, precisely it can be identified with 
the dual of the pre-Hilbert space one obtains endowing $\D$ with the norm $\|\DeltaE f\|_{\H}$, smaller
than the canonical norm of $\D$. Arguing as in Section~\ref{subsec:completion},
we can and will identify $\Ddual$ with the finiteness domain in $\D'$
of the lower semicontinuous functional 
\begin{equation}
  \label{eq:277}
  \frac 12 \|\ell\|_{\Ddual}^2:=
  \sup_{f\in\D}\duality{}{}\ell f -\frac 12 \int_X |\DeltaE f|^2\,\d\mm.
\end{equation}
By duality, any element $h\in\H$ induces an element $\DeltaE h\in\Ddual$,
via the relation 
$$\duality{\D'}\D{\DeltaE h} f=\int_X h \;\DeltaE f\,\d\mm.$$
We shall also make use of the space
$W^{1,2}(0,T;\H,\Ddual)$, fitting in our framework because both $\H$
and $\Ddual$ embed into the space $\D'$. 
Since $\Ddual\hookrightarrow \D'$ 
and the duality formula for complex interpolation yields
 $(\H,\D')_{1/2}=\V'$, \eqref{eq:11} yields 
\begin{equation}
  \label{eq:278}
  W^{1,2}(0,T;\H,\Ddual)\hookrightarrow 
  W^{1,2}(0,T;\H,\D')\hookrightarrow
  \rmC([0,T];\V').
\end{equation}
\begin{theorem}[Forward linearized equation]
  \label{thm:reg2}
  Let $P$ be a regular monotone nonlinearity as 
  in \eqref{eq:A1} and let $\rho\in L^2(0,T;\H)$. 
  \begin{enumerate}[\rm (L1)]
  \item    \label{item:reg2.2}
    For every $\bar w\in \Vdual1$, $T>0$ there exists a unique solution
    $w\in W^{1,2}(0,T;\H,\Ddual)$ of 
    \begin{equation}
      \label{eq:81}
      \frac\d\dt w=\DeltaE(P '(\varrho)w),\qquad
      w_0=\bar w
    \end{equation}
    in the weak formulation (recall \eqref{eq:279} and \eqref{eq:278}) 
    \begin{equation}
      \label{eq:52}
      \duality{\V'}{\V}{w_s}{\vartheta_s}
      -\int_{0}^{s} \int_X 
    \Big(\partial_t \vartheta_t + P'(\varrho_{t}) \DeltaE\vartheta_t\Big)
    w_{t}\,\d\mm\,\d t=
    \duality{\V'}{\V}{\bar w}{\vartheta_0} 
    \quad
    \forall s\in [0,T],
    \end{equation}
    for every $\vartheta\in W^{1,2}(0,T;\D,\H)$. In addition, 
    the function $w$ satisfies
    \begin{equation}
      \label{eq:131}
      \int_0^t \int_X P'(\varrho_r)|w_r|^2\,\d\mm\,\d r+
      \frac 12 \|w_t\|_{\Vdual1}^2=\frac 12\|\bar w\|_{\Vdual1}^2\qquad \forall t\in [0,T]
    \end{equation}
    and, for every solution $\varphi$ of \eqref{eq:82} with $\psi\equiv0$ one has
    \begin{equation}
      \label{eq:130}
      \duality{\V'}\V{w_t}{\varphi_t}=
      \duality{\V'}\V{\bar w}{\varphi_0}.
    \end{equation}
    \item 
      \label{item:reg2.3}
      If $\bar w=\DeltaE \bar\zeta$ 
      for some $\bar \zeta\in \V$, then 
      $w_t=\DeltaE \zeta_t$ for every $t\in [0,T]$, where
      $\zeta\in W^{1,2}(0,T;\D,\H)$ is the solution of
      \eqref{eq:82for} with $\psi\equiv0$.
    \item If $\varrho^n\to \varrho^\infty$ in $L^2(0,T;\H)$,
      $\bar w^n\to \bar w^\infty$ in $\Vdual1$ and $w^n$,
      $n\in \N\cup\{\infty\}$, are the
    corresponding solutions of \eqref{eq:81}, 
    then $w^n\to
    w^\infty$ strongly in $W^{1,2}(0,T;\H,\Ddual)$.
  \end{enumerate}
\end{theorem}

\begin{proof}[Proof of Theorem~\ref{thm:reg2}] 
  Let us first show the second claim: 
  if $w_t=\DeltaE \zeta_t\in W^{1,2}(0,T;\H;\Ddual)$
  for the solution $\zeta\in W^{1,2}(0,T;\D;\H)$ of \eqref{eq:82for} 
  and if $\vartheta$ is any function in $W^{1,2}(0,T;\D,\H)$ we have
  \begin{align*}
    \duality{\V'}{\V}{w_t}{\vartheta_t}=
    \duality{\V'}\V{ \DeltaE\zeta_t}{\vartheta_t}=
    -\cE(\zeta_t,\vartheta_t),
  \end{align*}
  so that $t\mapsto \duality{\V'}\V{ w_t}{\vartheta_t}$ 
  is absolutely continuous in $[0,T]$ and 
  for $\Leb 1$-a.e.\ $t\in (0,T)$ 
  its derivative is given by
  \begin{align*}
    -\frac \d{\dt}\cE(\zeta_t,\vartheta_t)&=
    \int_X \Big(\DeltaE \zeta_t \,\dot\vartheta_t+
    \DeltaE \vartheta_t \,\dot\zeta_t\Big)\,\d\mm=
    \int_X w_t\Big(\dot \vartheta_t+
    P'(\varrho_t)\,\DeltaE\vartheta_t\Big)
    \,\d\mm.
  \end{align*}
  A further integration in time yields \eqref{eq:52}.
  In this case \eqref{eq:131} is a consequence of 
  \eqref{eq:44} with $\psi\equiv0$, by noticing that 
  \begin{displaymath}
    \cE(\zeta_t,\zeta_t)=\cE^*(w_t,w_t)=\|w_t\|^2_{\Vdual1},\qquad
    \dot\zeta_t=P'(\varrho_t)w_t.
  \end{displaymath}
  The uniqueness of the solution to \eqref{eq:52} is clear
  thanks to \eqref{eq:130}.
  
  The general result stated in the first claim for arbitrary 
  $\bar w\in \Vdual1$ follows by the linearity of the problem,
  the estimate \eqref{eq:131}, and the density 
  of the set $\{\DeltaE \bar\zeta:\ \bar \zeta\in \V\}$ 
  in $\Vdual1$, see Lemma~\ref{lem:Alica}(b).

  The proof of (L3) is completely analogous to the proof
  of (BA3) in Theorem~\ref{prop:backward-linearization}:
  the weak convergence of $w^n$ to $w^\infty$ in
  $W^{1,2}(0,T;\H,\Ddual)$ follows by the a priori estimate 
  \eqref{eq:131}, the linearity of the problem w.r.t. $w$ for given $\varrho$
  and the uniqueness of its solution. 
  Strong convergence can then be obtained
  by standard lower semicontinuity arguments and 
  Remark~\ref{rem:trick}, by passing to the limit in \eqref{eq:131}.
\end{proof}

\begin{theorem}[Perturbation properties]
 \label{thm:precise-limit}
 Let us suppose that 
 $P$ 
 is a regular monotone nonlinearity as in \eqref{eq:A1}.
 Let $\bar\varrho_\eps:=
 \bar\varrho+\eps\bar w_\eps$ with $\bar\varrho,\,\bar\varrho_\eps\in
 L^2\cap  L^\infty (X,\mm)$, $\bar\varrho_\eps$ uniformly bounded in
 $L^2\cap L^\infty(X,\mm)$, and $\bar w_\eps\to \bar w$ strongly
 in $\Vdual1$ as $\eps\down0$. Let
 $\varrho_{\eps,t}$ (resp. $\varrho_t$) be the solutions provided by
 Theorem~\ref{thm:nonlin-diff} with initial datum $\bar\varrho_\epsilon$ (resp. $\bar\varrho$) and set
 $$w_{\eps,t}:=\frac{\varrho_{\eps,t}-\varrho_t}{\eps}.$$
 Then for every $t\ge0$
 there exists the limit $\lim_{\eps\down0}w_{\eps,t}=w_t$
 strongly in $\Vdual1$, the limit function $w$ belongs to  $W^{1,2}(0,T;\H,\Ddual)$  and satisfies \eqref{eq:81}.
\end{theorem}
\begin{proof} 
  By the Lipschitz estimate \eqref{eq:51} of 
  $\sfS:\Vdual\cE\to L^2(0,T;\H)\cap L^\infty(0,T;\Vdual\cE)$ 
  we know that $(w_\eps)$ is bounded in $L^2(0,T;\H)$ and in
  $L^\infty(0,T;\Vdual1)$, in particular this gives $\varrho_\varepsilon\to\varrho$ in $L^2(0,T;\H)$. 
  We can then find a subsequence
  $\eps_n\down0$ such that $w_{\eps_n}\to w$ weakly in $L^2(0,T;\H)$ 
  and weakly$^*$ in
  $L^\infty(0,T;\Vdual1)$.
   
  Since 
  $P\in \rmC^1(\R)$ and there exists a constant $R>0$ 
  such that $|\varrho_\eps|\le R,\ |\varrho|\le R$,
  we can use the inequalities (depending on the parameter $\delta>0$
  and on the fixed constant $R$)
 \begin{equation}
   \label{eq:35}
   \big|P(\varrho_\eps)-P(\varrho)-P'(\varrho)(\varrho_\eps-\varrho)\big|\le
   \delta|\varrho_\eps-\varrho|+C_\delta|\varrho_\eps-\varrho|^2,
 \end{equation}
 and the uniform bound of  $\eps^{-1}(\varrho_\eps-\varrho)$ 
 in $L^2(0,T;\H)$ to obtain 
 \begin{equation}
   \label{eq:36}
   \eps_n^{-1}\big(P(\varrho_{\eps_n})-P(\varrho)\big)\weakto P'(\varrho)w\quad
   \text{weakly in }L^2(0,T;\H).
 \end{equation}
In fact, since $P$ is Lipschitz,
$\eps^{-1}(P(\varrho_\eps)-P(\varrho))$ 
is also uniformly bounded in $L^2(0,T;\H)$ 
thus we can use a bounded test function $\zeta\in L^2(0,T;\H)$
to characterize the weak limit in \eqref{eq:36}.
For such a test function, denoting by $E$ an upper bound 
of $\eps^{-1}\|\varrho_\eps-\varrho\|_{L^2(0,T;\H)} $, 
we have
\begin{align*}
  \Big|\int_0^T\int_X &
  \Big(\frac{P(\varrho_\eps)-P(\varrho)}\eps
 -P'(\varrho)\frac{\varrho_\eps-\varrho}\eps\Big)\zeta\,\d\mm\,\d
 t\Big|
 \le \delta\,
 E\,\|\zeta\|_{L^2(0,T;\H)}+
  \eps\, C_\delta\,E^2\,\sup|\zeta|
\end{align*}
 thus showing \eqref{eq:36} as $\delta>0$ is arbitrary.
 
 Let us now consider for every $t>0$ and $\bar\varphi\in \V$ the
 solution $\varphi$
 of \eqref{eq:82} with final condition $\varphi_t=\bar\varphi$ and arbitrary $\psi\in L^2(0,T;\H)$,
 thus satisfying (by the Leibniz rule)
 \begin{displaymath}
   \int_X w_{\eps,t}\bar\varphi\,\d\mm=
   \eps^{-1}\int_0^t
   \int_X \Big(\big(P(\varrho_{\eps,r})-P(\varrho_r)\big)\DeltaE\varphi_r+
   (\varrho_{\eps,r}-\varrho_r)\dot \varphi_r\Big)\,\d\mm\,\d r+
   \int_X \bar w_\eps \varphi_0\,\d\mm.
 \end{displaymath}
 Since $\dot \varphi,\,\DeltaE\varphi\in L^2(0,T;\H)$ we obtain that
 for every $t>0$ the sequence $(w_{\eps_n,t})$ converges weakly in $\V'$
 (and thus in $\Vdual1$, since it is uniformly bounded in $\Vdual1$) 
 and the limit $\hat w_t$ will satisfy
 \begin{equation}
   \label{eq:43}
   \langle \hat w_t,\bar\varphi\rangle=
   \int_0^t \int_X w_r\psi_r \,\d\mm\,\d r+
   \int_X \bar w \varphi_0\,\d\mm.
 \end{equation}
 Choosing in particular $\psi\equiv0$,
 the previous formula identifies the limit, so that
 $\hat w_t=w_t$ for $\Leb 1$-a.e.\ $t\in (0,T)$
 and moreover the limit does not depend on the particular subsequence
 $(\eps_n)$. Since $\psi$ is arbitrary, we also get
 that $w$ satisfies \eqref{eq:81}
 in the weak sense of \eqref{eq:52}.

 In order to prove strong convergence of $w_\eps$ to $w$
 in $\Vdual1$ for every $t\in [0,T]$ and
 in $L^2(0,T;\H)$, we start from \eqref{eq:132}
 written for $\varrho^1:=\varrho$ and $\varrho^2:=\varrho_\eps$.
 Since
 \begin{displaymath}
   \liminf_{\eps\down0}\frac 1{\eps^2}\cE^*(\varrho(t)-\varrho_\eps(t),\varrho(t)-\varrho_\eps(t))=
   \liminf_{\eps\down0}\cE^*(w_\eps(t),w_\eps(t))\ge \cE^*(w(t),w(t)),
 \end{displaymath}
 and the limit $w$ satisfies \eqref{eq:131},
 by the argument of Remark~\ref{rem:trick} it is sufficient to prove
 that
 \begin{equation}
   \label{eq:133}
   \liminf_{\eps\down0} \frac 1{\eps^2}\int_0^t \int_X
   (\varrho-\varrho_\eps)(P(\varrho)-P(\varrho_\eps))\,\d\mm\,\d s\ge
   \int_0^t \int_X P'(\varrho)|w|^2\,\d\mm\,\d s.
 \end{equation}
 Setting
 \begin{displaymath}
   Z_\eps:=
   \begin{cases}
\displaystyle{\frac{P(\varrho)-P(\varrho_\eps)}{\varrho-\varrho_\eps}}&
     \text{if }\varrho\neq\varrho_\eps\\
     P'(\varrho)&\text{if }\varrho=\varrho_\eps,
   \end{cases}
 \end{displaymath}
 we obtain a family of nonnegative and uniformly bounded
 functions that satisfies $Z_{\eps_n}\to P'(\varrho)$ $\mm\otimes \Leb 1$-a.e.\ in $X\times (0,T)$
 whenever $\varrho_{\eps_n}\to\varrho$ $\mm\otimes \Leb 1$-a.e.\ in $X\times (0,T)$.
 On the other hand
 \begin{displaymath}
\frac1{\eps^2}(\varrho-\varrho_\eps)(P(\varrho)-P(\varrho_\eps))=
   Z_\eps |w_\eps|^2.
 \end{displaymath}
 We conclude by applying Lemma~\ref{le:sitrovaaltrove?} to a subsequence $(\eps_n)$ on which the $\liminf$ in
 \eqref{eq:133} is attained and convergence $\mm\otimes \Leb 1$-a.e.\ in $X\times (0,T)$ holds.

\end{proof}

Let $\varrho\in \ND0T$ be the solution provided by Theorem~\ref{thm:nonlin-diff} with initial datum $\bar\varrho\in \V$.
By applying Theorem~\ref{thm:precise-limit} to the difference quotients
$$
\frac{1}{\eps}(\varrho_{t+\eps}-\varrho_t)
$$
and using the strong differentiability of $t\mapsto\varrho_t$  with respect to $\Vdual1$ (see \eqref{eq:76}) we obtain
the following corollary.
\begin{corollary}
  \label{cor:rho-derivative}
  Let $\varrho\in \ND0T$ be the solution provided by
  Theorem~\ref{thm:nonlin-diff} with initial datum $\bar\varrho\in
   \V\cap L^\infty(X,\mm)$. 
  Then $w:=\frac \d\dt \varrho$ is 
  a solution to \eqref{eq:52}, 
  with initial datum $\bar w=\DeltaE P(\bar\varrho)$.
\end{corollary}

\part{Continuity equation and curvature conditions in metric measure spaces}

\section{Preliminaries}

\subsection{Absolutely continuous curves, Lipschitz functions and
  slopes}
\label{subsec:AC}
Let $(X,\sfd)$ be a complete metric space,
possibly extended (i.e.~the distance $\sfd$ can take the value $+\infty$).
A curve $\gamma:[a,b]\to X$ 
belongs to $\AC p{[a,b]}{(X,\sfd)}$, $1\leq p\leq\infty$, if
there exists $v\in L^p(a,b)$ such that 
\begin{equation}
  \label{eq:47o}
  \sfd(\gamma(s),\gamma(t))\le \int_s^t v(r)\,\d r\quad\forevery a\le
  s\le t\le b.
\end{equation}
We will often use the shorter notation 
$\AC p{[a,b]}{X}$ whenever the choice of the distance $\sfd$ will be
clear from the context.
The metric velocity of $\gamma$, defined by
\begin{equation}
  \label{eq:46o}
  |\dot \gamma|(r):=\lim_{h\to0}\frac{\sfd(\gamma(r+h),\gamma(r))}{|h|},
\end{equation}
exists for $\Leb 1$-a.e.\ $r\in (a,b)$, belongs to $L^p(a,b)$,
and provides the minimal function $v$, up to $\Leb{1}$-negligible sets, such that \eqref{eq:47o} holds.
We set 
\begin{equation}
  \label{eq:157}
  \cA_p(\gamma):=
  \begin{cases}
    \displaystyle \int_a^b|\dot\gamma|^p(r)\, \d r&\text{if }\gamma\in \AC
    p{[a,b]}X,\\
    +\infty&\text{otherwise.}
  \end{cases}
\end{equation}
Notice that $\sfd^p(\gamma(a),\gamma(b))\le (b-a)^{p-1}\cA_p(\gamma)$.

A continuous function $\gamma:[0,1]\to X$ is a length minimizing constant speed curve if 
$\cA_1(\gamma)=\sfd(\gamma(0),\gamma(1))=
|\dot \gamma|(t)$ for $\Leb 1$-a.e.~$t\in (0,1)$, or, equivalently, if
$\cA_p(\gamma)=\sfd^p(\gamma(0),\gamma(1))$ for 
some (and thus every) $p>1$. In the sequel, by geodesic we always mean
a length minimizing constant speed curve.

The extended metric space $(X,\sfd)$ is a \emph{length space} if
\begin{equation}
  \label{eq:21}
  \sfd(x_0,x_1)=\inf\Big\{\cA_1(\gamma):
  \gamma\in \AC{}{[0,1]}X,\ \gamma(i)=x_i
  \Big\}\quad
  \forevery x_0,x_1\in X.
\end{equation}
The collection of all Lipschitz  
real functions defined in $X$ will be denoted by
$\Lip(X)$, while $\Lip_b$ will denote the subspace of bounded Lipschitz functions.

 The slopes $|\rmD^\pm\varphi|$, the local Lipschitz constant $|\rmD\varphi|$ and
 the asymptotic Lipschitz constant $|\rmD^*\varphi|$ of $\varphi\in \Lip_b(X)$
are respectively defined by
\begin{gather}
  \label{eq:5o}
  |\rmD^\pm\varphi|(x):=\limsup_{y\to x}\frac
  {\big(\varphi(y)-\varphi(x)\big)_\pm}{\sfd(y,x)},\quad
  |\rmD\varphi|(x):=\limsup_{y\to x}\frac
  {|\varphi(y)-\varphi(x)|}{\sfd(y,x)},
  \\
  |\rmD^*\varphi|(x):=\limsup_{{y,z\to x}\atop{y\neq z}}\frac
  {|\varphi(y)-\varphi(z)|}{\sfd(y,z)}=
  \lim_{r\down0} \Lip(f,B_r(x)),
  \intertext{%
    with the convention that all the above quantities are $0$ if $x$ is an
    isolated point.
    Notice that $|\rmD^*\varphi|$ is an u.s.c.\ function
    and that, whenever $(X,\sfd)$ is a length space,}
  \label{eq:18}
  |\rmD^*\varphi|(x) =\limsup_{y\to x}|\rmD \varphi|(y),  
  \qquad{ \Lip(\varphi)=\sup_{x\in X}|\rmD\varphi(x)|=
    \sup_{x\in X}|\rmD^*\varphi(x)|.}
\end{gather}
{\color{black} For $\varphi\in\Lip_b(X)$ we shall also use the upper gradient property
\begin{equation}\label{def:upper_gradient}
|\varphi(\gamma(1))-\varphi(\gamma(0))|
\leq\int_0^1|\rmD^*\varphi|(\gamma(t))|\dot\gamma(t)|\,\d t
\end{equation}
whose proof easily follows by approximating $|\rmD^*\varphi|$ from above with the Lipschitz constant in 
balls and then estimating the derivative of the absolutely continuous map $\varphi\circ\gamma$.}

\subsection{The Hopf-Lax evolution formula}\label{sec:5.2}

Let us suppose that $(X,\sfd)$ is a metric space; 
the Hopf-Lax evolution map
$\sfQ_t:\rmC_b(X)\to \rmC_b(X)$, $t\geq 0$, is defined by $\sfQ_0f=f$ and
\begin{equation}\label{eq:11o}
  \sfQ_t f(x):=\inf_{y\in X} f(y)+\frac {\sfd^2(y,x)}{2t}\qquad t>0.
\end{equation}
We shall need the pointwise {properties}  
\begin{equation}
  \label{eq:211}
  \inf_X f\le \sfQ_t f(x)\le \sup_X f\quad\forevery x\in X,\ t\ge0,
\end{equation}
\begin{equation}\label{eq:identities0}
  -\frac{\d^+}{\d t}\sfQ_tf(x)\ge 
  \frac 12 |\rmD^* \sfQ_tf|^2(x)
  \quad\forevery x\in X, \ t\ge0
\end{equation}
(these are proved in Proposition~3.3 and Proposition~3.4
of \cite{AGS11c}, $\d^+/\d t$ denotes the right derivative). 

When $(X,\sfd)$ is a length space $(\sfQ_t)_{t\ge0}$ is a semigroup and we have the refined identity
\cite[Thm.~3.6]{AGS11a}
\begin{equation}\label{eq:identities}
-\frac{\d^+}{\d t}\sfQ_tf(x)=\frac{1}{2}|\rmD \sfQ_t f|^2(x)\quad
\forevery x\in X,\ t>0.
\end{equation} 
Inequality \eqref{eq:identities0} and the length property of $X$ yield the a priori bounds
\begin{equation}\label{eq:aprioriQt}
  {\rm Lip}(\sfQ_t f)\leq 2\,{\rm Lip}(f)\quad\forall t\geq 0,\qquad{\rm Lip}\bigl(\sfQ_\cdot f(x))\leq 2\,\bigl[{\rm Lip}(f)\bigr]^2\quad\forall x\in X.
  \end{equation}

\subsection{Measures, couplings, Wasserstein distance}\label{sec:5.3}

Let $(X,\sfd)$ be a complete and separable metric space. 
We denote by $\mathscr B(X)$ the collection of its Borel sets and by
$\Probabilities X$ the set of all Borel probability measures on $X$
endowed with the weak topology induced by the
duality with the class $\rmC_b(X)$ of bounded and continuous functions in $X$.
If $\mm$ is a nonnegative $\sigma$-finite Borel measure of $X$, $\Probabilitiesac X\mm$ denotes
the convex subset of the probabiliy measures absolutely continuous w.r.t.~$\mm$.
$\Probabilitiesp X$ denotes the set of probability measures $\mu\in
\Probabilities X$ with finite $p$-moment, i.e.
\begin{displaymath}
  \int_X \sfd^p(x,x_0)\,\d\mu(x)<\infty\quad
  \text{for some (and thus any) $x_0\in X$.}
\end{displaymath}
If $(Y,\sfd_Y)$ is another sparable metric space, $\rr:X\to Y$ is
a Borel map and $\mu\in \Probabilities X$, $\rr_\sharp \mu$ denotes
the push-forward measure in $\Probabilities Y$ defined by 
$\rr_\sharp\mu(B):=\mu(\rr^{-1}(B))$ for every $B\in \mathscr B(Y)$.

For every $p\in [1,\infty)$, the $L^p$-Wasserstein (extended) 
distance $W_p$ between two measures $\mu_0,\,\mu_1\in\Probabilities X$
is defined as
\begin{equation}
  \label{eq:153}
  W_p^p(\mu_1,\mu_2):=\inf\Big\{\int_{X\times
    X}\sfd^p(x_1,x_2)\,\d\mmu(x_1,x_2):
  \mmu\in \Probabilities {X\times X},\ 
  \pi^i_\sharp \mmu=\mu_i\Big\},
\end{equation}
where $\pi^i:X\times X\to X$, $i=1,\,2$, denote the projections
$\pi^i(x_1,x_2)=x_i$. A measure $\mmu$ with $\pi^i_\sharp\mmu=\mu_i$
as in \eqref{eq:153} is called a coupling between $\mu_1$ and
$\mu_2$. If $\mu_1,\,\mu_2\in \Probabilitiesp X$ then
a coupling $\mmu$ minimizing \eqref{eq:153} exists,
$W_p(\mu_0,\mu_1)<\infty,$ and $(\Probabilitiesp X,W_p)$ 
is a complete and separable metric space;
it is also a length space if $X$ is a length space.
Notice that if $X$ is unbounded
$(\Probabilities X,W_p)$ is an extended metric space, even if
$\sfd$ is a finite distance on $X$.

The dual Kantorovich characterization of $W_p$ provides the useful
representation formula (here stated only in the case $p=2$)
\begin{equation}
  \label{eq:154}
  \frac 12 W_2^2(\mu_0,\mu_1)=\sup\Big\{
  \int_X \sfQ_1\varphi\,\d\mu_1-\int_X \varphi\,\d\mu_0:
  \varphi\in \Lip_b(X)\Big\},
\end{equation}
where $(\sfQ_t)_{t>0}$ is defined in \eqref{eq:11o}.

\subsection{$W_p$-absolutely continuous curves and dynamic plans}\label{sec:5.4}

A dynamic plan $\ppi$ is a 
Borel probability measure on
$\rmC([0,1]; X)$. 
For each dynamic plan $\ppi$ one can consider the (weakly) continuous curve 
$\mu=(\mu_s)_{s\in [0,1]}\subset \Probabilities X$ defined by
$\mu(s):=(\rme_s)_\sharp \ppi$, $s\in [0,1]$
 (we will often write $\mu_s$ instead of $\mu(s)$ and
 we will also use an analogous 
 notation for ``time dependent'' densities or functions);
 here
\begin{equation}\label{def:evaluation}
\rme_s:\rmC([0,1]; X)\to X,\qquad 
\rme_s(\gamma):=\gamma(s)
\end{equation}
is the evaluation map at time $s\in [0,1]$.

We say that $\ppi$ has finite $p$-energy, $p\in [1,\infty)$, if 
\begin{equation}
    \label{eq:18bis}
    \mathscr A_p(\ppi):=\int \cA_p(\gamma)
  \,\d\ppi(\gamma)<\infty,
\end{equation}
a condition that in particular yields
$\gamma\in 
  \mathrm{AC}^p([0,1];{X})$ for $\ppi\text{-almost every }\gamma$.
If for some $p>1$ the dynamic plan $\ppi$ has finite $p$-energy,
it is not hard to show that the 
induced curve $\mu$ belongs to $\AC p{[0,1]}{(\Probabilities X,W_p)}$
and that
\begin{equation}
  \label{eq:156}
  |\dot\mu_s|^p\le \int |\dot\gamma_s|^p\,\d\ppi(\gamma)\
  \text{for $\Leb 1$-a.e.~$s\in (0,1)$},\quad
  \text{so that}\ \ 
  \int_0^1 |\dot \mu_s|^p\,\d s\le  \mathscr A_p(\ppi),
\end{equation}
where $|\dot\mu_s|$ denotes the metric derivative of the curve $\mu$
in $(\Probabilities X,W_p)$. Notice that the second inequality in 
\eqref{eq:156} can also be written as $\mathcal A_p(\mu)\leq\mathscr A_p(\ppi)$.
The converse inequalities, which involve a special choice of $\ppi$, provide a metric version of the
so-called superposition principle, and their proof is less elementary.

\begin{theorem}[\cite{Lisini07}] \label{thm:lisini}
For any $\mu\in \AC p{[0,1]}{(\Probabilities X,W_p)} $
there exists a dynamic plan $\ppi$ with finite $p$-energy
such that
\begin{equation}
  \label{eq:158}
    \mu_t=(\rme_t)_\sharp\ppi\quad\forevery t\in [0,1],\qquad
   \int_0^1 |\dot \mu_t|^p\,\d t= \mathscr A_p(\ppi).
\end{equation}
\end{theorem}

We say that the dynamic plan $\ppi$ is 
$p$-\emph{tightened to $\mu$} if \eqref{eq:158} holds. For this class of
plans equality holds in \eqref{eq:156}, namely
\begin{equation}
  \label{eq:156bis}
  |\dot\mu_s|^p= \int |\dot\gamma_s|^p\,\d\ppi(\gamma)\
  \text{for $\Leb 1$-a.e.~$s\in (0,1)$.}
\end{equation}
Focusing now on the case $p=2$, 
the distinguished class of optimal geodesic plans 
$\mathrm{GeoOpt}(X)$ consists of those dynamic plans
whose $2$-action coincides with the squared $L^2$-Wasserstein distance
between the marginals at the end points:
\begin{equation}
  \label{eq:178}
  \ppi\in \mathrm{GeoOpt}(X) \quad \text{if}\quad
  \mathscr A_2(\ppi)=W_2^2(\mu_0,\mu_1),\quad
  \mu_i=(\rme_i)_\sharp \ppi.
\end{equation}
It is not difficult to check that \eqref{eq:178} is equivalent to
\begin{equation}
  \label{eq:179}
  \ppi\text{-a.e.\ $\gamma$ is a geodesic and }
  (\rme_0,\rme_1)_\sharp \ppi\text{ is an optimal coupling between }
  \mu_0,\,\mu_1.
\end{equation}
{\color{black} 
It follows that $\ppi\in \mathrm{GeoOpt}(X)$ is always 2-tightened to the curve of
its marginals, and that 
a curve $\mu\in \Lip([0,1];(\ProbabilitiesTwo X,W_2))$
is a geodesic if and only if there exists $\ppi\in \mathrm{GeoOpt}(X)$
having $\mu$ has curve of marginals (Theorem~\ref{thm:lisini} is needed to prove the
``only if" implication).}

Finally, when a reference $\sigma$-finite and nonnegative Borel
measure $\mm$
is fixed, we say that $\ppi\in
\Probabilities{\rmC([0,1];\Probabilities X)}$ 
is a \emph{test plan} if it has finite $2$-energy and
there exists a constant $R>0$ such that
\begin{equation}
  \label{eq:39}
  \mu_t:=(\rme_t )_\sharp \ppi=\varrho_t\mm\ll\mm,\quad
  \varrho_t \le R\quad\text{$\mm$-a.e. in $X$ for every }t\in [0,1].
\end{equation}

\subsection{Metric measure spaces and the Cheeger energy}
\label{subsec:Cheeger}

In this paper a \emph{metric measure space} 
$(X,\sfd,\mm)$ will always consist of:
\begin{itemize}
\item a complete and separable metric space $(X,\sfd)$;
\item a
  nonnegative Borel measure $\mm$ having full support and satisfying
  the growth condition
  \begin{equation}
    \label{eq:173}
    \mm(B_r(x_0))\le A\rme^{B r^2}\quad\text{for some constants }
    A,\,B\ge0, \text{ and some }x_0\in X.
  \end{equation}
\end{itemize}
The Cheeger energy of a function $f\in L^2(X,\mm)$ is
defined as 
\begin{equation}
  \label{eq:14}
  \C(f):=\inf\Big\{\liminf_{n\to\infty}\frac 12\int_X |\rmD
  f_n|^2\,\d\mm:
  f_n\in \Lip_b(X),\quad
  f_n\to f\text{ in }L^2(X,\mm)\Big\}.
\end{equation}
If $f\in L^2(X,\mm)$ with $\C(f)<\infty$, then there exists a unique
function $|\rmD f|_w\in L^2(X,\mm)$, called \emph{minimal weak gradient of $f$},
satisfying the two conditions
\begin{equation}
  \label{eq:53}
  \begin{gathered}
    \Lip_b(X)\cap L^2(X,\mm)\ni f_n\weakto f,\ |\rmD f_n|\weakto
    G\quad \text{in }L^2(X,\mm)\quad\Rightarrow\quad |\rmD f|_w\le G\,\,\text{$\mm$-a.e.}\\
    \C(f)=\frac 12 \int_X |\rmD f|_w^2\,\d\mm.
  \end{gathered}
\end{equation}
In \eqref{eq:14} we can also replace $|\rmD f|$ with
$|\rmD^* f|$ since
a further approximation result of \cite[\S8.3]{AGS11c}
(see \cite{Ambrosio-Colombo-Dimarino12} for a detailed proof)
yields
for every $f\in L^2(X,\mm)$ with $\C(f)<\infty$
\begin{equation}
  \label{eq:105bis}
  \exists\,f_n\in \Lip_b(X)\cap L^2(X,\mm):\quad
  f_n\to f,\quad |\rmD^* f_n|\to |\rmD f|_w\quad\text{strongly in }L^2(X,\mm).
\end{equation}
We will denote by $W^{1,2}(X,\sfd,\mm)$ 
the vector space
of the $L^2(X,\mm)$ functions with finite Cheeger energy
endowed with the canonical norm
\begin{equation}
  \label{eq:239}
  \|f\|_{W^{1,2}(X,\sfd,\mm)}^2:=\|f\|_{L^2(X,\mm)}^2+2\C(f)
\end{equation}
that induces on $W^{1,2}(X,\sfd,\mm)$ a Banach space structure.
We say that $\C$ is a quadratic form if it satisfies the parallelogram
identity
\begin{equation}
  \label{eq:1}
  \C(f+g)+\C(f-g)=2\C(f)+2\C(g)\quad\forevery f,g\in W^{1,2}(X,\sfd,\mm).
\end{equation}
In this case we will denote by $\cE$ the associated bilinear Dirichlet
form, so that $\C(f)=\frac 12\cE(f,f)$; 
if $\cB$ is the $\mm$-completion 
of the collection of Borel sets in $X$,
we are in the setting of Section~\ref{subsec:Dir-Form};
keeping that notation, $\H=L^2(X,\mm)$ and $\V$ is the separable
Hilbert space $W^{1,2}(X,\sfd,\mm)$ 
endowed with the norm \eqref{eq:239}.
Under the quadraticity assumption on $\C$ it is possible to prove \cite[Thm.~4.18]{AGS11b} that
\eqref{eq:1} can be localized, namely
\begin{equation}
  \label{eq:188}
  |\rmD (f+g)|_w^2+
  |\rmD (f-g)|_w^2=2|\rmD f|_w^2+2|\rmD g|_w^2\quad
  \mm\text{-a.e.~in $X$.}
\end{equation}
It follows that
\begin{equation}
  \label{eq:189}
  (f,g)\mapsto \Gbil fg:=\frac 14 |\rmD (f+g)|_w^2-\frac 14
  |\rmD (f-g)|_w^2=\lim_{\eps\down0}\frac{|\rmD (f+\eps g)|_w^2-|\rmD f|_w^2}{2\eps}
\end{equation}
is a strongly continuous bilinear map from $\V$ to $L^1(X,\mm)$, with
$\Gq f=|\rmD f|_w^2$. The operator $\Gamma$ is the \emph{Carr\'e du Champ}
associated to $\cE$ and $\cE$ 
is a strongly local Dirichlet form enjoying
 useful $\Gamma$-calculus
properties, see e.g. \cite{Bouleau-Hirsch91,AGS12,Savare12},
and the mass preserving property
\eqref{eq:271} (thanks to \eqref{eq:173}). {\color{black} In the measure-metric setting
we will still use the symbol $\DeltaE$ to denote the linear operator  
$-\Delta:\V\to \V'$ associated to $\cE$, corresponding in the classical cases to Laplace's operator with homogenous Neumann boundary conditions.} We also set 
\begin{equation}\label{eq:defDD}
\D:=\big\{f\in \V:\ \DeltaE f\in \H\big\},
\end{equation}
the domain of $\DeltaE$ as unbounded selfadjoint operator
in $\H$, endowed with the Hilbertian norm 
 $\|f\|_\D^2:=\|f\|^2_\V+\|\DeltaE f\|^2_\H$.
The operator $-\DeltaE$ generates a measure preserving 
Markov semigroup $(\sfP_t)_{t\ge0}$ 
in each $L^p(X,\mm)$, $1\leq p\leq\infty$.

\GGGG
Recall that the Fisher information of a nonnegative function $f\in
L^1(X,\mm)$ is defined as 
\begin{equation}
  \label{eq:209}
  \mathsf F(f):=4\cE(\sqrt f,\sqrt f)=8\C(\sqrt f)=
  \int_{\{f>0\}} \frac{\Gamma(f)}f\,\d\mm
\end{equation}
with the usual convention $\mathsf F(f)=+\infty$ whenever $\sqrt
f\not\in W^{1,2}(X,\sfd,\mm)$.
\subsection{Entropy estimates of the quadratic moment and of the Fisher
  information along nonlinear diffusion equations}

In this section we will derive a basic estimate involving quadratic
moments, logarithmic
entropy, and Fisher information along the solutions of the nonlinear
diffusion equation \eqref{eq:270} in the metric-measure setting of the
previous section \ref{subsec:Cheeger}. In order to deal with arbitrary
measures satisfying the growth condition \eqref{eq:173}, we follow the
approach of \cite{AGS11a}: we  will derive the estimates for a reference
measure
with finite mass and then we will extend them to the general case by
an approximation argument. A basic difference here is related to the 
structure of the equations, which are not $L^2$ gradient flows; we
will thus use the $L^1$-setting by taking advantage of the
$m$-accretiveness
of the operator $\bar A$ of Theorem \ref{thm:m-accretive}.

Let us first focus on the approximation argument. 
Taking \eqref{eq:173} into account, we fix a point $x_0\in X$ and
we set 
\begin{equation}
  \label{eq:85}
  \sfV(x):=\Big(a+b\,\sfd^2(x,x_0)\Big)^{1/2},\quad \sfV_k(x):=\sfV(x)\land k,
\end{equation}
for suitable constants $b:=B+1$, $a\ge (\log A)_+$ so that 
\begin{equation}
  \label{eq:86}
  \int_X \mathrm e^{-\sfV^2(x)}\,\d\mm
  =2b\mathrm e^{-a} \int_0^\infty r \rme^{-br^2} \mm(B_r(\bar x))\,\d
  r
  \le \mathrm e^{-a}A \int_0^\infty 2r\mathrm e^{-r^2}\,\d r\le 1.
\end{equation}
As in \cite[Theorem 4.20]{AGS11a} we consider the increasing sequence of
finite measures
\begin{equation}
  \label{eq:273}
  \mm_0:= \mathrm e^{-\sfV^2}\mm=\beta_0\mm,\quad
  \mm_k:= \beta_k \mm=\mathrm e^{\sfV_k^2}\mm_0,\quad
  \beta_k:=\mathrm e^{\sfV_k^2-\sfV^2},
  \quad k\in \N_0.
\end{equation}
Notice that $\beta_k$ is a bounded Lipschitz function and
$\beta_k^{-1}$ is locally Lipschitz. The map 
\begin{equation}
  \label{eq:280}
  Y_k:\varrho\mapsto \varrho/\beta_k\quad\text{is an isometry of $
  L^1(X,\mm)$ onto $L^1(X,\mm_k)$}.
\end{equation}
We will denote by $\C_k$ the Cheeger energy associated to the metric
measure space $(X,\sfd,\mm_k)$; by the invariance property
\cite[Lemma 4.11]{AGS11a} and \eqref{eq:188} $\C_k$ is also associated
to a symmetric Dirichlet form $\cE_k$ in $L^2(X,\mm_k)$, inducing a
selfadjoint operator $L_k$ with domain $\D_k\subset L^2(X,\mm_k)$. 
We fix a map $P:\R\to \R$ as in
\eqref{eq:A1} and we define the $m$-accretive operator $\bar A_k$ in
$L^1(X,\mm_k)$ as in \eqref{eq:236}, by taking the closure of the
graph of $A_k:=-L_k\circ P$ defined by \eqref{eq:109}. 
We eventually consider the realization of $\bar A_k$ in $L^1(X,\mm)$ 
\begin{equation}
  \label{eq:281}
  \widetilde A_k:= Y_k^{-1} \bar A_k Y_k;
\end{equation}
since $Y_k$ are isometries, $\widetilde A_k$ is $m$-accretive in
$L^1(X,\mm)$ and it generates a contraction semigroup $(\sfS^k_t)_{t\ge0}$  by
Crandall-Liggett Theorem as in Theorem \ref{thm:nonlin-diff}(ND4).
Notice that for every $\bar\varrho\in L^1(X,\mm)$ with
\begin{equation}
  Y_k\bar\varrho\in L^2(X,\mm_k)\quad
  \text{i.e.}\quad
  \int_X \rme^{\sfV^2}\bar\varrho^2\,\d\mm<\infty
  \label{eq:284}
\end{equation}
(in particular
when
$\bar\varrho$ belongs to $L^2(X,\mm)$ and has bounded support),
setting $\varrho^k_t:=\sfS^k_t \bar\varrho$ the curve $Y_k
\varrho^k_t$ is a strong solution of 
the equation $u'-L_kP(u)=0$ in $W^{1,2}(0,T;\V_k,\Vdual{\cE_k})$ and
for  every entropy function $W$ as in Theorem
\ref{thm:nonlin-diff} (ND2) we have
\begin{equation}
  \label{eq:283}
  \int_X W(\varrho^k_t/\beta_k)\beta_k\,\d\mm+
  \int_0^t \cE_k(P(\varrho^k_r/\beta_r),W'(\varrho^k_r/\beta_r))\,\d r=
  \int_X W(\bar\varrho/\beta_k)\beta_k\,\d\mm.
\end{equation}
\begin{theorem}
  \label{thm:semigroup-convergence}
  For every $\bar \varrho\in L^1(X,\mm)$ we have
  \begin{equation}
    \label{eq:282}
    \lim_{k\up\infty}\sfS^k_t\bar\varrho=\sfS_t\bar\varrho\quad\text{strongly
      in }L^1(X,\mm)
  \end{equation}
  and the limit is uniform in every compact interval $[0,T]$. 
\end{theorem}
\begin{proof}
  By Br\'ezis-Pazy Theorem \cite[Thm.~3.1]{Brezis-Pazy72}, in order to prove
  \eqref{eq:282} it is sufficient to check the pointwise convergence of the
  resolvent operators $\sfJ^k_\tau:=(I+\tau \widetilde A_k)^{-1}$ to
  $\sfJ_\tau=(I+\tau\bar A)^{-1}$ in
  $L^1(X,\mm)$, i.e.
  \begin{equation}
    \label{eq:285}
    \lim_{k\up\infty} \sfJ_\tau^k f=\sfJ_\tau f\quad
    \text{strongly in }L^1(X,\mm)\quad
    \text{for every }\tau>0, \ f\in L^1(X,\mm).
  \end{equation}
  Since $\sfJ_\tau^k,\sfJ_\tau$ are contractions, for every $g\in
  L^1(X,\mm)$ we have
  \begin{displaymath}
    \|\sfJ^k_\tau f-\sfJ_\tau f\|_{L^1}\le 
    \|\sfJ^k_\tau f-\sfJ_\tau^k g+\sfJ_\tau^k g-\sfJ_\tau g+\sfJ_\tau g-\sfJ_\tau f\|_{L^1}
    \le 2\| f- g\|_{L^1}+\|\sfJ^k_\tau g-\sfJ_\tau g\|_{L^1}
  \end{displaymath}
  so that by an approximation argument it is not restrictive to check \eqref{eq:285} for
  $ f\in L^1\cap L^2(X,\mm)$ with bounded support. 

  Let
  $f^k_\tau=\sfJ_\tau^k f\in L^1(X,\mm)$; recalling Theorem
  \ref{thm:m-accretive}, it is easy to check that
  $h^k_\tau=Y_k f^k_\tau=f^k_\tau/\beta_k\in L^2(X,\mm_k)$ is the solution of 
  \begin{displaymath}
    h^k_\tau-L_k P(h_\tau^k)=f/\beta_k
    \quad\text{in }L^2(X,\mm_k).
  \end{displaymath}
  If we denote by $P^{-1}:\R\to \R$ the inverse function of $P$, then
  $z^k_\tau=P(h^k_\tau)$ belongs to $\D_k\subset L^2(X,\mm_k)$ and 
  solves
  \begin{equation}
    \label{eq:287}
    P^{-1}(z^k_\tau)-L_k z^k_\tau=f/\beta_k.
  \end{equation}
  Introducing the uniformly convex function $V^*(r):=\int_0^r P^{-1}(x)\,\d x$
  which still satisfies the uniform quadratic bounds \eqref{eq:74_bis},
  the solution to problem \eqref{eq:287} can be characterized as the unique
  minimizer in $L^2(X,\mm_k)$ of the uniformly convex functional
  \begin{equation}
    \label{eq:288}
    z\mapsto \Phi^k(z):=\int_X V^*(z)\,\d\mm_k+\C_k(z)-\int_X f\,z\,\d\mm.
  \end{equation}
  Arguing as in the proof of \cite[Theorem 4.18]{AGS11a} it is not
  difficult to show that $z^k_\tau$ converges strongly to $z_\tau$ in
  $L^2(X,\mm_0)\subset L^2(X,\mm)$, where $z_\tau$ is
  the unique minimizer of 
  \begin{equation}
    \label{eq:288}
    z\mapsto \Phi(z):= \int_X V^*(z)\,\d\mm+\C(z)-\int_X f\,z\,\d\mm,
  \end{equation}
  with
  \begin{equation}
    \label{eq:289}
    \int_X V^*(z_\tau^k)\,\d\mm_k=
    \int_X V^*(z_\tau^k)\beta_k
    \,\d\mm
    \to 
    \int_X V^*(z_\tau)\,\d\mm  
    \quad
    \text{as }k\up\infty.
  \end{equation}
  Since every subsequence $n\mapsto k(n)$ admits a further
  subsequence $n\mapsto k'(n)$ along which $z^{k'(n)}_\tau\to z_\tau$ converges $\mm_0$ (and thus $\mm$)-a.e., 
  the Lipschitz character of $P$ yields
  $h^{k'(n)}_\tau \to h_\tau$ $\mm$-a.e.;  since $\beta_k\to 1$ uniformly on bounded sets 
  we also get $f^{k'(n)}_\tau\to f_\tau$ $\mm$-a.e.
  
  When $f\ge0$ the order preserving property
  \eqref{eq:226} shows that $f_\tau^k\ge0$ and the mass preserving
  property yields
  \begin{equation}
    \label{eq:291}
    \int_X f_\tau^k\,\d\mm=\int_X h_\tau^k\,\d\mm_k=\int_X f\,\d\mm=\int_X f_\tau\,\d\mm,
  \end{equation}
  so that $f^{k'(n)}_\tau\to f_\tau$ strongly in $L^1(X,\mm)$. Since
  the sequence $n\mapsto k(n)$ is arbitrary, we conclude that
  $f^k_\tau\to f_\tau$ strongly in $L^1(X,\mm)$ as $k\to\infty$. 
  When $f$ has arbitrary sign, we still use the monotonicity property 
  to obtain the pointwise bound $|f^k_\tau|\le \sfJ_\tau^k |f|$ and we
  conclude by applying
  a variant of the Lebesgue dominated convergence theorem.
\end{proof}
We consider now the logarithmic entropy density $U_\infty(r):=r\log r$,
$r\ge0$ and for given nonnegative measures $\mu\in \PP(X)$ and $k\in \N_0$,
we set
\begin{equation}
  \label{eq:148}
  \mathcal U_\infty^k(\mu):=\int_X
  U_\infty(\varrho/\beta_k)\beta_k\,\d\mm=
  \int_X \varrho \log(\varrho/\beta_k)\,\d\mm,\quad
  \mu=\varrho\mm\ll\mm;
\end{equation}
we will simply write $\mathcal U_\infty(\mu)$ when $k=\infty$ and
$\beta_k\equiv 1$; we will set $\mathcal U^k_\infty(\mu)=+\infty$ if
$\mu$ is not absolutely continuous w.r.t.~$\mm$.
The inequality
\begin{equation}
  \label{eq:185}
  U_\infty(r)\ge r-\rme^{-\sfv^2}-r\sfv^2\quad\text{for every }\sfv\in \R,\ r>0,
\end{equation}
and \eqref{eq:86} show that the negative part of the integrand in
\eqref{eq:148} is always integrable whenever $\mu\in \PP(X)$  and
$k<\infty$ with
\begin{equation}
  \label{eq:190}
  \mathcal U^k_\infty(\mu)+\int_X \sfV_k^2\,\d\mu\ge0.
\end{equation}
When $k=\infty$ and $\mu\in \PP_2(X)$ we also have
\begin{equation}
  \label{eq:190bis}
  \mathcal U_\infty(\mu)+\int_X \sfV^2\,\d\mu=\mathcal U^k_\infty(\mu)+\int_X \sfV_k^2\,\d\mu
  \ge0.
\end{equation}
Moreover, if $\mu_k=\varrho_k\mm$ is a sequence of probability
measures with $\varrho_k\to\varrho$ strongly in $L^1(X,\mm)$ with
$\mu=\varrho\mm\in \PP_2(X)$, by
\cite[Lemma 9.4.3]{AGS08} and writing $\mathcal U^k_\infty(\mu_k)+\int_X
\sfV_k^2\,\d\mu_k $ as the relative entropy of $\mu_k$ with respect to
the finite measure $\mm_0$
we have
\begin{equation}
  \label{eq:293}
  \liminf_{k\to\infty}\mathcal U^k_\infty(\mu_k)+\int_X
  \sfV_k^2\,\d\mu_k\ge 
  \mathcal U_\infty(\mu)+\int_X \sfV^2\,\d\mu.
\end{equation}
Even easier, since the sequence $\sfV_k$ is monotonically increasing, we have
\begin{equation}
  \label{eq:294}
  \liminf_{k\to\infty}\int_X
  \sfV_k^2\,\d\mu_k\ge 
  \int_X \sfV^2\,\d\mu.
\end{equation}
Finally, defining the relative Fisher information in $(X,\sfd,\mm_k)$
as in \eqref{eq:209} by
\begin{equation}
  \label{eq:295}
  \mathsf F_k(\varrho):=8\C_k(\sqrt{\varrho/\beta_k})
\end{equation}
and observing that $\|\sqrt{\varrho/\beta_k}\|_{L^2(X,\mm_k)}=
\int_X \varrho\,\d\mm$, \cite[Proposition 4.17]{AGS11a} yields
\begin{equation}
  \label{eq:296}
   \liminf_{k\to\infty} \mathsf F_k(\varrho_k)\ge \mathsf F(\varrho).
\end{equation}
\begin{theorem}[Entropy, quadratic moment and Fisher information]\label{thm:emF}
  In the\\
  metric-measure setting of Section~\ref{subsec:Cheeger},
  let $\bar\rho\in L^1(X,\mm)$
  satisfying $\bar \mu=\bar\rho\mm\in \PP_2(X) $ and $\mathcal U_\infty(\bar\mu)<\infty$, 
  and let $\rho$ be the corresponding solution of the nonlinear diffusion equation
  \eqref{eq:75} according to Theorem \ref{thm:nonlin-diff}.
  Then there exists a constant $C>0$ only depending on $\sfa,B$ of
  \eqref{eq:A1} and \eqref{eq:173} such
  that for every $t\in [0,T]$ the probability measures
  $\mu_t=\varrho_t\mm$ 
  belong to $\PP_2(X)$ and satisfy
  \begin{equation}
    \label{eq:221}
    \mathcal U_\infty(\mu_t)+2\int_X \sfV^2\,\d\mu_t
    +\frac \sfa2\int_0^t \sfF(\rho_r)\,\d r\le \mathrm e^{C t}
    \Big(\mathcal U_\infty(\bar\mu)+2\int_X \sfV^2\,\d\bar\mu\Big).
  \end{equation}
\end{theorem}
\begin{proof}
  By a standard approximation argument, the $L^1$-contraction property
  of Theorem \ref{thm:nonlin-diff} (ND4) and the lower
  semicontinuity of entropy, quadratic momentum and Fisher information
  \eqref{eq:293}, \eqref{eq:294}, \eqref{eq:296},
  it is not restrictive to assume that $\bar\rho$ also belongs to
  $L^2(X,\mm)$ and has bounded support. 
  By Theorem \ref{thm:semigroup-convergence}, 
  it is also sufficient to prove the analogous inequality
  \begin{equation}
    \label{eq:221bis}
    \mathcal U_\infty^k(\mu^k_t)+2\int_X \sfV_k^2\,\d\mu^k_t
    +\frac \sfa 2\int_0^t \sfF_k(\rho^k_r)\,\d r\le \mathrm e^{C t}
    \Big(\mathcal U^{k}_\infty(\bar\mu)+2\int_X \sfV_{k}^2\,\d\bar\mu\Big),
  \end{equation}
  where $\rho^k_t:=\sfS^k_t\bar\varrho$ and $\mu^k_t:=\varrho^k_t\mm$.
  
  Since $U_\infty$ does not satisfy the conditions of (ND2) of
  Theorem \ref{thm:nonlin-diff}, we cannot immediately compute its
  derivative along the solution of the nonlinear diffusion equation as
  in \eqref{eq:123}; we thus introduce the regularized logarithmic
  function
  \begin{equation}
    \label{eq:94}
    W_\eps(r):=(r+\eps)\big(\log(r+\eps)-\log \eps\big)-r=
    U_\infty(r+\eps)-U_\infty'(\eps)(r+\eps)+\eps
    ,\quad
    \eps>0,
  \end{equation}
  satisfying
  \begin{equation}
    \label{eq:108}
    W_\eps(0)=W_\eps'(0)=0,\quad 
    W_\eps'(r)=\log(r+\eps)-\log\eps,\quad
    W_\eps''(r)=\frac1{r+\eps}.
  \end{equation}
  Applying \eqref{eq:283} to $W_\eps$ 
  we obtain 
  \begin{equation}
    \label{eq:135}
    \int_X W_\eps(\varrho^k_t/\beta_k)\beta_k\,\d\mm+\int_0^t
    \cE_k(P(\varrho^k_r/\beta_k),W_\eps'(\varrho^k_r/\beta_k))\,\d r=  \int_X W_\eps(\bar\varrho/\beta_k)\beta_k\,\d\mm.
  \end{equation}
  Standard
  $\Gamma$-calculus (see e.g.~\cite{Bouleau-Hirsch91}) 
  and the fact that $\varrho^k_t/\beta_k\in D(\C_k)$ for 
  almost every $t\in [0,T]$ yield 
  \begin{align*}
    \Gamma(&P(\varrho^k_t/\beta_k), W_\eps'(\varrho^k_t/\beta_k))=
             \frac{P'(\varrho^k_t)}{\varrho^k_t/\beta_k+\eps}\Gamma(\varrho^k_t/\beta_k,\varrho^k_t/\beta_k)
             \ge 
             \frac{\sfa}{\varrho^k_t/\beta_k+\eps}\Gamma(\varrho^k_t/\beta_k,\varrho^k_t/\beta_k).
  \end{align*}
  %
  Setting $W^1_\eps(r):=r\log (r+\eps)$ and
  $W^2_\eps(r):=\eps(\log (r+\eps)-\log\eps)$ and using the fact that
  $\int_X \varrho^k_t \,\d\mm=\int_X \bar\varrho\,\d\mm$ and $W^2_\eps(r)\ge0$, \eqref{eq:135}
  yields
  \begin{align}
    \label{eq:246}
    \int_X W^1_\eps(\varrho^k_t/\beta_k)\beta_k\,\d\mm&+
    \sfa\int_0^t \int_X
    \frac{\Gamma(\varrho^k_r/\beta_k,\varrho^k_r/\beta_k)}{\varrho^k_t/\beta_k+\eps}\beta_k\,\d\mm\,\d r
    \\&\le  \label{eq:246bis}
    \int_X \Big(W^1_\eps(\bar\varrho/\beta_k)+W^2_\eps(\bar\varrho/\beta_k)\Big)\beta_k\,\d\mm.
  \end{align}
  We observe that 
  \begin{displaymath}
    W^1_\eps(r)\le r(r+\eps-1),\
    W^1_\eps(r)\down U_\infty(r),\quad W^2_\eps(r)\le r,\quad
    \lim_{\eps\down0}W^2_\eps(r)=0,
  \end{displaymath}
  so that we can pass to the limit in \eqref{eq:246}, \eqref{eq:246bis} as $\eps\down0$
  obtaining
  \begin{equation}
    \label{eq:246tris}
    \cU_\infty^k(\mu^k_t) +
    \sfa \int_0^t \sfF_k(\varrho^k_r)\,\d r\le 
    \cU_\infty^k(\bar\mu). 
  \end{equation}
  We now compute the time derivative of $t\mapsto \int_X
  \sfV_k^2\,\d\mu^k_t$ obtaining
  \begin{align*}
    2\frac\d{\d t}\int_X
    \sfV_k^2\,\d\mu^k_t&=
    2\cE_k(P(\varrho^k_t/\beta_k),\sfV_k^2)\le 
    \frac{4 \sqrt b}\sfa \int_X \sqrt{\Gamma(\varrho^k_t/\beta_k)}\sfV_k
    \beta_k\,\d\mm\\&\le
                      \frac{4 \sqrt b}\sfa \Big(\sfF_k(\varrho^k_t) \int_X
    \sfV_k^2\,\d\mu^k_t\Big)^{1/2}
                      \le \frac \sfa2  \sfF_k(\varrho^k_t) +\frac{8 b}{\sfa^2} \int_X
    \sfV_k^2\,\d\mu^k_t.
  \end{align*}
  Integrating in time and summing up with \eqref{eq:246tris} we obtain
  \begin{align*}
    \cU_\infty^k(\mu^k_t)+2 \int_X
    \sfV_k^2\,\d\mu^k_t+
    \frac\sfa2 \int_0^t \sfF_k(\varrho^k_r)\,\d r&\le 
     \cU_\infty^k(\bar \mu)+2 \int_X
    \sfV_k^2\,\d\bar \mu\\&+
                           \frac{8 b}{\sfa^2} \int_0^t \Big(\int_X
                           \sfV_k^2\,\d\mu^k_r\Big)\,\d r.
  \end{align*}
  Since $ \cU_\infty^k(\mu^k_t)+2 \int_X
    \sfV_k^2\,\d\mu^k_t\ge \int_X
    \sfV_k^2\,\d\mu^k_t$ Gronwall Lemma yields \eqref{eq:221bis} with
    $C:=\frac{8(B+1)}{\sfa^2}$.
\end{proof}

\EEE
\subsection{Weighted $\Gamma$-calculus}
\label{subsec:weighted}
In the metric-measure setting of Section~\ref{subsec:Cheeger},
consider a nonnegative function $\varrho\in L^\infty(X,\mm)$. Any $f\in L^p(X,\mm)$
obviously induces a function in $L^p(X,\nn)$, with $\nn=\rho\mm$, that we shall denote $\tilde f$;  
in the following we will often suppress the symbol\ $\tilde {}$ when there will be no risk of
ambiguity. 

Consider now
 the symmetric  and continuous bilinear form in $\V\times \V$
\begin{equation}
  \label{eq:68}
  \cE_\varrho(f,g):=\int_X \varrho\,\Gbil fg\,\d\mm\qquad
  f,\,g\in \V,
\end{equation}
which induces a seminorm: we will denote by
$\Vhom \varrho =\Vhom{\cE_\varrho}$ the abstract Hilbert spaces
constructed from $\cE_\varrho$ 
as in Section~\ref{subsec:completion}, namely the completion of the
quotient space of $\V$ induced by the equivalence relation $f\sim g$ if $\cE_\varrho(f-g,f-g)=0$, with respect
to the norm induced by the quotient scalar product.
If $\varphi\in \V$ 
then its equivalence class in $\Vhom\varrho$ will be denoted by
$\class\varphi\varrho$ (or still by $\varphi$ when there is no risk of confusion), whereas we will still use the symbol $\cE_\varrho$
to denote the scalar product in $\Vhom\varrho$.
\GGGG By locality, if $\varphi,\psi\in \V$ with
$\varphi=\psi$ $\mm$-a.e.~on $\{\varrho>0\}$ then 
$\varphi_\varrho=\psi_\varrho$. In the degenerate case when 
$\varrho\equiv0$ $\mm$-a.e., then $\Vhom\varrho$ reduces to the 
null vector space and everything becomes trivial. 

Notice that the quadratic form $\frac 12 \cE_\varrho$
is always larger than the Cheeger energy $\C_\nn$ induced by the measure 
$\nn=\varrho\mm$, in the sense that for every $f\in \V$
$\frac 12 \cE_\varrho(f)\ge \C_\nn(\tilde f)$, see also Lemma~\ref{le:paranoico} below.
When $\varrho\equiv 1$, $\V_1$ corresponds to the homogeneous space $\Vhom 1$ associated to
$\cE$ already introduced in Section~\ref{subsec:Cheeger}.

The following two simple results provide useful tools
to deal with the abstract spaces $\Vhom\varrho$.

\begin{lemma}[Extension of $\Gamma$ to the weighted spaces $\Vhom\varrho$]
  \label{le:useful1}
  Let $\varrho\in L^\infty_+(X,\mm)$, 
  and let $(\varphi_n)\subset\V$ be a Cauchy sequence
  with respect to the seminorm of $\Vhom\varrho$,
  thus converging to $\phi\in \Vhom\varrho$.
  Then $ \widetilde{\Gq {\varphi_n}}$ is
  strongly converging in $L^1(X,\varrho\mm)$ 
    to a limit that depends only on $\varrho$ and $\phi$ and that we will
    denote by $\tGq\varrho{\phi}$. 
    When $\phi=\class\varphi\varrho$ for some $\varphi\in \V$ then 
    $\tGq\varrho\phi= \widetilde{\Gq \varphi}$ $\varrho\mm$-a.e.~in $X$.
    The map 
    \begin{equation}
      \label{eq:99}
      \tGbil\varrho {\phi}{\psi}:=\frac 14
      \tGq\varrho{\phi+\psi}
      -\frac 14 \tGq\varrho{\phi-\psi}
    \end{equation}
    is a continuous bilinear map from $\Vhom\varrho$ to $L^1(X,\varrho\mm)$ and
    \eqref{eq:68} extends to $\V_\varrho$ as follows:
    \begin{equation}
      \label{eq:102}
      \cE_\varrho(\phi,\psi)=\int_X \varrho\,\tGbil\varrho {\phi}{\psi}\,\d\mm\qquad \phi,\,\psi\in\V_\varrho.
    \end{equation}
    \end{lemma}
\begin{proof}
  The convergence of $\Gq{\varphi_n}$ in $L^1(X,\varrho\mm)$ and the
  independence of the limit follow from the obvious inequality 
  \begin{align*}
    \int_X &\Big|\Gq{\psi_1}-\Gq{\psi_2}\Big|\,\varrho\,\d\mm=
    \int_X 
    \Gq{\psi_1-\psi_2}^{1/2}\Gq{\psi_1+\psi_2}^{1/2}\,\varrho\,\d\mm
    \le \|\psi_1-\psi_2\|_{\Vhom\varrho}\|\psi_1+\psi_2\|_{\Vhom\varrho},
  \end{align*}
  for every $\psi_1,\,\psi_2\in \V$. 
  When $\phi=\class\varphi\varrho$ then we can choose the constant
  sequence
  $\varphi_n\equiv \varphi$, thus showing that
  $\tGq\varrho\phi=\Gq\varphi$ $\varrho\mm$-a.e.~in $X$.
  It is immediate to check that $\tGq\varrho\cdot$ satisfies
  the parallelogram rule, so that the properties of $\tGbil\varrho {\cdot} {\cdot}$  
  defined in \eqref{eq:99}, \fn and \eqref{eq:102} follow from the corresponding
  properties of $\Gamma$ and $\cE_\varrho$ in $\V$.
    \end{proof}
    
    {\color{black} The following lemma shows that when the weight $\varrho$ satisfies a mild additional
    regularity assumption, then $\tGq\varrho{\varphi_\varrho}=\widetilde{\Gq \varphi}$ coincide with the minimal
    relaxed slope relative to the measure $\varrho\nn$.}
     
\GGGG
\begin{lemma}[Comparison with the weighted Cheeger energy]
  \label{le:paranoico}
  Let $\nn=\varrho\mm$ where $\varrho\in L^\infty(X,\mm)$
  is a nonnegative 
  function satisfying $\sqrt\varrho\in \V$, and let $\C_\nn$ be the
  Cheeger energy induced by $\nn$
  in $L^2(X,\nn)$ with associated 
  minimal weak gradient $|\rmD \cdot|_{w,\nn}$. 
  For every $\varphi\in \V$ we have  $\tilde \varphi\in D(\C_\nn)$ with 
  $|\rmD \tilde \varphi|_{w,\nn}=\widetilde{\Gq\varphi}$; in particular, one has the identifications
  $$|\rmD \tilde \varphi|_{w,\nn}=\widetilde{\Gq\varphi}=\tGq\varrho{\varphi_\varrho}\qquad
  \text{$\nn$-a.e. in $X$.}$$
\end{lemma}
\begin{proof}
  By the previous Lemma, setting $\phi=\varphi_\varrho,$ with
  $\varphi\in \V$, we have 
  $\tGq\varrho\phi=\widetilde{\Gq\varphi}$
  $\nn$-a.e. On the other hand, 
  \cite[Thm.~3.6]{AGMR12} yields  
  $\widetilde{\Gq\varphi}= |\rmD \tilde\varphi|_{w,\nn}$ $\nn$-a.e. in $X$.
\end{proof}
\EEE
\begin{lemma}[Stability]
  \label{le:stability}
  Let $\varrho_t\in L^\infty_+(X,\mm)$, $t\in [0,1]$, be a uniformly
  bounded family, continuous with respect to the convergence in
  $\mm$-measure, let $\varrho\in L^\infty_+(X,\mm)$ and 
  let $B_{t}:\V\to \V$ be a family of linear operators satisfying
  \begin{gather}
    \label{eq:110}
    \int_X \varrho_t \Gq{B_t\varphi}\,\d\mm\le C
    \int_X \varrho \Gq{\varphi}\,\d\mm\quad\forevery t\in [0,1],\
    \varphi\in \V,\\
    \label{eq:110bis}
    t\mapsto B_{t}\varphi \in \rmC([0,1];\V)\quad
    \forevery \varphi\in \V.
  \end{gather}
  Then $B_{t}$ can be extended by continuity to a 
  family of uniformly bounded linear operators from $\Vhom{\varrho}$ to
  $\Vhom{\varrho_t}$ such that 
  \begin{gather}
    \label{eq:111}
    \cE_{\varrho_t}(B_{t}\phi)\le\cE_{\varrho}(\phi),\quad 
    \forevery t\in [0,1],\
    \phi\in \Vhom\varrho,\\
    \label{eq:111bis}
    t\mapsto \varrho_t\tGq{\varrho_t}{B_t\phi} \in \rmC([0,1];L^1(X,\mm))\quad
    \forevery \phi\in \Vhom{\varrho}.
  \end{gather}
\end{lemma}
\begin{proof} Assumption
  \eqref{eq:110} shows that for every $t\in [0,1]$ the operator $B_t$ is compatible with the equivalence
  relations associated to
  $\Vhom\varrho$ and $\Vhom{\varrho_t}$, so that it can be extended by
  continuity to a linear map between the two spaces, still denoted $B_t$ and
  satisfying \eqref{eq:111}. Given any $\varphi\in\Vhom\varrho$, choosing $(\varphi_n)\subset\V$ such that the corresponding elements 
  $\classd{\varphi_n}\varrho$ converge to $\phi$ in $\Vhom\varrho$, the
  estimate
  \eqref{eq:110} shows that $\varrho_t\tGq{\varrho_t}{B_t\varphi_n}$ 
  converges uniformly in time to $\varrho_t\tGq{\varrho_t}{B_t\phi}$ in $L^1(X,\mm)$, 
  so that the continuity property \eqref{eq:111bis} follows from
  the continuity of each curve $t\mapsto \varrho_t\tGq{\varrho_t}{B_t\varphi_n}$.
\end{proof}

Finally, we discuss dual spaces, following the general scheme described in Section~\ref{subsec:completion},
see in particular Proposition~\ref{prop:allduals}.
The space $\Vdual \varrho$ is the realization of the dual of $\Vhom\varrho$ in $\V'$. It can be seen as 
the finiteness domain of the quadratic form 
\begin{equation}
  \label{eq:97}
  \frac 12\cEs\varrho (\ell,\ell):=\sup_{\varphi\in \V}\,\langle
  \ell,\varphi\rangle-\frac 12 \cE_\varrho (\varphi,\varphi),\qquad \ell \in \V'.
\end{equation}
We shall denote by $\cEs\varrho(\cdot,\cdot)$ the quadratic form on $\Vdual\varrho$ induced by
$\cEs\varrho$. We denote by $-A_\varrho$ the Riesz isomorphism between $\Vhom\varrho$ and
$\Vdual\varrho$, and by $-A^*_\varrho$ its inverse. It is characterized by
\begin{equation}
    \label{eq:112}
    \phi=-A_\varrho^*\ell\quad\Longleftrightarrow\quad
    \cE_{\varrho}(\phi,\psi)=\langle \ell,\psi\rangle\quad\forevery 
    \psi\in \Vhom{\varrho}.
  \end{equation}
Notice that it is equivalent in \eqref{eq:112} to require the validity of the equality for all $\psi\in\V$; in this sense,
\eqref{eq:112} corresponds in our abstract framework to the weak formulation of the
PDE $-{\rm div}(\varrho\nabla\phi)=\ell$ in \eqref{eq:deg_PDE}, and $-A_\varrho^*$ is the solution operator.
Since $-A_\varrho$ is the Riesz isomorphism, we get
\begin{equation}\label{eq:113bist}
\cE^*_\varrho(\ell,\ell)=\cE_\varrho(A_\varrho^*\ell,A_\varrho^*\ell).
\end{equation}
Correspondingly we set
\begin{equation}
  \label{eq:37}
  \tGqs\varrho \ell:=\tGq \varrho{A_\varrho^* \ell}\quad
  \text{whenever }\ell\in \Vdual\varrho.
\end{equation}
It is clear that
$\Gamma_\varrho^*:\Vdual\varrho
\mapsto L^1(X,\varrho\mm)$ 
is a nonnegative quadratic map.

\begin{lemma}[Dual characterization of $\Gamma_\varrho^*$]
  \label{le:Gamma-dual}
  For every $\ell\in \V'$ and $\varrho\in L^\infty_+(X,\mm)$ let
  us consider the (possibily empty) closed convex subset of
  $L^2(X,\varrho\mm)$ defined by
  \begin{equation}
    \label{eq:38}
    G(\varrho,\ell):=\Big\{g\in L^2(X,\varrho\mm):
    \big|\la \ell,\varphi\ra\big|\le 
    \int_X g\,\sqrt{ \Gq\varphi}\,\varrho\,\d\mm
    \quad \forevery\ \varphi\in\V\Big\}.
  \end{equation}
  Then $\ell\in \Vdual\varrho$ if and only if 
  $G(\varrho,\ell)$ is not empty; if $\ell\in \Vdual\varrho$ 
  then $\sqrt{\tGqs \varrho\ell}$ is the 
  element of minimal $L^2(X,\varrho\mm)$-norm in $G(\varrho,\ell)$.
\end{lemma}
\begin{proof}
  If $g\in G(\varrho,\ell)$ then
  \begin{displaymath}
    \la \ell,\varphi\ra\le \|g\|_{L^2(X,\varrho\mm)} 
    \Big(\cE_\varrho(\varphi,\varphi)\Big)^{1/2}\quad\forevery\varphi\in\V,
  \end{displaymath}
  so that $\ell\in \Vdual\varrho$ and
  \begin{equation}
    \label{eq:62}
    \int_X \tGqs\varrho\ell\,\varrho\,\d\mm=
    \cE^*_\varrho(\ell,\ell)\le \int_X g^2\,\varrho\,\d\mm.
  \end{equation}
  Conversely, let us suppose that $\ell\in \Vdual\varrho$ 
  and let $\phi=-A^*_\varrho\ell$; \eqref{eq:112} yields
  \begin{displaymath}
    |\la\ell,\psi\ra|\le 
    \int_X  \left| \tGbil\varrho\phi\psi \right|  \,\varrho \,\d\mm
    \le \int_X \sqrt{\tGq\varrho\phi}\,\sqrt{\tGq\varrho\psi}\,
    \varrho\,\d\mm
    \quad\forevery\psi\in \V
  \end{displaymath}
  so that $\displaystyle\sqrt{\tGqs\varrho\ell}=
  \sqrt{\tGq\varrho\phi}\in G(\varrho,\ell)$. 
  Combining with \eqref{eq:62} we conclude that 
  $\sqrt{\tGqs\varrho\ell}$ is the element of minimal norm in
  $G(\varrho,\ell)$.
\end{proof}

The 
following lower semicontinuity lemma with respect to the weak topology of $\Vdual{}$
will also be useful. {\color{black} Remembering the identification \eqref{eq:238}, the lemma is
also applicable to sequences weakly convergent in $\Vdual{\cE}$.}

\begin{lemma}
  \label{le:lsc}
  Let $\varrho_n\stackrel *{\weakto}\varrho$ in $L^\infty(X,\mm)$  be nonnegative
  and assume that $\ell_n\weakto \ell$ in $\Vdual{}$. Then
  \begin{equation}
    \label{eq:216}
    \liminf_{n\to\infty}\cEs{\varrho_n}(\ell_n,\ell_n)\ge 
    \cEs\varrho(\ell,\ell).
  \end{equation}
  If moreover
  $\varrho_n\to\varrho$ also in the strong
  topology of $L^1(X,\mm)$ and
  \begin{equation}
    \label{eq:65}
    \limsup_{n\to\infty}\cEs{\varrho_n}(\ell_n,\ell_n)\le
    \cEs\varrho(\ell,\ell)<\infty,
  \end{equation}
  then
  for every continuous and bounded function
  $Q:[0,\infty)\to [0,\infty)$ we have
  \begin{equation}
    \label{eq:66}
    \lim_{n\to\infty}
    \int_X Q(\varrho_n) \tGqs{\varrho_n}{\ell_n}\varrho_n\,\d\mm
    =\int_X Q(\varrho) \tGqs{\varrho}{\ell}\,\varrho\,\d\mm.
  \end{equation}
\end{lemma}
\begin{proof}
  Concerning \eqref{eq:216}, for every $\varphi\in \V$, we have
  \begin{displaymath}
    \langle \ell,\varphi\rangle-\frac 12 \int_X \varrho \Gq{\varphi}\,\d\mm=
    \lim_{n\to\infty}\Big(\langle\ell_n,\varphi\rangle-\frac 12 \int_X
    \varrho_n
    \Gq{\varphi}\,\d\mm\Big)\le 
    \liminf_{n\to\infty}\cEs{\varrho_n}(\ell_n,\ell_n).
  \end{displaymath}
  Taking the supremum with respect to 
  $\varphi\in \V$ we get \eqref{eq:216}.

  Let us consider the second part of the statement
  and let us set $g_n=\sqrt{\tGqs{\varrho_n}{\ell_n}}$,
  $h_n=g_n\varrho_n$.
  Since $h_n$ is uniformly bounded in $L^2(X,\mm)$,
  possibly extracting a suitable subsequence we 
  can assume that $h_n$ weakly converge in $L^2(X,\mm)$ to $h$. Since
  the measures $h_n\mm$, $\varrho_n\mm$ weakly converge respectively to $h\mm$ and $\varrho\mm$ and the 
  densities $g_n$ of $h_n\mm$ w.r.t.
  $\varrho_n\mm$ satisfy $\sup_n\|g_n\|_{L^2(X,\varrho_n\mm)}<\infty$ we can apply a standard joint
  lower semicontinuity lemma
  (see, for instance \cite[Lemma 9.4.3]{AGS08}) to write $h=g\varrho$ for some $g\in L^2(X,\varrho\mm)$, with
  \begin{equation}
    \label{eq:67}
    \int_X g^2\,\varrho\,\d\mm\leq
    \liminf_{n\to\infty}\int_X g_n^2\,\varrho_n\,\d\mm.
  \end{equation}
  Passing to the limit in the inequalities
  \begin{displaymath}
    |\la \ell_n,\psi\ra|\le 
    \int_X \sqrt{\tGqs{\varrho_n}{\ell_n}}\,\sqrt{\Gq\psi}\,
    \varrho_n\,\d\mm=
    \int_X h_n \,\sqrt{\Gq
      \psi}\,
    \d\mm
    \quad\forevery\psi\in \V,
  \end{displaymath}
  we get
  \begin{displaymath}
    |\la \ell,\psi\ra|\le 
    \int_X h \sqrt{\Gq
      \psi}\,
    \,\d\mm
    =\int_X g \sqrt{\Gq
      \psi}\,
    \varrho\,\d\mm
    \quad\forevery\psi\in \V,
  \end{displaymath}
  which shows that $g\in G(\varrho,\ell)$.
  On the other hand, \eqref{eq:65} and \eqref{eq:67} yield
  \begin{displaymath}
    \int_X g^2\,\varrho\,\d\mm
    \le \liminf_{n\to\infty} \int_X g_n^2\,\varrho_n
    \,\d\mm=
    \liminf_{n\to\infty} \cE^*_{\varrho_n}(\ell_n,\ell_n)\le 
    \cE^*_\varrho(\ell,\ell)    
    =\int_X \tGqs{\varrho}\ell\varrho\,\d\mm,
  \end{displaymath}
  so that Lemma~\ref{le:Gamma-dual} gives $g={\color{black}\sqrt{\tGqs\varrho\ell}}$.
  
  {\color{black} Setting now $\hat{\mm}_n=\varrho_n\mm$, $\hat{\mm}=\varrho\mm$, we know
  that $\limsup_n\int_Xg_n^2\,\d\hat{\mm}_n\leq\int_Xg^2\,d\hat{\mm}$, and \eqref{eq:66} can be written in
  the form
  $$
  \lim_{n\to\infty}\int_XQ(\varrho_n)g_n^2\,d\hat{\mm}_n=\int_XQ(\varrho)g^2\,d\hat{\mm}.
  $$
  This convergence property can be proved writing the integrals in terms of the measures
  $\theta_n:=(\varrho_n,g_n)_\sharp\hat{\mm}_n$ which converge in $\ProbabilitiesTwo{\R\times\R}$
  to $\theta=(\varrho,g)_\sharp\hat{\mm}$, using the test function $(u,v)\mapsto Q(u)|v|^2$.}
\end{proof}

\section{Absolutely continuous curves in Wasserstein spaces and 
continuity inequalities  in a metric setting}\label{sec:abscurWas}

In this section we extend to general metric spaces some 
aspects of the results of \cite[Chap.~8]{AGS08}.
Even if we will use only the case $p=2$, we
state some results in the general case for possible future reference.

Let $(X,\sfd)$ be a complete and separable metric space; we set 
$\tilde X:=X\times [0,1]$ and define
$\tilde\rme:\rmC([0,1],X)\times [0,1] \to \tilde X$ by
$\tilde\rme(\gamma,t):=(\gamma(t),t)$. For every
dynamic plan $\ppi$ we 
consider
the measures
\begin{equation}
  \label{eq:25}
  \lambda:=\Leb 1\restr{[0,1]},\qquad
  \tilde\ppi:=\ppi\otimes \lambda,\qquad 
  \tilde\mu:=\tilde\rme_\sharp\big(\tilde\ppi\big) \in \Probabilities{X\times[0,1]}. 
\end{equation}
Notice that the disintegration of $\tilde \mu$ 
with respect to time is exactly $((\rme_t)_\sharp\ppi)_{t\in [0,1]}$,
i.e.~$\tilde\mu$ admits the representation
\begin{equation}
  \label{eq:159}
  \tilde\mu=\int_0^1 \mu_t\,\d \lambda(t)\quad\text{with}\quad
  \mu_t:=(\rme_t)_\sharp \ppi.
\end{equation}
If $\ppi$ has finite $p$-energy for some $p\in (1,\infty)$, 
the Borel map $(\gamma,t)\mapsto \tilde \rmv(\gamma,t):=|\dot \gamma|(t)$ 
(defined where the metric derivative exists) belongs to
$L^p(\rmC([0,1];X)\times [0,1],\tilde\ppi)$, so that the mean velocity $v$ of $\ppi$ can be defined by
\begin{equation}
  \label{eq:26}
  \tilde\rme_\sharp(\tilde \rmv\,\tilde\ppi)= v\tilde\mu\quad\text{with}\quad
   v\in L^p(\tilde X,\tilde\mu),\quad
   v(x,t)=\int |\dot\gamma_t|\,\d\tilde\ppi_{x,t}(\gamma)
\end{equation}
(here $(\tilde\ppi_{x,t})_{(x,t)\in\tilde X}\subset\Probabilities{\rmC([0,1];X)}$ is the disintegration
of $\tilde\ppi$ w.r.t.\ its image $\tilde\mu$).  {\color{black} More precisely, Jensen's inequality gives
  \begin{equation}
    \label{eq:160}
    \int_{\tilde X}v^p\,\d\tilde\mu\le \cA_p(\ppi).
  \end{equation}}

In the next definition we make precise the concept of a square integrable
velocity density for a curve of probability measures: 
differently from \cite{AGS08}, here we can 
consider only the ``modulus'' of the velocity field, but
this already provides an interesting information in many situations.

\begin{definition}[Velocity density]
  Let $\mu\in \rmC([0,1];\Probabilities X)$, 
  $\tilde\mu:= \int \mu_t\,\d\lambda\in
  \Probabilities{\tilde X}$. We say that $v\in L^1(\tilde X,\tilde\mu)$
  is a velocity density for $\mu$ if 
  for every $\varphi\in \Lip_b(X)$ one has
  \begin{equation}
    \label{eq:28}
  \Big|  \int_X \varphi\,\d\mu_t-
    \int_X \varphi\,\d\mu_s\Big|\le 
    \int_{X\times (s,t)} |\rmD^* \varphi|\, v\,\d\tilde\mu\quad
    \forevery \ 0\le s<t\le 1.
  \end{equation}
  The set of velocity densities is a closed convex set in $L^1(\tilde
  X,\tilde\mu)$, {\color{black} and we say that $v$ is a $p$-velocity density if $v\in L^p(\tilde X,\tilde\mu)$.}
  We say that $\bar v\in L^p(\tilde X,\tilde\mu)$ is \emph{the minimal
   $p$-velocity density} if $\bar v$ is the element of minimal 
 $L^p(\tilde X,\tilde\mu)$-norm among all the velocity densities.
\end{definition}
\begin{remark}[Lipschitz test functions with bounded support]
  \label{rem:bounded-support}
  \upshape
  We obtain an equivalent definition by asking that \eqref{eq:28}
  holds
  for every test function $\varphi\in \Lip_b(X)$ \emph{with bounded
    support}:
  in fact, fixing $x_0\in X$ and the family of cut-off functions
  \begin{equation}
    \label{eq:241}
    \psi_R(x)=\eta(\sfd(x,x_0)/R)\quad
    \text{where}\quad
    \eta(y)=(1-(y-1)_+)_+,
  \end{equation}
  every $\varphi\in \Lip_b(X)$ can be approximated by the sequence
  $\varphi_n:=\varphi\cdot \psi_n$; if $v\in L^1(\tilde X,\tilde\mu)$ satisfies \eqref{eq:28} 
  for every Lipschitz function with bounded support, 
  we can use the dominated convergence theorem to pass to the limit as $n\to\infty$ in
  \begin{displaymath}
     \Big|  \int_X \varphi_n\,\d\mu_t-
    \int_X \varphi_n\,\d\mu_s\Big|\le 
    \int_{X\times (s,t)} |\rmD^* \varphi|\, v\,\d\tilde\mu
    +\sup|\varphi| \int_{(\overline{B}_{2n}(x_0)\setminus B_n(x_0))\times (s,t)} \, v\,\d\tilde\mu,
  \end{displaymath}
  since
  \begin{displaymath}
    |\rmD^*\varphi_n|(x)\le 
    |\rmD^*\varphi|(x)\,\psi_n(x)+
    \sup |\varphi| \nchi_{\overline{B}_{2n}(x_0)\setminus B_n(x_0)}(x).
  \end{displaymath}
\end{remark}
For $p\in (1,\infty)$, we are going to show that the minimal $p$-velocity density
exists for curves $\mu\in \AC p{[0,1]}{(\Probabilities X,W_p)}$
and that it is provided exactly by \eqref{eq:26}, for every
dynamic plan with finite $p$-energy $\ppi$ tightened to $\mu$.
Heuristically, this means that for a tightened plan $\ppi$ associated to $\mu$, while branching may occur, the speed of curves
at a given point at a given time is independent of the curve and given by the minimal $p$-velocity.
The starting point of our investigation is provided by the following
simple result.
\begin{lemma}[The mean velocity is a velocity density]
  \label{le:1}
  Let $\ppi$ be a dynamic plan with finite $p$-energy and let 
  $\mu$, $\tilde \mu$, $v$ be defined as in \eqref{eq:25}, \eqref{eq:159},
  \eqref{eq:26}.
  Then $v\in L^p(\tilde X,\tilde\mu)$ 
  is a velocity density for $\mu$.
\end{lemma}
\begin{proof}
  Immediate, since for all $\varphi\in \Lip_b(X)$ the upper gradient property of $|\rmD^*\varphi|$ yields
  \begin{align*}
    \int_X \varphi\,\d\mu_t-
    \int_X \varphi\,\d\mu_s&=
    \int \Big(\varphi(\gamma(t))-\varphi(\gamma(s))\Big)\,\d\ppi(\gamma)
    \le 
    \int \int_s^t |\rmD^* \varphi|(\gamma(r))|\dot \gamma|(r)\,\d
    r\,\d\ppi(\gamma)
    \\&=
    \int_{\rmC([0,1];X)\times (s,t)} |\rmD^* \varphi|(\gamma(r))
    \tilde\rmv(\gamma,r)
    \,\d\tilde\ppi(\gamma,r)=
    \int_{X\times (s,t)} |\rmD^*\varphi| v\,\d\tilde\mu.\qed
  \end{align*}
\endnobox
\end{proof}

The next Lemma shows that we can use a velocity density 
even with time-dependent test functions.
\begin{lemma}
  \label{le:prel}
  Let $\mu\in \rmC([0,1];\Probabilities X)$, 
  $\tilde\mu:= \int \mu_t\,\d\lambda\in
  \Probabilities{\tilde X}$ and let $v\in L^1(\tilde X,\tilde\mu)$
  be a velocity density for $\mu$.
  Then $\mu\in \mathrm{AC}([0,1];(\mathscr P_1(X),W_1))$ and
  for every $\varphi\in \Lip_b(\tilde X)$ one has 
  \begin{equation}
    \label{eq:29}
    \int_X \varphi_t\,\d\mu_t-
    \int_X \varphi_s\,\d\mu_s\le 
    \int_{X\times (s,t)} \big(\partial^+_r \varphi_r+|\rmD^* \varphi_r|\, v\big)\,\d\tilde\mu\quad
    \forevery 0\le s<t\le 1,
  \end{equation}
  where
  \begin{equation}
    \label{eq:31}
    \partial^+_r \varphi_r(x)=\limsup_{h\downarrow0}\frac1h\big(\varphi_{r+h}(x)-\varphi_r(x)\big).
  \end{equation}
\end{lemma}
\begin{proof}
  If $\varphi$ is $1$-Lipschitz then $|\rmD^*\varphi|\le 1$, so that
  from \eqref{eq:28} and the dual chacracterization of $W_1$ we easily get
  \begin{gather*}
    W_1(\mu_s,\mu_t)
    =\sup_{\varphi\in\Lip_b(X),\ \Lip(\varphi)\le 1}\Big|  \int_X \varphi\,\d\mu_t-
    \int_X \varphi\,\d\mu_s\Big|\le 
    \int_s^t m(r)\,\d r,\quad
    \text{where}\\
    m(r):=\int_X  v\,\d\mu_r,
    \quad\text{so that $m\in L^1(0,1)$.}
  \end{gather*}
  If we consider the map $\eta(s,t):=\int_X \varphi_s\,\d\mu_t$ and
  we call $L$ the Lipschitz constant of $\varphi$, we
  easily get for every $0\le s\le s'\le 1,\ 0\le t\le t'\le 1$
  \begin{displaymath}
    |\eta(s',t)-\eta(s,t)|\le L|s'-s|,\qquad
    |\eta(s,t')-\eta(s,t)|\le L\int_t^{t'} m(r)\,\d r,
  \end{displaymath}
  so that we can apply \cite[Lemma 4.3.4]{AGS08} to
  get the absolute continuity of $t\mapsto \eta(t,t)$ with
  \begin{equation}
    \label{eq:32}
    \frac\d{\d t}\eta(t,t)\le 
    \limsup_{h\down0} \frac 1h\int_X \varphi_t\,\d(\mu_t-\mu_{t-h})+
    \limsup_{h\down0} \frac 1h\int_X (\varphi_{t+h}-\varphi_t)\,\d\mu_t.
  \end{equation}
  Choosing a Lebesgue point of $t\mapsto \int_X |\rmD^*
  \varphi|v\,\d\mu_t$ and applying Fatou's Lemma we conclude that we can estimate
  from above the derivative of $t\mapsto\eta(t,t)$ by
  $$
  \int_X\partial^+_t \varphi_t\,\d\mu_t+\int_X|\rmD^* \varphi_t| v_t\,\d\mu_t.
  $$
  Since $t\mapsto\eta(t,t)$ is absolutely continuous, by integration we get the result.
\end{proof}

\begin{theorem}[The metric derivative can be estimated with any velocity density]
  \label{thm:upper-velocity}
  Let $\mu\in \rmC([0,1];\Probabilities X)$, 
  $\tilde\mu:= \int \mu_t\,\d\lambda\in
  \Probabilities{\tilde X}$ and let $v\in L^p(\tilde X,\tilde\mu)$
  be a $p$-velocity density for $\mu$, for some $p\in (1,\infty)$.
  Then $\mu\in \mathrm{AC}^p([0,1];(\Probabilities X,W_p))$ and
   \begin{equation}
     \label{eq:33}
     |\dot \mu_t|^p\le \int_X v^p_t\,\d\mu_t\quad
     \text{for $\lambda$-a.e.\ $t\in (0,1)$}.
   \end{equation}
\end{theorem}
\begin{proof}
  We give the proof in the case $p=2$, the general case is 
  completely analogous.
  With the notation of Kuwada's Lemma \cite[Lemma
  6.1]{AGS11a},
  denoting by $\sfQ_t\varphi$ the Hopf-Lax evolution map 
  given by \eqref{eq:11o}, one has

  \begin{displaymath}
    \frac 12 W^2_2(\mu_s,\mu_t)=
    \sup_{\varphi\in\Lip_b(X)} \int_X \sfQ_1\varphi\,\d\mu_t-
    \int_X \varphi\,\d\mu_s\qquad 0\le s\le t\le 1.
  \end{displaymath}
  Setting $\ell=t-s$ and recalling that \eqref{eq:identities0} gives
  \begin{displaymath}
    \partial_r^+ \sfQ_{r/\ell}\varphi\le-\frac{|\rmD^*
      \sfQ_{r/\ell}\varphi|^2}{2\ell}\quad\text{in }X\times[0,\ell],
  \end{displaymath}
  the inequality \eqref{eq:29} yields
  \begin{align*}
    \int_X \sfQ_1\varphi\,\d\mu_t-
    \int_X \varphi\,\d\mu_s&\le 
    \int_0^\ell \int_X \Big(-\frac{|\rmD^* \sfQ_{r/\ell}\varphi|^2}{2\ell}+
    |\rmD^* \sfQ_{r/\ell}\varphi|\,v_{s+r} \Big)\,\d\mu_{s+r}\,\d r
    \\&\le 
    \frac\ell2\int_0^\ell \int_X v_{s+r}^2\,\d\mu_{s+r}\,\d r,
  \end{align*}
  where we used that $2 |\rmD^* \sfQ_{r/\ell}\varphi|\,v_{s+r} \leq |\rmD^* \sfQ_{r/\ell}\varphi|^2/\ell+ \ell v_{s+r}^2$.    
  We conclude that 
  \begin{displaymath}
     \frac 12 W_2^2(\mu_s,\mu_t)\le \frac 12(t-s)\int_s^t \Big(\int_X
     v_{r}^2\,\d\mu_{r}\Big)\,\d r,
  \end{displaymath}
that yields first the $2$-absolute continuity of the curve $t\mapsto\mu_t$ in $(\ProbabilitiesTwo X, W_2)$.  
Also inequality \eqref{eq:33} follows, because we have
$$
|\dot{\mu}_t|^2= \lim_{h\to 0} \frac{W_2^2(\mu_{t+h}, \mu_t)}{h^2} \leq   \lim_{h\to 0} \frac{1}{h} \int_{t}^{t+h} \left(  \int_X
     v_{r}^2\,\d\mu_{r}\right) \d r= \int_X
     v_{t}^2\,\d\mu_{t}
$$
     for $\lambda$-a.e. $t \in (0,1)$.
\end{proof}

\begin{theorem}[Existence and characterization of the metric velocity density]\label{thm:MVD}\
  \begin{enumerate}[{\rm $[$M.1$]$}]
  \item A curve $\mu\in \rmC([0,1];\Probabilities X)$ belongs to
    $\mathrm{AC}^p([0,1];(\Probabilities X,W_p))$, $p\in (1,\infty)$,
    if and only if
    $\mu$ admits a velocity density in $L^p(\tilde X,\tilde\mu)$.
    In this case 
    there exists a unique (up to $\tilde\mu$-negligible sets)
    minimal $p$-velocity density $\bar v\in L^p(\tilde X,\tilde\mu)$ and
    \begin{equation}
      \label{eq:161}
      |\dot \mu_t|^p=\int_X \bar v^p\,\d\mu_t\quad
      \text{for $\lambda$-a.e.~$t\in (0,1)$.}
    \end{equation}
  \item  
    If $\ppi$ is a dynamical plan $p$-tightened to $\mu$ and the mean velocity $v$ of $\ppi$
    is defined as in \eqref{eq:26}, then $\bar v=v$ $\tilde\mu$-a.e.
    in $\tilde X$ and
  \begin{equation}
    \label{eq:27}
    \bar v(\gamma(t),t)= |\dot \gamma|(t)
    \quad\text{for $\tilde\ppi$-a.e.\
      $(\gamma,t)$.}
  \end{equation}
  In particular, the
  velocity of curves depends $\tilde\ppi$-a.e. only on $(\gamma(t),t)$ and it is independent of the choice of $\tilde\ppi$.
  \end{enumerate}
\end{theorem}
\begin{proof}
  The characterization of $\AC p{[0,1]}{(\Probabilitiesp X,W_p)}$ 
  in terms of the existence of a velocity density in $L^p(\tilde
  X,\tilde\mu)$ 
  \AAA follows in the  \emph{only if}  part from the combination of Theorem \ref{thm:lisini} with Lemma \ref{le:1}, and in the \emph{if} part  
  from Theorem~\ref{thm:upper-velocity}. \fn
  The existence and the uniqueness of the minimal $p$-velocity density
  is a consequence of the strict convexity of the $L^p$-norm.
  
  If $\ppi$ is a dynamic plan $p$-tightened to $\mu$
  and $v$ is defined in terms of \eqref{eq:26}, we can 
  combine \eqref{eq:160} and Theorem~\ref{thm:upper-velocity} (which provides the sharp lower
  bound on the $L^p$ norm of velocity densities)
  to obtain that $v$ is the minimal $p$-velocity density and that 
  $|\dot\mu_t|^p=\int_X v_t^p\,\d\mu_t$ for $\lambda$-a.e. $t\in (0,1)$,
  so \eqref{eq:161} follows.  Combining this information with
  \eqref{eq:156bis}  yields
  $$
  \int_X v_t^p\,\d\mu_t=\int_X |\dot\gamma_t|^p\,d\ppi(\gamma)\qquad\text{for $\lambda$-a.e. $t\in (0,1)$,}
  $$
  so that, recalling \eqref{eq:26}, we get 
  $$
  \int_{\tilde{X}}\biggl(\int |\dot\gamma_t|\,\d\tilde\ppi_{x,t}(\gamma)\biggr)^p\,\d\tilde\mu(x,t)=
  \int_{\tilde{X}}\int |\dot\gamma_t|^p\,\d\tilde\ppi_{x,t}(\gamma)\,\d\tilde\mu(x,t).
  $$
  It follows that, for $\tilde\mu$-a.e. $(x,t)$, $|\dot\gamma_t|$ is $\tilde\ppi_{x,t}$-equivalent to a constant. By the
  definition of $v$, this gives \eqref{eq:27} with $v$ in place of $\bar v$. Using the coincidence of $v$ and
  $\bar v$ we conclude.
 \end{proof}

\section{Weighted energy functionals along absolutely continuous curves}\label{sec:weights}

Let $\mm$ be a reference measure in $X$ such that 
$(X,\sfd,\mm)$ is a metric measure space according to
Section~\ref{subsec:Cheeger}, and set $\tilde\mm:=\mm\otimes \lambda$, with $\lambda=\Leb{1}\vert_{[0,1]}$.
Let $\frQ:[0,1]\times [0,\infty)\to [0,\infty]$ be 
a lower semicontinuous function satisfying
\begin{equation}
  \label{eq:165}
  \lim_{r\down0} r\,\weight sr=0\quad \forevery s\in [0,1].
\end{equation}
Our typical example will be of the form
\begin{equation}
  \label{eq:50}
  \weight sr=\omega(s)Q(r).
\end{equation}
Let us fix an exponent $p\in (1,\infty)$ and
let us consider a curve $\mu\in \AC p{[0,1]}{(\Probabilities X,W_p)}$. We denote
$\tilde\mu=\int_0^1 \mu_s\,\d\lambda(s)\in \Probabilities{\tilde X}$ {\color{black} (namely the probability measure
whose second marginal is $\lambda$ and whose disintegration w.r.t. the second variable is $\mu_s$, $s\in (0,1)$),} 
and by $v\in L^p(\tilde X,\tilde\mu)$ the minimal $p$-velocity density of $\mu$.
We suppose that 
\begin{equation}
  \label{eq:166}
  \tilde\mu=\varrho\tilde\mm\ll\tilde\mm, \quad\text{so that}\quad
  \varrho(s,\cdot)=\varrho_s(\cdot)=\frac{\d\mu_s}{\d\mm}\quad\text{for $\lambda$-a.e. $s\in (0,1)$},
\end{equation}
where $\d\nu/\d\mm$ denotes the Radon-Nikodym density
of the measure $\nu$ with respect to $\mm$.
Then we introduce the functional
\begin{equation}
  \label{eq:17}
  \begin{aligned}
    \Action{\wname}{\mu}\mm:= &
    \int_{\tilde X}
    \weight s{\varrho_s}
    \,v^p\,\d\tilde\mu
    =
    \int_{\tilde X}\varrho\, \weight s{\varrho_s}
    \,v^p\,\d\tilde\mm.
  \end{aligned}
\end{equation}
We omit to indicate the dependence on $p$ in the notation of the
functional $\mathcal A_\frQ$, since $p$ will be fixed throughout this
section.
Notice that when $\weight sr=1$ we have the usual action $\int_{\tilde X}v^p\,\d\tilde\mu=\cA_p(\mu)$, 
the functional is independent of $\mm$ and it makes sense even for curves not
contained in $\Probabilitiesac X\mm$.

If $\ppi$ is a dynamic plan $p$-tightened to $\mu$ (recall that this means $\mathscr A_p(\ppi)=\cA_p(\mu)$), thanks to 
\eqref{eq:27}
we have the equivalent expression
  \begin{equation}
    \label{eq:34}
    \begin{aligned}
      \Action
      {\wname}\mu\mm &= 
      \iint_0^1 \weight {s}{\varrho_s(\gamma(s))}
      |\dot\gamma|^p(s)
      \,\d s\,\d\ppi(\gamma).
    \end{aligned}
  \end{equation}

\begin{theorem}[Stability of the weighted action]
 \label{le:stabilitySigma}
 Let $(\mu_n)\subset\AC p{[0,1]}{(\Probabilities X,W_p)}$
 with $\tilde\mu_n=\varrho_n\tilde\mm\ll\tilde\mm$,
 such that
   \begin{align}
 \label{eq:162}
   \lim_{n\to\infty}\varrho_n=\varrho_\infty\quad\text{strongly in
   }L^1(\tilde X,\tilde\mm)
 \end{align}
 and, writing $\varrho_\infty\tilde\mm=:\tilde{\mu}_\infty=: \int_{0}^{1} \mu_{\infty}(s) \, \d \lambda(s)$,
 one has
 \begin{align}
   \label{eq:162bis}
\limsup_{n\to\infty}\cA_p(\mu_n)\leq\cA_p(\mu_\infty)<\infty.
 \end{align}
 Then
 \begin{equation}
 \label{eq:163}
  \liminf_{n\to\infty}\Action{\wname}{\mu_n}\mm\ge
  \Action{\wname}{\mu_\infty}\mm
 \end{equation}
 and, whenever $\frQ$ is continuous and bounded,
 \begin{equation}
   \label{eq:163bis}
\lim_{n\to\infty}\Action{\wname}{\mu_n}\mm=\Action{\wname}{\mu_\infty}\mm.
 \end{equation}
\end{theorem}
\begin{proof} {\color{black} In this proof we are going to apply standard facts from the theory of Young measures, we follow
\cite{AGS08}, even though the state space therein is a Hilbert space, because the vector structure plays
no role in the results we quote, the Polish structure being sufficent.}
Up to extraction of a subsequence, \eqref{eq:162} and equi-continuity in the weak topology yield
 \begin{equation}
   \label{eq:212}
   \mu_{n,s}\to\mu_{\infty,s}\text{ weakly in $\Probabilities X$ for every $s\in [0,1]$}.
 \end{equation}
 We can apply \cite[Thm.~5.4.4]{AGS08}
 first to the sequence $(\tilde\mu_n,v_n)$, with $v_n$ equal to the velocity densities of $\mu_n$.
 We find that the family of plans $\nnu_n:=(\ii\times v_n)_\sharp \tilde\mu_n$
 has a limit point $\nnu_\infty$ in $\Probabilities{\tilde X\times [0,\infty)}$
 whose first marginal is $\tilde\mu_\infty$ and satisfies
 {\color{black} (redefining $\nnu_n$ to be the subsequence converging to $\nnu_{\infty}$)}
 \begin{equation}\label{eq:162bisbis}
 \int_{\tilde{X}\times [0,\infty)}|y|^p\,d\nnu_\infty(x,s,y)\leq\liminf_{n\to\infty}
 \int_{\tilde{X}\times [0,\infty)}|y|^p\,d\nnu_n(x,s,y)\leq\cA_p(\mu_\infty).
 \end{equation}
 If $(\nu_{x,s})_{(x,s)\in \tilde X}\subset \Probabilities{[0,\infty)}$
 is the disintegration of $\nnu_\infty$ w.r.t.\ $\tilde\mu_\infty$,
 setting
 \begin{displaymath}
   v_\infty(x,s):=\int_0^\infty y\,\d\nu_{x,s}(y),
 \end{displaymath}
 we obtain from the previous inequality and Jensen's inequality
 that $v_\infty$ belongs to $L^p(\tilde{X},\tilde{\mu}_\infty)$, with
 $\|v_\infty\|_{L^p(\tilde{X},\tilde{\mu}_\infty)}^p\leq \cA_p(\mu_\infty)$.

   For every $\varphi\in \Lip_b(X)$ and $0\le r<s\le 1$,
 we can use the upper semicontinuity of $|\rmD^*\varphi|$ to pass to the limit in the family of
 inequalities corresponding to \eqref{eq:28}
 \begin{displaymath}
   \Big|  \int_X \varphi\,\d\mu_{n,s}-
   \int_X \varphi\,\d\mu_{n,r}\Big|\le
   \int_{X\times (r,s)} |\rmD^* \varphi|\, v_n\,\d\tilde\mu_n=
   \int_{X\times (r,s)\times [0\infty)}|\rmD^*\varphi|(x)y\,d\nnu_n(x,s,y),
     \end{displaymath}
 obtaining
   \begin{displaymath}
     \Big|  \int_X \varphi\,\d\mu_{\infty,s}-
   \int_X \varphi\,\d\mu_{\infty,r}\Big|\le
   \int_{X\times (r,s)} |\rmD^* \varphi|\, v_\infty\,\d\tilde\mu_\infty.
 \end{displaymath}
 It follows that $v_\infty$ is a $p$-velocity density for the curve $\mu_\infty$, so that
$\|v_\infty\|_{L^p(\tilde{X},\tilde{\mu}_\infty)}^p\geq\cA_p(\mu_\infty)$ and since we already proved the converse inequality,
equality holds. {\color{black}
 If $v_\infty^*$ is the minimal $p$-velocity density, from the equality $\|v_\infty^*\|_{L^p(\tilde{X},\tilde{\mu}_\infty)}^p= \cA_p(\mu_\infty)$
 we get $v_\infty=v_\infty^*$.}
 %
 Denoting now by $(\tilde x,y,r)=(x,s,y,r)$ the coordinates in
 $\tilde X\times [0,\infty)\times [0,\infty)$, let us now consider the plans
 $\ssigma_n:=(\tilde x,y,\varrho_n(\tilde x))_\sharp \nnu_n=
 (\tilde x,v_n(\tilde x),\varrho_n(\tilde x))_\sharp \tilde\mu_n
 \in\Probabilities{\tilde X\times [0,\infty)\times [0,\infty)}$.
 From \eqref{eq:162} we obtain the existence of $\zeta:[0,\infty)\to [0,\infty)$ with $\zeta(r)\to +\infty$ as $r\to\infty$ such that
 \begin{displaymath}
  \sup_{n\in\N}\int \zeta(r)\ssigma_n(\tilde x,y,r)=\sup_{n\in\N}\int_{\tilde X}\varrho_n\zeta(\varrho_n)\,\d\tilde\mm<\infty,
 \end{displaymath}
 so that $\ssigma_n$ are tight (the marginals of $\ssigma_n$
 with respect to the block of variables $(\tilde x,y)$ are $\nnu_n$, thus are tight).
 We can then extract a subsequence
 (still denoted by $\ssigma_n$) weakly converging
 to $\ssigma_\infty$, whose marginal w.r.t. $(\tilde x,y)$ is $\nnu_\infty$.

 The strong $L^1$ convergence in \eqref{eq:162} also shows that
 for every $\zeta\in \rmC_b(\tilde X\times \R)$ one has
 \begin{align*}
   \int \zeta(\tilde x, r)\,\d\ssigma_\infty(\tilde x,y,r)
   &=\lim_{n\to\infty}
   \int \zeta(\tilde x, r)\,\d\ssigma_n(\tilde x,y,r)=
   \lim_{n\to\infty}
   \int_{\tilde X} \zeta(\tilde x,\varrho_n(\tilde
   x))\varrho_n(\tilde x)\,\d\tilde\mm(\tilde x)
   \\&
   = \int_{\tilde X} \zeta(\tilde x,\varrho_\infty(\tilde
   x))\varrho_\infty(\tilde x)\,\d\tilde\mm(\tilde x)=
   \int_{\tilde X}\zeta(\tilde x,\varrho_\infty(\tilde
   x))\,\d\tilde\mu_\infty(\tilde x),
 \end{align*}
 so that $(\ii\times\varrho_\infty)_\sharp \tilde\mu_\infty$ is the marginal w.r.t. $(\tilde x,r)$ of $\ssigma_\infty$.

 Hence, disintegrating $\ssigma_\infty$ with respect to $\nnu_\infty$,
 we obtain $$\ssigma_\infty(d\tilde x,dy,dr)=\delta_{\rho_\infty(\tilde x)}(dr)\times\nu_{\tilde x}(dy) d\mu_\infty(\tilde x).$$
 Since the map $(\tilde x,y,r)\mapsto \weight sr y^p$
 is lower semicontinuous and nonnegative in $\tilde X\times [0,\infty)^2$,
 assuming (with no loss of generality) $1=\sup\frQ$, we get
 \begin{eqnarray*}
   \liminf_{n\to\infty}\Action{\wname}{\mu_n}\mm&=&
   \liminf_{n\to\infty}\int \weight sr y^p
   \,\d\ssigma_n(\tilde x,y,r)\ge
   \int \weight sr y^p
   \,\d\ssigma_{\infty}(\tilde x,y,r)
   \\&\geq&\int_{\tilde X}
   \weight s{\varrho_\infty}
   v_\infty^p\,\d\tilde\mu_\infty
\geq\Action{\wname}{\mu_\infty}\mm.
 \end{eqnarray*}
 By applying this property with $\frQ$ and $1-\frQ$, since $\Action{\wname}{\mu}\mm=\cA_p(\mu)$
 when $\frQ\equiv 1$ and \eqref{eq:162bis} holds,
 we can use Remark~\ref{rem:trick} to obtain \eqref{eq:163bis}. \qed
 \endnobox
\end{proof}

\section{Dynamic Kantorovich potentials, continuity equation and dual weighted Cheeger energies}\label{sec:dynamic}

In this section we will still consider a metric measure space $(X,\sfd,\mm)$
according to the definition of Section~\ref{subsec:Cheeger} 
and we will focus on the particular case
when 
\begin{equation}
\text{$p=2$ and the Cheeger energy $\C$
is quadratic (see \eqref{eq:1})},
\label{eq:240}
\end{equation}
so that $\V=W^{1,2}(X,\sfd,\mm)$ is a separable Hilbert space;
we will also consider continuous
curves $(\mu_s)_{s\in [0,1]}\subset \Probabilities X$
with uniformly bounded densities w.r.t.~$\mm$, i.e.
\begin{equation}
  \label{eq:184}
  \mu\in \rmC([0,1];\Probabilities X),\qquad
  \mu_s=\varrho_s\mm,\qquad
  R:=\sup_s\|\varrho_s\|_{L^\infty(X,\mm)}<\infty.
\end{equation}
Our main aim is to show that the weighted energies
$\cE_{\varrho_s}$ (or better, their dual forms $\cE_{\varrho_s}^*$)
provide a useful characterization of the minimal $2$-velocity of 
absolutely continuous curves $\mu$ in $(\ProbabilitiesTwo X,W_2)$, now
not only in the form of inequality as in \eqref{eq:28}, 
but in the form of equality,
see \eqref{eq:42}.

\begin{lemma}[Absolute continuity w.r.t. $\Vdual 1$]
  \label{le:w-s}
  Let $\mu$ be as in \eqref{eq:184}
  and let $v\in L^2(\tilde X,\tilde\mu)$ be a velocity density for $\mu$, i.e.  satisfying \eqref{eq:28}.
  Then for every $\varphi\in\V$ one has
    \begin{equation}
    \label{eq:28w}
    \Big|  \int_X \varphi\,(\varrho_t-\varrho_s)\,\d\mm\Big|\le 
    \int_{X\times (s,t)} |\rmD \varphi|_w\, v\,\d\tilde\mu\quad
    \forevery \ 0\le s<t\le 1.
  \end{equation}
  In addition $\varrho:[0,1]\to L^1_+\cap L^\infty_+(X,\mm)$ has finite $2$-energy with respect to 
  the $\Vdual 1$ norm, more precisely
  \begin{equation}
    \label{eq:10}
   \|\varrho_s-\varrho_t\|_{\Vdual1}^2
    \le R(t-s)
    \int_{X\times(s,t)}v^2\,\d\tilde\mu\quad
    \forevery 0\le s<t\le 1.
  \end{equation}
\end{lemma}
\begin{proof} In order to prove \eqref{eq:28w} we
  simply approximate $\varphi$ with a sequence of Lipschitz
  functions
  $\varphi_n$ strongly converging to $\varphi$ in $L^2(X,\mm)$ 
  such that $|\rmD^* \varphi_n|\to |\rmD\varphi|_w$ in $L^2(X,\mm)$
  and we pass to the limit in \eqref{eq:28}, using the fact that 
  $\mu_t=\varrho_t\mm$ with uniformly bounded densities.

  By \eqref{eq:28w} it follows that 
  \begin{align*}
    \Big|\int_X (\varrho_t-\varrho_s)\varphi\,\d\mm\Big|
    &\le
    \int_s^t 
    \Big(\int_X |\rmD\varphi|^2_w\,\varrho_r\d\mm\Big)^{1/2}
    \Big(\int_X  v_r^2\,\varrho_r\,\d\mm\Big)^{1/2}\,\d r
    \\&\le R^{1/2}\Big(\int_X |\rmD\varphi|^2_w\,\d\mm\Big)^{1/2}
    \int_s^t\Big(\int_X  v_r^2\,\varrho_r\,\d\mm\Big)^{1/2}\,\d r,
  \end{align*}
  and since $\varphi$ is arbitrary we obtain \eqref{eq:10}.
\end{proof}

\begin{theorem}[Dual Kantorovich potentials and links with the minimal velocity]
  \label{thm:dynamicKant}
  Let us assume that $\C$ is a quadratic form and let $\mu$ be as in \eqref{eq:184}.
  Then $\mu$ belongs to $\AC2{[0,1]}{(\Probabilities X,W_2)}$ 
  if and only if there exists $\ell\in L^2(0,1;\Vdual{})$ such that 
  \begin{equation}
    \label{eq:42}
    \int_X \varphi\,(\varrho_t-\varrho_s)\,\d\mm
    =\int_s^t \langle \ell(r),\varphi\rangle\,\d r
    \quad\forevery \varphi\in \V,
  \end{equation}
  and, recalling the definition \eqref{eq:97} of  $\cEs{\varrho_r}$,
  \begin{equation}
    \label{eq:155}
    \int_0^1 \cEs{\varrho_r}(\ell(r),\ell(r))\,\d r<\infty.
  \end{equation}
  In particular $\ell(r)\in\Vhom{\varrho_r}'$ for $\Leb{1}$-a.e. $r\in (0,1)$ and, moreover, it is linked to the minimal
  velocity density $v$ of $\mu$ by 
  \begin{equation}
    \label{eq:168}
    \int_X v_r^2\,\varrho_r\,\d\mm=\cEs{\varrho_r}(\ell(r),\ell(r))
    \quad\text{for $\Leb 1$-a.e.\ $r\in (0,1)$,}
  \end{equation}
    \begin{equation}
    \label{eq:113}
    v_r^2=\tGq{\varrho_r}{\phi_r}\quad\mu_r\text{-a.e.~in }X,\quad
    \quad\text{for $\Leb 1$-a.e.\ $r\in (0,1),$}
  \end{equation}
  where 
  $\phi_r=-A_\varrho^*\ell(r)\in \Vhom{\varrho_r}$ is the solution of \eqref{eq:112} with $\ell=\ell(r)$. 
  \end{theorem}
\begin{proof}
  If $\mu\in \AC2{[0,1]}{(\Probabilities X,W_2)}$ then
  the existence of $\ell$ and \eqref{eq:42} follow immediately by 
  Lemma~\ref{le:w-s}, Theorem~\ref{thm:MVD}
  and the fact that $\V'$ is a separable Hilbert space.
  Differentiating \eqref{eq:28w} with $v$ equal to the minimal velocity density
  in a Lebesgue point for
  $s\mapsto \int_X |\rmD\varphi|_w v_s\varrho_s\,\d\mm$
  and for a countable dense set of 
  test functions $\varphi$ in $\V$ we get
  \begin{equation}
    \label{eq:186}
    \cEs{\varrho_r}(\ell(r),\ell(r))\le \int_X v_r^2\,\varrho_r\,\d\mm
    \quad\text{for $\Leb 1$-a.e.\ $r\in (0,1)$},
  \end{equation}
  which in particular yields \eqref{eq:155}.
  
  In order to prove the converse implication (and that equality holds in \eqref{eq:168}, as well as \eqref{eq:113}), 
  let us start from $\mu$ as in \eqref{eq:184}, satisfying \eqref{eq:42} and \eqref{eq:155} for some $\ell\in L^2(0,1;\Vdual{})$.
  Let us consider
  $\psi\in \Lip(X)$ with bounded support, the solution
  $\phi_r=-A_\varrho^*\ell(r)\in \Vhom{\varrho_r}$ of \eqref{eq:112} with $\ell=\ell(r)$
  and $\psi_{\varrho_r}$ the equivalence class 
  associated to $\psi$ in $\Vhom{\varrho_r}$, so that
  \begin{displaymath}
    \langle \ell(r),\psi\rangle=
    \int \varrho_r \tGbil{\varrho_r}{\phi_r}{\psi_{\varrho_r}}\,\d\mm.
  \end{displaymath}
  Now observe that \eqref{eq:42} and \eqref{eq:155} yield for every
  $0\le s<t\le 1$
  \begin{displaymath}
    \Big|\int_X \psi\,\d\mu_t-\int_X \psi\,\d\mu_s\Big|\le 
    \int_s^t \Big|\int_X \varrho_r \tGbil{\varrho_r}{\phi_r}{\psi_{\varrho_r}}\,\d\mm\Big|\,\d r\le 
    \int_s^t\int_X \varrho_r \big(\tGq{\varrho_r}{\phi_r}\big)^{1/2}
    |\rmD^* {
      \psi}|\,\d\mm\,\d r
  \end{displaymath}
  since for $\psi\in \Lip_b(X)$ 
  \begin{displaymath}
    \tGq{\varrho_r}{\psi_{\varrho_r}}=\Gq{\psi}=
    |\rmD \psi|_w^2\le |\rmD^* \psi|^2
    \quad \varrho_r\mm\text{-a.e.~in }X.
  \end{displaymath}
  In view of Remark~\ref{rem:bounded-support},
  we see that $\hat v_r=\big(\tGq{\varrho_r}{\phi_r}\big)^{1/2}$ is
  a velocity density for the curve $\mu$.
  Applying Theorem~\ref{thm:upper-velocity} and \eqref{eq:113bist} we 
  get $\mu\in \AC2{[0,1]}{(\Probabilities X,W_2)}$. In addition, since
  \begin{displaymath}
   \int_X \hat v_r^2\,\varrho_r\d\mm=
   \int_X \tGq{\varrho_r}{\phi_r}\varrho_r\,\d\mm=
   \cE_{\varrho_r}(\phi_r,\phi_r)
    = \cEs{\varrho_r}(\ell(r),\ell(r))
    \quad\text{for $\Leb 1$-a.e.~$r\in (0,1)$},
  \end{displaymath}
  comparing with \eqref{eq:186} we obtain that $\hat v$ is the minimal velocity density $v$, 
  thus obtaining \eqref{eq:168} and \eqref{eq:113}.
\end{proof}

\section{The $\RCDS KN$ condition and its characterizations through weighted convexity
and evolution variational inequalities}\label{sec:9}

\subsection{Green functions on intervals}\label{subsec:9.1}

We define the function $\sfg:[0,1]\times [0,1]\to [0,1]$ by
\begin{equation}
  \label{eq:3}
  \fnd t{}s:=
  \begin{cases}
    (1-t)s&\text{if }s\in[0,t],\\
    t(1-s)&\text{if }s\in [t,1],
  \end{cases}
\end{equation}
so that for all $t\in (0,1)$ one has
\begin{equation}
  \label{eq:4}
  -\frac{\partial^2}{\partial s^2}\fnd t{}s=\delta_{t}\quad\text{in $\mathscr D'(0,1)$},\qquad 
  \fnd t{}{0}=\fnd t{}{1}=0.
\end{equation}
It is not difficult to check that (see e.g.\ \cite[Chap.~16]{Villani09}) the condition $u''\geq f$ can be characterized 
in terms of an integral inequality involving $\sfg$.

\begin{lemma}[Integral formulation of $u''\geq f$]
  \label{le:0}
  Let $u\in \rmC([0,1])$ and $f\in L^1(0,1)$. Then 
  \begin{equation}
    \label{eq:5}
    u''\ge f\quad\text{in }\mathscr D'(0,1),
  \end{equation}
  if and only if 
  for every $0\le r_0\le r_1\le 1$ and $t\in [0,1]$ one has  
  \begin{equation}
    \label{eq:7}
    u((1-t)r_0+tr_1)\le 
    (1-t)u(r_0)+tu(r_1)-(r_1-r_0)^2\int_{0}^1
    f((1-s)r_0+sr_1)\fnd t{}s\,\d s.
     \end{equation}
\end{lemma}
\begin{proof}
  In order to prove the implication from \eqref{eq:5} to \eqref{eq:7} 
  it is not restrictive to assume $u\in \rmC^2([0,1])$ and 
  $f\in \rmC([0,1])$.
  The proof of \eqref{eq:7} follows easily from the elementary identity
  \begin{displaymath}
    u((1-t)r_0+tr_1)=(1-t)u(r_0)+tu(r_1)-(r_{1}-r_{0})^{2}\int_0^1 u''((1-s)r_0+sr_1)\fnd t{}s\,\d s.
  \end{displaymath}
  Concerning the converse implication, we choose 
  $r_1:=r+h$, $r_0=r-h$ and $t=\frac 12$ obtaining
  \begin{displaymath}
    \frac 12 u(r+h)+\frac 12 u(r-h)-u(r)\ge 
    4h^2\int_{0}^1 f(r-h+2hs)\fnd {1/2}{}s\,\d s.
  \end{displaymath}
  Multiplying by $2h^{-2}$ and by a nonnegative test function
  $\zeta\in \rmC^\infty_c(0,1)$ 
  we get after an integration
  \begin{displaymath}
    \frac 1{h^2} \int_0^1 u(r)\big(\zeta(r+h)+\zeta(r-h)-2\zeta(r)\big)\,\d
    r\ge 
    8\int_0^1 \fnd{1/2}{}s 
    \Big(\int_0^1 f(r-h+2hs)\zeta(r)\,\d r\Big)\,\d s.
  \end{displaymath}
  Passing to the limit as $h\down0$ we obtain
  \begin{displaymath}
   \int_0^1 u\,\zeta''\,\d r\ge 8
   \int_0^1 \fnd{1/2}{}s\,\d s 
   \int_0^1 f\,\zeta\,\d r
   =\int_0^1 f\,\zeta\,\d r.
    \qed
  \end{displaymath}
  \endnobox
\end{proof}

In the next lemma we show that functions satisfying the weighted convexity condition 
\eqref{eq:7} are locally Lipschitz, this will allow us to apply Lemma~\ref{le:0}.

\begin{lemma}\label{lem:LocLipF}
  Let  ${\mathfrak D}\subset \R$,  ${\mathfrak D} \neq \{0\}$, 
  be a $\Q$-vector space and let $u:(0,1)\cap {\mathfrak D} \to \R$ satisfy  \eqref{eq:7} for some $f \in L^1_{\rm loc}(0,1)$, for every 
  $r_0,\,r_1\in (0,1)\cap {\mathfrak D}$, $t\in [0,1]$
  {such that $(1-t)r_0+tr_1\in{\mathfrak D}$}. Then $u$ is locally Lipschitz in $(0,1)$, 
  more precisely for every closed subinterval $[a,b]\subset  (0,1)$ there exists $C\geq 0$ such that
  \begin{equation}\label{eq:loclipf}
  |u(x)-u(y)|\leq C|x-y|  \qquad \forall x,\,y \in (a,b) \cap {\mathfrak D}.
  \end{equation} 
   \end{lemma}
   
\begin{proof}  Since the statement is local and $\mathfrak D$ is dense, we can assume with no loss of generality that
$f\in L^1(0,1)$, that $0,\,1\in\mathfrak D$ and that \eqref{eq:7} holds $r_0,\,r_t,\,r_1\in [0,1]\cap {\mathfrak D}$, with
$r_t:=(1-t)r_0+tr_1$.
First of all note that \eqref{eq:7} is equivalent to the following control on the incremental ratios: for every  $r_0,\,r_t,\,r_1\in  [0,1]\cap {\mathfrak D}$ one has 
\begin{equation}\label{eq:IncrRatf}
\frac{u(r_t)- u(r_0)}   {r_t-r_0} \leq \frac{ u(r_1)- u(r_t)}   {r_1-r_t}-\frac{r_1-r_0}{t(1-t)} \int_0^1 f(r_s)\fnd t{}s\,\d s.
\end{equation}
Observing that $0\leq \fnd t{}s \leq t (1-t)$, we can easily estimate the remainder in the last inequality by
\begin{equation}\label{eq:fL1est}
\left| \frac{r_1-r_0}{t(1-t)} \int_0^1 f(r_s)\fnd t{}s\,\d s  \right| \leq \int_{r_0}^{r_1} |f(r)| \, \d r = \|f\|_{L^1(r_0, r_1)}. 
\end{equation}

Given  $a<b\in (0,1)\cap {\mathfrak D}$,  for every $x,\,y \in {\mathfrak D}\cap (a,b)$, $x<y$,  we want to use \eqref{eq:IncrRatf} iteratively in order to estimate  the difference quotient ${|u(x)-u(y)|}/{|x-y|}$.
\\ Applying \eqref{eq:IncrRatf} with $r_0=0$, $r_1=x$, $r_t=a$ we obtain
\begin{equation}\label{eq:(15)f}
\frac{u(a)- u(0)}   {a} \leq \frac{ u(x)- u(a)}   {x-a} + \|f\|_{L^1(0,x)}.
\end{equation}
Analogously, choosing $r_0=a$, $r_1=y$, $r_t=x$ in  \eqref{eq:IncrRatf} yields
\begin{equation}\label{eq:(16)f}
\frac{u(x)- u(a)}   {x-a} \leq \frac{ u(y)- u(x)}   {y-x} + \|f\|_{L^1(a,y)}.
\end{equation}
Putting together  \eqref{eq:(15)f} and \eqref{eq:(16)f} we obtain the desired lower bound
$$\frac{ u(y)- u(x)}   {y-x} \geq  \frac{ u(a)- u(0)}   {a} - 2 \|f\|_{L^1(0,1)}. $$
Along the same lines one gets also the upper bound
$$\frac{ u(y)- u(x)}   {y-x} \leq   \frac{ u(1)- u(b)}   {1-b} + 2 \|f\|_{L^1(0,1)}.$$
Since the last two estimates hold for every  $x,\,y \in {\mathfrak D}\cap (a,b)$ with $x\neq y$, the proof is complete.
\end{proof}
The next lemma provides a subdifferential inequality, in a quantitative form involving $f$.

\begin{lemma}
  \label{le:2}
  Suppose that $u\in \rmC([0,1])$ satisfies $u''\ge f$ in $\mathscr D'(0,1)$
  for some $f\in L^1(0,1)$. Then, setting 
  $u'(0_+):=\limsup_{t\down0}(u(t)-u(0))/t$, we get
  \begin{equation}
    \label{eq:8}
    u(1)-u(0)-u'(0_+)\ge \int_0^1 f(s)(1-s)\,\d s.
  \end{equation}
\end{lemma}
\begin{proof}
  Notice that by \eqref{eq:7}
  \begin{displaymath}
    u(t)-u(0)\le t(u(1)-u(0))-\int_0^1 f(s)\fnd t{}s\,\d s.
  \end{displaymath}
  Dividing by $t$ and passing to the limit as $t\down0$, since
  $\lim_{t\down0}t^{-1}\fnd t{}s= 1-s$ pointwise in $(0,1]$ and 
  $0\le  t^{-1}\fnd t{}s\le (1-s)$, we get \eqref{eq:8}.
\end{proof}

A similar result holds for the solutions $u$ of
the differential inequality
\begin{equation}
  \label{eq:169}
  u\in \rmC([0,1]),\quad
  u''+\kappa \, u\le 0\quad\text{in }\mathscr D'(0,1),\qquad
  \kappa\in \R.
\end{equation}
In this case, choosing $[r_0,r_1]\subset [0,1]$ with
$\delta=r_1-r_0\in (0,1]$, we can compare 
the function $t\mapsto u((1-t)r_0+tr_1)$, which solves
$w''+\kappa\delta^2 w\leq 0$ in $\mathscr D'(0,1)$, with the solution
of the Dirichlet problem
\begin{equation}
  \label{eq:242}
  v''+\kappa\,\delta^2\, v=0\quad\text{in }(0,1),\quad v(0)=u(r_0),\ v(1)=u(r_1),
\end{equation}
given by 
\begin{equation}
  \label{eq:243}
  v(t)=u(r_0)\frac{\sin(\omega (1-t))}{\sin(\omega)}
  +u(r_1)\frac{\sin(\omega t)}{\sin(\omega)}
  \qquad\text{if }\kappa\,\delta^2=\omega^2\in (0,\pi^2),
\end{equation}
and by
\begin{equation}
  \label{eq:244}
  v(t)=u(r_0)\frac{\sinh(\omega (1-t))}{\sinh(\omega)}+u(r_1)\frac{\sinh(\omega t)}{\sinh(\omega)}
  \qquad\text{if }\kappa\,\delta^2=-\omega^2<0,
\end{equation}
observing that the comparison principle gives $u((1-t)r_0+tr_1)\ge v(t)$ for every $t\in [0,1]$.

By introducing the factors 
\begin{equation}
  \label{eq:170}
  \factor \kappa t\delta:=
  \begin{cases}
    +\infty&\text{if }\kappa\,\delta^2\ge \pi^2,\\
    \displaystyle
    \frac{\sin(\omega t)}{\sin(\omega )}&
    \text{if }\kappa\,\delta^2=\omega^2\in (0,\pi^2),\\
    t&\text{if }\kappa=0,\\
    \displaystyle
    \frac{\sinh(\omega t)}{\sinh(\omega)}&
    \text{if }\kappa\,\delta^2=-\omega^2<0,\\
  \end{cases}
\end{equation}
the solution $v$ of \eqref{eq:242}, thanks to \eqref{eq:243} and \eqref{eq:244}, can be expressed in the form
\begin{displaymath}
  v(t)=\factor \kappa{1-t}{r_1-r_0} u(r_0)+
  \factor \kappa t{r_1-r_0} u(r_1)
\end{displaymath}
and the following result holds (see for instance \cite[Thm.~14.28]{Villani09}):

\begin{lemma}\label{lem:DiffInesigma}
  Let $u\in\rmC([0,1])$ nonnegative and $\kappa \in \R$. 
  Then $u''+\kappa\, u\le 0$ in $\mathscr D'(0,1)$
  if and only if for every $t\in [0,1]$ and
  for every $0\le r_0<r_1\le 1$ with $\kappa (r_1-r_0)^2 <\pi^2$ one has
  \begin{equation}
    \label{eq:171}
    u((1-t)r_0+tr_1)\ge \sigma^{(1-t)}_\kappa(r_1-r_0)\,u(r_0)+
    \sigma^{(t)}_\kappa(r_1-r_0)\,u(r_1).
  \end{equation}
\end{lemma}

In the same spirit of Lemma~\ref{lem:LocLipF}, where we proved that ``weighted convex'' functions are locally Lipschitz, next  we show that functions satisfying the concavity condition \eqref{eq:171} have the same regularity; this will allow to apply Lemma~\ref{lem:DiffInesigma}.

\begin{lemma}\label{lem:LocLipsigma}
  Let  ${\mathfrak D}\subset \R$ be a $\Q$-vector space with 
  ${\mathfrak D} \neq \{0\}$. 
  Let $\kappa\in\R$ and let $u:[0,1]\cap {\mathfrak D} \to \R$  satisfy  \eqref{eq:171} 
  for every $r_0,\,r_1\in [0,1]\cap {\mathfrak D}$ with $\kappa
  (r_1-r_0)^2 <\pi^2$ and  
  $t \in [0,1]$ such that $(1-t)r_0 + t r_1 \in  {\mathfrak D}$.  
  Then the following hold.
  \begin{enumerate}
  \item[(a)] There exists $\varepsilon_0=\varepsilon_0(\kappa)>0$  with the following property:  if
  \begin{equation}\label{eq:uUBGrid}
  \sup_{n\in \N, n\leq \lfloor \frac{1}{\varepsilon} \rfloor} u(n \varepsilon)< \infty ,
  \end{equation}
  for some $\varepsilon \in (0, \varepsilon_0)\cap {\mathfrak D}$,
  then $\sup_{r\in {\mathfrak D} \cap [0,1]} u(r) < \infty$.
  \item[(b)] There exists $\varepsilon_0=\varepsilon_0(\kappa)>0$  with the following property:  if 
  \begin{equation}\label{eq:uLBGrid}
  \sup_{n\in \N, n\leq \lfloor \frac{1}{\varepsilon} \rfloor} |u(n \varepsilon)| < \infty ,
  \end{equation}
  for some $\varepsilon \in (0, \varepsilon_0)\cap {\mathfrak D}$, 
  then $\sup_{r\in {\mathfrak D} \cap [0,1]} |u(r)|<  \infty$.
  \item[(c)] If  in addition $u:[0,1]\cap {\mathfrak D} \to \R$ is   locally bounded then 
  $u$ is locally Lipschitz in $(0,1)$, i.e. for every $r \in (0,1)\cap {\mathfrak D}$ there exist $\varepsilon,\,C>0$ such that
  $[r-\varepsilon,r+\varepsilon]\subset [0,1]$ and
  \begin{equation}\label{eq:loclipsigma}
  |u(x)-u(y)|\leq C|x-y| \quad \forall x,\,y \in [r-\varepsilon,r+\varepsilon]\cap {\mathfrak D}.
  \end{equation}
  \end{enumerate}
      \end{lemma}
   
\begin{proof}   For simplicity of notation we can  assume $\Q \subset {\mathfrak D}$.

(a) Assume by contradiction the existence of a sequence $(s_n)\subset (0,1)\cap {\mathfrak D}$ such that $u(s_n)\to +\infty$. Clearly there exists $\bar{s}\in [0,1]$ such that, up to subsequences, $s_n\to \bar{s}$; let us start by assuming $\bar{s}=0$, without loss of generality we can also assume that  $s_n\in [0,\varepsilon/4]$ for every $n \in \N$ ($\varepsilon_0$ will be chosen later just depending on $\kappa$). Applying \eqref{eq:171}  with $r_0=s_n$, $r_t=\varepsilon$ and $r_1=2\varepsilon$ we get 
\begin{equation}\label{eq:uvare}
u(\varepsilon) \ge \sigma^{(1-t_n)}_\kappa(2\varepsilon-s_n)\,u(s_n) + \sigma^{(t_n)}_\kappa(2\varepsilon-s_n)\,u(2\varepsilon),
\end{equation}
where $t_n=\frac{\varepsilon-s_n}{2\varepsilon-s_n}\to \frac{1}{2}$ as $n\to \infty$.
\\By a Taylor expansion at $0$ of the function $r\to \sigma^{(1-t_n)}_{\kappa}(r)$ it is easy to see that 
\begin{equation}\label{eq:Taysi}
 \sigma^{(1-t_n)}_{\kappa}(2\varepsilon-s_n)=(1-t_n)+o_{\kappa}(\varepsilon_0) \ge \frac{1}{4} ,  
\end{equation}
provided  $\varepsilon_0=\varepsilon_0(\kappa)>0$ is chosen  small enough. But then, observing that $\inf_n  \sigma^{(t_n)}_{\kappa}(2\varepsilon-s_n) \, u(2 \varepsilon)> -\infty $, combining \eqref{eq:uvare} and  \eqref{eq:Taysi} we get
\begin{equation}\nonumber
u(\varepsilon)\geq  \frac{1}{4} u(s_n)+  \sigma^{(t_n)}_{\kappa}(2\varepsilon-s_n) \, u(2 \varepsilon) \to +\infty \quad \text{ as }  n\to +\infty, 
\end{equation}
contradicting \eqref{eq:uUBGrid}.  If instead $s_n\to 1$, applying \eqref{eq:171} with $r_0=1-2\vare$, $r_t=1-\vare$ and $r_1=s_n$, with analogous arguments we get
$$u(1-\varepsilon)\geq  \sigma^{(1-t_n)}_{\kappa}(s_n-(1-2\varepsilon)) \, u(1-2 \varepsilon)+  \frac{1}{4} u(s_n) \to +\infty \quad \text{ as }  n\to +\infty,$$
contradicting \eqref{eq:uUBGrid}. Finally, if $\lim_n s_n=\bar{s} \in (0,1)$ we can repeat the first argument with $0$ replaced by $\bar{s}$ thus reaching a contradiction. The proof of the first statement is then complete.
\\

(b) Let $\vare_0=\vare_0(\kappa)>0$ be as above. Since by the first statement we already know that $u$ is uniformly bounded  above, here  it is enough to prove a uniform bound from below. Applying \eqref{eq:171} to $r_0=n\vare$ and $r_1=(n+1) \vare$ for every $n\in \N \cap [0, \lfloor \frac{1}{\varepsilon} \rfloor]$, we get that 
$$u(r)\geq  \sigma^{(1-t_r)}_\kappa(\vare)\,u(n \vare ) + \sigma^{(t_r)}_\kappa(\vare)\,u((n+1) \vare) \geq- C  \sup_{n\in \N, n\leq \lfloor \frac{1}{\varepsilon} \rfloor} |u(n \varepsilon)| > -\infty \; , $$
for every $r \in [n \vare, (n+1) \vare]\cap {\mathfrak D}$, for some $C>0$ independent of $n$.
Applying the same argument to  $r_0=\lfloor \frac{1}{\varepsilon} \rfloor \vare$, $r_1=1$ we also obtain a uniform lower bound on  $[\lfloor \frac{1}{\varepsilon} \rfloor, 1]\cap {\mathfrak D}$ and the conclusion follows. 
\\

(c) Since the statement is local and $\mathfrak D$ is dense, we can assume with no loss of generality
that $u:[0,1]\cap\mathfrak D\to\R$ is bounded, that $0,\,1\in\mathfrak D$ and that
\eqref{eq:171} holds for every $r_0,\,r_1\in [0,1]\cap {\mathfrak D}$ with $\kappa (r_1-r_0)^2 <\pi^2$ and
$t\in [0,1]\cap\Q$.
First of all note that \eqref{eq:171} is equivalent to the following control on distorted incremental ratios: 
for every  $r_0,\,r_1\in  [0,1]\cap {\mathfrak D}, \; t\in [0,1]\cap\Q$ with $\kappa (r_1-r_0)^2 <\pi^2$ it holds
\begin{equation}\label{eq:IncrRat}
\frac{u(r_t)-\frac{1}{1-t}\sigma^{(1-t)}_\kappa(r_1-r_0) \; u(r_0)}   {r_t-r_0} \geq \frac{\frac{1}{t} \sigma^{(t)}_\kappa(r_1-r_0) \; u(r_1)- u(r_t)}   {r_1-r_t},
\end{equation}
where $r_t:=(1-t)r_0+tr_1\in  [0,1]\cap {\mathfrak D}$.

Given $r\in (0,1)\cap {\mathfrak D}$, $\varepsilon>0$ with
$\varepsilon<\min\{r,1-r\}$ and $4 \kappa \vare^2 < \pi^2$, and $x,\,y \in {\mathfrak D}\cap [r-\varepsilon, r+\varepsilon]$, $x<y$,  
we want to use \eqref{eq:IncrRat} iteratively in order to estimate  the difference quotient $|u(x)-u(y)|/|x-y|$. We will prove that this is
possible provided $\varepsilon$ is sufficiently small.
\\ At first apply \eqref{eq:IncrRat} with $r_0=0$, $r_1=x$, $r_t=r-\varepsilon$.  Noting that $1-t=\frac{x-(r-\varepsilon)}{x}\leq C_r \varepsilon$, 
with a first order Taylor expansion at $t=1$ 
of the explicit expression \eqref{eq:170} of $\frac{1}{1-t} \sigma^{(1-t)}_\kappa(x)$ one checks the existence of $C_r>0$, $\varepsilon_r>0$
satisfying (with the above choice of $t=t(x,r)$)
\begin{equation}\label{eq:sigma1-t}
\biggl|\frac{1}{1-t} \sigma^{(1-t)}_\kappa(x)\biggr\vert \leq C_r\quad \text{for all $x\in [r-\varepsilon,r+\varepsilon]$},
\text{ for every } \varepsilon\in (0, \varepsilon_r).
\end{equation} 
Analogously, possibly enlarging $C_r$ and reducing $\varepsilon_r$ we can also achieve 
\begin{equation}\label{eq:sigmat}
\left| \frac{1}{t} \sigma^{(t)}_\kappa(x)-1 \right| \leq C_r (x-(r-\varepsilon))  
\text{ for every } \varepsilon\in (0, \varepsilon_r).
\end{equation} 
The combination of \eqref{eq:IncrRat}, \eqref{eq:sigma1-t} and \eqref{eq:sigmat} gives
\begin{equation}\label{eq:IncrRat1}
\frac{u(r-\varepsilon)+C_r|u(0)|}   {r-\varepsilon} \geq \frac{ u(x)- u(r-\varepsilon)}   {x-(r-\varepsilon)} - C_r \quad 
\text{ for every } \varepsilon\in (0, \varepsilon_r)\cap {\mathfrak D}.
\end{equation}
Observing that $| \frac{1}{t} \sigma^{(t)}_\kappa(y-(r-\varepsilon))-1| \leq C_r t (y-(r-\varepsilon))$,  applying  \eqref{eq:IncrRat}  with $r_0=r-\varepsilon$, $r_1=y$, $r_t=x$, yields\begin{equation}\label{eq:IncrRat2}
\frac{u(x)- u(r-\varepsilon)}   {x-(r-\varepsilon)} +C_r \geq \frac{ u(y)- u(x)}   {y-x} \; , \text{ for every } \varepsilon\in (0, \varepsilon_r)\cap {\mathfrak D}.
\end{equation}
Putting together  \eqref{eq:IncrRat1} and \eqref{eq:IncrRat2} we obtain the desired upper bound
$$\frac{ u(y)- u(x)}   {y-x} \leq  
\frac{u(r-\varepsilon)+C_r |u(0)|}   {r-\varepsilon} +2 C_r.$$
Along the same lines one gets also the lower bound
$$\frac{ u(y)- u(x)}   {y-x} \geq  \frac{u(r+\varepsilon)- u(y)}   {(r+\varepsilon)-y} -C_r  \geq  
\frac{-C_r |u(1)|- u(r+\varepsilon)}   {1-(r+\varepsilon)} -2 C_r. $$
Since the last two estimates hold for every  $x,\,y \in {\mathfrak D}\cap [r-\varepsilon, r+\varepsilon]$ with $x\neq y$, the proof is complete.
\end{proof}

\subsection{Entropies and their
regularizations}
\label{subsec:regularent}

Let $(X,\sfd,\mm)$ be a metric measure space as in Section~\ref{subsec:Cheeger}.
 We consider 
continuous and convex entropy functions $U:[0,\infty)\to \R$ 
with locally Lipschitz derivative in $(0,\infty)$ and $U(0)=0$.
We set
\begin{equation}
  \label{eq:40}
  P(r):=rU'(r)-U(r),\quad
  Q(r):=r^{-1}P(r)\in \Lip_{\rm loc}(0,\infty),\quad
  R(r):=rP'(r)-P(r).
\end{equation}
The induced entropy functional is defined by 
\begin{equation}
  \label{eq:41}
  \mathcal U(\mu):=\int_X U(\varrho)\,\d\mm+U'(\infty)\mu^\perp(X)
  \quad\text{if }\mu=\varrho\mm+\mu^\perp, \quad
  \mu^\perp\perp\mm,
\end{equation}
where $U'(\infty)=\lim_{r\to\infty}U'(r)$. Since $U(0)=0$ and the negative part of $U$ grows
at most linearly, $\mathcal U$ is well defined and with values in $(-\infty,+\infty]$ if $\mu$ has bounded support,
{\color{black} more general cases are discussed below.}

We say that $P$ is \emph{regular} if, for some constant $\sfa=\sfa(P)>0$, one has
  \begin{equation}
  \label{eq:A1bis}
  P\in \rmC^1([0,\infty)),\quad
  P(0)=0,\quad
  0<\sfa\le P'(r)\le \sfa^{-1}\quad\forevery r\ge0.
\end{equation}
Notice that in this case $Q$, extended at $0$ with the value $P'(0)$, is continuous
in $[0,\infty)$ and it satisfies the analogous bounds
\begin{equation}
  \label{eq:175}
  \sfa\le Q(r)\le \sfa^{-1}\quad\forevery r\ge0.
\end{equation}
When $P$ is regular, we still denote by
$P:\R\to \R$ its odd extension, namely $P(-r):=-P(r)$ for every $r\ge0.$

Once a regular function $P$ 
is assigned, a corresponding entropy function 
$U$ can be determined up to a linear term by the formula
\begin{equation}
  \label{eq:47}
  U(r)=r\int_{1}^r \frac {P(s)}{s^2}\,\d s,
\end{equation}
so that \eqref{eq:A1bis} yields
\begin{equation}\label{eq:47-bis}
  \sfa \big|r\log r\big|\le |U(r)|\le \sfa^{-1} \big|r\log r\big|\qquad\text{for every $r\geq 0$.}
\end{equation}
Motivated by \eqref{eq:47}, we call the entropies $U$ satisfying $U(1)=0$, \emph{normalized}. Notice that $P$ uniquely
determines the normalized entropy $U$.
Thus in the case of regular $P$, the asymptotic behaviour 
of $U$ near $r=0$ or $r=\infty$ is controlled
by the one of the logarithmic entropy functional $\mathcal U_\infty$ 
associated to $U_\infty$, namely
\begin{equation}
  \label{eq:176}
  U_\infty(r):=r \log r ,\quad 
  P_\infty(r)=r,\quad Q_\infty(r)=1,\quad
  R_\infty(r)=0.
\end{equation}
In particular, using \eqref{eq:173} one can prove that
$\mathcal U(\mu)$ is always well defined, with values in $(-\infty,+\infty]$, if $\mu\in \ProbabilitiesTwo
X$, see \cite[\S 7,1]{AGS11a}.
The choice of the base point $1$ in the integral formula \eqref{eq:47} 
provides, thanks to Jensen's inequality, the lower bound $\mathcal U(\mu)\ge0$ whenever
$\mm\in\Probabilities X$. {\color{black} In addition, we shall extensively use the lower semicontinuity
of the entropy functionals \eqref{eq:41} w.r.t. convergence in $\ProbabilitiesTwo{X}$, see for instance
\cite{Villani09}}.
\begin{remark}[Regularized entropies]
  \label{rem:regularization}
  \upshape
  Let $P\in \rmC^1((0,\infty))$ with $P'(r)>0$ for every $r>0$
  and $0=P(0)=\lim_{r\down0}P(r)$.
  It is easy to approximate $P$ by regular functions:
  we set for $0<\eps<M<\infty$
  \begin{equation}
    \label{eq:207}
    P_{\eps}(r):=
  P(r+\eps)-P(\eps),\qquad
  P_{\eps,M}(r):=
  \begin{cases}
    P_{\eps}(r)&\text{if }0\le r\le M,\\
    P_{\eps}(M)+(r-M)P_{\eps}'(M)&\text{if }r>M.
  \end{cases}
  \end{equation}
  Notice that 
  \begin{equation}
    \label{eq:217}
    r P_{\eps}'(r)-P_\eps(r)=
    R(r+\eps)-R(\eps)+\eps(P'(\eps)-P'(r+\eps)).
  \end{equation}
\end{remark}
Besides \eqref{eq:176}, 
our main example is provided by the family depending on $N\in (1,\infty)$
\begin{equation}
  \label{eq:167}
  \begin{aligned}
    U_N(r):={}&Nr (1-r^{-1/N}),\quad P_N(r)=r^{1-1/N},\quad
    Q_N(r):=r^{-1/N},\\
    R_N(r)={}&-\frac {r^{1-1/N}}N=-\frac 1N P_N(r)
  \end{aligned}
\end{equation}
together with the regularized functions $P_{N,\eps}$ and
$P_{N,\eps,M}$ as in \eqref{eq:207}.

Notice that a simple computation provides:
\begin{equation}
  \label{eq:83}
  R_{N,\eps}(r)=-\frac{1}N P_{N,\eps}(r)
  +\eps (P_{N,\eps}'(0)- P_{N,\eps}'(r))\quad \forevery r\in [0,\infty),
  \end{equation}
  so that the concavity and the monotonicity of $P_{N,\eps}$ give
  \begin{equation}
  \label{eq:92}
  -\frac{1}N P_{N,\eps}(r) + (1-\frac 1N)\eps^{1-1/N}\ge
  R_{N,\eps}(r)\ge -\frac{1}N P_{N,\eps}(r)\quad \forevery r\in [0,\infty).
\end{equation}

The entropies corresponding to $U_N$, {\color{black} according to \eqref{eq:41}}, will be denoted with $\mathcal U_N$:
{\color{black}
\begin{equation}\label{def:CalUN}
\mathcal U_N(\mu):=N-N \int_X \varrho^{1-\frac{1}{N}} \, \d\mm  \qquad\text{if }\mu=\varrho\mm+\mu^\perp, \quad
  \mu^\perp\perp\mm .
\end{equation}
In particular $\mathcal U_N(\mu)=\int_X U_N(\varrho)\,\d\mm$ whenever $\mu=\varrho\mm$ is absolutely continuous.}

\subsection{The $\CDS KN$ condition and its characterization via weighted action convexity}\label{subsec:9.3}

In this section we start by recalling what does it mean for a metric measure space to have  ``Ricci tensor bounded below by $K\in \R$ and dimension bounded above by $N\in (1, \infty]$'', this corresponds to the the so-called curvature dimension conditions $\CD KN $ or to 
the reduced curvature dimension conditions $\CDS K N$. First, let us recall the notion of  $\CD K\infty$ space  introduced independently by Lott-Villani \cite{Lott-Villani09} and Sturm \cite{Sturm06I}  (see also \cite{Villani09} for a comprehensive treatment).

\begin{definition}[$\CD K\infty$ condition]
  Let $K\in\R$. We say that $(X,\sfd,\mm)$ satisfies the $\CD K\infty$ condition if 
  for every $\mu_i=\varrho_i\mm\in D(\cU_\infty)\cap\ProbabilitiesTwo
  X$,
  $i=0,\,1$, there exists a $W_2$-geodesic $(\mu_s)_{s\in [0,1]}$ 
  connecting $\mu_0$ to $\mu_1$ such that
  \begin{equation}
    \label{eq:182}
    \cU_\infty(\mu_s)\le (1-s)\cU_\infty(\mu_0)+
    s\cU_\infty(\mu_1)-\frac K2{s(1-s)}W_2^2(\mu_0,\mu_1)\qquad\text{$\forall s\in (0,1)$}.
  \end{equation}
  If morever \eqref{eq:182} is satisfied along 
  \emph{any} geodesic $\mu_s$ connecting $\mu_0$ to $\mu_1$, we say that 
  $(X,\sfd,\mm)$ is a \emph{strong} $\CD K\infty$ space.
\end{definition}

It is well known that  smooth Riemannian manifolds with Ricci curvature bounded below by $K$ are $\CD K\infty$-spaces; one  reason of the geometric relevance of such spaces is that they form a class which is stable under measured Gromov-Hausdorff convergence (for proper spaces see \cite{Lott-Villani09}, for normalized spaces with finite total volume, see \cite{Sturm06I}, for the general case without any finiteness or local compactness assumption see \cite{GMS13}).

In strong $\CD K\infty$ spaces $(X,\sfd,\mm)$, quite stronger metric properties have been proved in
\cite{Rajala-Sturm12}; we list them in the next proposition. 

\begin{proposition}[Properties of strong $\CD K\infty$ spaces] \label{prop:scdk}
Let $(X,\sfd,\mm)$ be a strong $\CD K\infty$ space {\color{black} and let
$\mu_0=\varrho_0 \mm,\,\mu_1=\varrho_1 \mm \in  \Probabilitiesac X\mm \cap \ProbabilitiesTwo X$.} Then:
\begin{enumerate}[{\rm [RS1]}]
\item There exists only one optimal geodesic plan $\ppi\in
  \mathrm{GeoOpt}(\mu_0,\mu_1)$ 
  (and thus only one geodesic connecting
  $\mu_0,\,\mu_1$);
  \item $\ppi$ is concentrated on a set of nonbranching
  geodesics and it is induced by a map;
  \item 
  all the interpolated measures $\mu_s=(\rme_s)_\sharp\ppi$ are absolutely continuous w.r.t.\
  $\mm$; if moreover $\mu_0, \mu_1\in D(\mathcal U_\infty)$, then  $\mu_s$ have uniformly bounded logarithmic entropies $\mathcal U_\infty(\mu_s)$.
  \item 
  if $\varrho_0,\,\varrho_1$ are $\mm$-essentially bounded and have bounded supports, then the interpolated measures $\mu_s=\varrho_s \mm=(\rme_s)_\sharp\ppi$ have uniformly bounded densities. More  precisely the following estimate holds: 
  \begin{equation}\label{eq:rhosLinfty}
  \|\varrho_s\|_{L^\infty(X,\mm)} \leq e^{K^- D^2/12} \max\{ \| \varrho_0\|_{L^\infty(X,\mm)}, \| \varrho_1\|_{L^\infty(X,\mm)}  \} ,
  \end{equation} 
where $D:=\diam(\supp \varrho_0 \cup \supp \varrho_1)$ and $K^-:=\max\{0, -K\}$.
\end{enumerate}
\end{proposition}

As bibliographical remark let us mention also  \cite{GigliGAFA} about existence of optimal maps in non branching spaces;  also note that  [RS4] is well known \cite[Thm. 30.32, (30.51)]{Villani09}, \cite[\S 3]{AGS11b} as soon as the branching phenomenon is ruled out. Remarkably, this property holds even without the non-branching assumption \cite{Rajala12}.

\begin{remark} \label{rem:BddSupp} {\rm
Notice also the following general fact, holding regardless of curvature assumptions: if $\mu_0, \,\mu_1 \in \Probabilities X$ have bounded support,  then there exists a bounded subset $E$ of $X$
 containing all the images of the geodesics from a point of $\supp \mu_0$ to a point of $\supp \mu_1$; in particular we have that  $\supp[ (\rme_s)_\sharp\ppi ] \subset E$ for every $s \in [0,1]$ and every $\ppi\in
  \mathrm{GeoOpt}(\mu_0,\mu_1)$.}
\end{remark}  

\begin{lemma}
  \label{le:CD-bs}
  $(X,\sfd,\mm)$ is a strong $\CD K\infty$ space if and only if 
  every couple of measures $\mu_0,\,\mu_1\in D(\cU_\infty)\cap \ProbabilitiesTwo X$ 
  with bounded support can be connected by a $W_2$-geodesic and
  \eqref{eq:182} is satisfied along any geodesic connecting $\mu_0$ to $\mu_1$.  
\end{lemma}
\begin{proof}
  Let us first prove that 
  every couple $\mu_0,\,\mu_1\in D(\cU_\infty)\cap \ProbabilitiesTwo X$
  can
  be connected by a $W_2$-geodesic.
  For $\bar x\in X$ fixed and $N$ sufficiently big, we can define the measures
  $\mu_i^N:=\frac 1{c_N}\mu_i\llcorner B_N(\bar x)\in
  \ProbabilitiesTwo X$. 

  By choosing a constant $C>B$ (recall \eqref{eq:173}),
  we can introduce the normalized probability measure $\bar\mm\in \ProbabilitiesTwo X$
  \begin{displaymath}
    \bar\mm:=\frac 1{\mathrm z} \rme^{-C\sfd^2(x,\bar x)}\mm\quad\text{with}\quad
    \mathrm z:=\int_X \rme^{-C\sfd^2(x,\bar x)}\,\d\mm(x),
  \end{displaymath}
  and the corresponding relative entropy functional $\tilde \cU_\infty$, satisfying the identity
  \begin{equation}
    \label{eq:13}
    \cU_\infty(\mu)=
    \tilde\cU_\infty(\mu)-C\int_X \sfd^2(x,\bar x)\,\d\mu-\log\mathrm z.
  \end{equation}
  Let us denote by $\tilde\varrho_i,\tilde\varrho_i^N$ the
  densities
  of $\mu_i, \mu_i^N$ w.r.t.~$\tilde\mm$.
  From $c_N\uparrow 1$ it is easy to check that
  $W_2(\mu_i^N,\mu_i)\to 0$ and 
$\|\tilde\varrho^N_i-\tilde\varrho_i\|_{L^1(X,\tilde\mm)}\to0$ as
$N\uparrow\infty$. 
Since $\tilde\varrho_i^N\le c_N^{-1}\tilde\varrho_i$, 
the uniform bound 
$$-\rme^{-1}\le \tilde\varrho_i^N\log (\tilde\varrho_i^N)\le
\tilde\varrho_i^N(\log \tilde \varrho_i-\log c_N)\le c_N \tilde \varrho_i
(\log \tilde \varrho_i)_+ -c_N\log c_N \tilde \varrho_i$$
and the fact that $\tilde \mm(X)$ is finite
yields $\tilde\cU_\infty(\mu^N_i)\to \tilde\cU_\infty(\mu_i)$ as $N\uparrow
\infty$ and therefore, by \eqref{eq:13}, 
$\cU_\infty(\mu^N_i)\to \cU_\infty(\mu_i)$ as $N\to
\infty$.

Since $\mu_i^N$ have bounded support we can
find a $W_2$-geodesic $(\mu_s^N)_{s\in [0,1]}$ connecting them and
satisfying the corresponding uniform bound 
  \begin{equation}
    \label{eq:182bis}
    \cU_\infty(\mu_s^N)\le (1-s)\cU_\infty(\mu_0^N)+
    s\cU_\infty(\mu_1^N)-\frac K2{s(1-s)}W_2^2(\mu_0^N,\mu_1^N)\qquad\text{$\forall s\in (0,1)$},
  \end{equation}
 which
in particular shows that $\tilde\cU_\infty(\mu_s^N)\le S<\infty$
for every $s\in [0,1]$ and $N$ sufficiently big. Since the sublevels
of 
$\tilde \cU_\infty$ are relatively compact in $\Probabilities X$ and
the curves
$[0,1]\ni s\mapsto \mu_s^N$ are equi-Lipschitz with respect to $W_2$, 
we can extract (see e.g.~\cite[Prop.~3.3.1]{AGS08})
an increasing subsequence $h\mapsto N(h)$
and a limit geodesic $(\mu_s)_{s\in [0,1]}$ such that $\mu_s^{N(h)}\to
\mu_s$ weakly in $\Probabilities X$ as $h\to\infty$. In particular
$(\mu_s)_{s\in [0,1]}$ is a geodesic connecting $\mu_0$ to $\mu_1$.
%

Let us now prove that \eqref{eq:182} holds along any geodesic 
connecting $\mu_0,\,\mu_1 \in D(\cU_\infty)\cap\ProbabilitiesTwo
X$. Let $\mu_s=(\rme_s)_\sharp \ppi$ be
a geodesic induced by $\ppi\in \mathrm{GeoOpt}(\mu_0,\mu_1)$;
we consider  
$$
\Gamma_R:=\left\{\gamma\in\rmC^0([0,T];X):\ \gamma([0,1])\subset\overline{B}_R(\bar x)\right\},\qquad
c_R:=\ppi(\Gamma_R),\quad
\ppi^R:=\frac 1{c_R}\ppi\llcorner\Gamma_R
$$
and $\mu_s^R:=(\rme_s)_\sharp \ppi^R$; since $\ppi^R\in
\mathrm{GeoOpt}(\mu_0^R,\mu_1^R)$, 
$(\mu_s^R)_{s\in [0,1]}$ is a geodesic and 
the measures $\mu_s^R$ have bounded
support in $\overline{B}_R(\bar x)$.
Thus for every $R>0$ one has that
$\mu_0^R,\mu_s^R,\mu_1^R$ satisfy \eqref{eq:182};
arguing as in the previous step, 
we can pass to the limit as $R\to\infty$ using the facts that 
$W_2(\mu_s^R,\mu_s)\to 0$ for every $s\in [0,1]$,
$\cU_\infty(\mu_i^R)\to \cU_\infty(\mu_i)$ if $i=0,\,1$, and 
$\liminf_{R\to\infty}\cU_\infty(\mu_s^R)\ge \cU_\infty(\mu_s)$ and we obtain
the corresponding inequality for $\mu_0,\mu_s,\mu_1$.
\end{proof}
Next, let us recall the definition of reduced curvature dimension condition $\CDS KN$ introduced by Bacher-Sturm \cite{Bacher-Sturm10}.

\begin{definition}[$\CDS KN$ condition]\label{def:CDS}
  We say that $(X,\sfd,\mm)$ satisfies the
  $\CDS KN$ condition, $N\in [1,\infty)$, if 
  for every $\mu_i=\varrho_i\mm\in \Probabilitiesac X\mm$, $i=0,1$,
  with bounded support there exists $\ppi\in \mathrm{GeoOpt}(\mu_0,\mu_1)$ such that 
  \begin{equation}
    \label{eq:177}
    \mathcal U_{M}(\mu_s)\le M- M 
    \int \Big(\sigma_{K/M}^{(1-s)}(\sfd(\gamma_0,\gamma_1))
    \varrho_0(\gamma_0)^{-1/M}+
    \sigma_{K/M}^{(s)}(\sfd(\gamma_0,\gamma_1))
    \varrho_1(\gamma_1)^{-1/M}\Big)\,\d\ppi(\gamma)
  \end{equation}
  for every $s\in [0,1]$ and $M\geq N$, where $\mu_s=(\rme_s)_\sharp\ppi$, the coefficients $\sigma$ are defined in \eqref{eq:170} and  $\mathcal U_M$ is defined in \eqref{def:CalUN}.
  
  If moreover 
  \eqref{eq:177} is satisfied  along \emph{any} 
  $\ppi\in \mathrm{GeoOpt}(\mu_0,\mu_1)$,
  we say that $(X,\sfd,\mm)$ satisfies 
  the \emph{strong} $\CDS KN$ condition.
\end{definition}

\begin{remark}\label{rem:notatCDS} {\rm
Definition \ref{def:CDS} coincides with the original definition of $\CDS K N$ spaces given in \cite{Bacher-Sturm10}. Note that the additional terms in the right hand side of \eqref{eq:177} are  due to our definition of entropy as $\mathcal U_M(\varrho \mm):= M \int_X  \varrho(1-\varrho^{-\frac{1}{M}}) \, \d \mm= M -M\int_X \varrho^{1-\frac{1}{M}} \, \d \mm$, while the one adopted in \cite{Bacher-Sturm10} was $-\int_X \varrho^{1-\frac{1}{M}} \, \d \mm$ (for absolutely continuous measures). This convention will be convenient in our work in order to use  regularized entropies and analyze the corresponding non linear diffusion semigroups.}
\end{remark}

It can be proved that 
a strong $\CDS KN$-space is also a strong $\CD K\infty$ space,
and thus properties [RS1-4] hold, see Lemma~\ref{le:CDSKN-CDKI} below, whose proof included for completeness
follows the lines of \cite[Prop.~1.6]{Sturm06II}. Conversely, any $\CDS KN$ space satisfying 
[RS1-4] is clearly strong. Therefore a $\CDS KN$ space is strong if and only if
[RS1-4] hold.

\begin{lemma}
  \label{le:CDSKN-CDKI}
  If $(X,\sfd,\mm)$ satisfies the (strong) $\CDS KN$ condition for
  some
  $K\in \R$, $N\in [1,\infty)$, then $(X,\sfd,\mm)$ is a (strong) $\CD
  K\infty$ space.
\end{lemma}
\begin{proof}
  By Lemma~\ref{le:CD-bs} 
  it is sufficient to prove \eqref{eq:182} 
  along every $W_2$-geodesic $(\mu_s)_{s\in [0,1]}$ 
  induced by $\ppi\in \mathrm{GeoOpt}(\mu_0,\mu_1)$
  with $\mu_s$ supported in a bounded set and $\mu_i\in
  D(\cU_\infty)$, $i=0,\,1$. 

  Let us first notice that for every $r\ge0$
  $$\lim_{N\to\infty} U_N(r)=U_\infty(r),\quad
  N\mapsto U_N(r)\text{ {\color{black} is increasing for all $r\in (0,\infty)$}} $$
  If $\mu_s=\varrho_s\mm$, since $\mu_s$ is supported in a bounded set
  with finite $\mm$-measure,  it is then not difficult to prove that 
  \begin{equation}
    \label{eq:15}
    \lim_{N\to\infty}\cU_N(\mu_s)=\cU_\infty(\mu_s)\quad\text{for
      every }s\in [0,1].
  \end{equation}
  The second important property concerns the coefficients
  $\sigma_\kappa^{(s)}(\delta)$: if $K>0$
  \begin{equation}
    \label{eq:23}
    M\Big(s-\sigma_{K/M}^{(s)}(\delta)\Big)
    =
      K\frac{s\sin(\sqrt{K/M}\delta)-\sin(\sqrt{K/M}s\delta)}{(K/M)\sin(\sqrt{K/M}\delta)}
      =\frac {\delta^2}6K\big(s^3-s\big)+o(1)
  \end{equation}
  as $M\up\infty$, and a similar property holds when $K\leq 0$.

  We thus get
  \begin{align*}
    &\cU_M(\mu_s)
    -
      (1-s)\cU_M(\mu_0)-s\cU_M(\mu_1)
      \\&\le
      M\int \Big(\big(s-\sigma_{K/M}^{(s)}(\sfd(\gamma_0,\gamma_1))\big)\varrho_1^{-1/M}(\gamma_1)+
      \big(1-s-\sigma_{K/M}^{(1-s)}(\sfd(\gamma_0,\gamma_1))\big)
          \varrho_0^{-1/M}(\gamma_0)\Big)\,\d\ppi(\gamma)
  \end{align*}
  and passing to the limit as $M\up\infty$ by applying \eqref{eq:15}
  and \eqref{eq:23} we obtain
  \begin{align*}
    \cU_\infty(\mu_s)
    -
      (1-s)\cU_\infty(\mu_0)+s\cU_\infty(\mu_1)
    &\le
          -\frac K2s(1-s)\int \sfd^2(\gamma_0, \gamma_1)\,\d\ppi(\gamma)
      \\&=
          -\frac K2s(1-s)W_2^2(\mu_0,\mu_1).
  \end{align*}
\end{proof}

Let us  also introduce a more general class of natural entropy functionals, used 
for instance in Lott-Villani's  approach of CD spaces \cite{Lott-Villani09, Villani09}
{\color{black} (compared to \cite{Villani09}, we add a few more regularity properties on $P$)}. 
They will play a crucial role in the next chapters.

\begin{definition}[{\cite[Def.~17.1]{Villani09}}] 
  \label{def:McCann}
  {\color{black} Let $U:[0,\infty)\to\R$ be a continuous and convex entropy function, with $U(0)=0$ and
  $U'$ locally Lipschitz in $(0,\infty)$.} We say that $U$ belongs to McCann's 
  class $\MC N$, $N\in [1,\infty]$, if the corresponding pressure function $P=rU'-U$ satisfies
  $P(0)=\lim_{r\down0}P(r)=0$ and  $r\mapsto r^{\frac{1}{N}-1}P(r)$ is nondecreasing, i.e. 
  \begin{equation}
    \label{eq:16}
   R(r)=rP'(r)-P(r)\ge -\frac 1N P(r) \quad \text{for ${\Leb 1}$-a.e. $r>0$.} 
  \end{equation}
  We say that $U$ is \emph{regular} and write
  $U\in \MCreg N$ if, in addition, $U$ is normalized and $P$ is regular according to \eqref{eq:A1bis}.
\end{definition}

If $P$ is regular, we can also write $P\in\MC N$ (resp. $P\in\MCreg N$) if the corresponding normalized entropy $U$ belongs to
$\MC N$ (resp. $\MCreg N$).
{\color{black} Directly from the convexity inequality $0=U(0)\geq U(r)-rU'(r)$, it is immediate to see that $P$ is nonnegative.} 
Moreover,  the function $V:(0, \infty)\to \R^+$ defined by
\begin{equation}\label{defprop:V}
V(r):= r^N  U(r^{-N}) \quad \text{is convex and nonincreasing}.
\end{equation}
The last condition is actually equivalent to $U\in \MC N$.

Before stating the next result, we introduce
a family of weighted energy functionals taylored to 
a pressure function $P$ as in \S~\ref{subsec:regularent}:
if $Q(r):=P(r)/r$, we consider the weight 
$\frQ^{(t)}(s,r):=\sfg(s,t)Q(r)$, where $\sfg$ is 
the Green function defined in \eqref{eq:3}. 

We adopt the notation of Section~\ref{sec:abscurWas}. 
If $\mu\in \AC 2{[0,1]}{(\Probabilities X,W_2)}$
with $\tilde\mu=\varrho\tilde\mm\ll\mm$ and $v$ is its minimal $2$-velocity
density, we set 
\begin{equation}
  \label{eq:172}
  \Mod tQ\mu\mm:=\Action {\frQ^{(t)}}\mu\mm=
  \int_{\tilde X}\sfg(s,t)Q(\varrho(x,s))v^2(x,s)\,\d\tilde\mu(x,s).
\end{equation}
In the following theorem we relate the $\CDS KN$ condition, defined in terms of  the distortion  coefficients $\sigma_{K/N}$,
to a modulus of convexity along Wasserstein geodesics of the entropies induced by maps $U\in \MC N$, very much like in the case $N=\infty$.
The main difference is that the modulus of convexity is not the squared Wasserstein distance, but the action $ \Mod tQ\mu\mm$
of \eqref{eq:172}.

\begin{theorem} \label{thm:Andrea_complete}
  Let us assume that {\upshape [RS1-4]} hold.
  The following properties are equivalent:
  \begin{enumerate}[{\rm [CD1]}]
    \item $(X,\sfd,\mm)$ is a 
        strong $\CDS KN$~space, for some  $K \in \R$ and $N\in [1,\infty)$.
  
  \item For every $\mu_0=\varrho_0\mm,\,\mu_1=\varrho_1 \mm\in \Probabilitiesac X\mm$ 
    with densities $\varrho_i$ $\mm$-essentially bounded with bounded support, the geodesic $(\mu_t)_{t\in [0,1]}$ 
    connecting $\mu_0$ to $\mu_1$ satisfies (with $Q_N(r)=r^{-1/N}$ as in \eqref{eq:167})
    \begin{equation}
      \label{eq:180N}
      \cU_N(\mu_t)\le 
      (1-t) \cU_N (\mu_0)+ t \,
      \cU_N(\mu_1)-K\Mod t{Q_N}{\mu}\mm\quad
    \forevery t\in [0,1].
    \end{equation}

    \item For every $\mu_0=\varrho_0\mm,\,\mu_1=\varrho_1 \mm\in \Probabilitiesac X\mm \cap \ProbabilitiesTwo X$, the geodesic $(\mu_t)_{t\in [0,1]}$ 
    connecting $\mu_0$ to $\mu_1$ satisfies   \eqref{eq:180N}.
     
  \item For every $\mu_0=\varrho_0\mm,\,\mu_1=\varrho_1 \mm\in \Probabilitiesac X\mm  \cap  \ProbabilitiesTwo {X}$   and every 
    $U\in \MC N$, the geodesic $(\mu_t)_{t\in [0,1]}$ 
    connecting $\mu_0$ to $\mu_1$ satisfies
    \begin{equation}
      \label{eq:180}
      \cU(\mu_t)\le 
      (1-t)\Reny {}{\mu_0}+
      t\Reny {}{\mu_1}-K\Mod t{Q}{\mu}\mm\quad
    \forevery t\in [0,1].
    \end{equation}

  \item For every $\mu_0,\,\mu_1\ \in \Probabilitiesac X\mm \cap 
    \ProbabilitiesTwo {X}$ 
    and every regular
    $U\in \MCreg N$  the geodesic $(\mu_t)_{t\in [0,1]}$ 
    connecting $\mu_0$ to $\mu_1$ satisfies \eqref{eq:180}.
  \end{enumerate}
  \end{theorem}
 
 \AAA
 \begin{remark}\label{rem:Levico}
If in Theorem~\ref{thm:Andrea_complete}, in all the items {\rm [CD3],[CD4],[CD5]}, the measures $\mu_{0},\mu_{1}$ are assumed to be with bounded support, then the equivalence with {\rm [CD1],[CD5]} holds with the same proof. Since in the last part of the paper we will work with measures which may not have bounded support,  it will be useful to have the stated form with the extension of {\rm [CD3],[CD4],[CD5]} to measures in $\ProbabilitiesTwo {X}$.
 \end{remark}
 \fn
 
 \begin{proof}
 The implications [CD4] $\Rightarrow$ [CD3] $\Rightarrow$ [CD2] and [CD4] $\Rightarrow$ [CD5] are trivial. \\ 
 We will prove [CD1] $\Rightarrow$ [CD2] $\Rightarrow$ [CD3] $\Rightarrow$ [CD4] $\Rightarrow$ [CD1]   and  [CD5] $\Rightarrow$ [CD2].
 \\
 
 [CD1] $\Rightarrow$ [CD2].  Let $\mu_0=\varrho_0\mm,\,\mu_1=\varrho_1 \mm\in \Probabilitiesac X\mm$ 
    with densities $\varrho_i$ $\mm$-essentially bounded with bounded support. By [RS1-4] there exists a unique geodesic $\mu_t=(\rm e_t)_{\sharp} \ppi$ 
    connecting $\mu_0$ to $\mu_1$, it is made of absolutely continuous measures with bounded densities and it is given by optimal maps:   $\varrho_t \mm=\mu_t=(T_t)_\sharp \mu_0$. Since $\ppi$ is concentrated on non-branching geodesics, we can apply \cite[Proposition 2.8 (iii)]{Bacher-Sturm10} to infer that for every $t\in (0,1)$ there exists a Borel subset $E_t\subset \supp \mu_0$ with $\mu_0(X\setminus E_t)=0$ such that 
  \begin{equation}\label{eq:BSLoc}
    \frak d_t(x)\ge \sigma_{K/N}^{(1-t)}(\sfd(x,T_1(x)))\,\frak d_0(x)+
    \sigma_{K/N}^{(t)}(\sfd(x,T_1(x)))\,\frak d_1(x), \quad \forall x \in E_t , 
  \end{equation}  
  where  $\frak d_t(x):=(\varrho_t^{-1/N}\circ T_t)(x)$.
  Moreover,  by \cite[Theorem 1.2]{RajalaDCDS}, the convexity property \eqref{eq:177} holds for all intermediate times (note that in this case one could argue directly by knowing that the $W_2$-geodesic is unique). It follows that, for any fixed countable $\Q$-vector space  $\mathfrak D \subset \R$, there exists a Borel subset $E_{\mathfrak D}\subset \supp \mu_0$ with $\mu_0(X\setminus E_{\mathfrak D})=0$ such that for every  $x \in E_{\mathfrak D}$, every  $r_0,r_1\in [0,1]\cap{\mathfrak D}$ and $t\in [0,1]\cap \Q$ it holds
\begin{equation}\label{eq:BSLocInt}
    \frak d_{r_t}(x)\ge \sigma_{K/N}^{(1-t)}(\sfd(T_{r_0}(x),T_{r_1}(x)))\,\frak d_{r_0}(x)+
    \sigma_{K/N}^{(t)}(\sfd(T_{r_0}(x),T_{r_1}(x)))\,\frak d_{r_1}(x),
  \end{equation}  
 where $r_t:=(1-t)r_0+tr_1$. For the moment simply  choose ${\mathfrak D}=\Q$ and, fixed some $n \in \N$, define
  $${\mathfrak F}:=\{m/n \, : \,  m\in \N, m\leq n\}.$$
  By Lemma~\ref{lem:LBrhot} below we have that the map ${\mathfrak F}\in r\mapsto {\frak d}_r(x)$ is uniformly bounded 
    in $[0,1]$ for  every $x \in E$, where $E\subset E_{\mathfrak D}$ satisfies $\mu_0(E)=1$. Observe also that  $\sfd(T_{r_0}(x),T_{r_1}(x))=(r_1-r_0) \sfd(x, T_1(x))$ is uniformly bounded since we are assuming $\mu_{i}$ to have bounded support, $i=0,\,1$.  Choosing $n\in \N$ large enough in the definition of ${\mathfrak F}$, by   Lemma~\ref{lem:LocLipsigma}(b) we infer that the map  $\Q \in r\mapsto {\frak d}_r(x)$ is uniformly bounded for every $x \in E$. 
But then part (c) of Lemma~\ref{lem:LocLipsigma} applied  to the function $[0,1]\cap \Q \ni r \mapsto  \frak d_{r}(x)\in \R^+$ gives that such a map is locally Lipschitz, so it admits a unique continuous extension $[0,1] \ni t\mapsto \bar{\frak d}_t(x) \in \R^+$. 
  
 Observing now that  
 $$\sigma_{K/N}^{(1-s)}(\sfd(T_{r_0}(x),T_{r_1}(x)))=\sigma_{K/N}^{(1-s)}((r_1-r_0)\sfd(x,T_{1}(x)))=\sigma_{\frac{K}{N}\sfd(x,T_{1}(x))^{2}}^{(1-s)}(r_1-r_0),$$ 
 Lemma \ref{lem:DiffInesigma} implies that  the continuous  map $t\mapsto   \bar{\frak d}_{t}(x)$ satisfies the differential inequality  
 \begin{equation}
    \label{eq:13'}
    \frac {\d^2}{\d t^2} \bar{\frak d}_t (x)\le -\frac KN  \sfd^2(x, T_1(x))  \; \bar{\frak d}_t (x) \quad \text{in }  \mathscr D'(0,1),  \text{ for every } x \in E  .
    \end{equation} 
  But then, Lemma \ref{le:0} gives  
  \begin{equation}
    \label{eq:12'}
   \bar{ \frak d}_{t}(x)\ge
    (1-t) \bar{ \frak d}_{0}(x)+
    t \bar{\frak d}_{1}(x)+
    \frac KN\int_{0}^{1} \bar{\frak d}_{s}(x) \sfd^2(x, T_1(x))  \,\fnd
    t{} s\,\d s\quad \forall t \in[0,1],\; \forall x \in E.
  \end{equation}
  We now claim that for every $t \in [0,1]$ it holds ${\frak d}_{t}= \bar{\frak d}_{t}$ $\mu_0$-a.e.  If it is not the case then there exists $\bar{t}\in[0,1]$ and a subset $F_{\bar{t}}\subset \supp \mu_0$ with $\mu_0(F_{\bar{t}})>0$ such that 
  \begin{equation}\label{eq:NonCont}
  \frak d_{\bar{t}}(x) \neq \bar{\frak d}_{\bar{t}}(x)=\lim_{n\to \infty} \frak d_{{t_n}}(x) \quad \forall x \in F_{\bar{t}}, \; t_n\in \Q\cap[0,1] \text{ with } t_n\to \bar{t}.  
  \end{equation}
  But choosing $\mathfrak D=\{q_1+q_2 \bar{t}\,:\, q_1,q_2\in \Q\}$, we get that  there exists a subset $E_{\mathfrak D}'
  \subset \supp \mu_0$ with $\mu_0(X\setminus E_{\mathfrak D}')=0$ such that the inequality \eqref{eq:BSLocInt}  holds for every $x \in E_{\mathfrak D}'$; therefore, repeating the arguments above,  Lemma \ref{lem:LocLipsigma} yields that   the function $[0,1]\cap \mathfrak D \ni r \mapsto  \frak d_{r}(x)\in \R^+$ is locally Lipschitz for every $x \in E_{\mathfrak D}'$. This is in contradiction with the discontinuity  \eqref{eq:NonCont} at $\bar{t}$, since  $\mu_0(F_{\bar{t}})>0$ and $\mu_0(X\setminus E_{\mathfrak D}')=0$.
  
   Integrating now \eqref{eq:12'} in $\d \mu_0(x)$, since by construction $\mu_0(X\setminus E)=0$ and ${\frak d}_{t}= \bar{\frak d}_{t}$ $\mu_0$-a.e, we get \eqref{eq:180N}. Indeed 
   $$ \int_E  \bar{\frak d}_{t} \, \d \mu_0= \int_E  \frak d_{t} \, \d \mu_0=\int_X \varrho_t^{-1/N}\circ T_t \, \d \mu_0 =\int_X\varrho_t^{1-{\frac{1}{N}}} \, \d \mm=1-\frac{1}{N}\cU_N(\mu_t);  $$
   and
   \begin{eqnarray}
   \int_{0}^{1} \left[\int_X \bar{\frak d}_{s}(x) \sfd^2(x, T_1(x)) \d \mu_0(x)\right]\fnd
    t{} s\,\d s &=&  \int_{0}^{1} \left[\int_X {\frak d}_{s}(x) \sfd^2(x, T_1(x)) \d \mu_0(x)\right]\fnd
    t{} s\,\d s  \nonumber \\
    &=& \int_{0}^{1} \left[\int_X\frak \varrho_{s}^{-\frac{1}{N}}(x) \, v^2(x)  \d \mu_s(x)\right]\fnd
    t{} s\d s \nonumber \\
    &=&   \int_{\tilde X}\sfg(s,t)Q_N(\varrho(x,s))v^2(x)\d\tilde\mu(x,s) \nonumber \\
    &=&\Mod t{Q_N}{\mu}\mm,  \label{eq:QNdsd2}
    \end{eqnarray}
    where we used the fact that, since the plan $\ppi\in
  \mathrm{GeoOpt}(\mu_0,\mu_1)$ is concentrated on constant speed geodesics and recalling \eqref{eq:27}, the minimal 2-velocity $v$ is constant in time and given by $v(x)=\sfd(x,T_1(x))$. 
  \AAA In particular,  \eqref{eq:QNdsd2} implies that whenever $\Mod t{Q_N}{\mu}\mm$ is finite then the function  $s\mapsto \bar{\frak d}_{s}(x) \sfd^2(x, T_1(x)) \fnd
    t{} s$ is integrable, and therefore $s\mapsto \bar{\frak d}_{s}(x)\in L^{1}_{loc}((0,1))$,   for $\mu_{0}$-a.e. $x \in X$. \fn
 \\
  
 [CD2] $\Rightarrow$ [CD3].  Let us start by assuming that  $\mu_0=\varrho_0\mm,\,\mu_1=\varrho_1 \mm$, $\varrho_t \mm=\mu_t=(T_t)_\sharp \mu_0=(\rm e_t)_{\sharp} \ppi$ are as in the above implication, i.e. $\varrho_i$, $i=0,1$, are $\mm$-essentially bounded with bounded support, so that by hypothesis   we know that  $\mu_t$ satisfies \eqref{eq:180N}. For any Borel subset $A\subset \supp \mu_0$ with $\mu_0(A)>0$, consider the localized and normalized measure $\mu_0^A:=\frac{1}{\mu_0(A)} \; \mu_0 \llcorner A= \frac{1}{\mu_0(A)} \chi_A  \; \mu_0$ and its push forwards $\mu_t^A:=(T_t)_\sharp \mu_0^A$. By cyclical monotonicity of the measure-theoretic support, it is  well known that  $ \frac{1}{\mu_0(A)} (\chi_A \circ \rm e_0) \ppi \in  \mathrm{GeoOpt}(\mu_0^A,\mu_1^A)$ so that $\mu^A_t$ is the $W_2$-geodesic from  $\mu^A_0$ to $\mu^A_1$ {\color{black} with essentially bounded densities $\varrho_t^A$
 satisfying $\varrho^A_t\circ T_t=\chi_A\varrho_t\circ T_t$.}
    Applying {\color{black} \eqref{eq:172}} and \eqref{eq:180N} to the geodesic $\mu_t^A$ gives the localized convexity inequality
     \begin{equation} \nonumber         \int_A{ \frak d}_{t} \, \d \mu_0 \geq
    (1-t) \int_A \frak d_{0}  \, \d \mu_0 +
    t  \int_A \frak d_{1}  \, \d \mu_0 +
    \frac KN \int_A \int_{0}^{1}  {\frak d}_{s}(x) \sfd^2(x, T_1(x))   \,\fnd
    t{} s\,\d s \, \d \mu_0(x),
      \end{equation}
for every  $t \in[0,1]$, where  $\frak d_t(x):=(\varrho_t^{-1/N}\circ T_t)(x)$ as before.  The arbitrariness of the Borel set $A$ implies that for
all $t\in [0,1]$ one has 
\begin{equation}
    \label{eq:LocSpace}
   {\frak d}_{t}(x) \geq
    (1-t) {\frak d}_{0} (x) + t  \frak d_{1}(x)  + \frac KN   \int_{0}^{1}  {\frak d}_{s}(x) \sfd^2(x, T_{1}(x))   \,\fnd
    t{} s\,\d s, \quad\text{for $\mu_0$-a.e. $x$.}
      \end{equation}

Now let instead $\mu_i=\varrho_i\mm \in  \Probabilitiesac X\mm \cap  \ProbabilitiesTwo X, \, i=0,1$, and     $\varrho_t \mm=\mu_t=(T_t)_\sharp \mu_0=(\rm e_t)_{\sharp} \ppi$ be the unique $W_2$-geodesic joining them. Consider the approximating  geodesic $\mu^k_t=\varrho^k_t \mm$ given by Lemma \ref{lem:ApproxGeod} below. Since $\varrho^k_i$ are $\mm$-essentially bounded with bounded support,  $\eqref{eq:LocSpace}$ holds for $\mu^k_t$  by assumption. 
It follows that  there exists $E_{k,t}\subset \supp \mu_0^k \subset \supp \mu_0$ with  $\mu^k_0(X\setminus E_{k,t})=0$ such that 
\begin{equation}
    \label{eq:LocSpacek}
   {\frak d}_{t}^k(x) \geq
    (1-t) {\frak d}_{0}^k  (x) + t  \frak d_{1}^k (x)  + \frac KN   \int_{0}^{1}  {\frak d}_{s}^k(x) \sfd^2(x, T_{1}(x))   \,\fnd
    t{} s\,\d s ,
 \end{equation}
for every  $x \in E_{k,t}$, where  ${\frak d}_t^k (x)={(\varrho^k_t}\circ T_t)^{-1/N}(x)$. Without loss of generality we may also assume $E_{k,t}\subset \{{\frak d}_t^k>0\}$.  Defining $E_t:=\bigcap_{k \in \N} E_{k,t}$, by using 
Lemma~\ref{lem:ApproxGeod}(4), we get that $E_t \subset \supp \mu_0$ and $\mu_0(X\setminus E_t)=0$. Moreover, observe that  \eqref{eq:LocSpacek} is still true for the renormalized measures $c_k \, \varrho_t^k$, since the constants just simplify from both sides thanks to the homogeneity of the entropy.  But then,  
Lemma~\ref{lem:ApproxGeod}(3) implies that for $\mu_0$-a.e. $x \in E$ one has ${\frak d}_{t}^k(x)={\frak d}_{t}(x)$, provided  $k$ is large enough. Passing to the limit for $k\to \infty$, we conclude that    \eqref{eq:LocSpace} holds and the thesis follows by integration in $\d\mu_0(x)$ as in the implication  [CD1] $\Rightarrow$ [CD2].
\\

[CD3] $\Rightarrow$ [CD4]    Let $\mu_i=\varrho_i\mm \in  \Probabilitiesac X\mm \cap  \ProbabilitiesTwo X$, $i=0,1$, and     $\varrho_t \mm=\mu_t=(T_t)_\sharp \mu_0=(\rm e_t)_{\sharp} \ppi$ be the unique $W_2$-geodesic joining them. Observing that the restriction to a subinterval $[r_0,r_1]\subset[0,1]$ of a geodesic is still a geodesic, the localization argument of the implication [CD2] $\Rightarrow$ [CD3] ensures that 
 for every $r_0,r_1,t \in [0,1]$ one has 
\begin{equation}
    \label{eq:CD411}
   {\frak d}_{r_t}(x) \geq
    (1-t) {\frak d}_{r_0} (x) + t  \frak d_{r_{1}}(x)  + \frac KN (r_1-r_0)^2  \int_{0}^{1}  {\frak d}_{r_s}(x) \sfd^2(x, T_{1}(x))   \,\fnd
    t{} s\,\d s \quad \mu_0\text{-a.e. }x,
      \end{equation}
where $r_t:=(1-t)r_0+tr_1$ as before. \AAA Note first of all that if $K=0$ the proof is simpler since the non-linear term in all the convexity inequalities just disappear. If $K\neq 0$ we first claim that for  $\mu_{0}$-a.e. $x \in X$ the function $s\mapsto \bar{\frak d}_{s}(x)$ belongs to $L^{1}_{loc}((0,1))$.    By Lemma \ref{lem:LBrhot}, we know that for $\mu_{0}$-a.e. $x\in X$ it holds $ \frak d_{0}(x),  \frak d_{1/2}(x),  \frak d_{1}(x) \in (0,\infty)$.  Specializing \eqref{eq:CD411}  to $r_{0}=0, r_{1}=1, t=1/2$ we get
$$
 \frac {K \sfd^2(x, T_{1}(x))}{2N} \left(  \int_{0}^{1/2} s  {\frak d}_{s}(x)   \,\d s+  \int_{1/2}^{1}  (1-s)  {\frak d}_{s}(x)  \,\d s \right) \leq {\frak d}_{1/2}(x) < \infty, \quad  \mu_0\text{-a.e. }x.
$$
In particular, if $K>0$, we get that $s\mapsto {\frak d}_{s}(x)$ belongs to $L^{1}_{loc}((0,1))$ for  $\mu_0\text{-a.e. }x$.  Also, if $K<0$, we may assume that $\Mod t{Q_N}{\mu}\mm<\infty$ otherwise the thesis of [CD4] trivializes, and then by \eqref{eq:QNdsd2} we get that  $s\mapsto \bar{\frak d}_{s}(x)\in L^{1}_{loc}((0,1))$. 
\fn
From \eqref{eq:CD411} it follows that  for any fixed countable $\Q$-vector space  $\mathfrak D \subset \R$, there exists a Borel subset $E_{\mathfrak D}\subset \supp \mu_0$ with $\mu_0(X\setminus E_{\mathfrak D})=0$ such that \eqref{eq:CD411} holds   for every  $x \in E_{\mathfrak D}$, every  $r_0,r_1\in [0,1]\cap{\mathfrak D}$ and $t\in [0,1]\cap \Q$. Since $ s\mapsto {\frak d}_{s}(x)$ in an element of $L^{1}_{loc}((0,1))$ for  $\mu_0\text{-a.e. }x$, choosing simply $\mathfrak D=\Q$, for every fixed $x\in E:=E_{\Q}$,  we can apply Lemma \ref{lem:LocLipF} to the function $[0,1]\cap \Q \ni r \mapsto  \frak d_{r}(x)\in \R^+$  and infer that such a map is locally Lipschitz; thus it admits a unique continuous extension $[0,1] \ni t\mapsto \bar{\frak d}_t(x) \in \R^+$ satisfying \eqref{eq:12'}.  Lemma \ref{le:0} gives then 
\begin{equation}\label{eq:13''}
\frac {\d^2}{\d t^2} \bar{\frak d}_t (x)\le -\frac KN  \sfd^2(x, T_1(x))  \; \bar{\frak d}_t (x) \quad \text{in }  \mathscr D'(0,1),  \text{ for every } x \in E \; .
\end{equation}
Given  now  $U \in  \MC N$, recalling \eqref{defprop:V} and taking \eqref{eq:13''} into account, we get the following chain of inequalities in distributional sense 
\begin{eqnarray}
\frac{\d^2}{\d t^2} V(\bar{\frak d}_t (x))&=& V''(\bar{\frak d}_t (x)) \; \left( \frac{\d}{\d t}  \bar{\frak d}_t (x)\right)^2+ V'(\bar{\frak d}_t (x)) \; \frac{\d^2}{\d t^2} \bar{\frak d}_t (x)\ \nonumber \\
                                                            &\geq&  -\frac KN   \; \bar{\frak d}_t (x) \; \sfd^2(x, T_1(x)) \;   V'(\bar{\frak d}_t (x))  \quad \text{in }  \mathscr D'(0,1),  \text{ for every } x \in E. \nonumber
\end{eqnarray} 
Applying again Lemma \ref{le:0}, this time with $u(t):= V(\bar{\frak d}_t (x))$,  we obtain
\begin{equation}\label{eq:VWeigConv}
V(\bar{\frak d}_t (x))\leq (1-t) V(\bar{\frak d}_0 (x)) + t V(\bar{\frak d}_1 (x))+  \frac KN  \int_0^1  \bar{\frak d}_s (x) \; \sfd^2(x, T_1(x)) \;   V'(\bar{\frak d}_s (x))   \,\fnd t{} s\,\d s,
\end{equation}
for every $x \in E$.  With the same argument as in the proof of [CD1] $\Rightarrow$ [CD2], we have that for every $t \in [0,1]$ it holds $\bar{\frak d}_t (x)={\frak d}_t (x)=(\varrho_t^{-1/N}\circ T_t)(x)$ for $\mu_0$-a.e. $x$. The desired inequality \eqref{eq:180} follows then by integrating  \eqref{eq:VWeigConv} in $\d \mu_0(x)$, since by construction $\mu_0(X\setminus E)=0$.  Indeed, recalling that $V({\frak d}_t (x))=V(\varrho_t^{-1/N}\circ T_t(x))=\frac {U(\varrho_t \circ T_t(x) )}{\varrho_t \circ T_t(x)}$, we have
\begin{eqnarray}
\int_E V(\bar{\frak d}_t)\, \d \mu_0&=& \int_X V({\frak d}_t)\, \d \mu_0=\int_X \frac {U(\varrho_t \circ T_t)}{\varrho_t \circ T_t} \, \d \mu_0= \int_X \frac {U(\varrho_t )}{\varrho_t} \, \d ((T_t)_\sharp \mu_0) \nonumber\\
                                                      &=&\int_X \frac {U(\varrho_t )}{\varrho_t} \, \d (\varrho_t \mm)= \int_X U(\varrho_t )  \, \d \mm = {\cal U}(\mu_t)\quad . \nonumber
\end{eqnarray}
For the action term in \eqref{eq:180} observe that, since the plan $\ppi\in
  \mathrm{GeoOpt}(\mu_0,\mu_1)$ is concentrated on constant speed geodesics and recalling \eqref{eq:27}, the minimal 2-velocity $v$ is constant in time and given by $v(x)=\sfd(x,T_1(x))$.  Therefore, noting that $Q(r)=-\frac{1}{N}r^{-\frac{1}{N}} V'(r^{-\frac{1}{N}})$ for   $\Leb{1}$-a.e. $r\in (0,1)$, we obtain
\begin{eqnarray}
&& \frac KN \int_E \left[ \int_0^1\bar{\frak d}_s (x) \; \sfd^2(x, T_1(x)) \;   V'(\bar{\frak d}_s (x))   \,\fnd t{} s\,\d s  \right] \d \mu_0(x) \nonumber\\
&& \quad \quad \quad =  K \int_0^1 \left[ \int_X   v^2(x) \frac{1}{N}  {\frak d}_s (x)  \;   V'({\frak d}_s (x)) \, \d \mu_0(x)  \right]\,\fnd t{} s\,\d s \nonumber \\
 &&\quad \quad \quad = -K   \int_0^1 \left[ \int_X   v^2(x)  Q(\varrho_s(x))\, \d \mu_s(x)  \right]\,\fnd t{} s\,\d s = -K\Mod t{Q}{\mu}\mm . \nonumber 
\end{eqnarray}
\\
 
 [CD4] $\Rightarrow$ [CD1].  Let $\mu_0=\varrho_0\mm,\,\mu_1=\varrho_1 \mm\in \Probabilitiesac X\mm$ 
    with densities $\varrho_i$ having bounded support, so in particular $\mu_i \in \ProbabilitiesTwo X$. By [RS1-4] there exists a unique $W_2$-geodesic $\mu_t=(\rm e_t)_{\sharp} \ppi$ 
    from $\mu_0$ to $\mu_1$, it is made of absolutely continuous measures and it is given by optimal maps:   $\varrho_t \mm=\mu_t=(T_t)_\sharp \mu_0$.  Choosing $U=U_N$,  we get that $\mu_t$ satisfies \eqref{eq:180N} by assumption. Localizing in space and time as above, we obtain \eqref{eq:CD411}, namely 
    \begin{equation}
   \nonumber
   {\frak d}_{r_t}(x) \geq
    (1-t) {\frak d}_{r_0} (x) + t  \frak d_{1}(x)  + \frac KN (r_1-r_0)^2  \int_{0}^{1}  {\frak d}_{r_s}(x) \sfd^2(x, T_1(x))   \,\fnd
    t{} s\,\d s \quad \mu_0\text{-a.e. }x.
      \end{equation}
It follows that, for any fixed countable $\Q$-vector space  $\mathfrak D \subset \R$, there exists a Borel subset $E_{\mathfrak D}\subset \supp \mu_0$ with $\mu_0(X\setminus E_{\mathfrak D})=0$ such that \eqref{eq:CD411} holds   for every  $x \in E_{\mathfrak D}$, every  $r_0,r_1\in [0,1]\cap{\mathfrak D}$ and $t\in [0,1]\cap \Q$. \AAA Since by assumption  $\mu_{0}$ and $\mu_{1}$ are bounded with bounded supports, it follows that  $\Mod t{Q_N}{\mu}\mm$ is finite; thus, from  \eqref{eq:QNdsd2}, we get that the map  $s\mapsto \bar{\frak d}_{s}(x)$ is an element of $L^{1}_{loc}((0,1))$ for $\mu_{0}$-a.e. $x \in X$\fn. Therefore  choosing $\mathfrak D=\Q$, for every fixed $x\in E:=E_{\Q}$,  we can apply Lemma \ref{lem:LocLipF} to the function $[0,1]\cap \Q \ni r \mapsto  \frak d_{r}(x)\in \R^+$ and infer that such a map is locally Lipschitz, so it admits a unique continuous extension $[0,1] \ni t\mapsto \bar{\frak d}_t(x) \in \R^+$ satisfying \eqref{eq:12'}.  
Lemma~\ref{le:0} gives then \eqref{eq:13'} and Lemma~\ref{lem:DiffInesigma} yields
  \begin{equation}\label{eq:BSLocIntf}
    \bar{\frak d}_{t}(x)\ge \sigma_{K/N}^{(1-t)}(\sfd(x,T_1 (x)))\, \bar{\frak d}_{0}(x)+
    \sigma_{K/N}^{(t)}(\sfd(x,T_{1}(x)))\,\bar{\frak d}_{1}(x), \forall x 	\in E,\forall t \in \Q\cap[0,1]. 
  \end{equation}  
  Arguing as in the implication [CD1] $\Rightarrow$ [CD2], we get that   for every $t \in [0,1]$ one has $\bar{\frak d}_{t}={\frak d}_{t}$, $\mu_0$-a.e. in $X$. Integrating \eqref{eq:BSLocIntf} in $\d \mu_0(x)$ gives \eqref{eq:177} for ${\cal U}_N$; since for every $M>N$ one has $U_M\in  \MC N$, the argument for any other $M> N$ is completely analogous: just replace $N$ with $M$ in the formulas above.
  \\  
    
 [CD5] $\Rightarrow$ [CD2].   Let $\mu_i=\varrho_i\mm \in  \Probabilitiesac X\mm$ with $\varrho_i$ $\mm$-essentially bounded having bounded supports, $i=0,1$, and     $\varrho_t \mm=\mu_t=(T_t)_\sharp \mu_0=(\rm e_t)_{\sharp} \ppi$ be the unique $W_2$-geodesic joining them.  Under our working assumptions we have that  $\varrho_t$, $t\in [0,1]$, are uniformly  $\mm$-essentially bounded with uniformly bounded supports. Given the $N$-dimensional entropy  $U(r):=Nr(1-r^{-1/N})$ with associated pressure $P(r):=r^{1-1/N}$, for every $k \in \N$ call $P_k$ the regularized pressure $P_k:=P_{1/k, k}$  where $P_{1/k, k}$ was defined in \eqref{eq:207}. Called $U_k$ the regularized and normalized entropy associated to $P_k$ as in \eqref{eq:47}, observe that
 \begin{equation}\label{eq:UnifConv}
 P_k\to P \text{ and }   U_k\to U  \text{ uniformly on } [0,R]  \text{ for every } R\in \R^+. 
 \end{equation}
 Since  $U_k\in \MCreg N$, by assumption for every $k \in \N$ we have
 \begin{eqnarray}
 \int_{\supp \mu_t} U_k(\varrho_t) \, \d \mm &\leq& (1-t) \int_{\supp \mu_0} U_k(\varrho_0) \, \d \mm  + t \,  \int_{\supp \mu_1} U_k(\varrho_1) \, \d \mm \nonumber \\
                                                                             && - K \int_0^1 \left[\int_{\supp \mu_s} P_k(\varrho_s)\, \sfd^2(x,T_1(x)) \, \d \mm(x) \right] \,\fnd t{} s\,\d s  \; ,  \label{eq:CD52}
\end{eqnarray} 
where we used that $Q(r)=P(r)/r$ by definition. Recalling that $\varrho_t$ are uniformly $\mm$-essentially bounded with uniformly bounded supports, we infer that $\mm\left( \bigcup_t \supp(\mu_t)\right)<\infty$ and  the uniform convergence \eqref{eq:UnifConv} allows to pass to the limit in \eqref{eq:CD52}, obtaining \eqref{eq:180N}.
 \end{proof} 
 
 \begin{lemma}\label{lem:ApproxGeod}
 Let $(X,\sfd,\mm)$ be a strong $\CD K \infty$ space, so that {\rm [RS1-4]} hold. Consider $\mu_i=\varrho_i \mm \in \Probabilitiesac X\mm \cap \ProbabilitiesTwo X$, $i=0,1$, and let $\ppi \in \rm{GeoOpt}(\mu_0, \mu_1)$ be the  plan representing the  $W_2$-geodesic  $\varrho_t \mm=\mu_t=({\rm e}_t)_\sharp \ppi=(T_t)_\sharp (\mu_0)$ from $\mu_0$ to $\mu_1$. 
 
 Then  there exist  sequences of measures $\mu_0^k=\varrho^k_0 \mm \in   \Probabilitiesac X\mm$ and  constants $c_k \uparrow 1$ such that the curve  $\varrho^k_t \mm:= \mu_t^k:=(T_t)_\sharp (\mu_0^k)$  is the $W_2$-geodesic from $\mu_0^k$ to $\mu_1^k$ and it satisfies the following:
 \begin{enumerate}
 \item[(1)] $\varrho^k_i $ are $\mm$-essentially bounded and with bounded support, $i=0,1$;
 \item[(2)]  $c_k \, \varrho^k_t \leq \varrho_t$  $\mm$-a.e.in $X$ for every $t \in [0,1]$;
 \item[(3)] for every $t \in [0,1]$ it holds  $c_k \, \rho_t^k(x)= \rho_t(x)$ for $\mm$-a.e.  $x \in X$, for $k$ large enough possibly depending on $x$;
 \item[(4)]  $\mu^k_0=c_k^{-1} \sigma_k\mu_0$ with $\sigma_k\uparrow 1$,  $\mu_0$-a.e. on $X$.
 \end{enumerate}
 \end{lemma}

 \begin{proof}
 Fix a base point $\bar{x}\in \supp \mu_0$, call $B_k:=B_k(\bar{x})$ the  ball of center $\bar{x}$ and radius $k\in \N$. For every $k \in \N$ consider first the densities $\bar{\varrho}_0^k:= \chi_{B_k} \, \min\{k, \varrho_0\}$ and the push forward measures $\bar{\mu}^k_1:=(T_1)_\sharp (\bar{\varrho}_0^k \mm)$.  Since clearly $\bar{\varrho}_0^k\leq \varrho_0$ and  $\bar{\varrho}_0^k= \varrho_0$ on $\{x \in B_k\,:\, \varrho_0(x)\leq k\}$, and since by assumption  $T_1$ is 
 {$\mu_0$}-essentially injective, we have 
 \begin{equation}\label{eq:bark1}
 \bar{\varrho}^k_1 \mm:= \bar{\mu}^k_1 \leq \mu_1 \quad \text{ and } \quad  \bar{\varrho}^k_1=\varrho_1 \quad \text{ on } \;   T_1( \{x \in B_k\,:\, \varrho_0(x)\leq k\} ) .
 \end{equation}
 Consider now $\tilde{\varrho}^k_1:=  \chi_{B_k} \, \min\{k, \bar{\varrho}_1^k\}$. Using again the  $\mu_0$-essential injectivity of $T_1$ and observing that $\tilde{\varrho}^k_1 \leq \bar{\varrho}^k_1 \leq {\varrho}_1$, we can define $\tilde{\mu}_0^k:=(T_1^{-1})_\sharp (\tilde{\varrho}^k_1 \mm)$. By construction we have
  \begin{equation}\label{eq:tildek0}
 \tilde{\varrho}^k_0 \mm:= \tilde{\mu}^k_0 \leq \bar{\varrho}_0^k \mm \leq \mu_0 \; \text{ and } \;  \tilde{\varrho}^k_0=\varrho_0 \text{ on }  T_1^{-1}\bigl(T_1\bigl( B_k\cap\bigl\{\max\{\varrho_0,\bar\varrho_1^k\}\leq k\bigr\} \bigr)\bigr);
 \end{equation}
  in particular we have that $ \tilde{\varrho}^k_i \leq k$ and $\supp \tilde{\varrho}^k_i \subset B_k$, $i=0,1$. Moreover, for $\mm$-a.e. $x$ we have  $\tilde{\varrho}^k_0(x)=\varrho_0(x)$ for $k$ large enough (possibly depending on $x$).
  \\ Setting $c_k:= \tilde{\mu}^k_0(X)$, $\mu_0^k:= c_k^{-1}  \tilde{\mu}^k_0$ and ${\mu}^k_t:=(T_t)_\sharp (\mu^k_0)$ we get the thesis. 
 \end{proof}
 
 \begin{lemma}\label{lem:LBrhot}
 Let $\mu_0\in \Probabilities X$ and let $T: \supp \mu_0 \to X$ be a  $\mu_0$-essentially injective map such that $\varrho_1 \mm:=\mu_1:=T_\sharp (\mu_0) \in \Probabilitiesac X\mm$. Then 
 \begin{equation}\label{eq:rho1=0}
 \mu_0(\{x \in \supp \mu_0\, : \, \varrho_1 (T (x))=0\})=0 .
 \end{equation}
In particular, given $\mu_t=\varrho_t\mm=(T_t)_\sharp (\mu_0)$  a $W_2$-geodesic as in [RS1-3], for any  finite subset ${\mathfrak F}\subset [0,1]$ we have
\begin{equation}\label{eq:LBrhor}
\mu_0\bigl(\{x \in \supp \mu_0\, : \, \min_{r\in{\mathfrak F}}\varrho_r (T_r(x))>0\}\bigr)=1 .
\end{equation}
 \end{lemma}

\begin{proof}
Let us consider the set $A:=\{x \in \supp \mu_0\, : \, \varrho_1 (T (x))=0\}$.  Since by assumption $\mu_1=T_\sharp (\mu_0)$ and $T$ is $\mu_0$-essentially injective, we have that $T$ is $\mu_1$-a.e. invertible and $\mu_0=(T^{-1})_\sharp (\mu_1)$. It follows that 
$$\mu_0(A)=\mu_0\big(T^{-1}(T(A))\big)=\mu_1(T(A))=\int_{T(A)} \varrho_1 \, \d \mm=0 , $$
since, by definition of $A$, we have  $\rho_1\equiv 0$ on $T(A)$. This proves the first statement. 
\\The second one is an easy consequence of the finiteness of ${\mathfrak F}$; indeed, called $$A_r:=\{x \in \supp \mu_0\, : \, \varrho_r (T_r (x))=0\},$$ by the first part of the lemma we have that $\mu_0(A_r)=0$ for every $r \in {\mathfrak F}$. Denoted with 
$$C_n:=\left\{x \in \supp \mu_0\, : \, \varrho_r (T_r (x))\geq \frac{1}{n} \text{ for every } r \in {\mathfrak F}\right\} ,$$
using the finiteness of ${\mathfrak F}$ we have 
 $$ \bigcup_{n \in \N} C_n = X \setminus \bigcup_{r \in {\mathfrak F}} A_r\, .$$
We conclude that $\bigcup_{n \in \N} C_n$ is of full $\mu_0$-measure and the proof is complete.
\end{proof}

\subsection{$\RCD K\infty$ spaces and a criterium for $\CDS KN$ via $\mathrm{EVI}$}

Let us first recall the definition of $\RCD K\infty$ spaces, introduced and characterized in
\cite{AGS11b} (see also \cite{AGMR12} for the present simplified axiomatization and extension to $\sigma$-finite measures); 
in the statements involving the so-called evolution variational inequalities,
characterized by differential inequalities involving the squared distance, the
entropy and suitable action functionals, we will use the notation
\begin{equation}
  \label{eq:214}
  \frac{\d^+}{\dt}\zeta(t):=
  \limsup_{h\down0}\frac{\zeta(t+h)-\zeta(t)}h
\end{equation}
for the upper right Dini derivative.

\begin{definition}[$\RCD K\infty$ metric measure spaces] \label{def:RCDI}
  A metric measure space $(X,\sfd,\mm)$ is an $\RCD K\infty$ space if
  it satisfies one of the following equivalent conditions:
  \begin{enumerate}[\rm ({RCD}1)]
  \item $(X,\sfd,\mm)$ satisfies the 
    $\CD K\infty$ condition and the Cheeger energy is quadratic.
 \item For every $\mu\in D(\cU_\infty)\cap \ProbabilitiesTwo X$ 
    there exists a curve $\mu_t=\sfH_t\mu$, $t\ge0$, such that 
    \begin{equation}
      \label{eq:213}
      \frac 12\frac {\d^+}\dt W_2^2(\mu_{t},\nu)
      +\Reny\infty\mu\le \Reny\infty\nu-
      \frac K2W_2^2(\mu_t,\nu)\quad
      \forevery t\ge0,\,\,\nu\in D(\cU_\infty).
    \end{equation}
\end{enumerate}
\end{definition}

Among the important consequences of the above property, we recall that: 
\begin{enumerate}
\item $\RCD K\infty$ spaces are \emph{strong} $\CD K\infty$ spaces
  and thus satisfy properties [RS1-4].
\item The map $(\sfH_t)_{t\ge0}$ is uniquely characterized by 
  \eqref{eq:213}, it is a $K$-contraction in $\ProbabilitiesTwo X$ and
  it coincides with the heat flow $\sfP_t$, i.e.
  \begin{equation}
    \label{eq:215}
    \sfH_t(\varrho\mm)=(\sfP_t\varrho)\mm\quad
    \forevery \varrho\mm\in D(\cU_\infty)\cap \ProbabilitiesTwo X.
  \end{equation}
\item Lipschitz functions essentially coincide with functions $f\in\V$ with $|\rmD f|_w\in L^\infty(X,\mm)$, more
precisely (recall that, according to \eqref{eq:253}, $\V_\infty$ stands for $\V\cap L^\infty(X,\mm)$):
  \begin{equation}\label{eq:233}
    \text{every $f\in \V_\infty$ with $|\rmD f|_w\le 1$ $\mm$-a.e. in $X$ admits a 
  $1$-Lipschitz representative.}
\end{equation}
\item The Cheeger energy satisfies the Bakry-\'Emery $\BE K\infty$ condition: we
  will
  discuss this aspect in the next Section~\ref{sec:BE}.
\end{enumerate}
We will show that a similar characterization holds 
for strong $\CDS KN$ spaces. 

In order to deal with a general class of entropy functionals 
$\cU$ with entropy density satisfying the McCann condition
$\MC N$ and 
arbitrary curvature bounds $K\in \R$, 
for every $\mu\in \AC2{[0,1]}{(\ProbabilitiesTwo X,W_2)}$ with $\mu_s\ll\mm$ for $\Leb{1}$-a.e. $s\in (0,1)$
we consider the weighted action functional
associated to $\weight sr=\omega(s)Q(r)$ as in \eqref{eq:17}, with $\omega(s):=1-s$:
\begin{equation}
  \label{eq:10bis}
  \EMod Q\mu\mm:=\Action{\omega Q}{\mu}{\mm}=\int_{\tilde X} (1-s)Q(\varrho(x,s))v^2(x,s)
  \,\d\tilde\mu(x,s).
\end{equation}
If $(X,\sfd,\mm)$ is a strong $\CD K\infty$ space 
then for every $\mu_0,\,\mu_1\in \Probabilitiesac X\mm$ 
we can also set
\begin{equation}
  \label{eq:183}
  \Action{\omega Q}{\mu_0,\mu_1}\mm:=
  \Action{\omega Q}\mu\mm,\quad
  \text{with $\mu$ the unique geodesic connecting 
    $\mu_0$ to $\mu_1$.}
\end{equation}
Since $\omega(1-s)+\omega(s)=1$, we obtain the useful identity
\begin{equation}\label{eq:time_reversal}
\Action{Q}{\mu_0,\mu_1}\mm=
\Action{\omega Q}{\mu_0,\mu_1}\mm+
\Action{\omega Q} {\mu_1,\mu_0}\mm.
\end{equation}
{\color{black}
We will need the following Lemma, proved in the case of the logarithmic entropy
in \cite[Thm.~3.6]{AGMR12}. The proof is analogous for regular entropies $U''$, since their second derivative
$U''$ still diverges like $z^{-1}$.

\begin{lemma}\label{lem:spaziopesatoU}
Let $U\in\MCreg N$ and let $\varrho\in\V\cap L^\infty(X,\mm)$ be satisfying
$$
\int_X \varrho U''(\varrho)^2\Gamma(\varrho)\,\d\mm<\infty.
$$
Then $U'(\varrho)\in\V_\varrho$ and
$$
\int_X\Gamma_\varrho(U'(\varrho),\varphi)\,\rho\mm=\int_X\Gamma(P(\varrho),\varphi)\qquad
\forall\varphi\in\V.
$$
\end{lemma}
} 

\begin{theorem}
  \label{thm:CD-EVI}
  Let $(X,\sfd,\mm)$ be a strong $\CDS KN$ space
  and let us suppose that the Cheeger energy $\C$ is quadratic 
  as in \eqref{eq:1}. 
  Let $U\in \MCreg N$, $P,\,Q$ as in \eqref{eq:40}, 
  $\Lambda:=\inf_{r\ge0} KQ(r)$ and let $(\sfS_t)_{t\ge0}$
  be the flow defined by Theorem~\ref{thm:nonlin-diff}.\\
  Then $\sfS_t$ induces a $\Lambda$-contraction in 
  $(\ProbabilitiesTwo X,W_2)$ and for every $\mu=\varrho\mm\in 
  D(\cU)\cap \ProbabilitiesTwo X$ 
  the curve $\mu_t:=(\sfS_t\varrho)\mm$ satisfies
  \begin{equation}
    \label{eq:218}
    \frac 12\frac{\d^+}\dt W_2^2(\mu_t,\nu)
    +\cU(\mu_t)\le \cU(\nu)-K \EMod Q{\mu_t,\nu}\mm
    \quad\forevery \nu\in D(\cU)\cap\ProbabilitiesTwo X,\,\,t\geq 0.
  \end{equation}
\end{theorem}
\begin{proof} The proof of \eqref{eq:218} follows the lines of \cite{AGMR12}
(where the case $\cU=\cU_\infty$ was considered), which extends to the $\sigma$-finite case the analogous
result proved with finite reference measures $\mm$ in \cite{AGS11b}. {\color{black} All technical difficulties are
due to the fact that $\mm$ is potentially unbounded, the proof being much more direct for finite measures $\mm$.} 

Specifically, first the proof is reduced to the case of measures $\mu=\varrho\mm$
and $\nu$ with $\varrho\in L^\infty(X,\mm)$ and $\nu$ with bounded support. 
{\color{black} First of all, notice that the combination of Theorem~\ref{thm:nonlin-diff} and Theorem~\ref{thm:dynamicKant} ensures that the curve $t\mapsto \mu_{t}$ is $W_{2}$-absolutely continuous. Then, using the dual formulation \eqref{eq:154}
of the optimal transport problem, {\color{black} and \eqref{eq:76}}
one can show that for $\Leb{1}$-a.e. $t>0$ one has (see \cite[Thm.~6.3]{AGMR12})
\begin{equation}\label{eq:aug3_1}
\frac{\d}{\d t}\frac 12 W_2^2(\mu_t,\nu)=
-\int_X  \Gamma  (\varphi_t,P(\varrho_t))\, \d\mm
\end{equation}
for \emph{any} optimal Kantorovich potential $\varphi_t\in\V$ from $\mu_t$ to $\nu$, and the existence
of potentials with this property is ensured by the boundedness of the support of $\sigma$ (see \cite[Prop~2.2]{AGMR12}).}

 On the other hand, one can also use
the calculus tools developed in \cite{AGS11a,AGS11b} to estimate (see \cite[Thm.~6.5]{AGMR12}) 
\begin{equation}\label{eq:aug3_2}
\cU(\nu)-\cU(\mu_t)-K \EMod Q{\mu_t,\nu}\mm
\geq- \int_X \Gamma_{\varrho_t}  (\varphi_t,U'(\varrho_t))\varrho_t\,\d\mm
\end{equation}
for \emph{some} optimal Kantorovich potential $\varphi_t$ from $\mu_t$ to $\nu$.
Using  {\color{black} Lemma~\ref{lem:spaziopesatoU}, whose application is justified by
Theorem~\ref{thm:emF},} to combine
 \eqref{eq:aug3_1} and \eqref{eq:aug3_2} gives \eqref{eq:218}.

In turn, the proof of \eqref{eq:aug3_2} goes as follows. First of all one notices that
\begin{equation}\label{eq:aug3_5}
\lim_{s\downarrow 0} \frac 1 s\Mod s{Q}{\mu_{\cdot,t}}\mm=\EMod Q{\mu_t,\nu}\mm,
\end{equation}
where $s\mapsto\mu_{s,t}$ is any constant speed geodesic joining $\mu_t$ to $\nu$. 
Indeed, setting $\mu_{s,t}=\varrho_{s,t}\mm$ and denoting by $v_{s,t}(x)$ the minimal velocity
density of $\mu_{\cdot,t}$, we can use the expression \eqref{eq:3} of $\sfg$ to write
$$
\Mod s{Q}{\mu_{\cdot,t}}\mm=\int_0^s (1-s)r\int_X Q(\varrho_{r,t})v^2_{r,t}\varrho_{r,t}\,\d\mm \,\d r+
\int_s^1 (1-r)s\int_X Q(\varrho_{r,t})v^2_{r,t}\varrho_{r,t}\,\d\mm \,\d r.
$$
Since the first term in the right hand side is $o(s)$ (recall that $Q$ is a bounded function), by monotone convergence we obtain
\eqref{eq:aug3_5}. 

Now, by the convexity inequality \eqref{eq:180} one has
\begin{equation}\label{eq:aug3}
\cU(\nu)-\cU(\mu_t)-\liminf_{s\downarrow 0}\frac 1s
K\Mod s{Q}{\mu_{\cdot,t}}\mm\geq
\limsup_{s\downarrow 0}
\frac{\cU(\mu_{s,t})-\cU(\mu_t)}{s}.
\end{equation}
In addition, if $\varrho_t$ decays sufficiently fast at infinity, one can estimate the 
directional derivative of $\cU$ as follows: 
\begin{equation}\label{eq:aug3_4}
\limsup_{s\downarrow 0}
\frac{\cU(\mu_{s,t})-\cU(\mu_t)}{s}\geq
\limsup_{s\downarrow 0}
\int_XU'(\varrho_t)\frac{\varrho_{s,t}-\varrho_t}{s}\,\d\mm\geq
\int_X\Gamma_{\varrho_t}(\varphi_t,U'(\varrho_t))\varrho_t\,\d\mm,
\end{equation}
\AAA where in the last step we used Theorem \ref{thm:dynamicKant}. \fn
The combination of \eqref{eq:aug3} and \eqref{eq:aug3_4} gives \eqref{eq:aug3_2}, 
taking \eqref{eq:aug3_5} into account. Then the decay
assumption on $\varrho_t$ is removed by an approximation argument, recovering \eqref{eq:aug3_2} in
the general case. This concludes the proof of \eqref{eq:218}.

Since geodesics have constant speed, from \eqref{eq:161} we obtain the identity
$$
\int_0^1 (1-r)\int_X v^2_{r,t}\varrho_{r,t}\,\d\mm \,\d r=
\frac{1}{2}W^2_2(\mu_t,\nu).
$$
Hence, from \eqref{eq:218}
we get the standard ${\sf EVI}$ condition \eqref{eq:213} with $\cU_\infty$ replaced by $\cU$ 
and $K$ replaced by $\Lambda$, and it
is well-known (see for instance \cite[Cor.~4.3.3]{AGS08}) that this leads to $\Lambda$-contractivity.
\end{proof}
Conversely, we can now prove adapting the proof of \cite{Daneri-Savare08} that the infinitesimal version of \eqref{eq:218} leads to
the strong $\CDS KN$ condition.

\begin{theorem}\label{thm:EVI-CD} 
  Let $(X,\sfd,\mm)$ be a strong $\CD K\infty$ metric measure space. 
  Suppose that for every $U\in \MCreg N$ and
  every $\bar\mu=\varrho\mm\in \ProbabilitiesTwo X$ 
  with $\varrho\in L^\infty(X,\mm)$ with bounded support there exists a 
  curve $\mu_t=\sfS_t\bar\mu\in \ProbabilitiesTwo X$, $t\ge0$,
  such that 
  \begin{equation}
    \label{eq:181}
    \limsup_{h\down0}\frac {W_2^2(\mu_h,\nu)-W_2^2(\bar\mu,\nu)}{2h}
    +\cU(\bar\mu)\le \cU(\bar\nu)-K \EMod Q{\bar\mu,\bar\nu}\mm
  \end{equation}
  for every \AAA  $\bar\nu=\sigma \mm\in \ProbabilitiesTwo X$ 
  with $\sigma\in L^\infty(X,\mm)$ \AAA with bounded support \EEE .
  Then $(X,\sfd,\mm)$ satisfies the strong $\CDS KN$ condition 
  and the Cheeger energy is quadratic. 
\end{theorem}
\begin{proof}
   We prove the validity of [CD2] of Theorem~\ref{thm:Andrea_complete}.
   So, let us fix $\mu_0,\,\mu_1\in D(\cU)$ with bounded densities and support and
   let $(\mu_s)_{s\in [0,1]}$ be the geodesic
  connecting $\mu_0$ to $\mu_1$. \AAA Notice that in virtue of  \cite[Thm.\ 1.3]{Rajala12}  we have that $\mu_{s}=\varrho_{s} \mm$ with $\varrho_{s}$ $\mm$-essentially bounded with bounded support. \EEE
  For a given $s\in (0,1)$, let 
  $\mu_{s,t}:=\mathsf S_t \mu_s$ be the curve 
  starting from $\bar\mu=\mu_s$ and satisfying \eqref{eq:181}.

  Choosing $\nu:=\mu_0$ and taking the 
  right upper derivative 
  at $t=0$ (still denoted for simplicity by ${\d^+}/{\dt}$) we get
  \begin{displaymath}
    \frac 12\frac {\d^+}{\dt}W_2^2(\mu_{s,t},\mu_0)\restr{t=0}+
    \Reny {}{\mu_s}-\Reny{}{\mu_0}\le 
    -K\EMod Q{\mu_s,\mu_0}\mm.
  \end{displaymath}
  Similarly, choosing $\nu:=\mu_1$, we get
    \begin{displaymath}
    \frac 12\frac {\d^+}{\dt}W_2^2(\mu_{s,t},\mu_1)\restr{t=0}+
    \Reny {}{\mu_s}-\Reny{}{\mu_1}\le 
    -K\,\EMod Q{\mu_s,\mu_1}\mm.
  \end{displaymath}
  Let us observe, as in \cite{Daneri-Savare08}, that 
  \begin{displaymath}
    \frac{\d^+}{\dt}\Big((1-s)W_2^2(\mu_{s,t},\mu_0)+s W_2^2(\mu_{s,t},\mu_1)\Big)\restr{t=0_+}\ge0
  \end{displaymath}
  since the inequality $(a+b)^2\leq a^2/s+b^2/(1-s)$ gives
  \begin{displaymath}
    (1-s)W_2^2(\mu_{s,t},\mu_0)+s W_2^2(\mu_{s,t},\mu_1)\ge 
    s(1-s)W_2^2(\mu_0,\mu_1)=
    (1-s)W_2^2(\mu_{s},\mu_0)+s W_2^2(\mu_{s},\mu_1).
  \end{displaymath}
  Hence, taking a convex combination of the two inequalities
  with weights $(1-s)$ and $s$ respectively, we obtain
  \begin{align*}
    (1-s)\Reny{}{\mu_0}
    +s \Reny{}{\mu_1}
    -\Reny{}{\mu_s}
    &\ge 
    (1-s)K\,\EMod Q{\mu_s,\mu_0}\mm
    +sK\,\EMod Q{\mu_s,\mu_1}\mm.
  \end{align*}
  Now observe that (for $\Theta_r=\int Q(\varrho_r)v_r^2\,\d\mu_r$, $s(1-\xi)=r$)
  \begin{displaymath}
    \EMod Q{\mu_s,\mu_0}\mm=
    s^2\int_0^1 \Theta_{s(1-\xi)} (1-\xi) \,\d\xi=
    \int_0^s \Theta_r\, r\,\d r
  \end{displaymath}
  and, analogously, that (for $\Theta_r$ as above, $s+(1-s)\xi=r$)
  \begin{displaymath}
    \EMod Q{\mu_s,\mu_1}\mm=
    (1-s)^2\int_0^1 \Theta_{s+(1-s)\xi} (1-\xi) \,\d \xi=
    \int_s^1 \Theta_r(1-r) \,\d r,
  \end{displaymath}
  so that the definition \eqref{eq:3} of $\sfg$ gives
  \begin{align*}
     (1-s)\EMod Q{\mu_s,\mu_0}\mm
    +s\EMod Q{\mu_s,\mu_1}\mm
    &=
    \int_0^s \Theta_r\, (1-s)r\,\d r+
    \int_s^1 \Theta_r s(1-r) \,\d r
    \\&=
    \int_0^1 \Theta_r \fnd s{}r\,\d r
    =\Mod sQ{\mu}\mm.
  \end{align*}
  This proves that \eqref{eq:180}
  holds \AAA for every $\mu_0,\,\mu_1\in D(\cU)$ with bounded densities and support; 
  taking Remark~\ref{rem:Levico} into account, we then get that [CD2] holds 
  and therefore $(X,\sfd,\mm)$ is a strong $\CDS K N$
  space. \fn  It remains to show that the Cheeger energy is quadratic;
  by applying the characterization of $\RCD K \infty$ spaces
  recalled in Definition~\ref{def:RCDI} 
  it is sufficient to check that \eqref{eq:181} yields \eqref{eq:213} 
  as a particular case. In fact, we can choose 
  the regular entropy  $U_\infty(r):=r \log r \in \MCreg N$ with $Q_{\infty}\equiv 1$,  
  and observe that 
  the associated weighted action on constant speed geodesics is nothing but half of the standard 2-action:
 $$  \Action{\omega Q_\infty}{\mu_0,\mu_1}\mm= \int_0^1  \int_X (1-s) \, v^2(x,s) \, \d \mu_s \, \d s =   \int_0^1 (1-s) |\dot \mu_s|^2 \, \d s =  \frac{1}{2} W_2^2(\mu_0,\mu_1)^2 ,$$ 
 where in the second equality we recalled \eqref{eq:161} and in the last one we used that $(\mu_s)_{s \in [0,1]}$ is a constant speed geodesic.  
\end{proof}

\part{Bakry-\'Emery condition and nonlinear diffusion}

\section{The Bakry-\'Emery condition}
\label{sec:BE}
 
In this section we will recall the basic assumptions related to the
Bakry-\'Emery condition
and we will prove 
some important properties related to them. 
In the case of a locally compact space
we will also establish a useful
local criterium to check this condition.

\subsection{The Bakry-\'Emery condition 
for local Dirichlet forms and interpolation estimates}
\label{subsec:BE-Dirichlet}
 The natural setting is provided
by a Polish topological space $(X,\tau)$ endowed with
a $\sigma$-finite reference Borel measure $\mm$ and 
a strongly local symmetric Dirichlet form $\cE$ in $L^2(X,\mm)$
enjoying a \emph{Carr\'e du Champ} $\Gamma:D(\cE)\times D(\cE)\to L^1(X,\mm)$ and a
$\Gamma$-calculus 
(see e.g.\ \cite[\S~2]{AGS12}).
All the estimates we are discussing in this 
section and in the next one, Section~\ref{sec:NDaction}, devoted to
action estimates for nonlinear diffusion equations
do not really need
an underlying compatible metric structure, as the one discussed in 
\cite[\S~3]{AGS12}. 
We refer to \cite[\S 2]{AGS12} for the basic notation and assumptions;
in any case, we will apply all the results to the case of the Cheeger
energy (thus assumed to be quadratic) 
of the metric measure space $(X,\sfd,\mm)$ and we keep the same
notation 
of the previous Section~\ref{subsec:Cheeger}, just using the calculus properties of 
the Dirichlet form that are related to the $\Gamma$-formalism.

In the following we set $\V_\infty:=\V\cap L^\infty(X,\mm)$,
$\D_\infty:=\D\cap L^\infty(X,\mm)$,
\begin{equation}
  \label{eq:100}
  \begin{cases}
  D_{L^p}(\DeltaE):=\big\{f\in \D{}\cap L^p(X,\mm):\ \DeltaE f\in
  L^p(X,\mm)\big\}\qquad p\in [1,\infty],\\
  D_\V(\DeltaE)=\big\{f\in \D:\ \DeltaE f\in \V\big\},
  \end{cases}
\end{equation}
  endowed with the norms
  \begin{equation}
  \label{eq:103}
  \|f\|_{D_{L^p}}:=\|f\|_\V+\|f-\DeltaE f\|_{L^2\cap L^p(X,\mm)},\quad
  \|f\|^2_{D_\V}:=\|f\|^2_{L^2(X,\mm)}+\|\DeltaE f\|^2_\V,
\end{equation}
and we introduce the multilinear form $\GGamma_2$ given by
\begin{equation}
\begin{aligned}
  \GGamma_2(f,g;\varphi):=&
  \frac 12\int_X \Big(\Gbil fg\, {\DeltaE \varphi}-
  \Gbil f{\DeltaE g}\varphi-\Gbil g{\DeltaE f}\varphi\Big)\,\d\mm
  \qquad
  (f,g,\varphi)\in D(\mathbf\Gamma_2),
\end{aligned}\label{eq:22}
\end{equation}
where $D(\mathbf{\Gamma}_2):=D_\V(\DeltaE)\times D_\V(\DeltaE) \times
  D_{L^\infty}(\DeltaE)$.
When $f=g$ we also set
\begin{equation}
  \label{eq:80} 
  \GGamma_2(f;\varphi):=\GGamma_2(f,f;\varphi)=
  \int_X \Big(\frac 12\Gq f\, {\DeltaE \varphi}-
  \Gbil f{\DeltaE f}\varphi\Big)\,\d\mm,
\end{equation}
so that 
\begin{equation}
  \label{eq:46}
\GGamma_2(f,g;\varphi)=  \frac14\GGamma_2(f+g;\varphi)-\frac 14 \GGamma_2(f-g;\varphi). 
\end{equation}
$\mathbf\Gamma_2$ provides a weak version 
(inspired by \cite{Bakry06,Bakry-Ledoux06}) of 
the Bakry-\'Emery condition \cite{Bakry-Emery84,Bakry92}.
\begin{definition}[Bakry-\'Emery conditions]
  \label{def:BE}
  Let $K\in \R$, $N\in [1,\infty]$.
  We say that the strongly local Dirichlet form $\cE$ satisfies the
  $\BE K N$
  condition, 
  if it admits a Carr\'e du Champ $\Gamma$ and
  for every $(f,\varphi)\in D_\V(\DeltaE)\times D_{L^\infty}(\DeltaE)$
  with $\varphi\ge0$ one has
  \begin{equation}
  \label{eq:9}
  \GGamma_2(f;\varphi)\ge K\int_X \Gq f\,\varphi\,\d\mm+\frac 1N\int_X
  (\DeltaE f)^2\varphi\,\d\mm.
\end{equation}

We say that a metric measure space 
$(X,\sfd,\mm)$ (see \S~\ref{subsec:Cheeger}) satisfies the 
\emph{metric} $\BE KN$ condition if 
the Cheeger energy is quadratic, the associated Dirichlet form
$\cE$ satisfies $\BE KN$, and
\begin{equation}
  \label{eq:69}
\text{any}\quad
f\in \V_\infty
\quad\text{with}\quad
\Gq f\in L^\infty(X,\mm)
\quad
\text{has a $1$-Lipschitz representative}.
\end{equation}
\end{definition}  

\begin{remark}[Pointwise gradient estimates for $\BE K\infty$]
  \upshape
  When $N=\infty$, the inequality \eqref{eq:9} is in fact equivalent (see \cite[Cor.~2.3]{AGS12} for a
  proof in the abstract setup of this section) to either of the following
  pointwise gradient estimates
  \begin{equation}
    \label{eq:77}
    \Gq{\sfP_t f}\le \rme^{-2K t}\,\sfP_t\big(\Gq f\big)\quad
    \text{$\mm$-a.e.\ in $X$, for every }f\in \V,
  \end{equation}
  \begin{equation}
    \label{eq:77bis}
    2\rmI_{2K}(t) \Gq{\sfP_t f}\le \sfP_t{f^2}-\big(\sfP_t f\big)^2
    \quad \mm\text{-a.e.\
      in }X,\quad
    \forevery
    t>0,\ f\in L^2(X,\mm),
  \end{equation}
  where $\rmI_K$ denotes the real function
  \begin{displaymath}
    \rmI_{K}(t):=\int_0^t \rme^{K r}\,\d r=
    \begin{cases}
      {\displaystyle\frac 1K(\rme^{K t}-1)}&\text{if }K\neq 0,\\
      t&\text{if }K=0.
    \end{cases}
  \end{displaymath}
\end{remark}
It will be useful to  have different expressions for $\GGamma_2(f;\varphi)$,
that make sense under weaker condition on $f,\,\varphi$.
Typically their equivalence 
will be proved by 
regularization arguments,
which will be based on the following
approximation result.

\begin{lemma} [Density of $D_\V(\DeltaE)\cap D_{L^\infty}(\DeltaE)$]   \label{le:approximation}
  The vector space $D_\V(\DeltaE)\cap D_{L^\infty}(\DeltaE)$ is dense in $D_\V(\DeltaE)$.
  In addition, if $f\in D_{L^p}(\DeltaE)$, $p\in [1,\infty]$ satisfies
  the uniform bounds $c_0\le f\le c_1$ $\mm$-a.e.~in $X$
  for some real constants $c_0,\,c_1$, then
  we can find an approximating sequence $(f_n)\subset D_\V(\DeltaE)\cap
  D_{L^\infty}(\DeltaE)$ converging to $f$ in
  $D_{L^p}(\DeltaE)$ with $f_n\to f$ in $\V$ and $\DeltaE f_n\to\DeltaE f$ 
  in $L^2\cap L^p$ if $p<\infty$ (in the weak$^*$ sense when $p=\infty$), 
  as $n\to\infty$ and satisfying the same bounds
  $c_0\le f_n\le c_1$ $\mm$-a.e.~in $X$.
\end{lemma}
\begin{proof}
  The proof of the density 
  of $D_\V(\DeltaE)\cap D_{L^\infty}(\DeltaE)$ in 
  $D_\V(\DeltaE)$ 
  has been given in 
  \cite[Lemma 4.2]{AMS13}.
  In order to prove the second approximation result,
  we introduce the mollified heat flow
  \begin{equation}
    \label{eq:55}
    \mathfrak H^\eps f:= \frac 1\eps \int_0^\infty \sfP_r f\,\kappa(r/\eps)\,\d r,
  \end{equation}
  where $\kappa\in \rmC^\infty_c(0,\infty)$ is a nonnegative regularization kernel
  with $\int_0^\infty \kappa(r)\,\d r=1$.

  Setting $f_n:=\mathfrak H^{1/n} f$, 
  since $f\in L^2\cap L^\infty(X,\mm)$
  it is not difficult to check that $f_n
  \in D_{\V}(\DeltaE)\cap D_{L^\infty}(\DeltaE)$. In addition,
  $c_0\le f_n\le c_1$, since the heat flow
  preserves global lower or upper bounds by constants.
  
    We then use the fact $\DeltaE$ is the generator
  of a strongly continuous semigroup in each
  $L^p(X,\mm)$ if $p<\infty$ 
  (and of a weak$^*$-continuous semigroup in $L^\infty(X,\mm)$).
\end{proof}

{\color{black} An immediate corollary of the previous density result is the possibility
to test the condition $\BE KN$ on a better class of test functions.}

\begin{corollary}
  \label{cor:obvious}
  If \eqref{eq:9} holds for every
  $f\in D_\V(\DeltaE)\cap D_{L^\infty}(\DeltaE)$ and 
  every nonnegative $\varphi\in D_{L^\infty}(\DeltaE)$,  
  then the $\BE KN$ condition holds.
\end{corollary}

A first representation of $\GGamma_2$ 
is provided by the following lemma,
whose proof is an easy consequence
of the Lebniz rule for $\Gamma$, see 
\cite[Lemma 4.1]{AMS13}.

\begin{lemma}
  \label{le:Gamma-equivalence1} 
  If $f\in D_\V(\DeltaE)\cap D_{L^\infty}(\DeltaE)$ and $\varphi\in D_{L^\infty}(\DeltaE)$ 
  then
  \begin{equation}
    \label{eq:30}
    \GGamma_2(f;\varphi)=\int_X \Big(\frac12 \Gq f\DeltaE\varphi+
    \DeltaE f\,\Gbil f\varphi+\varphi(\DeltaE f)^2\Big)\,\d\mm.
  \end{equation}
\end{lemma}
Recalling \eqref{eq:46} we also get
\begin{equation}
  \label{eq:49}
  \GGamma_2(f,g;\varphi)=\frac 12 \int_X \Big(\Gbil fg\DeltaE\varphi+
  \DeltaE f\,\Gbil g\varphi+
  \DeltaE g\,\Gbil f\varphi+2\varphi \, \DeltaE f\,\DeltaE g\Big)\,\d\mm.
\end{equation}
Notice that \eqref{eq:30}
makes sense even if $f,\,\varphi\in \D_\infty$, provided
$\Gq f$ and $\Gbil f\varphi$ belong to $L^2(X,\mm)$.
This extra integrability of $\Gamma$ is a general 
consequence of the $\BE K\infty$ condition.

\begin{theorem}[Gradient interpolation, {\cite[Thm.~3.1]{AMS13}}]
  \label{thm:interpolation}
  Assume that $\BE K\infty$ holds,
  let $\lambda\ge K_-$, $p\in \{2,\infty\}$,
  $f\in L^2\cap L^\infty(X,\mm)$ with $\DeltaE f\in
  L^p(X,\mm)$.
  Then $\Gq f\in L^{p}(X,\mm)$ 
  and
  \begin{equation}
    \label{eq:87}
    \big\|\Gq f\|_{L^p(X,\mm)}\le
    c\|f\|_{L^\infty(X,\mm)}\,\|\lambda f-\DeltaE f\|_{L^p(X,\mm)}
  \end{equation}
  for a universal constant $c$ independent of $\lambda,X,\mm,f$
  ($c=\sqrt{2\pi}$ when $p=\infty$).
  
  Moreover, if $f_n\in \D_\infty$ with $\sup_n \|f_n\|_{L^\infty(X,\mm)}<\infty$ and 
  $f_n\to f$ strongly in $\D$, then $\Gq {f_n}\to \Gq f$ 
  and $\Gq {f_n-f}\to0$ strongly in $L^2(X,\mm)$.
\end{theorem}
An important consequence of Theorem~\ref{thm:interpolation}
is that $\D_\infty$ is an algebra, also preserved by left composition with
functions $h\in \rmC^2(\R)$ vanishing at $0$:
this can be easily checked by the formula
\begin{equation}
  \label{eq:96}
  \DeltaE(fg)=f\DeltaE g+g\DeltaE f+2\Gbil fg,\qquad
  \DeltaE(h(f))=h'(f)\DeltaE f+h''(f)\Gq f
\end{equation}
using the fact that $\Gq f,\,\Gbil fg\in L^2(X,\mm)$ whenever
$f,\,g \in \D_\infty$.

Thanks to the improved integrability of $\Gamma$ 
given by Theorem~\ref{thm:interpolation}
and to the previous approximation result, 
we can now extend the domain of $\GGamma_2$
to the whole of $(\D_\infty)^3$.

\begin{corollary}[Extension of $\GGamma_2$]
  \label{cor:weaker-assumption-G2}
  If $\BE K\infty$ holds then $\GGamma_2$ can be extended to a continuous
  multilinear form in $\D_\infty\times \D_\infty\times \D_\infty$
  by \eqref{eq:49} and $\BE KN$ holds if and only if 
  \begin{equation}\label{eq:9999}
  \int_X \Big(\frac12 \Gq {f}\DeltaE\varphi+
    \DeltaE f\,\Gbil {f}{\varphi}+ (1-\frac 1N)\varphi(\DeltaE f)^2\Big)\,\d\mm
    \ge K\int_X \Gq {f}\,\varphi\,\d\mm.
  \end{equation}
  is satisfied by every choice of $f,\,\varphi\in \D_\infty$ with
  $\varphi\ge0$.
\end{corollary}
\begin{proof}
  Notice that \eqref{eq:49} makes sense if $f,\,g,\,\varphi\in \D_\infty$ 
  since $\Gq f,\,\Gq g,\,\Gq\varphi\in L^2(X,\mm)$ by Theorem~\ref{thm:interpolation}
  and that it provides an extension of $\GGamma_2$ by Lemma~\ref{le:Gamma-equivalence1}.
  
  In order to check \eqref{eq:9999} under the $\BE K N$ assumption
  whenever $f,\,\varphi\in \D_\infty$, $\varphi\ge0$, we first approximate  $f,\varphi$ in $\D_\infty$ with elements in $D_{\V}(\DeltaE)$ via the
  Heat flow,  and then we apply Lemma~\ref{le:approximation} with a diagonal argument to
  find $f_n,\,\varphi_n\in D_{\V}(\DeltaE)\cap D_{L^\infty}(\DeltaE)$ with $\varphi_n\ge 0$ such that 
  \eqref{eq:9} and \eqref{eq:30} yield
  \begin{displaymath}
    \int_X \Big(\frac12 \Gq {f_n}\DeltaE\varphi_n+
    \DeltaE f_n\,\Gbil {f_n}{\varphi_n}+ (1-\frac 1N)\varphi_n(\DeltaE f_n)^2\Big)\,\d\mm
    \ge K\int_X \Gq {f_n}\,\varphi_n\,\d\mm.
  \end{displaymath}
  \AAA Since, up to to subsequences, we can assume   
  \begin{align*}
   & f_n\to f,\ \varphi_n\to \varphi\quad\text{strongly in }\D \text{ and } \mm \text{-a.e.}, \quad  \|\varphi_n\|_{L^\infty(X,\mm)}\le \|\varphi\|_{L^\infty(X,\mm)}
\\ 
   & \|f_n\|_{L^\infty(X,\mm)}\le \|f\|_{L^\infty(X,\mm)},\quad  |\DeltaE f_n|	\leq g  \; \mm\text{-a.e., for some $g \in L^{2}(X,\mm)$  independent of $n$}    \end{align*}
   \fn
  we can apply the estimates stated in 
  Theorem~\ref{thm:interpolation}
  to pass to the limit in the previous inequality as $n\to\infty$.
  
  Conversely, if \eqref{eq:9999} holds for 
  every $f,\,\varphi\in \D_\infty$ with $\varphi\ge0$,
  it clearly holds for every $f\in D_\V(\DeltaE)\cap D_{L^\infty}(\DeltaE)$ 
  and nonnegative $\varphi\in D_{L^\infty}(\DeltaE)$,
  thus with the expression of $\GGamma_2$ given by \eqref{eq:80},
  thanks to Lemma~\ref{le:Gamma-equivalence1}.
  We can then apply Corollary~\ref{cor:obvious}.
\end{proof}

\subsection{Local and ``nonlinear'' characterization of 
the metric $\BE KN$ condition in locally compact spaces}
\label{subsec:localBE}

When $(X,\sfd,\mm)$ is a locally compact space
satisfying the metric $\BE K\infty$ condition,
the $\GGamma_2$ form enjoys 
a few localization properties, that will turn to be useful in the
following.

Let us first recall 
that if $(X,\sfd,\mm)$ satisfies the metric $\BE K\infty$ condition, then
$(X,\sfd)$ is a length space and
the Dirichlet form $\cE$ associated to the Cheeger energy
is quasi-regular
\cite[Thm. 4.1]{Savare12}.

In the locally compact case, the length condition also yields that
$(X,\sfd)$ is  proper 
(i.e.~every closed bounded subset of $X$
is compact) and thus geodesic 
(every couple of points can be joined by a minimal
geodesic), see, e.g.,
\cite[Prop. 2.5.22]{Burago-Burago-Ivanov01}.

A further important property
(see e.g.~\cite[Remark 6.3]{AMS13})
is that $\cE$
is regular,
i.e.~$\V\cap \rmC_c(X)$ is dense both in 
$\V$ (w.r.t.~the $\V$ norm) and 
in $\rmC_c(X)$ (w.r.t.~the uniform norm).
In particular, by Fukushima's theory 
(see e.g.~\cite{Chen-Fukushima12,Bouleau-Hirsch91}),
every $\varphi\in \V$ admits 
a $\cE$-quasi continuous representative $\tilde\varphi$ 
uniquely determined up to $\cE$-polar sets 
and every linear functional $\ell:\V\to\R$
which is nonnegative (i.e.~such that $\la\ell,\varphi\ra\ge0$ for
every nonnegative $\varphi\in \V$) 
can be uniquely represented by a $\sigma$-finite Borel measure
$\mu_\ell$ which does not charge $\cE$-polar sets, so that $\la
\ell,\varphi\ra=\int_X \tilde\varphi\,\d\mu_\ell$
for every $\varphi\in\V$. We refer to \cite[Sect.~5]{AMS13} 
for more details. We will often identify $\varphi$ with
$\tilde\varphi$,
when there is no risk of confusion. 

Before stating our locality results, we recall two useful facts, obtained
in \cite{AMS13} and slightly improving earlier results in \cite{Savare12}. See Corollary 5.7
for statement (i), and Lemma 6.7 of \cite{AMS13} for statement (ii) (more precisely, the statement of \cite[Lemma 6.7]{AMS13} 
deals with a  Lipschitz cut off function $\chi$ with 
$\DeltaE \nchi\in L^\infty(X,\mm)$ and $\Gq \nchi\in \V_\infty$, but
  since
$\nchi$ is built of the form $\eta\circ f$ with $\eta$ constant near 0, from Lemma \ref{le:pre-local}(i) below and
\eqref{eq:96} one can get also $\DeltaE\nchi\in\V_\infty$.) 

\begin{lemma}
  \label{le:pre-local}
  Let us suppose that $(X,\sfd,\mm)$ satisfies
  the \emph{metric} $\BE K\infty$ condition for some $K\in \R$.
  \begin{itemize}
  \item[(i)] For every $f,\,g\in D_{L^4}(\DeltaE)$ we have
    $\Gbil fg\in \V$ and the bounded 
    linear functional
     \begin{equation}
    \label{eq:71}
    \V\ni\varphi\mapsto 
    \int_X \Big(-\frac 12 \Gbil{\Gq f}\varphi
    +\DeltaE f \, \Gbil f\varphi+\big((\DeltaE f)^2-K\Gq f\big)\varphi\Big)
    \,\d\mm
  \end{equation}
  can be represented by a finite nonnegative Borel measure
  denoted by $\Gamma_{2,K}^*[f]$, satisfying 
    \begin{equation}
    \label{eq:104}
    \GGamma_2(f;\varphi)-K\int_X 
    \Gq f\varphi\,\d\mm=\int_X \varphi\,\d\Gamma_{2,K}^*[f]
  \end{equation}
  for every $f\in D_{L^4}(\DeltaE)\cap L^\infty(X,\mm)$
  and $\varphi\in \D_\infty$, where in \eqref{eq:104} we use the extension of $\GGamma_2(f;\varphi)$ provided by
  Corollary~\ref{cor:weaker-assumption-G2}.
  \item[(ii)] If $(X,\sfd)$ is locally compact,
  then for every compact set $E$ and every 
    open neighborhood $U\supset E$ there exists 
    a Lipschitz cutoff function 
    $\nchi:X\to[0,1]$ such that 
    $\supp(\nchi)\subset U$, 
    $\nchi\equiv 1$ in a neighborhood of $E$,
    $\DeltaE\nchi\in \V_\infty$ and
    $\Gq \nchi\in \V_\infty$.
  \end{itemize}
  \end{lemma}

\begin{corollary}[Locality w.r.t.~$\varphi$]
  \label{cor:loc-phi}
  Let $K\in\R$ and $N<\infty$. Let us suppose that $(X,\sfd)$ is locally compact
  and that $(X,\sfd,\mm)$ satisfies the \emph{metric} $\BE K\infty$ condition.
  If  \eqref{eq:9} holds for every 
  $f\in D_\V(\DeltaE)\cap D_{L^\infty}(\DeltaE)$ 
  and every nonnegative $\varphi\in D_{L^\infty}(\DeltaE)$ with compact support,
  then $(X,\sfd,\mm)$ satisfies the metric
  $\BE KN$ condition.
\end{corollary}
\begin{proof}
  We argue by contradiction: if $\BE KN$ does not hold,
  by Corollary~\ref{cor:obvious} 
  we can find $f\in D_\V(\DeltaE)\cap D_{L^\infty}(\DeltaE)$ 
  and a nonnegative $\varphi\in D_{L^\infty}(\DeltaE)$ 
  such that 
  \begin{displaymath}
    \GGamma_2(f;\varphi)-K\int_X \Gq f\varphi\,\d\mm
    -\frac 1N \int_X (\DeltaE f)^2\varphi\,\d\mm<0.
  \end{displaymath}
  Since  $D_\V(\DeltaE)\cap D_{L^\infty}(\DeltaE)\subset
  D_{L^4}(\DeltaE)\cap L^\infty(X,\mm)$ 
  we can apply the representation result \eqref{eq:104},
  thus obtaining that the measure
  \begin{displaymath}
    \mu:={\varphi}\Gamma_{2,K}^*[f]-\frac {\varphi}N(\DeltaE f)^2\mm
  \end{displaymath}
  has a nontrivial negative part. Since $X$ is Polish, we can find
  a compact set $E$ such that $\mu(E)< 0$;
  approximating $E$ by a sequence of 
  open set $U_n\downarrow E$, Lemma~\ref{le:pre-local}(ii) provides
  a corresponding sequence of 
  nonnegative test functions $\nchi_n\in D_{L^\infty}(\DeltaE)$ such
  that 
  \begin{displaymath}
    \lim_{n\to\infty}\int_X \nchi_n\,\d\mu=\mu(E)<0.
  \end{displaymath}
  Choosing $n$ sufficiently large, since ${\varphi}\nchi_n$ 
  has compact support and belongs to $D_{L^\infty}(\DeltaE)$, 
  this contradicts the assumptions of the Corollary.  
\end{proof}

\begin{theorem}[Local characterization of $\BE KN$]
  \label{thm:localization}
  Let us suppose that $(X,\sfd,\mm)$ satisfies
  the \emph{metric} $\BE K\infty$ condition for some $K\in \R$,
  and that $(X,\sfd)$ is locally compact.
  If \eqref{eq:9} with $N<\infty$ holds for every 
   $f\in D_{L^\infty}(\DeltaE) \cap D_{\V}(\DeltaE)$ with compact support and
  for every nonnegative $\varphi \in D_{L^\infty}(\DeltaE)$
  with compact support and 
  with $\inf_{\supp f}\varphi>0$,
  then $(X,\sfd,\mm)$ satisfies the 
  metric $\BE KN$ condition.
\end{theorem}
\begin{proof}
  By the previous Corollary, we have to check that \eqref{eq:9} holds if $f\in D_{\V}(\DeltaE)\cap D_{L^\infty}(\DeltaE)$ 
  and $\varphi\in D_{L^\infty}(\DeltaE)$ nonnegative with
  compact support. Choosing a cutoff function $\nchi\in D_{L^\infty}(\DeltaE)\cap D_{\V}(\DeltaE)$ with compact support, values in
  $[0,1]$ 
  and such that $\nchi\equiv 1$ on a neighborhood of 
  $\supp(\varphi)$ as in Lemma~\ref{le:pre-local}(ii),
  it is easy to check, using Theorem~\ref{thm:interpolation}, the locality properties
  of $\Gamma$, $\DeltaE$ as well as the computation rules
  \begin{displaymath}
    \nchi f\in \D_\infty,\quad
    \DeltaE(\nchi f)=\nchi\DeltaE f+2\Gbil\nchi f+f\DeltaE \nchi,\quad
    \DeltaE(\nchi f)=\nchi\DeltaE f\quad\text{on }\supp(\varphi),
  \end{displaymath}
  that $\nchi f\in D_{L^\infty}(\DeltaE)\cap D_{\V}(\DeltaE) \fn \subset D_{L^4}(\DeltaE)\cap
  L^\infty(X,\mm)$ and that
   \begin{displaymath}
     \begin{aligned}
       &\GGamma_2(f;\varphi)-K\int_X \Gamma(f)\varphi\,\d\mm=
       \GGamma_2(f;\nchi\varphi)-K\int_X \Gamma(f)\nchi\varphi\,\d\mm=
      \\&=
      \int_X \Big(-{\frac 12}\Gbil{\Gq f}{\nchi \varphi} +\DeltaE f \, \Gbil
      f{\nchi\varphi}+{(\DeltaE f)^2\nchi\varphi}
      -
      K\Gq f
      \nchi\varphi\Big)
      \,\d\mm
      \\&=
      \int_X \Big(-{\frac 12}\Gbil{\Gq {\nchi f}}{\varphi} +\DeltaE (\nchi f)\, \Gbil
      {\nchi f}{\varphi}+{(\DeltaE (\nchi f))^2\varphi}
      -K\Gq {\nchi f}
      \varphi\Big)
      \,\d\mm
      \\&=
      \GGamma_{2,K}^*(\nchi f;\varphi)=
      \lim_{\eps \down0} \GGamma_{2,K}^*(\nchi f;\psi_\eps),
    \end{aligned}
  \end{displaymath}
   where $\psi_\eps=\varphi+\eps\hat\nchi$ and
   $\hat\nchi\in D_{L^\infty}(\DeltaE)$ is another nonnegative 
   cutoff function
   with compact support 
   such that $\hat\nchi\equiv 1$ in an open neighborhood of 
  $\supp(\nchi f)$.
  Since by assumption $\GGamma_{2,K}^*(\nchi f;\psi_\eps)\geq 0$
  we conclude.
\end{proof}

\begin{theorem}[A nonlinear version of the $\BE KN$ condition]
  \label{thm:nonlinearBE}
   If the $\BE KN$ condition holds 
  and  $P\in\mathrm{DC}(N)$ is regular
  with $R(r)=rP'(r)-P(r)$, then
  for every $f\in \D_\infty$
  and every nonnegative function $\varphi\in \V_\infty$
  with $P(\varphi)\in \D_\infty$ we have
  \begin{equation}
    \label{eq:164}
    \GGamma_2(f;P(\varphi))
    +\int_X R(\varphi) \,(\DeltaE f)^2\,\d\mm
    \ge K\int_X \Gq f\,P(\varphi)\,\d\mm.
  \end{equation}
  Conversely, let us assume that
  $(X,\sfd,\mm)$ is locally compact and satisfies the 
  metric $\BE K\infty$-condition.
  If \eqref{eq:164} holds for every
  function $P=P_{N,\eps,M}$, $\eps,\,M>0$ as in \eqref{eq:207}
  and \eqref{eq:167},
  every 
  $f\in D_\V(\DeltaE)\cap D_{L^\infty}(\DeltaE)$ with 
  compact support and 
  every nonnegative $\varphi\in D_{L^\infty}(\DeltaE)$
  with compact support and $\inf_{\supp f}\varphi>0$, 
  then $(X,\sfd,\mm)$ satisfies the metric 
  $\BE KN$ condition.  
\end{theorem}

\begin{proof} The inequality
  \eqref{eq:164} is an obvious consequence of 
  $\BE KN$ (in the form of Corollary \ref{cor:weaker-assumption-G2})
  since $P\in\mathrm{DC}(N)$ yields $R(r)\ge -\frac 1N P(r)$ .
  
  In order to prove the second part of the statement, we apply 
  the previous Theorem~\ref{thm:localization}:
  we fix $f\in D_\V(\DeltaE)\cap D_{L^\infty}(\DeltaE)$   and $\varphi\in D_{L^\infty}(\DeltaE)$ nonnegative, 
  both with compact support and satisfying $\inf\{ \varphi(x):x\in \supp(f)\}>0$; with this
  choice of $f$ and $\varphi$ we need to prove \eqref{eq:9}.
  
  We fix $\eps>0$ and 
  we set $\tilde\varphi=P_{N,\eps}^{-1}(\varphi)$;
  since $\varphi$ is bounded, 
  $\tilde\varphi\in D_{L^\infty}(\DeltaE)$ 
  and therefore we can choose $M>0$ sufficiently large
  such that $\tilde\varphi\leq M$ and consequently
  $\varphi=P_{N,\eps,M}(\tilde\varphi)$.
  Applying \eqref{eq:164} with this choice of $f$ and $\tilde\varphi$ 
  and recalling the inequality \eqref{eq:92} we get
  \begin{align*}
    \GGamma_2(f;\varphi)
    -\frac 1N\int_X \varphi \,(\DeltaE f)^2\,\d\mm
    + \left(1-\frac{1}{N} \right)\eps^{1-1/N}\int_X (\DeltaE f)^2\,\d\mm
    \ge K\int_X \Gq f\,\varphi\,\d\mm.
  \end{align*}
  Passing to the limit as $\eps\down0$
  we get \eqref{eq:9}.
\end{proof}
\nc

\section{Nonlinear diffusion equations and action estimates}
\label{sec:NDaction}

In this section we give a rigorous proof of the crucial estimate
we briefly discussed in the formal calculations of Example \ref{ex:2}.
The estimate requires extra continuity and summability properties 
on $\Gq \varrho$ and $\Gq\varphi$, that 
will be provided by the interpolation estimates of Theorem~\ref{thm:interpolation}.

We will assume that $P$ is regular according to \eqref{eq:A1bis}, 
we introduce the functions $R(z)=zP'(z)-P(z)$ and $Q(r):=P(r)/r$,
and we recall the definition of the $\Gamma_2$ multilinear form
\begin{align*}
  \mathbf\Gamma_2(\vph;\varrho)=
  \int_X \Big(\frac12 \DeltaE \varrho\, \Gq \vph\,\d\mm+
  \varrho \big(\DeltaE\vph\big)^2+
  \Gbil \varrho\vph \DeltaE\vph\Big)\,\d\mm
\end{align*}
whenever $\varrho,\,\vph\in \D_\infty$ with $\Gq \varrho,\,\Gq\vph\in\H$.
Recall that, under the $\BE K\infty$ assumption, $f\in\D_\infty$ implies  $\Gq f\in \H$.
Notice also that $P(\varrho)\in  L^2(0,T;\D)$ and $\varrho$ bounded imply $\Gq {P(\varrho)}\in L^2(0,T;\H)$, so that
the regularity of $P$ and the chain rule yield $\Gq\varrho\in L^2(0,T;\H)$.

\begin{theorem}[Derivative of the Hamiltonian]
  \label{thm:crucial-estimate}
 \AAA Assume that $\BE K\infty$ holds. \fn Let  
  $\varrho\in \ND0T,\,\vph\in W^{1,2}(0,T;\D,\H)$ be bounded solutions, respectively, of 
  \begin{subequations}
    \begin{align}
      \partial_t \varrho - \DeltaE P (\varrho)=0
      \label{eq:PM} \\
      \partial_t \varphi +P'(\varrho)\DeltaE\varphi=0\label{eq:PMb}
    \end{align}
  \end{subequations}
with $\Gq \varrho,\,\Gq \varphi\in L^2(0,T;\H)$. Then the map 
  $t\mapsto \cE_{\varrho_t}(\vph_t)=\int_X \rh_t\Gq{\vph_t}\,\d\mm$ 
  is absolutely continuous in $[0,T]$ and we have
  \begin{equation}
    \label{eq:58}
    \frac\d{\d t}\frac 12\int_X \rh_t\Gq{\vph_t}\,\d\mm
    =\mathbf\Gamma_2(\vph_t;P(\rh_t))+
    \int_X R(\rh_t)(\DeltaE\vph_t)^2\,\d\mm\quad
    \text{$\Leb{1}$-a.e.\ in }(0,T).
  \end{equation}
\end{theorem}
The \emph{proof} is based on the following Lemma:

\begin{lemma}   \AAA Assume that $\BE K\infty$ holds. \fn   Let 
  $\varrho\in \ND0T,\,\vph\in W^{1,2}(0,T;\D,\H)$ be bounded with
  $\Gq \varrho,\Gq \varphi\in L^2(0,T; \H)$.
  Then, for every $\eta\in \rmC^\infty_\rmc(0,T)$ we have
  \begin{equation}
    \label{eq:88}
    \frac 12  \int_0^T\int_X \frac \d\dt(\varrho_t\eta_t)
    \Gq{\vph_t}\,\d\mm\,\d t=
    \int_0^T \eta_t \int_X \Big(\varrho_t \DeltaE\vph_t\,\frac \d\dt\vph_t+
    \Gbil{\varrho_t}{\vph_t}\,\frac \d\dt\vph_t\Big)\,\d\mm\,\d t.
  \end{equation}
\end{lemma}
\begin{proof}
  Let us consider the functions
  $\vph^\eps_t:=\eps^{-1}\int_0^\eps \vph_{t+r}\,\d r$:
  $t\mapsto \vph^\eps_t$ are differentiable in $\V$
  with
  $\frac \d\dt \vph^\eps_t=\eps^{-1}(\vph_{t+\eps}-\vph_t)$,
  so that 
  \begin{align*}
    \frac 12  \int_X \varrho_t \frac \d\dt\Big( \Gq{\vph^\eps_t}\Big)\,\d\mm&=
     \int_X \varrho_t \Gbil{\frac \d\dt\vph^\eps_t}{\vph^\eps_t}\,\d\mm=
     -\int_X \varrho_t \DeltaE\vph^\eps_t\frac \d\dt\vph^\eps_t\,\d\mm-
     \int_X \Gbil{\varrho_t}{\vph^\eps_t}\,\frac \d\dt\vph^\eps_t\,\d\mm.
  \end{align*}
  For every $\eta\in \rmC^\infty_\rmc(0,T)$ we thus have
  \begin{align*}
    \frac 12  \int_0^T\int_X \frac \d\dt(\varrho_t
    \eta_t)
    \Gq{\vph^\eps_t}\,\d\mm\,\d t&=
    \int_0^T \eta_t \Big(\int_X \varrho_t \DeltaE\vph^\eps_t\,\frac \d\dt\vph^\eps_t\,\d\mm+
    \int_X \Gbil{\varrho_t}{\vph^\eps_t}\,\frac \d\dt\vph^\eps_t\,\d\mm\Big)\,\d t.
  \end{align*}
  In order to pass to the limit as $\eps\down0$ in the last identity,
  we observe that $\frac \d\dt\vph^\eps_t\to \frac \d\dt\vph_t$ and that
  $\DeltaE\vph^\eps\to \DeltaE\vph$ strongly in $L^2(0,T;\H)$.
  Moreover, it is easy to check that the convexity 
  of $\zeta\mapsto \sqrt{\Gq \zeta}$ yields
  \begin{equation}
    \label{eq:90}
    \Gq {\vph_t^\eps}\le \frac 1\eps \int_0^\eps \Gq{\vph_{t+r}}\,\d r
    \quad\text{$\mm$-a.e.\ in $X$,}\quad \forevery t\in [0,T-\eps],
  \end{equation}
  so that the convolution inequality $\int_0^{T-\eps}\eps^{-1}\int_0^\eps\psi(t+r)\,\d r\,\dt\leq\int_0^T\psi(t)\,\dt$, for $\psi\geq 0$, gives
    \begin{equation}
    \label{eq:89}
    \int_0^{T-\eps}\int_X \big(\Gq {\vph_t^\eps}\big)^2\,\d\mm\,\d t\le 
    \frac 1\eps\int_0^{T-\eps} \int_0^\eps\int_X \big(\Gq
    {\vph_{t+r}}\big)^2\,\d\mm\,\d r\,\d t
    \le \int_0^{T}\int_X \big(\Gq {\vph_t}\big)^2\,\d\mm\,\d t.
  \end{equation}
  Since ${\vph_t^\eps}\to\vph_t$ strongly in $\V$  as $\eps\down0$, 
  we have $\Gq{\vph_t^\eps}\to \Gq{\vph_t}$ pointwise in $L^1(X,\mm)$, hence
  {\color{black}
  $$
  \liminf_{\eps\down 0}  \int_0^{T-\eps}\int_X \big(\Gq {\vph_t^\eps}\big)^2\,\d\mm\,\d t\ge
   \int_0^T\int_X \big(\Gq {\vph_t}\big)^2\,\d\mm\,\d t.  
  $$
  This, combined with \eqref{eq:89}, yields the strong convergence of 
  $\Gq{\vph^\eps}\chi_{(0,T-\eps)}$ to $\Gq\vph$} in $L^2(0,T;\H)$.
  The above mentioned convergences are then sufficient to get
  \eqref{eq:88}.
\end{proof}

\begin{proof}[Proof of Theorem~\ref{thm:crucial-estimate}]
  The map $t\mapsto \cE_{\varrho_t}(\varphi_t)$ is continuous since
  $t\mapsto \varrho_t$ is weakly$^*$ continuous in $L^\infty(X,\mm)$ and
  $t\mapsto \Gq{\vph_t}$ is strongly continuous in $L^1(X,\mm)$ (thanks to
  Theorem~\ref{thm:interpolation}).
  For every $\eta\in \rmC^\infty_\rmc(0,T)$,
  using the differentiability of $\varrho$ in $L^2(0,T;\H)$ we have
  \begin{align*}
    -\int_0^T \cE_{\varrho_t}(\vph_t)  \frac \d\dt\eta_t\,\d t&=
    -\frac 12\int_0^T \int_X \varrho_t \Gq {\vph_t}\,\d\mm\,\frac \d\dt\eta_t \d t\\&=
    -\frac 12\int_0^T \int_X \frac\d\dt (\varrho_t \eta_t) \Gq
    {\vph_t}\,\d\mm\,\d t
    +\frac 12\int_0^T \int_X (\frac\d\dt \varrho_t) \eta_t \Gq
    {\vph_t}\,\d\mm\,\d t\\&=
    -\frac 12\int_0^T \int_X \frac \d\dt (\varrho_t \eta_t) \Gq
    {\vph_t}\,\d\mm\,\d t
    +\frac 12\int_0^T \eta_t \Big(\int_X \DeltaE P(\varrho_t) \Gq
    {\vph_t}\,\d\mm\Big)\,\d t.
  \end{align*}
  On the other hand, \eqref{eq:88} yields
  \begin{align*}
    -\frac 12&\int_0^T \int_X \frac \d\dt (\varrho_t \eta_t) \Gq
    {\vph_t}\,\d\mm\,\d t=
    \int_0^T \eta_t \int_X \Big(\varrho_t P'(\varrho_t)\big(\DeltaE\vph_t\big)^2
    +\Gbil{\varrho_t}{\vph_t}\,P'(\varrho_t)\DeltaE\vph_t\Big)\,\d\mm\,\d t
    \\&=
    \int_0^T \eta_t \int_X \Big(P (\varrho_t)\big(\DeltaE\vph_t\big)^2
    +\Gbil{P(\varrho_t)}{\vph_t}\,\DeltaE\vph_t\Big)\,\d\mm\,\d t
    +\int_0^T \eta_t \int_X R(\varrho_t)\big(
    \DeltaE\vph_t\big)^2\,\d\mm\,\d t.
    \end{align*}
    Combining the two formulas, we get \eqref{eq:58}.
\end{proof}

\begin{theorem}[Action and dual action monotonicity]
  \label{thm:BEKNaction}
  Let us assume that the 
  $\BE KN$ condition holds, and that $P\in\MCreg N$.  
    \begin{enumerate}[\rm (i)]
  \item If $\varrho\in \ND0T,\vph\in W^{1,2}(0,T;\D,\H)$ 
    are bounded solutions of
    {\upshape (\ref{eq:PM},b)} then the map $t\mapsto \int_X \varrho_t
    \Gq{\vph_t}\,\d\mm$ is absolutely continuous in $[0,T]$ and we have
    \begin{equation}
      \label{eq:91}
      \frac\d{\d t}\frac 12\int_X \rh_t\Gq{\vph_t}\,\d\mm\ge
      K \int_X P(\varrho_t)\Gq{\vph_t}\,\d\mm
      \qquad
      \text{$\Leb 1$-a.e.\ in }(0,T).
    \end{equation}
  \item Setting
      \begin{equation}
        \label{eq:56}
        \Lambda:=\inf_{r>0} K Q(r)>-\infty,
      \end{equation}
      if $w\in W^{1,2}(0,T;\H,\Ddual)$ 
      is a solution of \eqref{eq:81} with $\bar w\in \Vdual{\bar\varrho}\subset\Vdual1$, then
      $w_t\in \Vdual{\varrho_t}$ for all $t\in [0,T]$, with
      \begin{equation}
        \label{eq:57}
        \cE_{\varrho_s}^*(w_s,w_s)\le \rme^{-2\Lambda (s-t)}\cE_{\varrho_t}^*(w_t,w_t)\quad
        \text{for every }0\le t<s\le T.
      \end{equation}
    \item If moreover  $\phi_t=-A^*_{\varrho_t}(w_t)\in \Vhom{\varrho_t}$ is the 
      potential associated to $w_t$ according to \eqref{eq:112}, i.e.
      \begin{equation}
        \label{eq:98}
        \cE_{\varrho_t}(\phi_t,\zeta)=\langle w_t,\zeta\rangle\quad
        \text{for every }\zeta\in \Vhom{\varrho_t},
      \end{equation}
      we have
      \begin{equation}
        \label{eq:101}
        \limsup_{h\downarrow0}\frac 1{2h}\Big(\cE_{\varrho_t}^*(w_t,w_t)-
        \cE_{\varrho_{t-h}}^*(w_{t-h},w_{t-h})\Big)\le -K\int_X Q(\varrho_t)\varrho_t
        \tGq {\varrho_t}{\phi_t}\,\d\mm.
      \end{equation}
  \end{enumerate}
\end{theorem}
\begin{proof}
  Since $\BE KN$ holds (and thus in particular $\BE K\infty$)
  we can apply \eqref{eq:58} of Theorem~\ref{thm:crucial-estimate},
  since the interpolation estimate \eqref{eq:87} and the regularity 
  properties of $\varrho$ and $\varphi$ yield $\Gq \varrho,\Gq \varphi\in
  L^2(0,T;\H)$. The estimate \eqref{eq:91} follows then by the combination of \eqref{eq:58} with  Theorem  \ref{thm:nonlinearBE}. 

  For $0\le t<s\le T$, let us now call
  $B_{s,t}:\V_\infty\to\V_\infty$ the linear map 
  that to each function $\bar\varphi\in \V_\infty$ associates the value at
  time $t$ of the unique solution $\varphi$ of (\ref{eq:PMb}) with final condition
  $\varphi_s=\bar\varphi$, given by Theorem~\ref{prop:backward-linearization}. 
  If \eqref{eq:91} holds and $\Lambda$ is defined as in \eqref{eq:56} we have
  \begin{equation}
    \label{eq:84}
    \int_X \varrho_t\Gq{\varphi_t}\,\d\mm\le 
    \rme^{-2\Lambda (s-t)}    \int_X \varrho_s\Gq{\bar\varphi}\,\d\mm,
  \end{equation}
  so that 
  \begin{eqnarray*}
    &&\langle \rme^{\Lambda(s-t)}\,w_s,\rme^{-\Lambda(s-t)}\bar\varphi\rangle -\frac
    12\cE_{\varrho_s}(\rme^{-\Lambda(s-t)} \bar\varphi,\rme^{-\Lambda(s-t)}\bar\varphi)
    \\&=&
    \langle w_s,\bar\varphi\rangle -\rme^{-2\Lambda(s-t)}\,\frac
    12\cE_{\varrho_s}(\bar\varphi,\bar\varphi)\topref{eq:130}=
    \langle w_t,B_{s,t}\bar\varphi\rangle -\rme^{-2\Lambda(s-t)}\frac
    12\cE_{\varrho_s}(\bar\varphi,\bar\varphi)\\&\topref{eq:84}\le& 
    \langle w_t,B_{s,t}\bar\varphi\rangle- \frac
    12\cE_{\varrho_t}(B_{s,t}\bar \varphi,B_{s,t}\bar \varphi)
    \topref{eq:97}\le\frac 12\cE_{\varrho_t}^*(w_t,w_t).
  \end{eqnarray*}
  Taking the supremum with respect to $\bar\varphi\in\V_\infty$ 
  we get \eqref{eq:57}.

  Similarly, we can choose a maximizing sequence $(\varphi_n)\subset\V_\infty$ in
  $$
  \frac 12 \cE_{\varrho_t}^*(w_t,w_t)=\sup_{\varphi\in\V_\infty}  \langle w_t,\varphi\rangle -\frac
    12\cE_{\varrho_t}(\varphi,\varphi),
  $$ 
  so that $\varphi_n$ converge in $\Vhom{\varrho_t}$ to the potential
  $\phi_t=-A_{\varrho_t}^*(w_t)$. Recalling \eqref{eq:91} and \eqref{eq:130} we have
  \begin{align*}
    \langle w_t,\varphi_n\rangle -\frac
    12\cE_{\varrho_t}(\varphi_n,\varphi_n)
    &\le 
    \langle w_{t-h},B_{t,t-h}\varphi_n\rangle -\frac
    12\cE_{\varrho_{t-h}}(B_{t,t-h}\varphi_n, B_{t,t-h}\varphi_n)
    \\&\qquad -
    K\int_{t-h}^t \int_X Q(\varrho_r) \varrho_r
    \Gq{B_{t,r}\varphi_n}\,\d\mm\,\d r .
  \end{align*}
  Passing to the limit as $n\to\infty$ and recalling 
  Lemma~\ref{le:stability} we get
  \begin{align*}
   \frac 12   \cE_{\varrho_t}^*(w_t,w_t)&\le 
    \langle w_{t-h},B_{t,t-h}\phi_t\rangle -\frac
    12\cE_{\varrho_{t-h}}(B_{t,t-h}\phi_t,B_{t,t-h}\phi_t)
    \\&\qquad -
    K\int_{t-h}^t \int_X Q(\varrho_r) \varrho_r
    \tGq{\varrho_r}{B_{t,r}\phi_t}\,\d\mm\,\d r 
    \\&\le \frac 12  \cE_{\varrho_{t-h}}^*(w_{t-h},w_{t-h}) -
    K\int_{t-h}^t \int_X Q(\varrho_r) \varrho_r
    \tGq{\varrho_r}{B_{t,r}\phi_t}\,\d\mm\,\d r.
  \end{align*}
  Dividing by $h$ and passing to the limit as $h\down0$,
  a further application of Lemma~\ref{le:stability} yields \eqref{eq:101}.
\end{proof}

\begin{corollary}\label{cor:era_senza_label}
  Let us assume that the 
  $\BE KN$ holds, and that $P\in \MCreg N$.  
  If $\varrho\in \ND0T$ is a 
  nonnegative bounded solution of \eqref{eq:75} 
  with $\sqrt{\bar\varrho}\in \V$
  then $w_t:=\frac \d\dt\varrho_t$ satisfies
  \begin{equation}
    \label{eq:139}
    \cE^*_{\varrho_t}(w_t,w_t)\le \rme^{-2\Lambda t}\cE^*_{\bar\varrho}(w_0,w_0)
    \le 4\sfa^{-2}\,e^{2\Lambda^-T}\cE(\sqrt{\bar\varrho},\sqrt{\bar\varrho})<\infty
  \end{equation}
  with $\sfa$ given by \eqref{eq:A1}.
\end{corollary}
\begin{proof}
  Since $w_0=\DeltaE P(\bar\varrho)$, we have for every $\varphi\in \V$
  \begin{eqnarray*}
   - \langle w_0,\varphi\rangle&=&
    \cE(P(\bar\varrho),\varphi)=
    \int_X P'(\bar\varrho)\Gbil{\bar\varrho}\varphi\,\d\mm
    \\&=&2\int_X P'(\bar\varrho)\sqrt {\bar\varrho}\,
    \Gbil {\sqrt{\bar\varrho}}{\varphi}\,\d\mm
    \\&\le& 2\sfa^{-2} \cE(\sqrt{\bar\varrho},\sqrt{\bar\varrho})+
    \frac 12 \cE_{\bar\varrho}(\varphi,\varphi),
  \end{eqnarray*}
  which yields $\cE^*_{\bar\varrho}(w_0,w_0)\le 
  4\sfa^{-2} \cE(\sqrt{\bar\varrho},\sqrt{\bar\varrho})$. Since, thanks to Corollary~\ref{cor:rho-derivative}, 
  $w$ solves \eqref{eq:81}, we can apply \eqref{eq:57} to obtain \eqref{eq:139}.
  \end{proof}

\section{The equivalence between $\BE KN$ and $\RCDS KN$} \label{sec:EquivBERCDS}

\subsection{Regular curves and regularized entropies} \label{subsec:RegularCurves}

Let us first recall the notion, adapted from \cite[Def.~4.10]{AGS12}, of regular curve.
{\color{black} Recall that $\cE(\cdot,\cdot)$ stands, in this metric context, for Cheeger's energy, here assumed to be quadratic.}

\begin{definition}[Regular curves]  \label{def:regcur}
  Let $\mu_s=\varrho_s\mm\in \ProbabilitiesTwo X$, $s\in [0,1]$.  
  We say that $\mu$ is a regular curve if:
  \begin{enumerate}[\rm (a)]
  \item There exists a constant $R>0$ 
    such that $\varrho_s\le R$ $\mm$-a.e.~in $X$ for every $s\in [0,1]$.
  \item $\mu\in \mathrm{Lip}([0,1];\ProbabilitiesTwo X)$
    and in particular \eqref{eq:10} {\color{black} and the identification between minimal velocity and metric derivative}
    yield $\varrho\in \Lip([0,1];\Vdual1)$.
  \item $g_s:=\sqrt {\varrho_s}\in \V$ and 
    there exists a constant $E>0$ such that 
    $\cE(g_s,g_s)\le E$ for every $s\in [0,1]$
    (in combination with (a), this 
    yields that $\varrho_s\in\V$ and 
    also $\cE(\varrho_s,\varrho_s)\le 4RE$ are uniformly bounded).
  \end{enumerate}
\end{definition}

The next approximation result is an improvement of \cite[Prop.~4.11]{AGS12}, since we are able to
approximate with curves having uniformly bounded densities (while in the original version only
a uniform bound on entropies was imposed). This improvement is possible thanks to
 \cite[Thm.\ 1.3]{Rajala12}.

\begin{lemma}[Approximation by regular curves]
  \label{le:appregentr}
  Let $(X,\sfd,\mm)$ be an $\RCD K\infty$ space and
  $\mu_0,\,\mu_1\in \ProbabilitiesTwo X$. Then there exist 
  a geodesic $(\mu_s)_{s\in [0,1]}$ connecting $\mu_0$ to $\mu_1$ 
  in $\ProbabilitiesTwo X$ and regular curves
  $(\mu^n_s)_{s\in [0,1]}$ with $\mu^n_s=\varrho_s^n\mm$, $n\in \N$, such that 
  \begin{equation}
    \label{eq:140}
    \lim_{n\to\infty}W_2(\mu^n_s,\mu_s)=0\quad
    \forevery s\in [0,1],\quad 
    \limsup_{n\to\infty}\int_0^1 |\dot\mu^n_s|^2\,\d s\le W_2^2(\mu_0,\mu_1).
  \end{equation} 
  \GGG Moreover, if $\mu_i=\varrho_i\mm$, $i=0,\,1$ then 
  \begin{equation}
    \label{eq:14001}
    \lim_{n\to\infty}\|\varrho^n_i-\varrho_i\|_{L^1(X,\mm)}=0\quad
    \text{for }i=0,1,
  \end{equation}
  and \EEE
  {\color{black} 
  \begin{equation}\label{eq:14000}
  \lim_{n\to\infty}\cU(\mu^n_i)=\cU(\mu_i)
  \qquad \text{for $i=0,1$ and for all $U\in \MCreg N$.}
  \end{equation}}
  \AAA 
  Finally, if $\varrho_{i}$, $i=0,\,1$, are $\mm$-essentially bounded with bounded supports then $\mu_{s}=\varrho_{s} \mm$  for each $s \in [0,1]$ with  $(\varrho_{s})_{s \in [0,1]}$  uniformly $\mm$-essentially bounded with bounded supports, and 
  \begin{equation}\label{eq:approxweak*Linfty}
  \varrho_{s}^{n} \to \varrho_{s} \text{ strongly in } L^{p}(X,\mm) \text{ for all $p\in [1,\infty)$ and in weak $*$-$L^{\infty}(X,\mm)\; $, for all $s \in [0,1]$}. 
  \end{equation}
  \EEE
\end{lemma}
\begin{proof}
  First of all we approximate $\mu_i$, $i=0,\,1$, in $\ProbabilitiesTwo
  X$ by two sequences $\nu_i^n=\sigma_i^n\mm$ with bounded support and 
  bounded densities $\sigma_i^n\in L^\infty(X,\mm)$. 
  \GGG
  Whenever $\mu_0=\varrho_0\mm$ (resp.~$\cU(\mu_0)<\infty$) 
  we can also choose $\nu_0^n$ so that 
  $\sigma_i^n\to \varrho_i$ strongly in $L^1(X,\mm)$ 
  (resp.~$\cU(\nu_i^n)\to \cU(\mu_i)$) as $n\to\infty$.
  \EEE
  Applying 
  \cite[Thm.\ 1.3]{Rajala12} we can find geodesics $(\nu^n_s)_{s\in [0,1]}$ in
  $\ProbabilitiesTwo X$
  connecting $\nu_0^n$ to $\nu_1^n$ 
  with uniformly bounded entropies and densities $\sigma^n_s$ satisfying 
  $\sup_{s\in [0,1]}\|\sigma_s^n\|_{L^\infty(X,\mm)}<\infty$ 
  for every $n\in \N$. 
  \GGG By setting $\tilde \nu^n_s:=\nu^n_{s+s/n}$ if $s\in [0,n/(n+1)]$, 
  $\tilde\nu^n_s:=\nu^n_1$ if $s\in [n/(n+1),1]$,
  we may also assume that $\nu^n$ is constant in a right neighborhood
  of $1$. \EEE
  Since  $\nu^n\in \mathrm{AC}^2(0,1;\ProbabilitiesTwo{X})$, 
  we can then apply the same argument of \cite[Prop.~4.11]{AGS12}
  (precisely, an averaging procedure w.r.t. $s$ and a short time action of
  the {\color{black} heat} semigroup, to gain $\V$ regularity)
  to construct regular curves $\nu^{n,k}=\sigma^{n,k}\mm$, 
  $k\in \N$, 
  in the sense of Definition~\ref{def:regcur} approximating 
  $\nu^n$ in energy and Wasserstein distance as $k\to\infty$. {\color{black} Notice also
  that the construction in \cite[Prop.~4.11]{AGS12} provides the monotonicity property
  {\GGG 
  $\cU(\nu^{n,k}_i)\leq\cU(\nu^n_i)$, $i=0,1$,}
  thanks to the convexity of $\cU$ and to 
  fact that $\cU$ decreases under the action of the heat semigroup, so
  that {\GGG $\|\sigma^{n,k}_i-\sigma^n_i\|_{L^1(X,\mm)}\to0$ and 
    $\cU(\nu^{n,k}_i)\to \cU(\nu^n_i)$ as 
    $k\to\infty$} by the lower semicontinuity of $\cU$.}
  A standard diagonal argument yields a subsequence
  $\mu^n_s:=\nu^{n,k_n}_s$ satisfying the properties stated in the Lemma.
  \AAA
 \\ If the starting measures satisfy $\mu_{i}=\varrho_{i} \mm$ with $\varrho_{i}$ $\mm$-essentially bounded with bounded supports, then by  \cite[Thm.\ 1.3]{Rajala12} there exists a $W_{2}$-geodesic $\mu_{s}=\varrho_{s} \mm$  for each $s \in [0,1]$ with  $(\varrho_{s})_{s \in [0,1]}$  uniformly $\mm$-essentially bounded with bounded supports. Recalling the regularity and continuity properties of the heat semigroup proved in \cite[Thm. 6.1]{AGS11b} (see also \cite{AGMR12}), we obtain that the approximations   $\mu^{n}_{s}$ (defined above by an averaging procedure w.r.t. $s$ and  a short time action of
  the heat semigroup) converge in $L^{1}(X,\mm)$ and are uniformly bounded in $L^{\infty}(X,\mm)$; the claimed convergence  \eqref{eq:approxweak*Linfty} follows.
  \EEE 
\end{proof}

Given $U:[0,\infty)\to\R$ continuous, with $U(0)=U(1)=0$  and $U'$ locally
Lipschitz in $(0,\infty)$, with $P(r)=rU'(r)-U(r)$ regular, we now introduce the 
regularized convex entropies $U_\eps\in\rmC^2([0,\infty))$, $\eps>0$, defined by 
\begin{eqnarray}
  \label{eq:48}
  U_\eps(r)&:=&(r+\eps)\int_0^r \frac{P(s)}{(s+\eps)^2}\,\d s
 -r\int_0^1 \frac{P(s)}{(s+\eps)^2}\,\d s\\&=&
  r \int_1^r \frac{P(s)}{(s+\eps)^2}\,\d s
  +\eps\int_0^r \frac{P(s)}{(s+\eps)^2}\,\d s,\nonumber
\end{eqnarray}
that satisfy (since $P(0)=0$)
\begin{equation}
  \label{eq:61}
  U_\eps(0)=0,\qquad
  U_\eps'(0)=-\int_0^1 \frac{P(s)}{(s+\eps)^2}\,\d s,\qquad
  U_\eps''(r)=\frac{P'(r)}{r+\eps}.
\end{equation}
Notice that,  since $U$ is normalized, for every $R>0$ there exists a constant $C_R$ such that
\begin{equation}
  \label{eq:224}
  \min\{U(r),0\}\leq U_\eps(r)\quad\forall r\in [0,1],\qquad 0\le U_\eps(r)\le C_R r\quad\forall r\in [1,R],
\end{equation}  
moreover one has the convergence property
 \begin{equation}\label{eq:224bis}
  \lim_{\eps\down0}U_\eps(r)=U(r).
\end{equation}
We also set
\begin{equation}
  \label{eq:144}
  Z(r):=\int_0^r \frac{P'(s)}{\sqrt s}\,\d s,\quad
\end{equation}
so that \eqref{eq:A1} gives
\begin{equation}\label{eq:144bis}
  2\sfa\sqrt r\le Z(r)\le 2\sfa^{-1}\sqrt r.
\end{equation}

\begin{lemma}[Derivative of the regularized Entropy]
  \label{le:1bis}
  Let $(\varrho_s)_{s\in [0,1]}$ be uniformly bounded densities in $W^{1,2}(0,1;\V,\Vdual1)$.
  Then the map
  $s\mapsto \int_X U_\eps(\varrho_s)\,\d\mm$ is absolutely continuous in
  $[0,1]$ and
  \begin{equation}
    \label{eq:129}
    \frac \d{\d s}\int_X U_\eps(\varrho_s)\,\d\mm= \duality{\V'}{\V}{\frac\d {\d s} \varrho_s}{U_\eps'(\varrho_s)} 
    \qquad\text{for $\Leb{1}$-a.e. $s\in (0,1)$.} 
  \end{equation}
\end{lemma}
\begin{proof}
  The convexity of $U_\eps$ yields
  \begin{align*}
    \int_X U_\eps(\varrho_s)\,\d \mm-
    \int_X U_\eps(\varrho_r)\,\d\mm&\le 
    \int_X U_\eps'(\varrho_s)(\varrho_s-\varrho_r) \,\d\mm
    \le \cE(U_\eps'(\varrho_s))^{1/2}
    \cE^*(\varrho_s-\varrho_r)^{1/2}
    \\&\le 
    \sup|U_\eps''|\,\cE(\varrho_s)^{1/2}
    \cE^*(\varrho_s-\varrho_r)^{1/2},
  \end{align*}
  so that \eqref{eq:A1} and the last identity in \eqref{eq:61} give
  \begin{equation}
    \label{eq:219}
      \Big|\int_X U_\eps(\varrho_s)\,\d \mm-
      \int_X U_\eps(\varrho_r)\,\d\mm\Big|\le 
      \frac 1{\sfa\eps} \max\Big(\cE(\varrho_s,\varrho_s)^{1/2},\cE(\varrho_r,\varrho_r)^{1/2}\Big)
      \cE^*(\varrho_s-\varrho_r)^{1/2}.
  \end{equation}
  This shows the absolute continuity 
  {\color{black} (see \cite[Lem.~1.2.6]{AGS08}).}
  The derivation of \eqref{eq:129} is then standard.
\end{proof}

\begin{lemma}
  \label{le:2bis}
  Let $\varrho\in \V_\infty$ be nonnegative.
  \begin{enumerate}[\rm (i)]
  \item $\sqrt \varrho\in \V$ if and only if $Z(\varrho)\in \V$ if and only if 
    $\int_{\{\varrho>0\}}\varrho^{-1}\Gq {P(\varrho)}\,\d\mm<\infty$.
    In this case
    \begin{equation}
      \label{eq:145}
      \cE(Z(\varrho),Z(\varrho))=\int_{\{\varrho>0\}}\frac{\Gq {P(\varrho)}}{\varrho}\,\d\mm
      =\lim_{\eps\down0}\int_X \varrho\Gq{U_\eps'(\varrho)}\,\d\mm.
    \end{equation}
    \item If $Z(\varrho)\in \V$ then $\DeltaE P(\varrho)\in \Vdual\varrho$,
    $U_\eps'(\varrho)\to A_\varrho^*(\DeltaE P(\varrho))$ in $\Vhom\varrho$ as $\eps\down0$ and
    \begin{equation}
      \label{eq:143}
      \lim_{\eps\downarrow 0}\int_X \varrho\Gq {U_\eps'(\varrho)}\,\d\mm=
      \int_X \Gq{Z(\varrho)}\,\d\mm=  \cEs{\varrho} (\DeltaE P(\varrho),\DeltaE P(\varrho)).
    \end{equation}
    Motivated by this,
    we will call $U'(\varrho)\in\V_{\varrho}$ the limit $A_\varrho^*(\DeltaE P(\varrho))$ of $U_\eps'(\varrho)$ in $\Vhom\varrho$.
  \item If $\mu_s=\varrho_s\mm$, $s\in [0,1]$, is a regular curve,
    then $s\mapsto \cU(\mu_s)$ is absolutely continuous
    and
    \begin{equation}
      \label{eq:223}
      \frac\d{\d s}\cU(\mu_s)=   \duality{\V_{\varrho_s}'}{\V_{\varrho_s}}{\frac\d {\d s}\varrho_s}{U'(\varrho_s)}
      \quad\text{for $\Leb1$-a.e.~$s\in (0,1).$}
      \end{equation}
  \item If $Z(\varrho)\in \V$, $\varrho_t=\sfS_t\varrho$, 
    $\BE KN$ holds and $\Lambda$ is defined as in \eqref{eq:56}, then 
    \begin{equation}
      \label{eq:146}
      Z(\varrho_t)\in \V,\quad
      \cE(Z(\varrho_t),Z(\varrho_t))\le \rme^{-2\Lambda t}\cE(Z(\varrho),Z(\varrho))\qquad\forall t\geq 0.
    \end{equation}
    \GGG In particular, if $\mu=\varrho\mm\in \PP_2(X)$ then
    $t\mapsto \varrho_t\mm$ is a Lipschitz curve in $[0,T]$ with 
    respect to the Wasserstein distance in $\PP_2(X)$ 
    with Lipschitz constant bounded by 
    $\rme^{-\Lambda T} \sqrt{\cE(Z(\varrho))}$.
  \end{enumerate}
\end{lemma}
\begin{proof}
The proof of the first claim is standard, see e.g.\ \cite[Lemma~4.10]{AGS11a}.

In order to prove (ii), let us first notice that
\begin{equation}
  \label{eq:142}
  \cEs{\varrho}(\DeltaE P(\varrho),\DeltaE P(\varrho))\le \cE(Z(\varrho),Z(\varrho)).
\end{equation}
In fact for every $\varphi\in \V$ there holds
\begin{align*}
  -\langle \DeltaE P(\varrho),\varphi\rangle&=
  \int_X P'(\varrho)\Gbil\varrho\varphi\,\d\mm=
  \int_X \sqrt\varrho \,\Gbil{Z(\varrho)}\varphi\,\d\mm\le 
  \cE(Z(\varrho),Z(\varrho))^{1/2} \cE_\varrho(\varphi,\varphi)^{1/2}.
\end{align*}
On the other hand, choosing as test functions
$\varphi_\eps:=-U_\eps'(\varrho)$, taking 
the last identity in \eqref{eq:61} into account we get
\begin{align*}
  \cE_\varrho(\varphi_\eps,\varphi_\eps)&=\int_X \varrho \Gq{U_\eps'(\varrho)}\,\d\mm
  \le\int_X (\varrho+\eps) (U_\eps''(\varrho))^2 \Gq{\varrho}\,\d\mm\le 
  \cE(Z(\varrho),Z(\varrho)),\\
  \langle \DeltaE P(\varrho),\varphi_\eps\rangle&=
  \int_X \Gbil{P(\varrho)}{U_\eps'(\varrho)}\,\d\mm=
  \int_X \frac {1}{\varrho+\eps}\Gq{P(\varrho)}\,\d\mm\uparrow
  \cE(Z(\varrho),Z(\varrho))\quad\text{as }\eps\down0.
\end{align*}
This shows that $\{\varphi_\eps\}_{\eps>0}$ is an optimal family as $\eps\to 0$,
thus we can apply Proposition~\ref{prop:allduals}(b) to obtain that $\varphi_\eps$ converge in $\Vhom\varrho$ to a
$-A_\varrho^*(\DeltaE P(\varrho))$, and that \eqref{eq:143} holds.

In order to prove \eqref{eq:223} we pass to the limit as $\eps\downarrow 0$ in the identity obtained
integrating \eqref{eq:129}
\begin{equation}
  \label{eq:225}
  \int_X U_\eps(\varrho_t)\,\d\mm-\int_X U_\eps(\varrho_s)\,\d\mm
  =\int_s^t \duality{\V_{\varrho_r}'}{\V_{\varrho_r}}{\frac{\d}{\d r}\varrho_r}{U_\eps'(\varrho_r)} \,\d r\quad
  \forevery 0\le s\le t\le 1.
\end{equation}
Indeed,  in the left hand side it is sufficient to apply 
the dominated convergence theorem, thanks
to the uniform bounds of \eqref{eq:224} and \eqref{eq:47-bis}.
Since the curve $\mu$ is regular, 
the modulus of the integrand in the right hand side is bounded from above by
\begin{displaymath}
  \frac 12 \cE(Z(\varrho_r),Z(\varrho_r))+\frac 12\cEs{\varrho_r}(\frac\d{\d r}\varrho_r,\frac\d{\d r}\varrho_r)\le C\quad
  \forevery r\in [0,1],
\end{displaymath}
so that we can pass to the limit thanks to (ii).

The inequality \eqref{eq:146} follows by \eqref{eq:139}, 
the fact that $\frac\d\dt \varrho_t=\DeltaE P(\varrho_t)$ and
\eqref{eq:143}. 

\GGG In order to prove the last statement, 
we apply Theorem \ref{thm:dynamicKant}, the estimate
\eqref{eq:143} which provides an explicit expression of 
the metric Wasserstein velocity, and \eqref{eq:146}.
\end{proof}

\subsection{$\BE KN$ yields $\mathrm{EVI}$ for regular 
  entropy functionals in $\MC N$} \label{subsec:BEEVI}
  
\begin{theorem}[$\BE KN$ implies contractivity] \label{thm:BEKNContr} 
  Let us assume that metric $\BE KN$ holds and that $P\in\MCreg N$. 
  If $\Lambda$ is defined as in \eqref{eq:56}, then the nonlinear
  diffusion semigroup $\sfS$
  \GGG 
  defined by Theorem \ref{thm:nonlin-diff}
  \EEE is $\Lambda$-contractive in $\ProbabilitiesTwo
  X$, i.e. {\GGG 
    for all 
    $\mu_0=\varrho\mm,\ \nu_0=\sigma\mm\in
    \ProbabilitiesTwo{X}$ one has}
  \begin{equation}
    \label{eq:93}
    W_2(\mu_t,\nu_t)\le \rme^{-\Lambda t}\,
    W_2(\mu_0,\nu_0)\quad\text{with}\quad
    \mu_t=(\sfS_t\varrho)\mm,\,\,
    \nu_t=(\sfS_t\sigma)\mm.
  \end{equation}
\end{theorem}
\begin{proof}
   We assume first that $\varrho$ and $\sigma$ are the extreme points of a regular curve
  $\bar{\mu}_s=\bar{\varrho}_s \mm$.
  We set $\mu_{s,t}=\varrho_{s,t}\mm$, with $\varrho_{s,t}=\sfS_t \bar{\varrho}_s$. 
  Since $\bar{\mu}_s$ is Lipschitz with respect to $W_2$ and $\bar{\varrho}_s$ are
  uniformly bounded, $s\mapsto \bar{\varrho}_s$ is also Lipschitz and weakly
  differentiable with
  respect to $\Vdual1$: we set $w_{s,t}:=\partial_s \varrho_{s,t}$.

  By Kantorovich duality, 
  \begin{equation}
    \label{eq:95}
    \frac 12 W_2^2(\mu_{0,t},\mu_{1,t})=
    \sup\Big\{\int_X \sfQ_1\varphi\,\d\mu_{1,t}-\int_X \varphi\,\d\mu_{0,t}\Big\}
  \end{equation}
  where $\varphi$ runs among all Lipschitz functions
  with bounded support. If $\varphi$ is such a function with Lipschitz
  constant $L$, setting
  $\varphi_s:=Q_s \varphi$, 
  the map $\eta(s,r):= \int_X \varphi_s\,\d\mu_{r,t}$ is Lipschitz:
  in fact, recalling that 
  \begin{displaymath}
    \Lip(\varphi_s)\le 2 L,\qquad
    \sup_{x\in X} |\varphi_s(x)-\varphi_r(x)|\le 2L^2 |s-r|
  \end{displaymath}
  we easily have
  \begin{displaymath}
    |\eta(s,r)-\eta(s',r)|\le 2L^2 |s-s'|,\qquad
    |\eta(s,r)-\eta(s,r')|\le 2L \sqrt{\mm(S)}\,
    \|\varrho_{s,r}-\varrho_{s,r'}\|_{\Vdual1},
  \end{displaymath}
  where $S$ is a bounded set containing all the supports of $\varphi_s$, $s\in [0,1]$.
  From \eqref{eq:identities} we eventually find
  \begin{align*}
    \frac\d{\d s}\int_X \varphi_s \,\d\mu_{s,t}\le 
    -\frac 12 \int_X |\rmD\varphi_s|^2\,\d\mu_{s,t}+
    \langle w_{s,t},\varphi_{s}\rangle.
  \end{align*}
  Denoting now by $r\mapsto\varphi_{s,r}$ the solution of the backward
  linearized equation \eqref{eq:82} (corresponding to \eqref{eq:PMb}) in the interval $[0,t]$ 
  with final condition $\varphi_{s,t}:=\varphi_s$, recalling Corollary~\ref{cor:rho-derivative} 
  we get by \eqref{eq:130} of Theorem~\ref{thm:reg2} and 
  \eqref{eq:91} of Theorem~\ref{thm:BEKNaction}
    \begin{displaymath}
    \langle w_{s,t},\varphi_{s}\rangle=
    \langle w_{s,0},\varphi_{s,0}\rangle=
    \int_X \varphi_{s,0} \, \partial_s \varrho_s \,\d\mm,\qquad
    \int_X |\rmD\varphi_s|_w^2\,\d\mu_{s,t}
    \ge \rme^{2\Lambda t} \int_X|\rmD\varphi_{s,0}|_w^2\,\d\mu_s,
  \end{displaymath}
  and therefore the relations \eqref{eq:161} and \eqref{eq:168} between minimal $2$-velocity and metric derivative,
  together with Lemma~\ref{le:w-s},  give
  \begin{align*}
    \int_X \varphi_1\,\d\mu_{1,t}-\int_X
    \varphi_0\,\d\mu_{0,t}
    &\le 
    \int_0^1 \Big( -\frac 12 \int_X |\rmD\varphi_s|_w^2\,\d\mu_{s,t}+
    \langle w_{s,t},\varphi_{s}\rangle\Big)\,\d s
    \\&\le 
    \int_0^1 \Big( -\frac 12 \rme^{2\Lambda t}\int_X |\rmD\varphi_{s,0}|_w^2\,\d\mu_{s}+
    \int_X (\partial_s\varrho_s) \, \varphi_{s,0}\,\d\mm\Big)\,\d s
    \\&\le 
    \frac 12 \rme^{-2\Lambda t}\int_0^1 |\dot \mu_s|^2\,\d s.
  \end{align*}
  Taking now the supremum with respect to $\varphi$ we get
  $W_2^2((\sfS_t\varrho)\mm,(\sfS_t\sigma)\mm)\leq e^{-2\Lambda t}\int_0^1 |\dot \mu_s|^2\,\d s$. 
  Using Lemma~\ref{le:appregentr} 
  \GGG and the contraction property of $\sfS_t$ 
  in $L^1(X,\mm)$ \EEE
  we obtain the same bound for an arbitrary couple
  of initial measures.
\end{proof}

Let us recall the notation (see \eqref{eq:17})
$$
\Action{Q} \mu \mm=\int_0^1\int_X Q(\varrho_s)v^2_s\rho_s\,\d\mm \d s
$$
for the weighted action of a curve $\mu_s=\varrho_s\mm$ w.r.t. $\mm$, where
$v_s$ is the velocity density of the curve.

\begin{theorem}
  [Action monotonicity]\label{thm:BEEVI1}
  Let us assume that metric $\BE KN$ holds, and that $P\in\MCreg N$.
  Let $\mu_s=\varrho_s\mm$, $s\in [0,1]$, be a regular curve 
  and let {\color{black} $\mu_{s,t}:=\varrho_{s,t}\mm$ with $\varrho_{s,t}=(\sfS_t\varrho_s)$.}
  Denoting by $\mu_{\cdot,t}$ the curve $s\mapsto \mu_{s,t}$,
  we have 
  \begin{equation}
    \label{eq:116}
  \frac{1}{2} \fn  \cA_2(\mu_{\cdot,t_1}) 
    +   K\int_{t_0}^{t_1}\Action{Q} {\mu_{\cdot,t}} \mm\,\d t\le \frac{1}{2} \fn
    \cA_2(\mu_{\cdot,t_0}) \qquad 0\leq t_0\leq t_1\leq 1.
  \end{equation}  
\end{theorem}
\begin{proof}
  \AAA It is sufficient to prove that the map $t\mapsto  \cA_2(\mu_{\cdot,t})$ is absolutely continuous and satisfies \fn for every $t>0$
  \begin{equation}
    \label{eq:117}
    \limsup_{h\down0}
    \frac 1{2h}\Big(
    \cA_2(\mu_{\cdot,t}) 
    -
    \cA_2(\mu_{\cdot,t-h}) 
    \Big)\le -K \Action Q{\mu_{\cdot,t}}\mm.
  \end{equation}
  Let us fix $t>h>0$; thanks to Theorem~\ref{thm:nonlin-diff} and Theorem~\ref{thm:precise-limit}, the curves $\varrho_{\cdot,t}$ and
  $\varrho_{\cdot,t-h}$ are $\Leb{1}$-a.e. in $(0,1)$ differentiable in $\V'$, 
  with derivatives $w_{s,t}\in\V'_{\varrho_{s,t}}$, $w_{s,t-h}\in\V'_{\varrho_{s,t-h}}$. 
  
  Recall also the relations \eqref{eq:168} and \eqref{eq:113} of Theorem~\ref{thm:dynamicKant}, linking
  the minimal velocity density of a regular curve $\nu_s=\varrho_s\mm$, its $\V'$ derivative $\ell_s$
  and the potential $\phi_r=-A_{\varrho_s}^*(\ell_s)$. 
    
  By \eqref{eq:168} we get 
  \begin{displaymath}
    \frac 1{2h}\Big(
    \cA_2(\mu_{\cdot,t}) 
    -
    \cA_2(\mu_{\cdot,t-h}) \Big)
    =\frac 1{2h}\int_0^1 \Big(\cE_{\varrho_{s,t}}^*(w_{s,t},w_{s,t})-
    \cE_{\varrho_{s,t-h}}^*(w_{s,t-h},w_{s,t-h})
\Big)\,\d s.
  \end{displaymath}
  Recalling \eqref{eq:57} and the definition \eqref{eq:56} of $\Lambda$, one has 
  $$
  \cE_{\varrho_{s,t}}^*(w_{s,t},w_{s,t})-
    \cE_{\varrho_{s,t-h}}^*(w_{s,t-h},w_{s,t-h})
    \leq \bigl(e^{-2\Lambda h}-1\bigr)\cE_{\varrho_{s,t-h}}^*(w_{s,t-h},w_{s,t-h})
  $$
   which is uniformly bounded (using \eqref{eq:57} once more) by $C(t)h$, if $h<t/2$. \AAA Therefore the curve $t\mapsto  \cA_2(\mu_{\cdot,t})$ is absolutely continuous; moreover, \fn applying \eqref{eq:101},  \eqref{eq:168} and Fatou's Lemma we get 
 \begin{displaymath}
   \limsup_{h\down0}\frac 1{2h}\Big(    
    \cA_2(\mu_{\cdot,t}) 
    -
    \cA_2(\mu_{\cdot,t-h}) \Big)\le 
    -K\int_0^1 \int_X
    Q(\varrho_{s,t})\varrho_{s,t}v_{s,t}^2\,\d\mm\,\d s,
    \end{displaymath} 
    where $v_{\cdot,t}$ is the minimal velocity density of $\mu_{\cdot,t}$.
    \qed
  \endnobox
  \end{proof}

Let us now refine the previous argument. In this refinement we shall use the
weighted action
$$
\Action {\AAA id \fn Q}\mu\mm=\int_0^1\int_X \AAA s \fn Q(\varrho_s)v^2_s\rho_s\,\d\mm \, \d s,
$$
where \AAA $id(s)=s$. Notice that the weighted action appearing in the
${\sf EVI}$ property \eqref{eq:218} is $\Action { \omega Q}\mu\mm$, with $\omega(s)=1-s$; in other words $\Action {\omega  Q}\mu\mm$ corresponds to the $s$-time reversed  weighted action   $\Action { id  Q}\mu\mm$. \fn

\begin{theorem}[Action and energy monotonicity]
  \label{thm:BEEVI2}
  Let us assume that metric $\BE KN$ holds and that $P\in\MCreg N$. 
  Let $\mu_s=\varrho_s\mm$, $s\in [0,1]$, be a regular curve and let
  {\color{black} $\mu_{s,t}:=\varrho_{s,t}\mm$ with $\varrho_{s,t}=(\sfS_{st}\varrho_s)$.}
  Denoting by $\mu_{\cdot,t}$ the curve $s\mapsto \mu_{s,t}$,
  we have 
  \begin{equation}
    \label{eq:116.2}
   \frac{1}{2}  \cA_2(\mu_{\cdot,t}) +t\cU(\mu_{1,t})
    +K\int_{0}^t\Action {\AAA id \fn Q}{\mu_{\cdot,r}}\mm\,\d r\le 
   \frac{1}{2}  \cA_2(\mu_{\cdot,0}) +t\cU(\mu_{0,0}).
  \end{equation}  
  \end{theorem}

\begin{proof}
Since by assumption $U$ is continuous and convex, by \eqref{eq:122} we already know that the 
map $t\mapsto \cU(\mu_{1,t})$ is nonincreasing; thus it \fn is sufficient to prove that 
    \begin{equation}
    \label{eq:220}
     \limsup_{h\down0}
    \frac 1{2h}\Big(
    \cA_2(\mu_{\cdot,t}) 
    -
    \cA_2(\mu_{\cdot,t-h}) 
    \Big)\le \cU(\mu_{0,0})-\cU(\mu_{1,t})-K \Action {\AAA id \fn Q}{\mu_{\cdot,t}}\mm.
  \end{equation}
  We thus fix $0<h<t$.  Recalling Theorem~\ref{thm:MVD} and Theorem~\ref{thm:dynamicKant},  
  we have  
    \begin{equation}
    \label{eq:229}
    \cA_2(\mu_{\cdot,t})=  \int_0^1 \cEs{\varrho_{s,t}}(\partial_{s}\varrho_{s,t})\,\d s,\qquad
    \cA_2(\mu_{\cdot,t-h})=  \int_0^1 \cEs{\varrho_{s,t-h}}(\partial_{s} \varrho_{s,t-h})\,\d s. 
  \end{equation}   
It is easy to check that for every $\tau>0$ the curve $s\mapsto
\mu_{s,t-h+\tau}$ is Lipschitz in $\ProbabilitiesTwo X$ and 
$s\mapsto \varrho_{s,t-h+\tau}$ is Lipschitz in $\Vdual1$,
since for every $0\le s_1<s_2\le 1$
\begin{align*}
  \|\varrho_{s_1,t-h+\tau}-\varrho_{s_2,t-h+\tau}\|_{\Vdual1}&\le 
  \|\varrho_{s_1,t-h+\tau}-\sfS_{s_1 \tau} \varrho_{s_2,t-h} \fn\|_{\Vdual1}+
  \|\sfS_{s_1 \tau} \varrho_{s_2,t-h} \fn-\varrho_{s_2,t-h+\tau}\|_{\Vdual1} \\
  &\le  \|\varrho_{s_1,t-h}-\varrho_{s_2,t-h} \|_{\Vdual1}+C \, \tau (s_2-s_1) ,
\end{align*}
for some constant $C$ independent of $s_{1},s_{2}$ and $\tau$,  
where in the last inequality we used the contractivity \eqref{eq:51}
of $\sfS$  in $\Vdual1$ and Theorem~\ref{thm:nonlin-diff}
(ND3).
\\ \GGG A similar argument shows the Lipschitz property with respect to
the Wasserstein distance:
\begin{align*}
  W_2(\mu_{s_1,t-h+ \tau},\mu_{s_2,t-h+\tau}) &\le 
  W_2\big(\mu_{s_1,t-h+\tau},(\sfS_{s_1 \tau} \varrho_{s_2,t-h}) \mm \big)+
  W_2\big((\sfS_{s_1 \tau} \varrho_{s_2,t-h})  \mm, \mu_{s_2,t-h+\tau}\big) \\
  & \le
  \rme^{-\Lambda s_1\tau} W_2(\mu_{s_1,t-h},\mu_{s_2,t-h})+
  C' \, \tau (s_2-s_1) ,
\end{align*}
where we applied \eqref{eq:93} and point (iv) of Lemma \ref{le:2bis}:
notice that, along the regular curve $\mu_s=\varrho_s \mm$, the quantity 
$\cE(\sqrt{\varrho_s},\sqrt{\varrho_s})$ is uniformly bounded,
so that $\cE(Z(\varrho_{s,t-h}),Z(\varrho_{s,t-h}))$ is also uniformly
bounded by \eqref{eq:146}. \EEE


For every $r\in [0,1], u\in [0,t]$, also the curves $s\mapsto \varrho_{s,r}^u:=\sfS_{ru}\varrho_{s,t-h}$   are regular:
we set $z_{s,r}^u:=\partial_s \varrho_{s,r}^u$.
We have 
\begin{displaymath}
  \lim_{k\to0}\frac{\varrho_{s,r+k}^u-\varrho_{s,r}^u}k=
  u\DeltaE P(\varrho_{s,r}^u)\quad
  \text{for every } u\in [0,t],\quad s,\,r\in [0,1].
\end{displaymath}
Since $\varrho^h_{s,s}=\varrho_{s,t}$,
it follows that the derivative of 
$s\mapsto \varrho_{s,t}$ in $\Vdual 1$ is
\begin{displaymath}
  \partial_s \varrho_{s,t}=   \partial_s \big( \sfS_{sh}  \varrho_{s,t-h}\big)= z_{s,s}^h+h
  \DeltaE P(\varrho_{s,t}).
\end{displaymath}

Applying Lemma~\ref{le:2bis}(ii,iii) we get
\begin{align}
  \label{eq:120}
    \frac\partial{\partial s}\int_X U(\varrho_{s,t}) \,\d\mm
  &=\notag
  \duality{\V'_{\varrho_{s,t}}}{\V_{\varrho_{s,t}}}{z_{s,s}^h+h
  \DeltaE P(\varrho_{s,t})}{U'(\varrho_{s,t})} \quad
  \Leb 1\text{-a.e.\ in }(0,1).
  \end{align}
For every $s\in [0,1]$, let $\varphi^n_{s,t}\in \V$ be an optimal
sequence for $\partial_{s} 	\varrho_{s,t}$, thus satisfying
\begin{displaymath}
  \frac 12\cEs{\varrho_{s,t}}(\partial_{s} 	\varrho_{s,t}, \partial_{s} 	\varrho_{s,t})=
  \lim_{n\to\infty} \duality{\V'_{\varrho_{s,t}}}{\V_{\varrho_{s,t}}}{ z_{s,s}^h+h
  \DeltaE P(\varrho_{s,t})}{\varphi^n_{s,t}}-
  \frac 12 \int_X \varrho_{s,t} \tGq {\varrho_{s,t}}  {\varphi_{s,t}^n} \,\d\mm.
\end{displaymath}
Let $\upsilon_{s,t}:=-U'(\varrho_{s,t})\in \Vhom{\varrho_{s,t}}$ and 
$\psi^n_{s,t}:=\varphi_{s,t}^n-h\upsilon_{s,t}$.
We get
\begin{align*}
  &\duality{\V'_{\varrho_{s,t}}}{\V_{\varrho_{s,t}}}{z^h_{s,s}+h\DeltaE P(\varrho_{s,t})}{\varphi_{s,t}^n}-
  \frac 12 \int_X \varrho_{s,t} \tGq {\varrho_{s,t}} {\varphi_{s,t}^n}\,\d\mm
  \\&=
  \duality{\V'_{\varrho_{s,t}}}{\V_{\varrho_{s,t}}}{z_{s,s}^h+h\DeltaE P(\varrho_{s,t})}{\psi^n_{s,t}
  +h\upsilon_{s,t}}-
  \frac 12 \int_X \varrho_{s,t} \tGq{\varrho_{s,t}}{\psi^n_{s,t}+h\upsilon_{s,t}}\,\d\mm
  \\&=
  -h\frac\partial{\partial s}\int_X U(\varrho_{s,t})\,\d\mm
  +\duality{\V'_{\varrho_{s,t}}}{\V_{\varrho_{s,t}}}{z_{s,s}^h}{\psi^{n}_{s,t}}-\frac 12 \cE_{\varrho_{s,t}}(\psi^n_{s,t},\psi^n_{s,t})+h \;
    \duality{\V'_{\varrho_{s,t}}}{\V_{\varrho_{s,t}}}{\DeltaE P(\varrho_{s,t})}{\psi^n_{s,t}}
    \\&\qquad
    -h\int_X \varrho_{s,t}
    \tGbil{\varrho_{s,t}}{\psi^n_{s,t}}{\upsilon_{s,t}}\,\d\mm
    -\frac {h^2}2\int_X \varrho_{s,t}\Gq{\upsilon_{s,t}}\,\d\mm
    \\&\le -h\frac\partial{\partial s}\int_X U(\varrho_{s,t})\,\d\mm
    +\duality{\V'_{\varrho_{s,t}}}{\V_{\varrho_{s,t}}}{z_{s,s}^h}{\psi^{n}_{s,t}}
    -\frac 12 \cE_{\varrho_{s,t}}(\psi^n_{s,t},\psi^n_{s,t}),
\end{align*}
where  we used Lemma \ref{le:2bis} (ii) to get  the second equality, and to simplify the 
third and second to last terms in order to obtain the last inequality.  
We observe that 
$\psi^n_{s,t}$ is an optimal sequence for 
$z_{s,s}^h$: we will denote by $\psi_{s,t}$ its limit in
$\V'_{\varrho_{s,t}}$ and by $\phi_{s,t}$ the limit of 
$\varphi^n_{s,t}$. They are related by
\begin{equation}
  \label{eq:230}
  \phi_{s,t}=\psi_{s,t}+h\upsilon_{s,t}.
\end{equation}
Passing to the limit in the previous inequality we obtain
\begin{equation}
  \label{eq:232}
 \frac 12\cEs{\varrho_{s,t}}(\partial_{s} 	\varrho_{s,t}, \partial_{s} 	\varrho_{s,t}) \le
  -h\frac\partial{\partial s}\int_X U(\varrho_{s,t})\,\d\mm
  +\frac{1}{2} \cEs{\varrho_{s,t}}(z^h_{s,s},z^h_{s,s}).
\end{equation}
\AAA
Observe that $u \mapsto \varrho_{s,r}^u:=\sfS_{ru}\varrho_{s,t-h}$ and  $u \mapsto z_{s,r}^u:= \partial_{s} \varrho^{u}_{s,r}$   satisfy respectively
$$ 
\partial_{u} \varrho^{u}_{s,r} = r \DeltaE (P(\varrho^{u}_{s,r})), \quad \partial_{u} z^{u}_{s,r}= r \DeltaE(P'(\varrho^{u}_{s,r}) z^{u}_{s,r}),
$$
where the second equation follows from Theorem \ref{thm:precise-limit}. Setting $r=s$ we get
$$
\partial_{u} \varrho^{u}_{s,s} = s \DeltaE (P(\varrho^{u}_{s,s})), \quad   \partial_{u} z^{u}_{s,s} = s \DeltaE (P'(\varrho^{u}_{s,s}) \, z^{u}_{s,s}).
$$
\fn
Let now $B^t_{s,r}$ be the operator, given by Theorem~\ref{prop:backward-linearization}, 
mapping $\zeta\in \V$ into the solution $\zeta_{s,r}$ of 
\begin{equation}
  \label{eq:231}
  \frac\d{\d r}\zeta_{s,r}=- \AAA s \fn  P'(\varrho_{s,r})\DeltaE\zeta_{s,r}\qquad
  r\in [0,t],\quad
  \zeta_{s,t}:=\zeta.
\end{equation}
Theorem \ref{thm:BEKNaction}
 and the fact that $z^0_{s,s}=\partial_{s} \varrho_{s,t-h}$ and $\varrho^{u}_{s,s}=\varrho_{s,t-h+u}$ yield
 \begin{align*}
& \frac {1}{2} \left[ \cEs{\varrho_{s,t}}(z^h_{s,s},z^h_{s,s}) -  \cEs{\varrho_{s,t-h}}(\partial_{s} \varrho_{s, t-h}, \partial_{s} \varrho_{s, t-h}) \right]  \\
& \qquad  \qquad \qquad  \le 
  -K \AAA s \fn   \int_{t-h}^t \int_X Q(\varrho_{s,r})\varrho_{s,r}
  \tGq{\varrho_{s,r}}{B^t_{s,r}(\psi_{s,t})}\,\d\mm\,\d r.
  \end{align*}
Using the estimate
\begin{displaymath}
  \tGq{\varrho_{s,r}}{B^t_{s,r}(\psi_{s,t})}\le 
  (1+\delta) \tGq{\varrho_{s,r}}{B^t_{s,r}(\phi_{s})}+
  h^2 \Big(1+\frac 1\delta \Big) \tGq{\varrho_{s,r}}{B^t_{s,r}(\upsilon_{s,t})}
\end{displaymath}
and the uniform bound
\begin{displaymath}
  \int_X
  \varrho_{s,r}\tGq{\varrho_{s,r}}{B^t_{s,r}(\upsilon_{s,t})}\,\d\mm\le 
  C \int_X
  \varrho_{s,t}\tGq{\varrho_{s,t}}{\upsilon_{s,t}}\,\d\mm\le C',
\end{displaymath}
Lemma~\ref{le:stability} eventually yields
\begin{displaymath}
  \limsup_{h\down0}
  \frac1 { 2 h}\Big(\cEs{\varrho_{s,t}}(z^h_{s,s},z^h_{s,s})-
  \cEs{\varrho_{s,t-h}}(\partial_{s} \varrho_{s, t-h}, \partial_{s} \varrho_{s, t-h}) \Big)\le
  -K  \AAA s \fn  
  \int_X Q(\varrho_{s,t})\varrho_{s,t}
  \tGq{\varrho_{s,t}}{\phi_{s,t}}\,\d\mm.
\end{displaymath}
\AAA
Combining this estimate with  \eqref{eq:232}, we get 
\begin{align*}
  \limsup_{h\down0}
  \frac1 { 2 h}\Big(\cEs{\varrho_{s,t}}(\partial_{s} \varrho_{s,t}, \partial_{s} \varrho_{s,t})&-
  \cEs{\varrho_{s,t-h}}(\partial_{s} \varrho_{s, t-h}, \partial_{s} \varrho_{s, t-h}) \Big) \\
  & \leq -K  s   
  \int_X Q(\varrho_{s,t})\varrho_{s,t}
  \tGq{\varrho_{s,t}}{\phi_{s,t}}\,\d\mm - \frac\partial{\partial s}\int_X U(\varrho_{s,t})\,\d\mm.
\end{align*}
By recalling
\eqref{eq:229} and Theorem~\ref{thm:dynamicKant}, the integration  w.r.t.~$s$ in $(0,1)$ of the last inequality gives \eqref{eq:220}.
\fn
\end{proof}

\begin{theorem}[$\BE KN$ implies $\CDS K N$]
  \label{cor:MAIN}
  Let us suppose that $(X,\sfd,\mm)$ is a metric measure space satisfying the metric $\BE KN$ condition.
    Then for every entropy function $U$ in $\MCreg N$ and every \AAA  
    $\bar{\mu}=\varrho \mm$ with $\varrho$ $\mm$-essentially bounded with bounded support, \EEE 
  the curve $\mu_t=(\sfS_t\varrho)\mm$ is the
  unique solution of the Evolution Variational Inequality 
  \eqref{eq:181}. 
  In particular $(X,\sfd,\mm)$ is a strong $\CDS KN$ space.  
\end{theorem}
\begin{proof}
  Under the above conditions, one can apply \cite[Cor.~4.18]{AGS12} 
  to obtain that $(X,\sfd,\mm)$ is an $\RCD K\infty$ space, 
  in particular the assumptions of Lemma~\ref{le:appregentr} are satisfied.
  Now let $\sfS_t$ be the solution of 
  the nonlinear diffusion semigroup of  Theorem~\ref{thm:nonlin-diff} 
  and let \AAA $\bar{\nu}=\sigma \mm$ with $\sigma$ $\mm$-essentially bounded with bounded support\EEE;
  we consider a family of regular
  curves $\mu^{(n)}_{s}=\varrho^{(n)}_s\mm$ approximating a geodesic
  $\mu_s$ 
  from $\bar{\nu}$ to $\bar{\mu}$ in the sense of Lemma~\ref{le:appregentr}
  and we set 
  $\mu^{(n)}_{s,t}=(\sfS_{st}\varrho^{(n)}_s)\mm$.
  Applying \eqref{eq:116.2} of Theorem~\ref{thm:BEEVI2} we get
  \begin{equation}
    \label{eq:234}
   \frac{1}{2}  \cA_2(\mu^{(n)}_{\cdot,t}) +t\cU(\mu^{(n)}_{1,t})
    +K\int_{0}^{t}\Action {\AAA id \fn Q}{\mu^{(n)}_{\cdot,r}}\mm\,\d r\le 
   \frac{1}{2}  \cA_2(\mu^{(n)}_{\cdot,0}) +t\cU(\mu^{(n)}_{0,0}).
  \end{equation}
  \AAA Dividing by $t>0$ and letting $n\to \infty, t\down 0$  we get
   \begin{equation}
    \label{eq:234bis}
 \limsup_{t\down 0} \limsup_{n\to \infty}  \left( \frac{\cA_2(\mu^{(n)}_{\cdot,t})-  \cA_2(\mu^{(n)}_{\cdot,0})}{2t} + \cU(\mu^{(n)}_{1,t})
    +\frac{K}{t}\int_{0}^{t}\Action { id  Q}{\mu^{(n)}_{\cdot,r}}\mm\,\d r \right)\le 
  \limsup_{n\to \infty}    \cU(\mu^{(n)}_{0,0}).
  \end{equation}
 Next we pass to the limit in the different terms, setting $\mu_{s,t}=(\sfS_{st}\varrho_s)\mm$. First of all, combining \eqref{eq:140} with the lower semicontinuity of the 2-actions and recalling that $\mu_{(\cdot)}$ is a $W_{2}$-geodesic we get
\begin{align}
 \limsup_{t\down 0} \limsup_{n\to \infty}  \frac{1}{2t}  \Big( \cA_2(\mu^{(n)}_{\cdot,t})-  \cA_2(\mu^{(n)}_{\cdot,0}) \Big)  & \geq \limsup_{t\down 0} \frac{1}{2t}  \Big(\cA_2(\mu_{\cdot,t})-  \cA_2(\mu_{\cdot, 0}) \Big) \nonumber \\
 &\geq    \limsup_{t\down0}\frac 1{2t}\Big(
    W_2^2(\bar{\nu}, \mu_{1,t})-W_2^2(\bar{\nu}, \bar{\mu})  \Big). \label{eq:limsuptnA2}
\end{align}
In virtue of \eqref{eq:14000} and of  the lower semicontinuity of the entropy we also get
\begin{equation}  \label{eq:limcU}
\liminf_{t\down 0} \liminf_{n\to \infty} \cU(\mu^{(n)}_{1,t}) \geq \liminf_{t\down 0}  \cU(\mu_{1,t}) \geq \cU(\bar{\mu}), \quad   \lim_{n\to \infty}    \cU(\mu^{(n)}_{0,0})=\cU(\bar{\nu}).
\end{equation}
Regarding the term with the integral of the actions we claim that the \emph{joint limit} as $t\down 0, n\to \infty$  exists with value
\begin{equation}\label{eq:limitActionQ}
\lim_{t\down 0, n\to \infty} \frac{1}{t}\int_{0}^{t}\Action { id  Q}{\mu^{(n)}_{\cdot,r}}\mm\,\d r =  \, \Action {id Q}{\bar{\nu}, \bar{\mu}}\mm.
\end{equation}
 In order to prove \eqref{eq:limitActionQ}, we first show that 
\begin{equation}\label{eq:limitActionQ1}
\lim_{t\down 0, n\to \infty} \Action { id  Q}{\mu^{(n)}_{\cdot,t}}\mm = \Action { id  Q}{\mu_{\cdot}}\mm =  \, \Action {id Q}{\bar{\nu}, \bar{\mu}}\mm.
\end{equation}
  In order to show the convergence we wish to apply Theorem \ref{le:stabilitySigma}, let us then verify its assumptions.
  \\Recalling that  by  Lemma~\ref{le:appregentr} we have $\varrho^{(n)}_{s}\to \varrho_{s}$ strongly in $L^{1}(X,\mm)$ for every $s \in [0,1]$,  using the $L^{1}$-contractivity  and $L^{1}$-continuity of the semigroup proved in Theorem \ref{thm:nonlin-diff} (ND4),
  we obtain
  \begin{align*}
 \limsup_{t\down 0, n\to \infty} \|\varrho_{s}-\varrho_{s,t}^{(n)}\|_{L^{1}(X,\mm)} & \leq     
 \limsup_{n\to \infty, t\down 0} \Big( \|\varrho_{s}-\varrho_{s,t}\|_{L^{1}(X,\mm)} +  \|\varrho_{s,t}-\varrho_{s,t}^{(n)}\|_{L^{1}(X,\mm)} \Big) \nonumber \\
& \leq     \limsup_{t\down 0}  \|\varrho_{s}-\varrho_{s,t}\|_{L^{1}(X,\mm)} + \limsup_{n\to \infty}  \|\varrho_{s}-\varrho_{s}^{(n)}\|_{L^{1}(X,\mm)} =0,
  \end{align*}
 which in turn implies (by dominated convergence)
 \begin{equation}\label{eq:convnstSL1}
 \varrho^{(n)}_{\cdot, t} \to \varrho_{\cdot} \quad \text{as } n\to \infty, t\down 0,  \text{ strongly in } L^{1}(\tilde{X},\tilde{\mm}).  
 \end{equation}
 Now let $(t_{n})_{n\in\N}$ be any sequence with $t_{n} \down 0$. First of all, by the  lower  semicontinuity of the 2-actions we have
 $\liminf_n  \cA_2(\mu^{(n)}_{\cdot,t_{n}}) \geq  \cA_2(\mu_{\cdot, 0})$.
 On the other hand, by Theorem~\ref{thm:BEEVI1}  we have
 $$
 \cA_2(\mu^{(n)}_{\cdot,t})  \leq-  2  K\int_{0}^{t}\Action{Q} {\mu_{\cdot,s}^{(n)}} \mm\,\d s +  \cA_2(\mu^{(n)}_{\cdot,0})  \leq 2 |K| (\sup |Q|)  \int_{0}^{t}\cA_{2} (\mu_{\cdot,s}^{(n)})\,\d s +  \cA_2(\mu^{(n)}_{\cdot,0})    
 $$
 which, by Gronwall Lemma, implies  $\sup_{t \in [0,1]} \cA_2(\mu^{(n)}_{\cdot,t}) \leq C=C(\mu_{\cdot},K, \sup |Q|))$. 
 \\Therefore, again by Theorem \ref{thm:BEEVI1} we get
 \begin{align*}
  \limsup_{n\to \infty} \cA_2(\mu^{(n)}_{\cdot,t_{n}}) &  \leq   \limsup_{n\to \infty} \Big(-  2  K\int_{0}^{t_{n}}\Action{Q} {\mu_{\cdot,t}} \mm\,\d t +  \cA_2(\mu^{(n)}_{\cdot,0})  \Big) \nonumber \\
  & \leq   \limsup_{n\to \infty} \Big(C(\mu_{\cdot},K, \sup |Q|)) t_{n} +  \cA_2(\mu^{(n)}_{\cdot,0})  \Big)  =   \limsup_{n\to \infty}  \cA_2(\mu^{(n)}_{\cdot,0}) =  \cA_2(\mu_{\cdot,0}).
 \end{align*}
It follows that $\lim_{n\to \infty}  \cA_2(\mu^{(n)}_{\cdot,t_{n}})=  \cA_2(\mu_{\cdot,0})$ for any sequence $t_{n}\down 0$ and then  
$$\lim_{n\to \infty, t\down 0}  \cA_2(\mu^{(n)}_{\cdot,t})=  \cA_2(\mu_{\cdot,0}).$$
 We can then apply Theorem \ref{le:stabilitySigma} and obtain the claim \eqref{eq:limitActionQ1} and then  \eqref{eq:limitActionQ}.
\\Putting together  \eqref{eq:limsuptnA2}, \eqref{eq:limcU}  and  \eqref{eq:limitActionQ}  we obtain
    \begin{equation}
    \label{eq:210pre}
    \limsup_{t\down0}\frac 1{2t}\Big(
    W_2^2(\mu_{1,t},\bar{\nu})-W_2^2(\bar{\mu}, \bar{\nu})
    \Big)+\cU(\bar{\mu})
    +K\Action {id  Q}{\bar{\nu}, \bar{\mu}}\mm\le 
    \cU(\bar{\nu}).
  \end{equation}
  Recalling  that $\omega(s)=1-s$, and that $\mu_{1,t}=(\sfS_t \varrho) \mm=\mu_t$, the last identity is equivalent to 
   \begin{equation}
    \label{eq:210}
    \limsup_{t\down0}\frac 1{2t}\Big(
    W_2^2(\mu_t,\bar{\nu})-W_2^2(\bar{\mu},\bar{\nu})
    \Big)+\cU(\bar{\mu})
    +K\Action {\omega Q}{\bar{\mu}, \bar{\nu}}\mm\le 
    \cU(\bar{\nu}).
  \end{equation}
  \fn
  \AAA 
  This proves   \eqref{eq:181}; therefore the  strong $\CDS KN$ property, is an immediate consequence of Theorem~\ref{thm:EVI-CD}.
  \end{proof}

\subsection{$\RCDS KN$  implies $\BE KN$}\label{sec:RCDtoBE}
In this section we will assume that 
$(X,\sfd,\mm)$ is an $\RCDS KN$ space
and we will show that the Cheeger energy satisfies $\BE KN$.
By \cite{AGS12} we already know that 
$\BE K\infty$ holds.

In the following, we consider an entropy density function
$U=U_{N,\eps,M}\in \MCreg N$ of the form given by \eqref{eq:167} 
through the regularization \eqref{eq:207} and we will denote
by $(\sfS_t)_{t\ge0}$ the nonlinear diffusion flow
provided by Theorem~\ref{thm:nonlin-diff}
and satisfying the EVI property \eqref{eq:218} by Theorem~\ref{thm:CD-EVI}.

\begin{lemma}\label{lem:geo_interpolation} Let $\mu_s=\varrho_s\mm$ be a Lipschitz curve 
in $\ProbabilitiesTwo X$  such that $s \mapsto  \entv(\mu_s)$ is continuous. \fn For a given integer $J$,  consider the 
uniform partition $0=s_0<s_1<\cdots<s_J=1$ of the time interval $[0,1]$ of size $\sigma:=J^{-1}$ and the
piecewise geodesic $\mu^J_s=\varrho^J_s\mm$, $s\in [0,1]$, 
obtained by glueing all the geodesics connecting $\mu_{s_{j-1}}$ 
  to $\mu_{s_j}$.

Then  $\varrho^J_{(\cdot)} \to \varrho_{(\cdot)}$ in $L^1(X \times [0,1], \mm \otimes \LL^{1})$.
\end{lemma}
\fn

\begin{proof}
First of all, since $\mu_{(\cdot)}$ is a Lipschitz curve  in $\ProbabilitiesTwo X$, it is clear that the geodesic interpolation  converges, i.e.   $\mu_{(\cdot)}^J \to \mu_{(\cdot)}$ in $C^0([0,1], \ProbabilitiesTwo X)$. Therefore for every $s \in [0,1]$ we have $\mu_s^J\to \mu_s$ weakly and thus  (see for instance \cite[Lemma 9.4.3]{AGS08})
\begin{equation}\label{eq:lscEnt}
\entv(\mu_s) \leq \liminf_{J\to \infty} \entv(\mu_s^J), \quad \forall s \in [0,1].
\end{equation}
On the other hand it is not difficult to prove also the converse inequality
\begin{equation}\label{eq:uscEnt}
\entv(\mu_s) \geq \limsup_{J\to \infty} \entv(\mu_s^J), \quad \forall s \in [0,1].
\end{equation}
Indeed, the $K$-geodesic convexity of the entropy along geodesics ensured by $\RCD K \infty$ yields
\begin{equation}\label{eq:entKConv}
\entv(\mu^J_{(1-t)s_j+t s_{j+1}}) \leq (1-t) \entv(\mu_{s_j})+t \entv (\mu_{s_{j+1}}) - K \frac{t (1-t)}{2} W_2^2(\mu_{s_j}, \mu_{s_{j+1}}),
\end{equation} 
for all $t \in [0,1]$. 
Since the maps $s\mapsto \entv(\mu_s) \in \R^+$ and $s\mapsto \mu_s\in  \ProbabilitiesTwo X$ are continuous, we get \eqref{eq:uscEnt} by passing to the limit as $J\to \infty$ in  \eqref{eq:entKConv}.

\AAA
From the convergence  $\mu_{(\cdot)}^J \to \mu_{(\cdot)}$ in $C^0([0,1], \ProbabilitiesTwo X)$ we infer that the family $\{\mu_{s}^J, \mu_{s}\}_{s \in [0,1], J\in \N}$ is tight.
The thesis then follows from the following Lemma \ref{lem:W2EntL1Conv} combined with the  Dominated Convergence Theorem.   \fn
\end{proof}
 \AAA
We next state a  well known consequence of the strict convexity of the
function $t\mapsto t \log t$ on $[0,\infty)$
(see e.g.~\cite[Theorem 3]{Visintin84}).
\begin{lemma}\label{lem:W2EntL1Conv}
For $n \in \N$, let $\varrho_{n}\mm=\mu_{n}\in \Probabilities X$ and  $\varrho \mm=\mu \in \Probabilities X$ be such that 
\begin{itemize}
\item $\mu_{n}\to \mu$ weakly in  $\Probabilities X$,
\item $ \entv(\mu_{n})\to \entv(\mu)$ as $n \to \infty$.
\end{itemize}
Then  $\varrho_{n} \to \varrho$ strongly in $L^{1}(X,\mm)$.
\end{lemma}

\fn

\begin{lemma}
  \label{le:action-estimate}
  Let $\mu_s=\varrho_s\mm$ be a Lipschitz curve 
  in $\ProbabilitiesTwo X$ with $s\mapsto\varrho_s$ continuous w.r.t. the $L^1(X,\mm)$ topology
  and $\sup_s\|\varrho_s\|_{L^\infty(X,\mm)}<\infty$. Then, defining 
  $\mu_{s,t}=\varrho_{s,t}\mm$ with $\varrho_{s,t}=(\sfS_t\varrho_s)$, one has 
  \begin{equation}
    \label{eq:21bis}
    \frac{1}{2}\frac {\d^+}{\d t}\cA_2(\mu_{\cdot,t})\le -K
    \Action{Q}{\mu_{\cdot, t}}\mm
    \quad\text{for every $t\ge 0$.}
  \end{equation}
\end{lemma}
\begin{proof} {\color{black} The $L^1(X,\mm)$ contractivity of $\sfS$ ensures that $s\mapsto\varrho_{s,t}$, $t\geq 0$ are equi-continuous
in $L^1(0,1;L^1(X,\mm))$, while the embedding \eqref{eq:12} provides the continuity of $t\mapsto\varrho_{s,t}$ when $s$ is fixed; combining these
properties we know that $(s,t)\mapsto\varrho_{s,t}$ is continuous w.r.t.
  the $L^1(X,\mm)$ topology.} In addition, it is easily seen that the $L^\infty(X,\mm)$ norms of $\varrho_{s,t}$ are uniformly
  bounded, and $s\mapsto\mu_{s,t}=\varrho_{s,t} \mm$ is a Lipschitz curve 
  in $\ProbabilitiesTwo X$. 
  
  
   For a fixed
  integer $J$ we consider the uniform partition $0=s_0<s_1<\cdots<s_J=1$ 
  of the time interval $[0,1]$ of size $\sigma:=J^{-1}$, and the corresponding piecewise geodesic approximation $\mu^J_{s,t}$ of $\mu_{s,t}$.

  Summing up the Evolution Variational Inequality
  \eqref{eq:218} for $\mu_{s_{j-1},t}$ and test measure
  $\mu_{s_j,t}$ and the corresponding one for $\mu_{s_{j},t}$ and test measure
  $\mu_{s_{j-1},t}$ we  use the Leibniz rule \cite[Lemma~4.3.4]{AGS08} to get
  that $t\mapsto W_2^2(\mu_{{s_{j-1},t}},\mu_{s_j,t})$ is locally absolutely continuous in $[0,\infty)$, and that   
  \begin{align*}
  \frac{1}{2}   \frac {\d}{\dt}W_2^2(\mu_{{s_{j-1},t}},\mu_{s_j,t})&\le-
     K\,\Big(\Action{\omega Q}{\mu_{s_{j-1},t},\mu_{s_j,t}}\mm
  +
  \Action{\omega Q}{\mu_{s_{j},t},\mu_{s_{j-1},t}}\mm\Big)
  \end{align*}
  for $j=1,\ldots,J$ and $\Leb{1}$-a.e. $t>0$.
  Denoting by $\mu^J_{\cdot, t}$ the piecewise geodesic curve as in the previous lemma, 
  we obviously have
  $$
  \cA_2(\mu^J_{\cdot,t})=\frac{1}{\sigma}\sum_{j=1}^J W_2^2(\mu_{{s_{j-1},t}},\mu_{s_j,t}),
  $$
  while \eqref{eq:time_reversal} gives
  $$
  \Action{Q}{\mu^J_{\cdot,t}}\mm= \frac{1}{\sigma}
  \sum_{j=1}^J \Big(\Action{\omega Q}{\mu_{s_{j-1},t},\mu_{s_j,t}}\mm+
  \Action {\omega Q}{\mu_{s_{j},t},\mu_{s_{j-1},t}}\mm\Big) .
   $$
  We end up with
  \begin{equation}
    \label{eq:22bis}
 \frac{1}{2}   \frac {\d}{\d t}\cA_2(\mu^J_{\cdot,t})  \le 
    -K \Action{Q}{\mu^J_{\cdot,t}}\mm\qquad\text{for $\Leb{1}$-a.e. $t>0$},
  \end{equation}
  or, in the equivalent integral form,
  \begin{equation}
    \label{eq:23bis}
   \frac{1}{2} \cA_2(\mu^J_{\cdot,t_2})-
   \frac{1}{2} \cA_2(\mu^J_{\cdot,t_1}) \le 
    -K\int_{t_1}^{t_2} \Action{Q}{\mu^J_{\cdot,t}}\mm\,\dt\quad\qquad 0\le t_1<t_2. 
  \end{equation}
  By Lemma~\ref{lem:geo_interpolation}, we know that the curves $\mu^J_{\cdot, t}$ converge to the curves $\mu_{\cdot,t}$ in 
  $L^1(X\times[0,1], \mm \otimes \LL^1)$ as $J\to \infty$.  This enables us to apply Theorem~\ref{le:stabilitySigma}
  {\color{black} (notice that \eqref{eq:162bis} holds because the piecewise geodesic interpolation does not increase the action)},
  so that we can pass to the limit as $J\up\infty$ 
   in \eqref{eq:23bis} and use \eqref{eq:163bis} to get 
   \begin{equation}\label{eq:23ter}
  \frac{1}{2}  \cA_2(\mu_{\cdot,t_2})-
  \frac{1}{2}  \cA_2(\mu_{\cdot,t_1}) \le 
    -K\int_{t_1}^{t_2} \Action{Q}{\mu_{\cdot,t}}\mm\,\dt \qquad \text{ for all } 0\le t_1<t_2\leq T.
   \end{equation}
  \end{proof}

\begin{corollary}
  \label{cor:rough-exponential-estimate}
  Under the same assumptions and notation of the previous
  Lemma~\ref{le:action-estimate}, 
  if $\Lambda$ is defined as in \eqref{eq:56} then 
  \begin{equation}
    \label{eq:249}
    \cA_2(\mu_{\cdot,t})\le \rme^{-2\Lambda t} \cA_2(\mu_{\cdot,0})\quad\forevery t\geq 0.
  \end{equation}
  In particular, if $L$ is the Lipschitz constant of 
  the initial curve $(\mu_s)_{s\in [0,1]}$ in $(\PP_2(X),W_2)$ 
  and $s\mapsto \varrho_{s,t}\in \rmC^1([0,1];\Vdual{})$ then 
  \begin{equation}
    \label{eq:149}
    \cE_{\varrho_{s,t}}^*(\partial_s \varrho_{s,t},\partial_s\varrho_{s,t})\le \rme^{-2\Lambda t}\,L^2\qquad
    \forall\,s\in [0,1],\,\,\forall t\geq 0.
  \end{equation}
\end{corollary}
\begin{proof} The action estimate \eqref{eq:249} follows easily by \eqref{eq:21bis}
  and the fact that  the definition of $\Lambda$ gives $-K\Action{Q}{\mu_{\cdot, t}}\mm\le -\Lambda \cA_2(\mu_{\cdot, t})$. 
      
  By repeating the estimate above to every subinterval of $[0,1]$,  the identity \eqref{eq:168} of Theorem~\ref{thm:dynamicKant} 
  and the equality \eqref{eq:161} between minimal velocity and metric derivative  yield  
  $$\cE_{\varrho_{s,t}}^*(\partial_s \varrho_{s,t},\partial_s\varrho_{s,t})\le \rme^{-2\Lambda t}\,L^2\qquad
    \LL^1\text{-a.e.}\,s\in [0,1],\,\,\forall t\geq 0. $$
  The thesis  \eqref{eq:149} then follows by the lower semicontinuity of the map $s\mapsto \cE_{\varrho_{s,t}}^*(\partial_s \varrho_{s,t},\partial_s\varrho_{s,t})$  ensured by  Lemma~\ref{le:lsc}, since the maps $s\mapsto \partial_s\varrho_{s,t}, s \mapsto \varrho_{s,t}$ are 
  continuous in $\Vdual{}$ and weak$^*$-$L^\infty(X,\mm)$ respectively.  
\end{proof}

{\color{black}
We can now prove the implication from $\RCDS KN$ to $\BE KN$; we adopt a perturbation argument similar to the
one independently found in \cite{Bolley-Gentil-Guillin-Kuwada14}.} 

\begin{theorem}
  \label{thm:RCD-BE}
  If $(X,\sfd,\mm)$ satisfies $\RCDS KN$ then the metric $\BE KN$ condition holds.
\end{theorem}
\begin{proof}
  Let us first remark that 
  $(X,\sfd,\mm)$ satisfies the metric $\BE K\infty$ 
  condition and that $(X,\sfd)$ is locally compact;
  in order to check $\BE KN$ we can thus apply Theorem~\ref{thm:nonlinearBE}.
  
  We fix $f\in D_\V(\DeltaE)\cap D_{L^\infty}(\DeltaE)$ with compact support 
  and $\mu=\varrho\mm\in \Probabilities X$ with compactly supported density $\varrho\in D_{L^\infty}(\DeltaE)$ 
  satisfying $0<r_0\le \varrho$ $\mm$-a.e.~on the support of $f$. With these choices, our goal is to prove
  the inequality 
   \begin{equation}
    \label{eq:152}
    \GGamma_2(f;P(\varrho))
    +\int_X R(\varrho) \,(\DeltaE f)^2\,\d\mm
    \ge K\int_X \Gq f\,P(\varrho)\,\d\mm.
  \end{equation} 
    We define 
  $$
    \psi:=-\varrho\DeltaE f-\Gbil\varrho f.
  $$
  Since $\varrho$ and $f$ are Lipschitz in $X$, recalling 
  Theorem~\ref{thm:interpolation} and Lemma~\ref{le:pre-local} one has $\psi\in\V_\infty$ and
  \begin{equation}
    \label{eq:147}
    |\psi|\le a\varrho
    \quad \text{for some constant }a>0.
  \end{equation}
  In addition, $\psi$ has compact support and 
  \begin{equation}
    \label{eq:248}
    \int_X\psi\zeta\,\d\mm=
    \int_X \varrho\, \Gbil f\zeta\,\d\mm\quad\forall\zeta\in\V,
    \qquad
    \int_X \psi\,\d\mm=0,
   \end{equation}
   \begin{equation}\label{eq:248bis}
   \frac{1}{2} \cE^*_\varrho(\psi,\psi)=  \frac{1}{2} \cE_{\varrho}(f,f)=
    \langle \psi,f\rangle-\frac 12 \int_X \varrho\Gq f\,\d\mm.
  \end{equation}
  
  We then set $\varrho_s:=\varrho+s\psi$, so that $\partial_s\varrho_s\equiv \psi$, and we observe that 
  \eqref{eq:147} gives 
  \begin{equation}\label{eq:245}
    (1-as)\varrho\le \varrho_s\le (1+as)\varrho.
  \end{equation}
  This, together with \eqref{eq:248}, implies that 
  $\varrho_s \mm\in\Probabilities X$ for $s\in [0,1/a]$; 
  moreover,  \eqref{eq:245} also gives
  $(1-as)\cE_{\varrho}(\varphi)\le \cE_{\varrho_s}(\varphi) \le (1+as)\cE_{\varrho}(\varphi)$ for all $\varphi\in\V$,
  so that by duality we get
  \begin{equation}
    \label{eq:247}
    (1+as)^{-1}\cE^*_{\varrho}(\psi,\psi)\le\cE^*_{\varrho_s}(\psi,\psi)
    \le (1-as)^{-1}\cE^*_{\varrho}(\psi,\psi).
  \end{equation}
  It follows that $\varrho_s$ is Lipschitz in $\ProbabilitiesTwo X$ by 
  Theorem~\ref{thm:dynamicKant} and
  \begin{equation}
    \label{eq:174}
   \lim_{s\down0}\cE^{*}_{\varrho_s}(\psi,\psi) =
    \cE^*_{\varrho}(\psi,\psi)=
    \cE_\varrho(f,f).    
  \end{equation}
  
  We set $\varrho_s^t:=\sfS_t \varrho_s$,
  $w_s^t:=\partial_s \varrho_s^t$, $\varrho^t=\sfS_t\varrho$. Recall that, thanks to Corollary~\ref{cor:rho-derivative},
  $t\mapsto w_s^t$ belong to $W^{1,2}(0,T;\H,\Ddual)\subset C([0,T];\V')$
  and solve the PDE $\partial_t w=\DeltaE (P'(\rho_s^t)w)$ of Theorem~\ref{thm:reg2} with the initial condition $\bar w=\partial_s\varrho_s=\psi$.
 The contraction property of $\sfS$ in $L^1(X,\mm)$ and the integrability of $\psi$ yield
  \begin{equation}\label{eq:2555}
    \|\varrho_s^t-\varrho^t\|_{L^1(X,\mm)}\le 
    \|\varrho_s-\varrho\|_{L^1(X,\mm)}= s\|\psi\|_{L^1(X,\mm)}\quad\forall s\in
    (0,1/a),\,\,\forall t\in [0,T].
  \end{equation}
 Combining Theorem~\ref{thm:reg2},
  the estimate \eqref{eq:149} and \eqref{eq:247}
  we also get
  \begin{equation}
    \label{eq:250}
    \sup_{0\le s\le S}\cE^*_{\varrho_s^t}(w_s^t,w^t_s)\le
    \frac{\rme^{-2\Lambda t}}
    {1-aS}\cE^*_\varrho(\psi,\psi)\quad \forall t\ge0,\,\,\forall S\in (0,1/a).
  \end{equation}
  Theorem~\ref{thm:reg2}(L3) in combination with \eqref{eq:245}  and \eqref{eq:2555} also shows that 
  \begin{equation}
    \label{eq:255}
    \lim_{s\down0}\sup_{0\le t\le T}\|w_s^t-w^t_0\|_{\Vdual1}=0\quad
    \forevery T>0\quad\text{and}\quad
    \lim_{s,t\down0}\|w_s^t -\psi\|_{\Vdual1}=0.
  \end{equation}
 Combining the lower semicontinuity property \eqref{eq:216}  with \eqref{eq:250}, 
  \eqref{eq:255} and recalling   \eqref{eq:248bis},  we get
    \begin{equation}
    \label{eq:251}
   \lim_{s,t\downarrow 0}  \cE^*_{\varrho_s^t}(w_s^t,w^t_s)= \cE^*_{\varrho}(\psi, \psi)=\cE_{\varrho}(f,f); 
  \end{equation}
  we are then in position to apply Lemma~\ref{le:lsc} and infer that \fn
  \begin{equation}
    \label{eq:252}
    \lim_{s,t\down0} \int_X Q(\varrho_s^t)\varrho_s^t
    \tGqs{\varrho_s^t}{w_s^t}\,\d\mm=
    \int_X Q(\varrho)\varrho
    \Gq{f}\,\d\mm.
  \end{equation}
  Moreover, by \eqref{eq:2555} and \eqref{eq:255} we can find a nondecreasing function
  $(0,1)\ni t\mapsto S(t)>0$ 
  with $S(t)\le t^2$, such that 
  \begin{displaymath}
    \label{eq:254}
    \lim_{t\down0} \sup_{0<s<S(t)}t^{-1}\|w_s^t-w^t_0\|_{\Vdual1}=0,
    \qquad
    \lim_{t\down0} \sup_{0<s<S(t)}t^{-1}\|\varrho_s^t-\varrho^t\|_{L^1(X,\mm)}=0,
  \end{displaymath}
  so that 
  \begin{equation}
    \label{eq:257}
    \lim_{t\down0} \media_0^{S(t)} \frac1t\langle w_s^t-\psi,f\rangle
    \,\d s=
    \lim_{t\down0} \frac1t\langle w^t_0-\psi,f\rangle
  \end{equation}
  and 
  \begin{equation}
    \label{eq:256}
    \lim_{t\down0} \media_0^{S(t)}
    \int_X \frac 1t (\varrho_s^t-\varrho)\Gq f\,\d\mm\,\d s=
    \lim_{t\down0} 
    \int_X \frac 1t (\varrho^t-\varrho)\Gq f\,\d\mm,
  \end{equation}
  provided the limits in the right hand sides exist.
  Eventually,  \eqref{eq:248bis}, \eqref{eq:247} and $1-as\ge \frac 12 $ yield
  \begin{displaymath}
   \frac{1}{2} \cE^*_{\varrho_s}(\psi,\psi)\le \frac{1}{2} 
    (1+2as)\cE^*_\varrho(\psi,\psi)=
    \langle \psi,f\rangle -\frac 12 \int_X \varrho\Gq f\,\d\mm+
    as \,\cE^*_\varrho(\psi,\psi)
  \end{displaymath}
  so that the bound $S(t)\le t^2$ yields
  \begin{equation}
    \label{eq:258}
   \frac{1}{2}  \media_0^{S(t)}\cE^*_{\varrho_s}(\psi,\psi)\,\d s\le 
    \langle \psi,f\rangle -\frac 12 \int_X \varrho\Gq f\,\d\mm
    + \frac{1}{2} a t^2 \cE_\varrho(f,f).
  \end{equation}
  Combining Theorem~\ref{thm:dynamicKant} and Lemma~\ref{le:action-estimate} (applied to the 
  rescaled curves in the interval $(0,S(t))$) we get 
  \begin{equation}
    \label{eq:141}
  \frac{1}{2}  \media_0^{S(t)} \cEs{\varrho_{s}^t}(w_s^t,w_s^t)\,\d s+
    K
    \int_0^t \media_0^{S(t)} \int_X Q(\varrho_{s}^r)\varrho_{s}^r 
    \tGqs {\varrho_s^r}{w_s^r}
    \,\d\mm\,\d s\,\d r
    \le \frac{1}{2}  \media_0^{S(t)} \cEs{\varrho_{s}}(\psi,\psi)\,\d s,
  \end{equation}
    so that \eqref{eq:258} and
  the very definition of $\cE^*_{\varrho_s^t}$ yield
 \begin{align*}
    \media_0^{S(t)} &\Big(\langle w^t_{s},f\rangle-
    \frac 12\int_X {\varrho^t_{s}}\Gq f\,\d\mm\Big)\,\d s+
     K
    \int_0^t \media_0^{S(t)} \int_X Q(\varrho_{s}^r)\varrho_{s}^r 
    \tGqs {\varrho_s^r}{w_s^r}
    \,\d\mm\,\d s\,\d r
    \\&  \le 
     \langle \psi,f\rangle - \frac{1}{2} \int_X \varrho\Gq f\,\d\mm
    + \frac{1}{2}  at^2 \cE_\varrho(f,f),
  \end{align*}
  and, dividing by $t>0$,
  \begin{align*}
   \fn  \media_0^{S(t)} &\Big(\frac 1t\langle w^t_{s}-\psi,f\rangle-
    \frac 12\int_X \frac1t (\varrho^t_{s}-\varrho)
    \Gq f\,\d\mm\Big)\,\d s+
     K
    \media_0^t \media_0^{S(t)} \int_X Q(\varrho_{s}^r)\varrho_{s}^r 
    \tGqs {\varrho_s^r}{w_s^r}
    \,\d\mm\,\d s\,\d r
    \\&\le 
    \frac{1}{2}  t a\cE_\varrho(f,f).
  \end{align*}
  Passing to the limit as $t\down0$ and 
  recalling \eqref{eq:252}, 
  \eqref{eq:257} and \eqref{eq:256} we eventually get
  \begin{equation}
    \label{eq:150}
    \lim_{t\down0}\langle \frac{w^t_0-\psi}t,f\rangle
    - \frac{1}{2}
    \lim_{t\down0}
     \int_X \frac{\varrho^t-\varrho}t\Gq f\,\d\mm+
    K
    \int_X Q(\varrho)\varrho 
    \Gq f\,\d\mm\le 0.
  \end{equation}
  Observe now that 
  \begin{align*}
    \frac 1t\langle w^t_0-\psi,f\rangle=&
    \media_0^t \langle\DeltaE (P'(\varrho^r_0)w^r_0),f\rangle\,\d r
    =\media_0^t \langle P'(\varrho^r_0)w^r_0,\DeltaE f\rangle\,\d r
    =\media_0^t \langle w^r_0,P'(\varrho^r_0)\DeltaE f\rangle\,\d r.
  \end{align*}
  We can then pass to the limit since $w^r_0\to \psi$ in
  $\Vdual1$, $P'(\varrho^r_0)\to P'(\varrho)$ 
  in $\V$ (thanks to \eqref{eq:138}) with uniform $L^\infty$ bound
  and $\DeltaE f\in \V_\infty$.
  We get, by the definition of $\psi$, that 
  \begin{equation}
    \label{eq:259}
    \lim_{t\down0}\frac 1t\langle w^t_0-\psi,f\rangle=
    \langle\psi ,P'(\varrho)\DeltaE f\rangle=
   -  \int_X \Big(\varrho P'(\varrho)(\DeltaE f)^2+
    \Gbil {P(\varrho)}f\,\DeltaE f\Big)\,\d\mm.
  \end{equation}
  Similarly, since $\frac1t(\varrho^t-\varrho)\to \DeltaE P(\varrho)$
  in $\V'$, $\Gq f\in \V$ and $P(\varrho)\in \D_\infty$, we obtain
  \begin{equation}
    \label{eq:260}
      \lim_{t\down0}
      \int_X \frac{\varrho^t-\varrho}t\Gq f\,\d\mm=
    \int_X \DeltaE P(\varrho)\Gq f\,\d\mm.
  \end{equation}
  Combining \eqref{eq:150} with
  \eqref{eq:259} and \eqref{eq:260}
  we obtain
  \begin{equation}
    \label{eq:151}
   - \int_X P'(\varrho) (\varrho (\DeltaE f)^2 +\Gbil{P(\varrho)} f)
    \DeltaE f+ \frac12\DeltaE P(\varrho)\Gq f\,\d\mm
    \le - K\int_X P(\varrho)\Gq f\,\d\mm
  \end{equation}
  and finally \eqref{eq:152} is achieved.
    By applying Theorem~\ref{thm:nonlinearBE}
  we then get $\BE KN$.
\end{proof}


\begin{thebibliography}{10}

\bibitem{Ambrosio-Colombo-Dimarino12}
{\sc L.~{Ambrosio}, M.~{Colombo}, and S.~{Di Marino}}, {\em {Sobolev spaces in
  metric measure spaces: reflexivity and lower semicontinuity of slope}}, ArXiv
  eprint: 1212.3779. To appear on Advanced Studies in Pure Mathematics,
  (2012).

\bibitem{AGMR12}
{\sc L.~Ambrosio, N.~Gigli, A.~Mondino, and T.~Rajala}, {\em Riemannian {R}icci
  curvature lower bounds in metric measure spaces with $\sigma$-finite
  measure}, Trans. Amer. Math. Soc., 367 (2015), pp.~4661--4701.

\bibitem{AGS08}
{\sc L.~Ambrosio, N.~Gigli, and G.~Savar{\'e}}, {\em Gradient flows in metric
  spaces and in the space of probability measures}, Lectures in Mathematics ETH
  Z\"urich, Birkh\"auser Verlag, Basel, second~ed., 2008.

\bibitem{AGS11c}
\leavevmode\vrule height 2pt depth -1.6pt width 23pt, {\em Density of
  {L}ipschitz functions and equivalence of weak gradients in metric measure
  spaces}, Revista Matematica Iberoamericana,  (2013), pp.~969--986.

\bibitem{AGS11a}
\leavevmode\vrule height 2pt depth -1.6pt width 23pt, {\em Calculus and heat
  flow in metric measure spaces and applications to spaces with {R}icci bounds
  from below}, Inventiones Mathematicae,  (2014), pp.~289--391.

\bibitem{AGS11b}
\leavevmode\vrule height 2pt depth -1.6pt width 23pt, {\em Metric measure
  spaces with {R}iemannian {R}icci curvature bounded from below}, Duke
  Mathematical Journal,  (2014), pp.~1405--1490.

\bibitem{AGS12}
\leavevmode\vrule height 2pt depth -1.6pt width 23pt, {\em Bakry-\'{E}mery
  curvature-dimension condition and {R}iemannian {R}icci curvature bounds},
  Ann. Probab., 43 (2015), pp.~339--404.

\bibitem{AMS13}
{\sc L.~Ambrosio, A.~Mondino, and G.~Savar\'e}, {\em On the {B}akry-\'{E}mery
  condition, the gradient estimates and the {L}ocal-to-{G}lobal property of
  {RCD}$^*$({K},{N}) metric measure spaces}, Journal of Geometric Analysis,
  DOI: 10.1007/s12220-014-9537-7 (in press).

\bibitem{Ane-et-al00}
{\sc C.~An\'e, S.~Blach\`ere, D.~{Chafa\"\i}, P.~Foug\`eres, I.~Gentil,
  F.~Malreu, C.~Roberto, and G.~Scheffer}, {\em Sur les in\'egalit\'es de
  {S}obolev logarithmiques}, no.~10 in Panoramas et Synth\`eses, Soci\'et\'e
  Math\'ematique de France, 2000.

\bibitem{Bacher-Sturm10}
{\sc K.~Bacher and K.-T. Sturm}, {\em Localization and tensorization properties
  of the curvature-dimension condition for metric measure spaces}, J. Funct.
  Anal., 259 (2010), pp.~28--56.

\bibitem{Bakry92}
{\sc D.~Bakry}, {\em L'hypercontractivit\'e et son utilisation en th\'eorie des
  semigroupes}, in Lectures on probability theory ({S}aint-{F}lour, 1992),
  vol.~1581 of Lecture Notes in Math., Springer, Berlin, 1994, pp.~1--114.

\bibitem{Bakry06}
\leavevmode\vrule height 2pt depth -1.6pt width 23pt, {\em Functional
  inequalities for {M}arkov semigroups}, in Probability measures on groups:
  recent directions and trends, Tata Inst. Fund. Res., Mumbai, 2006,
  pp.~91--147.

\bibitem{Bakry-Emery84}
{\sc D.~Bakry and M.~\'Emery}, {\em Diffusions hypercontractives}, in
  S\'eminaire de probabilit\'es, XIX, 1983/84, vol.~1123, Springer, Berlin,
  1985, pp.~177--206.

\bibitem{Bakry-Gentil-Ledoux14}
{\sc D.~Bakry, I.~Gentil, and M.~Ledoux}, {\em Analysis and Geometry of Markov
  Diffusion Operators}, vol.~348 of Grundlehren der mathematischen
  Wissenschaften, Springer, 2014.

\bibitem{Bakry-Ledoux06}
{\sc D.~Bakry and M.~Ledoux}, {\em {A logarithmic Sobolev form of the Li-Yau
  parabolic inequality}}, Rev. Mat. Iberoamericana, 22 (2006), p.~683.
  
\bibitem{Bolley-Gentil-Guillin-Kuwada14}
{\sc F.~Bolley, I.~Gentil, A.~Guillin, and K.~Kuwada}, {\em Equivalence between dimensional
contractions in Wasserstein distance and the curvature-dimension condition.}
ArXiv:1510.07793. 

\bibitem{Bouleau-Hirsch91}
{\sc N.~Bouleau and F.~Hirsch}, {\em {D}irichlet forms and analysis on {W}iener
  sapces}, vol.~14 of De Gruyter studies in Mathematics, De Gruyter, 1991.

\bibitem{Brezis70}
{\sc H.~Br{\'e}zis}, {\em On some degenerate nonlinear parabolic equations}, in
  Nonlinear Functional Analysis (Proc. Sympos. Pure Math., Vol. XVIII, Part 1,
  Chicago, Ill., 1968), Amer. Math. Soc., Providence, R.I., 1970, pp.~28--38.

\bibitem{Brezis71}
\leavevmode\vrule height 2pt depth -1.6pt width 23pt, {\em Monotonicity methods
  in {H}ilbert spaces and some applications to nonlinear partial differential
  equations}, in Contribution to {N}onlinear {F}unctional {A}nalysis, Proc.\
  {S}ympos.\ {M}ath.\ {R}es.\ {C}enter, {U}niv.\ {W}isconsin, {M}adison, 1971,
  Academic Press, New York, 1971, pp.~101--156.

\bibitem{Brezis71b}
\leavevmode\vrule height 2pt depth -1.6pt width 23pt, {\em Propri\'et\'es
  r\'egularisantes de certains semi-groupes non lin\'eaires}, Israel J. Math.,
  9 (1971), pp.~513--534.

\bibitem{Brezis73}
\leavevmode\vrule height 2pt depth -1.6pt width 23pt, {\em Op\'erateurs
  maximaux monotones et semi-groupes de contractions dans les espaces de
  {H}ilbert}, North-Holland Publishing Co., Amsterdam, 1973.
\newblock North-Holland Mathematics Studies, No. 5. Notas de Matem\'atica (50).

\bibitem{Brezis-Pazy72}
{\sc H.~Br{\'e}zis and A.~Pazy}, {\em Convergence and approximation
of semigroups of nonlinear operators in {B}anach spaces}, J. Funct. Anal., 9 (1972), pp.~63--74. 


\bibitem{Burago-Burago-Ivanov01}
{\sc D.~Burago, Y.~Burago, and S.~Ivanov}, {\em A course in metric geometry},
  vol.~33 of Graduate Studies in Mathematics, American Mathematical Society,
  Providence, RI, 2001.

\bibitem{Carrillo-Lisini-Savare-Slepcev09}
{\sc J.~A. Carrillo, S.~Lisini, G.~Savar{\'e}, and D.~Slep{c}ev}, {\em
  Nonlinear mobility continuity equations and generalized displacement
  convexity}, J. Funct. Anal., 258 (2010), pp.~1273--1309.

\bibitem{Chen-Fukushima12}
{\sc Z.-Q. Chen and M.~Fukushima}, {\em Symmetric {M}arkov processes, time
  change, and boundary theory}, vol.~35 of London Mathematical Society
  Monographs Series, Princeton University Press, Princeton, NJ, 2012.

\bibitem{Cordero-McCann-Schmuckenschlager01}
{\sc D.~Cordero-Erausquin, R.~J. McCann, and M.~Schmuckenschl{\"a}ger}, {\em A
  {R}iemannian interpolation inequality \`a la {B}orell, {B}rascamp and
  {L}ieb}, Invent. Math., 146 (2001), pp.~219--257.

\bibitem{Daneri-Savare08}
{\sc S.~Daneri and G.~Savar\'e}, {\em {E}ulerian calculus for the displacement
  convexity in the {W}asserstein distance}, SIAM J. Math. Anal., 40 (2008),
  pp.~1104--1122.

\bibitem{Dolbeault-Nazaret-Savare08b}
{\sc J.~Dolbeault, B.~Nazaret, and G.~Savar\'e}, {\em A new class of transport
  distances between measures}, Calc. Var. Partial Differential Equations, 34
  (2009), pp.~193--231.

\bibitem{Erbar-Kuwada-Sturm}
{\sc M.~Erbar, K.~Kuwada, and K.-T. Sturm}, {\em On the equivalence of the
  entropic curvature-dimension condition and {B}ochner's inequality on metric
  measure spaces}, Invent. Math., 201 (2015), pp.~993--1071.

\bibitem{Gigli12}
{\sc N.~Gigli}, {\em On the differential structure of metric measure spaces and
  applications},  To appear on Memoirs of
  the AMS,  (2012).
\href {http://arxiv.org/abs/1205.6622} {\path{arXiv:1205.6622}}.

\bibitem{GigliGAFA}
\leavevmode\vrule height 2pt depth -1.6pt width 23pt, {\em Optimal maps in non branching spaces with {R}icci curvature
  bounded from below}, Geom. Funct. Anal., 22 (2012), pp.~990--999.

\bibitem{Gigli-BangXian}
{\sc N.~Gigli and B.-X. Han}, {\em The continuity equation on metric measure
  spaces}, Calc. Var. Partial Differential Equations, 53 (2015), pp.~149--177.

\bibitem{GMS13}
{\sc N.~Gigli, A.~Mondino, and G.~Savar\'e}, {\em Convergence of pointed
  non-compact metric measure spaces and stability of {R}icci curvature bounds and
  heat flows}, Preprint arXiv:1311.4907, To appear in Proceedings of London
  Math. Soc. (in press).

\bibitem{Koskela-Shanmugalingam-Zhou12}
{\sc P.~Koskela, N.~Shanmugalingam, and Y.~Zhou}, {\em Geometry and analysis of
  {D}irichlet forms ({II})}, J. Funct. Anal., 267 (2014), pp.~2437--2477.

\bibitem{Koskela_Zhou}
{\sc P.~Koskela and Y.~Zhou}, {\em Geometry and analysis of {D}irichlet forms},
  Advances in Math., 231 (2012), pp.~2755--2801.

\bibitem{Ledoux01}
{\sc M.~Ledoux}, {\em The concentration of measure phenomenon}, vol.~89 of
  Mathematical Surveys and Monographs, American Mathematical Society,
  Providence, RI, 2001.

\bibitem{Ledoux04}
\leavevmode\vrule height 2pt depth -1.6pt width 23pt, {\em Spectral gap,
  logarithmic {S}obolev constant, and geometric bounds}, in Surveys in
  differential geometry. {V}ol. {IX}, Surv. Differ. Geom., IX, Int. Press,
  Somerville, MA, 2004, pp.~219--240.

\bibitem{Ledoux11}
\leavevmode\vrule height 2pt depth -1.6pt width 23pt, {\em From concentration
  to isoperimetry: semigroup proofs}, in Concentration, functional inequalities
  and isoperimetry, vol.~545 of Contemp. Math., Amer. Math. Soc., Providence,
  RI, 2011, pp.~155--166.

\bibitem{Liero-Mielke13}
{\sc M.~Liero and A.~Mielke}, {\em Gradient structures and geodesic convexity
  for reaction-diffusion systems}, Philos. Trans. R. Soc. Lond. Ser. A Math.
  Phys. Eng. Sci., 371 (2013), pp.~20120346, 28.

\bibitem{Lions-Magenes72}
{\sc J.-L. Lions and E.~Magenes}, {\em Non-homogeneous boundary value problems
  and applications. {V}ol. {I}-{I}{I}}, Springer-Verlag, New York, 1972.
\newblock Translated from the French by P. Kenneth, Die Grundlehren der
  mathematischen Wissenschaften, Band 182.

\bibitem{Lisini07}
{\sc S.~Lisini}, {\em Characterization of absolutely continuous curves in
  {W}asserstein spaces}, Calc. Var. Partial Differential Equations, 28 (2007),
  pp.~85--120.

\bibitem{Lott-Villani09}
{\sc J.~Lott and C.~Villani}, {\em Ricci curvature for metric-measure spaces
  via optimal transport}, Ann. of Math. (2), 169 (2009), pp.~903--991.

\bibitem{Otto01}
{\sc F.~Otto}, {\em The geometry of dissipative evolution equations: the porous
  medium equation}, Comm. Partial Differential Equations, 26 (2001),
  pp.~101--174.

\bibitem{Otto-Villani00}
{\sc F.~Otto and C.~Villani}, {\em Generalization of an inequality by
  {T}alagrand and links with the logarithmic {S}obolev inequality}, J. Funct.
  Anal., 173 (2000), pp.~361--400.

\bibitem{Otto-Westdickenberg05}
{\sc F.~Otto and M.~Westdickenberg}, {\em Eulerian calculus for the contraction
  in the {W}asserstein distance}, SIAM J. Math. Anal., 37 (2005),
  pp.~1227--1255 (electronic).


\bibitem{Rajala12}
{\sc T.~Rajala}, {\em Interpolated
  measures with bounded density in metric spaces satisfying the
  curvature-dimension conditions of {S}turm}, J. Funct. Anal., 263 (2012),
  pp.~896--924.
  
 \bibitem{RajalaDCDS} \leavevmode\vrule height 2pt depth -1.6pt width 23pt,
 {\em Improved geodesics for the reduced curvature-dimension condition in branching metric spaces},
 Discrete Contin. Dyn. Syst., 33 (2013), no. 7, pp.~3043--3056.

\bibitem{Rajala-Sturm12}
{\sc T.~{Rajala} and K.-T. {Sturm}}, {\em {Non-branching geodesics and optimal
  maps in strong CD(K,$\infty$)-spaces}}, Calc. Var. Partial Differential Equations,
  50 (2014), pp.~831-846.

\bibitem{Savare12}
{\sc G.~Savar\'e}, {\em Self-improvement of the {B}akry-\'{E}mery condition and
  {W}asserstein contraction of the heat flow in {$RCD(K,\infty)$} metric
  measure spaces}, Discrete Contin. Dyn. Syst., 34 (2014), pp.~1641--1661.

\bibitem{Sturm_intrinsic}
{\sc K.-T. Sturm}, {\em Is a diffusion process determined by its intrinsic metric?}, 
Chaos Solitons Fractals, 8 (1997), pp.~1855--1860.

\bibitem{Sturm06I}
\leavevmode\vrule height 2pt depth -1.6pt width 23pt, {\em On the geometry of metric measure spaces. {I}}, Acta
  Math., 196 (2006), pp.~65--131.

\bibitem{Sturm06II}
\leavevmode\vrule height 2pt depth -1.6pt width 23pt, {\em On the geometry of
  metric measure spaces. {II}}, Acta Math., 196 (2006), pp.~133--177.

\bibitem{Sturm-VonRenesse05}
{\sc K.-T. Sturm and M.-K. von Renesse}, {\em Transport inequalities, gradient
  estimates, entropy, and {R}icci curvature}, Comm. Pure Appl. Math., 58
  (2005), pp.~923--940.

\bibitem{Triebel78}
{\sc H.~Triebel}, {\em Interpolation theory, function spaces, differential
  operators}, vol.~18 of North-Holland Mathematical Library, North-Holland
  Publishing Co., Amsterdam-New York, 1978.

\bibitem{Villani09}
{\sc C.~Villani}, {\em Optimal transport. Old and new}, vol.~338 of Grundlehren
  der Mathematischen Wissenschaften, Springer-Verlag, Berlin, 2009.

\bibitem{Visintin84}
{\sc A.~Visintin}, {\em Strong convergence results related to strict
  convexity}, Comm. Partial Differential Equations, 9 (1984), pp.~439--466.

\end{thebibliography}
\def\cprime{$'$}

\end{document}